\documentclass[12pt,reqno]{amsart}
\usepackage[letterpaper,margin=1in]{geometry} 
\usepackage{amsmath,amssymb}
\usepackage{enumitem}
\usepackage{setspace}
\usepackage[dvipsnames]{xcolor}
\usepackage{soul}
\usepackage{amsthm}
\usepackage{mathrsfs}
\DeclareMathOperator{\Sym}{Sym}

\usepackage{enumitem}
\newcommand{\Ad}{\operatorname{Ad}}
\newcommand{\Span}{\operatorname{Span}}

\usepackage{booktabs}
\usepackage{makecell}
\usepackage{tabularx}

\newtheorem{theorem}{Theorem}[section] 
\newtheorem{lemma}[theorem]{Lemma}
\newtheorem{proposition}[theorem]{Proposition}
\newtheorem{definition}[theorem]{Definition}
\newtheorem{corollary}[theorem]{Corollary}
\newtheorem{remark}[theorem]{Remark}
\newtheorem{example}[theorem]{Example}
\usepackage{tikz-cd}

\newcommand{\Deck}{\mathrm{Deck}}

\newcommand{\Isom}{\mathrm{Isom}}

\newcommand{\Flag}{\mathrm{Flag}}
\newcommand{\Cham}{\mathrm{Cham}}

\usepackage{longtable}

\usepackage{amsmath,amssymb}
\usepackage{graphicx}
\usepackage{upgreek}
\usepackage{mathrsfs}
\makeatletter
\newcommand{\listeqno}{
  \refstepcounter{equation}
  \hfill\mbox{\tagform@{\theequation}}
}
\makeatother

\numberwithin{equation}{section}

\providecommand{\K}{\mathtt{K}}
\providecommand{\N}{\mathtt{N}}
\providecommand{\B}{\mathtt{B}}
\providecommand{\f}{\mathfrak{f}}
\providecommand{\aaa}{\mathfrak{a}}
\providecommand{\h}{\mathfrak{h}}
\newcommand{\E}{\mathtt{E}}
\newcommand{\dist}{\operatorname{dist}}
\newcommand{\R}{\mathbb{R}}

\DeclareMathOperator{\Fix}{Fix}

\newcommand{\id}{\mathrm{id}}
\newcommand{\Spin}{\mathrm{Spin}}

\newcommand{\SO}{\mathrm{SO}}
\newcommand{\SU}{\mathrm{SU}}

\newcommand{\HH}{\mathbb{H}}
\newcommand{\Hull}[1]{\operatorname{Hull}\bigl\{#1\bigr\}}

\title{
Homogeneous $\mathbb Z/2$-Harmonic Forms and Spinors on $\R^4$ from Regular 4-Polytopes  }

\author[C. H.~Taubes]{Clifford Taubes$^\dag$}
\author[Y. Wu]{Yingying Wu$^\lozenge$}

\address{Department of Mathematics, Harvard University, Cambridge, MA}
\email{chtaubes@math.harvard.edu}

\address{Department of Mathematics, University of Houston, Houston, TX}
\email{ywu68@uh.edu}

\begin{document}

\begin{abstract}
We describe novel local singularity models for $\mathbb Z/2$ harmonic 1-forms, self-dual 2-forms and spinors in dimension 4.  These models are homogeneous versions on $\R^4$ whose singular sets are cones on the 1-skeletal of certain regular 4-dimensional polytopes. 
\end{abstract}

\maketitle

\setcounter{tocdepth}{1}
\tableofcontents

\part{Analytic Theory of $\mathbb{Z}/2$--Harmonic Forms on $S^3\setminus\Gamma$}

\medskip

\section{Introduction}

Homogeneous, $\mathbb{Z}/2$ harmonic differential forms and spinors on $\mathbb{R}^4$ arise as rescaling limits of divergent sequences to certain 4-dimensional generalizations of the Seiberg--Witten and (anti)-self-dual Yang--Mills equations.  (See e.g.\ \cite{T1,T2,T3}.)  What follows momentarily is a working definition of these objects of interest.  First some preliminaries to set the stage: The set of length $1$ vectors in $\mathbb{R}^4$ is the $3$-sphere, denoted here by $S^3$.  The radial projection map from $\mathbb{R}^4 \setminus \{0\}$ to $S^3$ that sends any given non-zero vector $x$ to the unit vector along the ray from the origin to $x$ is denoted by $\pi$.  The multiplicative group $(0,\infty)$ acts on $\mathbb{R}^4$ as the $1$-parameter family of
coordinate rescaling diffeomorphisms using the rule whereby any given $\lambda \in (0,\infty)$ sends any given vector $x$ to $\lambda x$. Next some notation: The upcoming definition refers to a ``Euclidean line bundle''.  This is a real line bundle with fiber metric.  Given that metric, there is a canonical definition of a (covariant) derivative for sections of the line bundle (and thus for differential form/spinor valued sections).  This is because the line bundle is locally isometric to the product $\mathbb{R}$-bundle with the isometry being ambiguous only up to multiplication by $\pm 1$ which commutes with differentiation.

\medskip

\begin{definition}
\emph{A homogeneous, $\mathbb{Z}/2$ harmonic differential form or spinor on $\mathbb{R}^4$ is characterized by a data set $(Z,\mathcal{I},v)$ consisting of the following elements:}
\begin{itemize}
  \item \emph{What is denoted by $Z$ signifies the inverse image via $\pi$ of a closed, non-empty set in $S^3$ with finite $1$-dimensional Hausdorff measure in $S^3$.}

  \item \emph{What is denoted by $\mathcal{I} \to \mathbb{R}^4 - Z$ signifies the pull-back via $\pi$ of a real (Euclidean) line bundle defined on the complement of $Z \cap S^3$ which has no extension to any ball in $S^3$ that contains any point of $Z \cap S^3$.}

  \item \emph{What is denoted by $v$ signifies an $\mathcal{I}$-valued harmonic, differential form or spinor defined over $\mathbb{R}^4 - Z$ that pulls-back as a non-zero multiple of itself via the action of $(0,\infty)$ on $\mathbb{R}^4$.  In addition, the conditions below must be met:}
\begin{enumerate}[label=\alph*)]
  \item \emph{The norm of $v$ extends over $Z$ to define a H\"older continuous function on $\mathbb{R}^4$ whose zero locus contains $Z$.}

  \item \emph{The square of the norm of the first derivative of any components of $v$ has finite integral over the complement of $Z$ in any ball in $\mathbb{R}^4$.}
\end{enumerate}
\end{itemize}
\end{definition}

Of particular interest in this paper are the cases where $v$ is a 1-form, self-dual 2-form, or self-dual spinor. (The self-dual spinor bundle on $\mathbb{R}^4$ is the product, rank 2, complex vector bundle $\mathbb{R}^4\times\mathbb{C}^2$.)

To put the upcoming theorems in a historical context, note first that examples of homogeneous $\mathbb{Z}/2$ harmonic differential forms and spinors on $\mathbb{R}^3$ were constructed previously, by the authors (see \cite{taubes2020examples} and also \cite{taubes2024topological}), by \cite{chen2024existence} and by \cite{FS}. As noted in \cite{T4}, these objects on $\mathbb{R}^3$ can be pulled up to $\mathbb{R}^4$ via the projection map to any given hyperplane and then those pull-backs give examples on $\mathbb{R}^4$. In addition, there are examples whereby $Z$ is the zero locus of a homogeneous, holomorphic function on $\mathbb{C}^2$ (given some isometry between $\mathbb{C}^2$ and $\mathbb{R}^4$). See for example \cite{T4} and \cite{HMT1}. In the former case, $Z$ is the union of half-planes in $\mathbb{R}^4$ through the origin that share a common axis; and in the latter case, $Z$ is a union of complex lines through the origin.

\medskip

With the preceding as background, what follows is our main theorem.

\medskip

\begin{theorem}\label{thm:main}
There exist homogeneous, $\mathbb{Z}/2$ harmonic 1-forms, self-dual 2-forms and self-dual spinors on $\mathbb{R}^4$ with the set $Z$ being the cone on the union of the edges and vertices of the following 4-dimensional analogs of the Platonic solids: The 4-dimensional tetrahedron (the 5-cell), the 4-dimensional cube (the 8-cell), and the polytopes known as the 24-cell, the 120-cell and the 600-cell.
\end{theorem}

\medskip

See Part II for a detailed description of these polytopes.

\medskip

In the cases of our theorem, the set $Z$ is the union of the cones on geodesic arcs as is the case for the pull-backs from $\mathbb{R}^3$ and for the holomorphic examples. However, in no cases is $Z$ a union of half-planes and/or full planes. (The upcoming Section \ref{ch:pre}-\ref{sec:600cell} gives a detailed description of the polytopes that are mentioned in the theorem.) By way of a look ahead to the proof: Our constructions exploit the symmetry groups of these 4-dimensional polytopes in the manner of the constructions in \cite{taubes2020examples} which exploit the symmetry groups of the Platonic solids. Even so, the group theory here is different (and pretty) and the analysis has some novelties. (The constructions for $\mathbb{R}^3$ in \cite{chen2024existence} and \cite{FS} also exploit symmetry groups, but in a more sophisticated way.) An almost contemporaneous preprint \cite{franceschini2026minimal} develops related microlocal methods in the study of two-valued harmonic functions and branched minimal submanifolds with stratified branching sets relevant to some of the local regularity and asymptotic questions considered here.
Other recent papers with related results on $\mathbb{Z}/2$ harmonic 1-forms and spinors are the paper by Bera and Walpuski \cite{bera2025dirac} and the paper by Jiahuang Chen \cite{chen2025perturbation}.
\section{Homogeneous, Harmonic on $\mathbf{R^4}$ from Eigensections on $\mathbf{S^3}$}

The basic observation (which is truly ancient) is that homogeneous, harmonic differential forms and spinors on $\mathbb{R}^4$ come from eigensections of corresponding differential operators acting on sections of some vector bundle over $S^3$. To elaborate: The assertion that $\nu$ is homogeneous says in effect that there exists $\kappa\in\mathbb{R}$ such that the pull-back of $\nu$ under a rescaling diffeomorphism $x\mapsto\lambda x$ is equal to $\lambda^\kappa \nu$ for any given $\lambda\in (0,\infty)$. Note in this regard that $Z$ is invariant under these coordinate rescalings so these diffeomorphisms map $\mathbb{R}^4-Z$ to itself. Meanwhile, the fact that $\mathcal{I}$ is pulled up from $S^3-(Z\cap S^3)$ implies that $\mathcal{I}$ at any given point in $\mathbb{R}^4-Z$ is canonically isomorphic to the pull-back of $\mathcal{I}$ via any rescaling diffeomorphism. This is why it makes sense to say that the $\lambda$-pull-back of $\nu$ at a given point in $\mathbb{R}^4-Z$ is equal to $\lambda^\kappa \nu$ at that point. In particular, homogeneity of $\nu$ implies that $\nu$ is completely determined by its restriction to $S^3-(Z\cap S^3)$. This homogeneity also implies the following: Let $r$ denote the radial coordinate $x\mapsto |x|$. The Lie derivative of $\nu$ along the vector field $r\partial_r$ (which generates the 1-parameter family of rescaling diffeomorphisms) is equal to $\kappa\nu$. (In the case when $\nu$ is a spinor, this Lie derivative is defined using the product structure on the appropriate vector bundle that is defined by parallel transport along rays from the origin using the Euclidean metric’s Levi–Civita connection.) As explained momentarily, this property of the Lie derivative has fundamental implications when $\nu$ is a harmonic differential form or spinor.

A simple case to consider first is the case when $\nu$ is an $\mathcal{I}$-valued, homogeneous, harmonic function. The assertion that $\nu$ is harmonic (thus $\Delta \nu=0$ with $\Delta$ denoting the sum of the second derivatives in the four coordinate directions) is equivalent to the assertion that $\nu$’s restriction to $S^3-(Z\cap S^3)$ is an $\mathcal{I}$-valued eigensection of the standard Laplacian on $S^3-(Z\cap S^3)$ with eigenvalue $\kappa(\kappa+2)$, thus
\begin{align}
- \Delta^{\perp} \nu = \kappa(\kappa+2)\nu
\end{align}
where $\Delta^{\perp}$ denotes here and subsequently the standard Laplacian on the 3-sphere. (The Laplace eigensections for the $S^2-(Z\cap S^2)$ analog are studied in \cite{taubes2020examples, taubes2024topological}.)

As explained in what follows, any given $\mathcal{I}$-valued, homogeneous harmonic function on $\mathbb{R}^4 - Z$ with the norm of its exterior derivative extending across $Z$ as a H\"older continuous function that vanishes on $Z$ can be used to construct homogeneous, $\mathbb{Z}/2$ harmonic 1-forms and spinors for the same $Z$ (but not directly any homogeneous, $\mathbb{Z}/2$ harmonic, self-dual 2-forms). To elaborate, let $\phi$ denote such a function. Then $d\phi$ is a homogeneous, $\mathbb{Z}/2$ harmonic 1-form for the given set $Z$. With regards to self-dual spinors: Let $\eta$ denote a constant anti-self-dual spinor and let $\mathfrak{D}$ denote the standard Dirac operator on $\mathbb{R}^4$. Then $\mathfrak{D}(\phi \eta)$ is a homogeneous, $\mathbb{Z}/2$ harmonic, self-dual spinor for that same locus $Z$.

These last two examples are instances of a general construction: Let $\mathbb{V}_0$ denote a vector space (real or complex); and let $\{\gamma_1,\dots,\gamma_4\}$ denote a set of anti-symmetric (or skew-adjoint if $\mathbb{V}_0$ is complex) endomorphisms of $\mathbb{V}_0$ that obey the Clifford algebra relations
\begin{itemize}[topsep=6pt]
\item $\gamma_i \gamma_j + \gamma_j \gamma_i = 0\;\text{for each distinct index pair},$
\item $\gamma_i^{\,2} = -\mathbb{I}\;\text{for each }i$.
\listeqno
\end{itemize}
(Here and below, $\mathbb{I}$ denotes the identity endomorphism of $\mathbb{V}_0$.)  Let $\{\partial_i\}_{i=1,2,3,4}$ denote the directional derivatives along the coordinate axes in $\mathbb{R}^4$ acting component-wise on sections of the product vector bundle $\mathbb{R}^4 \times \mathbb{V}_0$. Set
\begin{align}
\mathfrak{D} = \gamma_1 \partial_1 + \cdots + \gamma_4 \partial_4.
\end{align}
This is a first-order differential operator acting on sections of $\mathbb{R}^4\times \mathbb{V}_0$ obeying \(\mathfrak{D}^{\,2} = -\Delta\, \mathbb{I}\) with $\Delta$ denoting the standard Laplacian acting component-wise also. Take $\phi$ as in the previous paragraph and let $\eta$ denote any given element in $\mathbb{V}_0$. Then
\begin{align}\label{eq:2.3}
v = \mathfrak{D}(\phi \eta)
\end{align}
is annihilated by $\mathfrak{D}$; and as such, it defines a $\mathbb{Z}/2$ homogeneous, harmonic (with respect to $\mathfrak{D}$) section of $\mathbb{V}$ (it is a section of the vector bundle $(\mathbb{R}^4\times \mathbb{V}_0)\otimes\mathcal{I}$ on $\mathbb{R}^4 - Z$ that obeys the conditions set forth by the third bullet of Definition 1.1).

By way of a parenthetical remark: Any $\mathfrak{D}$-harmonic section of $(\mathbb{R}^4\times \mathbb{V}_0)\otimes\mathcal{I}$ can be written as a linear combination of sections of the form depicted in \eqref{eq:2.3} if one is not concerned with the behavior along $Z$. To elaborate, let $(x_1,\dots,x_4)$ denote the Euclidean coordinates on $\mathbb{R}^4$, and let
\begin{align}
x\!\cdot\!\gamma \;=\; x_1 \gamma_1 + \cdots + x_4 \gamma_4.
\end{align}
The crucial observation is that if $\mathfrak{D}\psi =0$, then $(x\!\cdot\!\gamma)\,\psi$ is annihilated by the Laplacian on $\mathbb{R}^4$ and
\begin{align}
\psi = -4\mathfrak{D}\!\big((x\!\cdot\!\gamma)\,\psi\big).
\end{align}
Even so, it is not clear \emph{a priori} that the norms of the derivatives of the components of $(x\!\cdot\!\gamma)\,\psi$ must extend across $Z$ as H\"older continuous functions that vanish on $Z$ if the norm of $\psi$ does.

The construction of homogeneous, $\mathbb{Z}/2$ harmonic self-dual 2-forms requires a different approach. (Yet, examples of homogeneous, harmonic, $\mathcal{I}$-valued solutions to the equation \(*dw + du = 0\) for a pair $\nu = (w,u)$ of $\mathcal{I}$-valued self-dual 2-form and $\mathcal{I}$-valued function can be obtained via an appropriate version of \eqref{eq:2.3}. However, because the function component of $\nu$ won’t be identically zero, the self-dual 2-form part isn’t harmonic.)

To say more about the upcoming construction for $\mathbb{Z}/2$ harmonic, self-dual 2-forms, let $\nu$ denote for the moment a harmonic, $\mathcal{I}$-valued self-dual 2-form on $\mathbb{R}^4 - Z$ for some scaling-invariant, codimension-2 set $Z$. Let $\{\hat e_1,\hat e_2,\hat e_3\}$ denote an oriented, orthonormal basis for $T^*S^3$. The self-dual 2-form $\nu$ can be written schematically using this basis as
\begin{align}\label{eq:2.4}
\nu
= \nu_a\Big(dr\wedge r\,\pi^*\hat e_a 
\;+\; \frac12\, r^2\,\varepsilon^{abc}\,\pi^*(\hat e^b\wedge\hat e^c)\Big)
\end{align}
with the notation as follows: The repeated indices are implicitly summed over $\{1,2,3\}$ with it understood that $\{\nu_1,\nu_2,\nu_3\}$ denote $\mathcal{I}$-valued functions on $\mathbb{R}^4 - Z$. Meanwhile, $\varepsilon^{abc}$ denotes the totally anti-symmetric 3-tensor with $\varepsilon^{123}=1$.

The self-dual 2-form $\nu$ is homogeneous and harmonic if and only if the three conditions below are met by $\{\nu_1,\nu_2,\nu_3\}$. This list uses $\nu^\perp$ to denote the 1-form 
\begin{align}
\nu^\perp = \nu_a\,\pi^*\hat e^a.
\end{align}
The symbol \(*\) below denotes the metric Hodge star on $S^3$:
\begin{itemize}[topsep=6pt]
\item $r\partial_r \nu^\perp     = (\kappa-2)\nu^\perp$,
\item $*d^{\perp}\nu^\perp        = \kappa\,\nu^\perp$,
\item $d^{\perp} * \nu^\perp      = 0$.
\listeqno\label{eq:2.5}
\end{itemize}
These bullets say in effect that $\nu^\perp$ is the product of $r^{\,\kappa-2}$ times the pull-back from $S^3 - (S^3 \cap Z)$ via the radial projection \(\pi\) of an eigensection of the operator \(*d\) acting on coclosed, $\mathcal{I}$-valued 1-forms on $S^3-(S^3\cap Z)$.

To put \eqref{eq:2.5} in some sort of perspective and to motivate some of the focus in the subsequent sections of this paper, reintroduce the vector
space \(\mathbb{V}_0\) and then the operator
\begin{align}
\mathfrak{D} \;=\; \gamma_1 \partial_1 + \gamma_2 \partial_2
               + \gamma_3 \partial_3 + \gamma_4 \partial_4
\end{align}
which acts on sections over \(\mathbb{R}^4 - Z\) of the tensor product of the line bundle \(\mathcal{I}\) with the product \(\mathbb{V}_0\)-bundle. Let
\begin{align}
\gamma_r \;=\; r^{-1}\bigl(x_1 \gamma_1 + \cdots + x_4 \gamma_4\bigr)
\end{align}
which is an endomorphism of that bundle with square equal to $-1$ times the identity endomorphism. The equation \(\mathfrak{D}\nu = 0\) can be written alternately in the schematic form
\begin{align}\label{eq:2.6}
r\,\partial_r \nu \;=\; \mathcal{D}\nu
\end{align}
with \(\mathcal{D}\) being a homogeneous, first order, elliptic differential operator that differentiates only in directions tangential to the constant \(r\)–spheres in \(\mathbb{R}^4\).  The assumption that \(\nu\) is homogeneous implies that \(r\,\partial_r\nu\) is equal to \(\lambda\nu\) with \(\lambda\) constant.  Hence, \(\nu\) is both harmonic and homogeneous if and only if its restriction to the standard \(r = 1\) sphere is an eigensection of the operator \(\mathcal{D}\) over \(S^3 - (S^3 \cap Z)\).  Conversely, if \(\nu\) is an eigensection of \(\mathcal{D}\) on \(S^3 - (S^3 \cap Z)\) with eigenvalue \(\lambda\), and if the norm of \(\nu\) vanishes on \(Z\) and is H\"older continuous across \(S^3 \cap Z\), then \(r^{\,\lambda}\,\pi^*\nu\) will be a homogeneous, harmonic (with respect to \(\mathfrak{D}\)) section of \((\mathbb{R}^4 \times \mathbb{V}_0)\otimes\mathcal{I}\) over \(\mathbb{R}^4 - Z\) that is described by the third bullet of Definition~1.1.

By way of an example: Take \(\mathfrak{D}\) to be the Dirac operator acting on \(\mathbb{R}^4\)'s spinor bundle, this being the \(\mathbb{V}_0 = \mathbb{C}^2 \oplus \mathbb{C}^2\) case (the left-hand \(\mathbb{C}^2\) summand gives the self-dual spinor bundle and the right-hand summand gives the anti-self-dual spinor bundle).  A self-dual, \(\mathcal{I}\)-valued spinor is homogeneous and harmonic if
\begin{align}
\nu \;=\; r^{\,\lambda}\,\pi^*\psi
\end{align}
with \(\psi\) an eigensection of the 3-sphere's Dirac operator on \(\mathcal{I}\)-valued spinors over \(S^3 - (S^3 \cap Z)\) with eigenvalue \(\lambda + \tfrac{3}{2}\). (The relevant version of \(\mathfrak{D}\) differs from that Dirac operator by a constant multiple of the identity endomorphism.)

\section{Some Topology Regarding the Set $\mathbf{Z}$ and the Line Bundle $\boldsymbol{\mathcal{I}}$}\label{sec:3}

For the purposes of what is to come, the set \(Z\) will be the cone in \(\mathbb{R}^4\) on certain sorts of embedded graphs in \(S^3\).  To elaborate: Let \(\Gamma \subset S^3\) denote a non-empty, connected graph with each vertex having even valency (which is to say that an even number of edges end at each vertex).  The implicit assumption is that the edges of \(\Gamma\) are closed subsets of embedded arcs in \(S^3\), and that the tangent vectors to the edges that are incident to any given vertex of \(\Gamma\) are distinct.  (There will be an extra assumption later to the effect that the edges of \(\Gamma\) are geodesic arcs.)  The upcoming Lemma~3.1 makes a formal statement to the effect that an appropriate line bundle \(\mathcal{I}\) on \(S^3 - \Gamma\) exists if and only if each vertex of \(\Gamma\) has even valency. The lemma refers to a \emph{linking circle} of \(\Gamma\). This is an embedded circle in \(S^3 - \Gamma\) bounding a disk in \(S^3\) that intersects \(\Gamma\) in a single point, that being a transversal intersection with the interior of an edge of \(\Gamma\).

\begin{lemma}
\emph{Let \(\Gamma\) denote a graph in \(S^3\).  There is a class in \(H^1(S^3 - \Gamma;\,\mathbb{Z}/2)\) with non-zero restriction to any linking circle of \(\Gamma\) if and only if each vertex of \(\Gamma\) has even valency. This is to say that there is a real line bundle over \(S^3 - \Gamma\), which restricts as the M\"obius line bundle to each linking circle of \(\Gamma\).} 
\end{lemma}

This lemma is proved momentarily.  The line bundle \(\mathcal{I}\) for the constructions that follow will always be the line bundle from Lemma~3.1.  By way of relevant instances for Lemma~3.1: The image in \(S^3\) via the radial projection map \(\pi\) of the union of the edges and vertices of any of the polytopes from Theorem~1.2 is a graph in \(S^3\) whose vertices all have even valency, this being four for the 5-cell, 8-cell and 120-cell, and eight and twelve in the respective cases of the 24-cell and 600-cell.

\medskip

\noindent\textit{Proof of Lemma 3.1.}\;
Since real line bundles are classified by their first Stiefel–Whitney class, and since any \(\mathbb{Z}/2\mathbb{Z}\) cohomology class is the first Stiefel–Whitney class of a real line bundle, it is sufficient to prove the lemma's assertion regarding \(H^1(S^3 - \Gamma;\,\mathbb{Z}/2)\).  To this end, let \(g\) denote the rank of \(H_1(S^3 - \Gamma;\,\mathbb{Z})\).  A regular neighborhood of \(\Gamma\) is a handlebody whose boundary is homeomorphic to a Riemann surface with genus equal to \(g\).  (This surface can be built in the following way: Let \(V\) denote the set of vertices of \(\Gamma\) and let \(E\) denote the set of edges.  Use each vertex to label a 2-sphere and each edge to label a cylinder.  The ends of the cylinders are then attached to the complements of disjoint disks in the 2-spheres using the edge–vertex attachment rule that defines \(\Gamma\).)  In particular, since the regular neighborhood is a handlebody which deformation retracts to \(\Gamma\), it follows that the first homology of \(S^3 - \Gamma\) with \(\mathbb{Z}/2\) coefficients has the same rank as the first homology of \(\Gamma\); thus a direct sum of \(g\) copies of \(\mathbb{Z}/2\).

To continue, let \(\Sigma\) denote the boundary of a regular neighborhood of \(\Gamma\) in \(S^3 - \Gamma\).  This surface separates \(S^3\) into two parts, one that contains \(\Gamma\) and the other that is disjoint from \(\Gamma\); it is a deformation retract of \(S^3 - \Gamma\).  The ``interior of \(\Sigma\)'' will denote the part that contains \(\Gamma\).  Let \(H_\Gamma \subset H_1(\Sigma;\,\mathbb{Z})\) denote the kernel of the inclusion map to \(H_1(S^3 - \Gamma;\,\mathbb{Z})\).  The inclusion of \(\Sigma\) into the closure of its interior identifies \(H_\Gamma\) with \(H_1(\Gamma;\,\mathbb{Z})\). Let \(H_\Gamma^{\perp}\) denote the symplectic dual of \(H_\Gamma\) in \(H_1(\Sigma;\,\mathbb{Z})\).  This is the kernel of the inclusion map sending \(H_1(\Sigma;\,\mathbb{Z})\) to the first homology of the interior of \(\Sigma\). (The intersection form on \(H_1(\Sigma;\,\mathbb{Z})\) defines a non-degenerate, skew-symmetric form on \(H_1(\Sigma;\,\mathbb{Z})\).)  The inclusion map identifies \(H_\Gamma^{\perp}\) with \(H_1(S^3 - \Gamma;\,\mathbb{Z})\).

With the preceding understood: The assertion of the lemma follows if there is a surface in \(S^3 - \Gamma\) that intersects \(\Sigma\) transversally and, in this regard, has intersection number \(1\) with each generator of \(H_\Gamma^{\perp}\). The task is to find this surface when the vertices
of \(\Gamma\) have even valencies.

First, by way of terminology, define an \emph{edge loop} in \(\Gamma\) to be a closed path in \(\Gamma\) consisting of a concatenation of edges that begin and end at the same vertex such that no edge is traversed twice. To construct an edge loop, start at a chosen vertex of $\Gamma$ and start walking along an edge until its end, then continue from that end vertex along another edge, and so on, with it understood that no edge should be traversed more than once and that the walk should end when the starting vertex is reached. If all edges are not contained in this edge loop, then there is a vertex with at least two untouched incident edges, so a second edge loop can be constructed starting at that vertex that doesn't cross any edge from the first edge loop. If all edges are not contained in these two edge loops, then the process can be repeated because there will be a vertex with at least two uncrossed, incident edges. And then repeat again (and again) if needed until all edges have been crossed exactly once. Let $[L_1,\dots,L_p]$ denote the set of loops so constructed. These define distinct classes in $H_1(\Gamma;\mathbb{Z})$ (because they have distinct edges), so their union bounds a surface in $S^3 - \Gamma$. This surface will have intersection number $1$ with each generator of $H_\Gamma^{\perp}$ since its boundary crosses each edge exactly once.

To complete the proof, suppose that $\Gamma$ has a vertex with odd valency and there is a class obeying the conditions of the lemma. That class restricted to a small radius $2$-sphere centered on the odd valency vertex would be non-trivial on the linking circles of the points where that $2$-sphere intersects the edges of $\Gamma$. Thus, it will be non-trivial on a circle in that $2$-sphere that encloses all of those intersection points (since there is an odd number). But that conclusion is nonsensical since that circle bounds a disk in the complement of those intersection points.

\section{Some Initial Analysis}

First, some notation:  Here and subsequently, \(\langle\,\cdot\,,\,\cdot\,\rangle\) is used to denote the inner product on \(T^*S^3\), on the spinor bundle for \(S^3\), on tensor bundles constructed from the latter, and in general on any other vector bundle with a specified inner product.  The Levi--Civita covariant derivative on sections of \(T^*S^3\), the corresponding covariant derivative on sections of the spinor bundle, and on tensor bundles constructed from \(T^*S^3\) and the spinor bundle will all be denoted by \(\nabla\).  Covariant derivatives are denoted by \(\nabla\) in what follows.

With this notation, let \(\mathbb{V}\) denote a real or complex vector bundle over \(S^3\) with a fiber metric and metric compatible connection.  Let \(\Gamma\) denote a graph in \(S^3\) with smoothly embedded edges and such that the edges abutting each vertex appear near the origin in a Gaussian coordinate chart centered at that vertex as smoothly embedded arcs that are asymptotically approaching the origin to a corresponding set of disjoint rays from the origin.  Given \(\mathbb{V}\), the tensor product \(\mathbb{V}\otimes \mathcal{I}\) is a vector bundle over \(S^3 - \Gamma\) with an inner product and compatible connection, and thus a covariant derivative (because, up to multiplication by \(\pm 1\), \(\mathbb{V}\) and \(\mathbb{V}\otimes\mathcal{I}\) are canonically isomorphic over any open ball in \(S^3-\Gamma\)).  This covariant derivative is denoted by \(\nabla\) in what follows.

Granted now this covariant derivative:  The positive square root of the functional
\begin{align}\label{eq:4.1}
\psi \longmapsto \int_{S^3-\Gamma} |\nabla \psi|^2
\end{align}
defines a norm on the space of smooth, compactly supported sections of \(\mathbb{V}\otimes\mathcal{I}\) over \(S^3-\Gamma\).

Of particular note with regard to this norm:  The non-triviality of \(\mathcal{I}\) on the linking circles implies the inequality
\begin{align}\label{eq:4.2}
\int_{S^3-\Gamma} \frac{1}{\operatorname{dist}(\cdot,\Gamma)^2}\,|\psi|^2
\;\leq\;
c_0 \int_{S^3-\Gamma} |\nabla \psi|^2,
\end{align}
where
\[\displaystyle\dist(\cdot,\Gamma):=\min_{x\in \Gamma}\dist(\cdot, x).\]

The norm on the space of compactly supported sections of \(\mathbb{V}\otimes\mathcal{I}\) over \(S^3-\Gamma\) whose square is depicted in \eqref{eq:4.1} is denoted by \(\|\cdot\|_{\mathbb{H}}\). The completion of this space of compactly supported sections using the norm \(\|\cdot\|_{\mathbb{H}}\) is a Hilbert space which is denoted by \(\mathbb{H}\) in what follows.

Meanwhile, let \(\mathbb{L}\) denote the Hilbert space completion of that same space of compactly supported sections of \(\mathbb{V}\otimes\mathcal{I}\) over \(S^3-\Gamma\) using the norm whose square is the functional
\begin{align}\label{eq:4.3}
\psi \longmapsto \int_{S^3-\Gamma} |\psi|^2 .
\end{align}
This norm on \(\mathbb{L}\) is denoted by \(\|\cdot\|_{\mathbb{L}}\).

Of interest in what follows are the functionals on \(\mathbb{H}\) that can be written as
\begin{align}\label{eq:4.4}
\psi \longmapsto
\mathcal{E}(\psi) \;\equiv\;
\int_{S^3-\Gamma} \langle \nabla \psi,\nabla \psi\rangle
\;+\;
\int_{S^3-\Gamma} \langle \psi,\mathcal{R}\psi\rangle ,
\end{align}
with \(\mathcal{R}\) here (and subsequently) denoting the extension to \(\mathbb{V}\otimes\mathcal{I}\) of a symmetric (or Hermitian) endomorphism of the bundle \(\mathbb{V}\).  (The inequality in \eqref{eq:4.2} implies that \(\mathcal{E}\) is bounded by a \(c_0\)-multiple of the square of the \(\mathbb{H}\)-norm.)

Let $\mathtt{E}$ denote a given real number.  A non-zero section \(\psi\) from \(\mathbb{H}\) will be deemed to be a \((\nabla^{\dagger}\nabla + \mathcal{R})\)-eigensection with eigenvalue \(\mathtt{E}\) in the event that the identity below holds whenever \(\eta \in \mathbb{H}\):
\begin{align}\label{eq:4.5}
\int_{S^3-\Gamma} \langle \nabla \eta,\nabla \psi\rangle
\;+\;
\int_{S^3-\Gamma} \langle \eta,\mathcal{R}\psi\rangle
\;=\;
\mathtt{E} \int_{S^3-\Gamma} \langle \eta,\psi\rangle .
\end{align}

Any given \((\nabla^{\dagger}\nabla + \mathcal{R})\)-eigensection is smooth on \(S^3 - \Gamma\) and it obeys there the eigenvalue equation \((\nabla^{\dagger}\nabla + \mathcal{R})\psi = \mathtt{E}\psi\). (The unwritten convention henceforth is that eigensections are normalized to have \(\mathbb{L}\)-norm equal to \(1\).) The following lemma says in effect that
\((\nabla^{\dagger}\nabla + \mathcal{R})\)-eigensections are ubiquitous.

\begin{lemma}
\emph{The set of \((\nabla^{\dagger}\nabla + \mathcal{R})\)-eigenvalues is discrete, bounded from below and lacks accumulation point.  Moreover, each such eigenvalue has finite multiplicity.  Corresponding eigenvectors for distinct eigenvalues are \(\mathbb{L}\)-orthogonal.  Once an \(\mathbb{L}\)-orthonormal basis is chosen for each eigenspace, the resulting collection of eigensections forms an orthonormal basis for \(\mathbb{L}\).}

\medskip
This lemma is proved at the end of this section.  The next lemma makes an assertion regarding the norm of a \((\nabla^{\dagger}\nabla + \mathcal{R})\)-eigensection near \(\Gamma\).
\end{lemma}

\begin{lemma}
\emph{The norm of any \((\nabla^{\dagger}\nabla + \mathcal{R})\)-eigensection from \(\mathbb{H}\)
extends across \(\Gamma\) as a H\"older continuous function that vanishes on
\(\Gamma\).}
\end{lemma}

This lemma is proved in Section~8 using some key technical lemmas in Section~7.  But note in this regard that there are steps in the proof of Section~7's technical lemma where \(\Gamma\)'s edges are assumed to be geodesic arcs.  This assumption is obeyed in the cases from Theorem~1.2 since the graph \(\Gamma\) is the radial projection to \(S^3\) of the edges and vertices of a polytope in \(\mathbb{R}^4\).  Lemma~4.2 is true without the simplifying assumption of geodesic arc edges, but the proof in the general case adds some length to the arguments and isn't given in this paper.

An example for Lemma~4.2 has \(\mathbb{V}\) being the product \(\mathbb{R}\)-bundle and \(\mathcal{R} = 0\); thus \((\nabla^{\dagger}\nabla + \mathcal{R})\) in this case is \((-1)\) times the standard spherical Laplacian on \(S^3\), but acting on sections of \(\mathcal{I}\) over \(S^3-\Gamma\).  This case is relevant for the construction of homogeneous, \(\mathbb{Z}/2\) harmonic \(1\)-forms and spinors on \(\mathbb{R}^4\).  A second relevant example is that where \(\mathbb{V} = T^*S^3\) and where \(\mathcal{R}\) is the Ricci curvature tensor in its incarnation as a symmetric endomorphism of \(T^*S^3\).  This case is relevant for the construction of homogeneous, \(\mathbb{Z}/2\) harmonic self-dual \(2\)-forms on \(\mathbb{R}^4\).  An instance of a third example (which is discussed next) is also relevant to the \(\mathbb{Z}/2\) harmonic, self-dual \(2\)-form case.

Suppose now that \(\mathbb{V}\) is also a Clifford module for \(T^*S^3\).  This is to say that there is a monomorphism
$\tau : T^*S^3 \longrightarrow \operatorname{End}(\mathbb{V})$ obeying the rules below:
\begin{itemize}[topsep=6pt]
\item $\tau(u)\,\tau(v) + \tau(v)\,\tau(u) = -2\langle u,v\rangle$ for all $u,v \in T^*S^3$;
\item $\tau^\dagger = -\,\tau$;
\item $\nabla \tau = \tau \nabla$.
\listeqno\label{eq:4.6}
\end{itemize}
Here, $\tau^\dagger$ signifies either the transpose or Hermitian conjugate of $\tau$ as the case may be.  An additional assumption is (imposed henceforth) that this Clifford algebra action is compatible with the given connection on $\mathbb{V}$.  Granted this Clifford module structure, there is a corresponding, first-order, elliptic and symmetric differential operator acting on sections of $\mathbb{V}$.  This operator is denoted by $\mathcal{D}$; and having chosen an oriented, orthonormal frame for $T^*S^3$, it can be written as
\begin{align}
\mathcal{D} = \tau_a \nabla_a ,
\end{align}
where $\{\tau_1,\tau_2,\tau_3\}$ are the images of the basis $1$-forms via $\tau$ and where $\{\nabla_1,\nabla_2,\nabla_3\}$ are the directional covariant derivatives in the direction of the dual basis for $TS^3$.  This operator $\mathcal{D}$ also acts on the space of smooth, compactly supported sections of $\mathbb{V}\otimes\mathcal{I}$ over $S^3-\Gamma$; and as such, it extends to define a bounded, linear map from $\mathbb{H}$ to $\mathbb{L}$.  A given element (call it $\psi$) from $\mathbb{H}$ is said in what follows to be
an eigensection for $\mathcal{D}$ when it is not identically zero and there exists a real number (denoted by $\lambda$) such that $\mathcal{D}\psi = \lambda\psi$ holds as an identity of sections of $\mathbb{L}$.  (A convention henceforth:  All $\mathcal{D}$-eigensections have $\mathbb{L}$-norm equal to $1$.)

Because $\mathcal{D}$ is elliptic, its eigensections in $\mathbb{H}$ are smooth sections of $\mathbb{V}\otimes\mathcal{I}$ on $S^3-\Gamma$.  The lemma that follows makes an assertion about the norm of $\psi$.
\begin{lemma}
\emph{The norm of an eigensection for $\mathcal{D}$ from $\mathbb{H}$ extends
over $\Gamma$ to define a H\"older continuous function on $S^3$ whose zero
locus contains $\Gamma$.}
\end{lemma}
\noindent\textit{Proof of Lemma 4.3.}\;
This is a corollary to Lemma~4.2 because an eigensection of \(\mathcal{D}\) from \(\mathbb{H}\) is also eigensection of \(\mathcal{D}^2\), a \(\mathcal{D}^2\) eigensection being an element \(\psi\) of \(\mathbb{H}\) such that the identity
\begin{align}\label{eq:4.8}
\int \langle \mathcal{D}\eta,\mathcal{D}\psi\rangle \;=\;
\mathtt{E} \int \langle \eta,\psi\rangle
\end{align}
holds for all \(\eta \in \mathbb{H}\) with \(\mathtt{E}\) being independent of \(\eta\). Meanwhile, a Bochner–Weitzenb\"ock formula for \(\mathcal{D}^2\) can be used to identify the left-hand side of \eqref{eq:4.8} as an instance of the left-hand side of \eqref{eq:4.5}.  (The integration by parts to go from \eqref{eq:4.8} to \eqref{eq:4.5} won't get
hung up on \(\Gamma\) when both \(\psi\) and \(\eta\) are from \(\mathbb{H}\).)

As explained directly, a \(\mathcal{D}^2\)-eigensection from \(\mathbb{H}\) can be used to construct a section from the Hilbert space \(\mathbb{L}\) which is (formally) a \(\mathcal{D}\)-eigensection.  Indeed, supposing that \(\psi\) denotes a \(\mathcal{D}^2\) eigensection from \(\mathbb{H}\), then both the \(+\) and \(-\) versions of
\begin{align}\label{eq:4.9}
\phi_{\pm} \;=\; \psi \,\pm\, \frac{1}{\sqrt{\mathtt{E}}}\,\mathcal{D}\psi
\end{align}
are in \(\mathbb{L}\) and they obey the identity \(\mathcal{D}\phi_{\pm} = \pm\sqrt{\mathtt{E}}\,\phi_{\pm}\) on \(S^3-\Gamma\).  In particular, \(\phi_{\pm}\) is a \(\mathcal{D}\)-eigensection if it is in \(\mathbb{H}\) and not identically
zero.  This may or may not be the case.  Of course, \(\psi\) is in \(\mathbb{H}\) by assumption, but \(\mathcal{D}\psi\) need not be.  But, if \(\mathcal{D}\psi\) is indeed in \(\mathbb{H}\), then one of \(\phi_{+}\) and \(\phi_{-}\) can be used to construct a harmonic, homogeneous \(\mathbb{Z}/2\) object on \(\mathbb{R}^4\) as explained in Section~2 (assuming
that \(\mathcal{D}\) comes via \eqref{eq:2.6}).

To elaborate about \(\mathcal{D}\psi\) being in \(\mathbb{H}\):  Not-with-standing the fact that \(\mathcal{D}\) is a symmetric operator, and not-with-standing what Lemmas~4.1 and~4.2 imply for \(\mathcal{D}^2\), the operator \(\mathcal{D}\) is only semi-Fredholm; it has an infinite dimensional cokernel (see \cite{Tak,HMT2,bera2025dirac}).  This is to say that there is an infinite-dimensional subspace in \(\mathbb{L}\) that is annihilated by \(\mathcal{D}\).  An element in this subspace is smooth on \(S^3-\Gamma\), but it need not be in \(\mathbb{H}\).  Indeed, it need not be in \(\mathbb{H}\) for the same reason that a solution from \(\mathbb{L}\) to the equation \(\mathcal{D}\psi = \lambda\psi\) need not be in \(\mathbb{H}\), which is this:  An integration by parts is needed to use the integrals of \(|\mathcal{D}\psi|^2\) and \(|\psi|^2\) to bound the integral of \(|\nabla\psi|^2\), and this integration by parts on \(S^3-\Gamma\) gets hung up along \(\Gamma\) if \(\psi\) is in \(\mathbb{L}\) but not in \(\mathbb{H}\).

With Lemma~4.3 in hand, the question at hand is this:  When does a \(\mathcal{D}^2\)-eigensection from \(\mathbb{H}\) give a \(\mathcal{D}\)-eigensection from \(\mathbb{H}\) via the formula in \eqref{eq:4.9}; thus when is \(\mathcal{D}\psi\) in \(\mathbb{H}\) also?  The next lemma gives a sufficient condition for this.
\begin{lemma}
\emph{In the context of \eqref{eq:4.5}, suppose that \(\psi\) denotes a
\((\nabla^{\dagger}\nabla + \mathcal{R})\)-eigensection from \(\mathbb{H}\).  Then
\(\nabla\psi\) is in the \(\mathbb{V}' \equiv \mathbb{V}\otimes T^*S^3\) version of
\(\mathbb{H}\) if, for every point \(q\) from any edge of \(\Gamma\), there exist
constants \(c_q > 0\) and \(\rho_q\) such that the following is true:  For any
given \(r \in (0,\rho_q]\), let \(\partial B_r\) denote the ball of radius \(r\) centered at
\(q\) and let \(A(r)\) denote the area of \(\partial B_r\) (which is \(4\pi\sin^2(r)\)).
Then}
\begin{align}
\frac{1}{A(r)} \int_{\partial B_r} |\psi|^2 \;\le\; c_q r^2 .
\end{align}
\end{lemma}
Section~8 explains how this lemma follows from those two technical lemmas in Section~7.

With regards to \(\mathcal{D}\psi\) being in \(\mathbb{H}\):  It is possible for this to happen not-with-standing the fact that \(\nabla\psi\) is not in \(\mathbb{H}\).  This possibility is exploited in the upcoming proof of Theorem~1.2's existence assertion for \(\mathbb{Z}/2\) self-dual harmonic \(2\)-forms.

To summarize:  A \(\mathcal{D}^2\)-eigensection \(\psi\) from \(\mathbb{H}\) supplies a \(\mathcal{D}\)-eigensection in \(\mathbb{H}\) via \eqref{eq:4.9} if the average of \(|\psi|^2\) over the radius \(r\) sphere centered at any interior point of an edge of \(\Gamma\) is bounded by a constant multiple of \(r^2\).  In any event, if \(\mathcal{D}\psi\) is in \(\mathbb{H}\), then the formula in \eqref{eq:4.9} supplies an eigensection of \(\mathcal{D}\) in \(\mathbb{H}\).  And, when \(\mathcal{D}\) comes via \eqref{eq:2.6}, that eigensection for \(\mathcal{D}\) from \(\mathbb{H}\) can be used to
construct a homogeneous, \(\mathcal{D}\)-harmonic, \(\mathbb{Z}/2\) object on \(\mathbb{R}^4 - Z\).

\medskip

\noindent\textit{Proof of Lemma 4.1.}\;
The proof follows the standard min–max variational route to obtain the \((\nabla^{\dagger}\nabla + \mathcal{R})\)-eigenvalues and eigenvectors as the respective critical values and critical points of the functional \(\mathcal{E}\) from \eqref{eq:4.4} subject to the constraint that \(\|\psi\|_{\mathbb{L}} = 1\). A standard Sobolev inequality (applied to the \(L^2_1\) function \(|\psi|\) and \(\psi\)'s components with respect to a basis for \(\mathbb{V}\)) and the Rellich lemma can be brought to bear to implement the min–max procedure just as if no \(\Gamma\) were present (see for example \cite{N}).  (In this regard: The Rellich lemma implies that a sequence in \(\mathbb{H}\) with bounded \(\mathbb{H}\)-norm has a Cauchy subsequence with respect to the \(\mathbb{L}\)-norm.)

\section{Equivariant Analysis}\label{sec:5}

The subsequent construction of \((\nabla^{\dagger}\nabla + \mathcal{R})\)-eigensections from \(\mathbb{H}\) that obey the condition set forth in Lemma~4.4 exploits the isometric action on \(S^3\) of certain finite groups (which is the \(S^3\)-analog of the strategy employed in \cite{taubes2020examples}).  To set the stage, suppose for the moment that \(G\) is a finite group acting on \(S^3\) via orientation-preserving isometries.  Assume in addition that all elements
in \(G\) map \(\Gamma\) to itself.  The assumption in what follows is that \(G\) lifts to an action on \(\mathcal{I}\).  To elaborate:  The pull-back of
\(\mathcal{I}\) by the action of any element from \(G\) is isometric to \(\mathcal{I}\).  Even so, that isometry will be ambiguous up to the action
of \(\{-1, 1\}\). With that understood, let \(\widetilde G\) denote the set of pairs of the form \((g,\delta)\) with \(g \in G\) and \(\delta\) being an isometry from \(g^*\mathcal{I}\) to \(\mathcal{I}\).  This set \(\widetilde G\) is itself a group, it being a central extension of \(G\) by \(\{-1,1\}\); see Lemma \ref{lem:12.5}.  The action of the group \(G\) lifts to a fiberwise linear action on \(\mathcal{I}\) if there exist, for each \(g \in G\), an isomorphism \(\delta_g : g^*\mathcal{I} \to \mathcal{I}\), and if these isomorphisms are compatible with the group law in the sense that
\begin{align}
\delta_g \circ \delta_{g'} \;=\; \delta_{g'g}
\end{align}
for any two elements \(g,g' \in G\).  (This implies in particular that \(\delta_{(\cdot)}\) for the identity element in \(G\) is the identity isomorphism from \(\mathcal{I}\) to \(\mathcal{I}\).)

Supposing that the \(G\) action on \(S^3\) lifts to an action on \(\mathcal{I}\) as described above, then \(G\) acts on the space of sections of \(\mathcal{I}\) as follows:  Let \(s\) denote a given section of \(\mathcal{I}\) and let \(g\) denote a given element in \(G\).  Then \(g\) sends the section \(s\) to the section whose value at any given \(x \in S^3-\Gamma\) is \(\delta_g s(g(x))\).  That given section \(s\) is then \(G\)-invariant when \(\delta_g s(g(x)) = s(x)\) for each \(x \in S^3-\Gamma\).  A very simple example of how this works is given in the next paragraph.  As explained in Section~\ref{ch:pre}-\ref{sec:600cell}, other examples are the symmetry groups of the polytopes from Theorem~1.2.

By way of an example:  Suppose that \(\Gamma\)'s intersection with a given small radius ball in \(S^3\) is a subset of the set of \(|x| = 1\) points on a plane through the origin in \(\mathbb{R}^4\).  No generality is lost by taking that plane to be the plane where \(x_3 = x_4 = 0\).  Let \(m\) denote an odd integer greater than \(2\) and let \(G\) denote the group \(\mathbb{Z}/m\mathbb{Z}\) whose action on \(\mathbb{R}^4\) and \(S^3\) is generated by the \(2\pi/m\) rotation of the coordinates \((x_3,x_4)\).  To elaborate, let \(z\) denote the complex coordinate \(x_3 + \mathrm{i}x_4\).  Then the generator of this group sends \((x_1,x_2,z)\) to \((x_1,x_2,e^{2\pi\mathrm{i}/m}z)\). Meanwhile, the line bundle \(\mathcal{I}\) can be viewed as the intersection of the ball with the subset in \((\mathbb{R}^2 \times \mathbb{C}) \times \mathbb{C}\) consisting of the elements \((x_1,x_2,z,u)\) with \(z \neq 0\), with \(x_1^2 + x_2^2 + |z|^2 = 1\) and such that
\begin{align}\label{eq:5.1}
\bar{z}\,u \;=\; |z|\,\bar{u} .
\end{align}
To see about the lift of the group \(\mathbb{Z}/m\mathbb{Z}\), view that group as the multiplicative group of \(m^{\text{th}}\) roots in \(\mathbb{C}\) of the number \(1\).  Let \(g\) denote the generator \(e^{2\pi\mathrm{i}/m}\).  The pull-back via \(g\) of the line bundle \(\mathcal{I}\) consists of the set
\((x_1,x_2,x,u)\) with \((x_1,x_2,z)\) as before but now with the constraint \(\bar z\,e^{-2\pi\mathrm{i}/m} u = |z|\,\bar{u}\).  Define the element
\(\delta_g\) for the \(\mathbb{Z}/m\mathbb{Z}\) generator \(g\) (which is \(e^{2\pi\mathrm{i}/m}\)) according to the rule whereby any given \(u\) obeying the
constraint \(\bar z\,e^{-2\pi\mathrm{i}/m} u = |z|\,\bar{u}\) is sent by \(\delta_g\) to \(u' \equiv -\,e^{-\mathrm{i}\pi/m}u\) (which is in \(\mathcal{I}\) since \(\bar{z}\,u' = |z|\,\bar{u}'\)).  If \(m\) is odd, then the \(m^{\text{th}}\) power of the pair \((g,\delta_g)\) is equal to the identity in \(\widetilde G\); and so it defines a lift of the \(G\) action on \(S^3-\Gamma\) to a fiberwise linear action on \(\mathcal{I}\).  (There is no lift of the \(\mathbb{Z}/m\mathbb{Z}\) to an action on \(\mathcal{I}\) when \(m\) is even.)

Supposing now that $G$ denotes a subgroup of $\mathrm{SO}(4)$ (the isometry group of $S^3$) that maps $\Gamma$ to itself, and that this group $G$ does lift in the manner just described to a fiberwise linear action on $\mathcal{I}$, then there is an infinite dimensional subspace in $\mathbb{H}$ of $G$-invariant sections of $\mathcal{I}$.  Indeed, if there exists just one non-trivial, $G$-invariant section, then there are infinitely many linearly independent ones because any such map can be multiplied by the pull-back to $S$ of a $G$-invariant function on $S^3-\Gamma$.  Meanwhile, a $G$-invariant section can be obtained by taking a section with compact support in some small ball (disjoint from $\Gamma$) in a fundamental domain for $G$ in $S^3$ and then averaging that
section over its $G$-orbit.

In the preceding local example where $G$ is $\mathbb{Z}/m\mathbb{Z}$ with $m$ an odd, positive integer:  Supposing that $k$ is a positive integer and $a$ is
a complex number, then
\begin{align}\label{eq:5.2}
(x_1,x_2,z,u) \;=\;
\biggl(x_1,x_2,z,\,
a\,\frac{1}{\sqrt{|z|}}\,z^{k}
\;+\;
\bar{a}\,\bar{z}^{\,k-1}\sqrt{|z|}\biggr)
\end{align}
is a section of $\mathcal{I}$; but it is a $G$-invariant section of $\mathcal{I}$ if and only if $m$ is an odd integer and $k$ is congruent modulo $m$ to $\tfrac{1}{2}(m+1)$.  To put the section above in a more familiar form, take $\sqrt{z}$ to be defined on the complement of the negative $x_3$–axis. Then $u_0 \equiv \sqrt{z}\,\frac{1}{\sqrt{|z|}}$ is a unit length section of $\mathcal{I}$ on the complement of that axis (it satisfies \eqref{eq:5.1} there). The section $u$ in \eqref{eq:5.2} is a real multiple of this basis section, that multiple being the real part of $a z^{k-1/2}$.

The following lemma has more to say about $G$-invariant sections of $\mathcal{I}$.

\begin{lemma}
\emph{Let $G$ denote a finite subgroup of $\mathrm{SO}(4)$ mapping $\Gamma$ to itself and lifting to a fiberwise linear action on $\mathcal{I}$.  Granted this assumption, then there exists an infinite (but countable) set of $G$-invariant eigensections in $\mathbb{H}$ for the standard Laplacian; and these span the subspace in $\mathbb{L}$ of $G$-invariant sections of $\mathcal{I}$.}
\end{lemma}

Lemma~5.1 follows as an instance of the next lemma which makes an analogous assertion for $(\nabla^{\dagger}\nabla + \mathcal{R})$–eigenvectors.  The upcoming assertion requires \emph{a priori} that $G$ satisfies two constraints: The first requires the $G$–action on $S^3-\Gamma$ lifts to a fiberwise, isometric action on the total space of both  $\mathbb{V}$ and $\mathbb{V}\otimes\mathcal{I}$. The second is that the operator $(\nabla^{\dagger}\nabla + \mathcal{R})$ when acting on the sections of $\mathbb{V}\otimes\mathcal{I}$ is $G$–equivariant.

\begin{lemma}
\emph{Suppose that $G$ is a finite subgroup of $\mathrm{SO}(4)$ that maps $\Gamma$ to itself.  Assuming the preceding two constraints on the group $G$, then there exists an infinite (but countable) set of $G$-invariant $(\nabla^{\dagger}\nabla + \mathcal{R})$–eigensections in $\mathbb{H}$; and these span the subspace in $\mathbb{L}$ of $G$-invariant sections of $\mathbb{V}\otimes\mathcal{I}$.}
\end{lemma}
\noindent\textit{Proof of Lemmas 5.1 and 5.2.}\; Use the same min–max eigenvalue characterization as for the proof of Lemma~4.1 after restricting the domain of $\mathcal{E}$ to the subspace in $\mathbb{H}$ of $G$–invariant sections (with $\mathbb{L}$–norm equal to $1$). Doing that leads to an $\mathbb{L}$–orthonormal set of $G$-invariant sections from $\mathbb{H}$ with the following two properties:
\begin{itemize}
\item Any $G$-invariant section from $\mathbb{L}$ is a linear combination of sections from this set.
\item If $\psi$ is from this set, and if $\eta$ is any $G$-invariant eigensection from $\mathbb{H}$, then (4.5) holds with $\mathtt{E}$ being independent of $\eta$.\listeqno\label{eq:5.3}
\end{itemize}

The two lemmas follow from \eqref{eq:5.3} if it is the case that \eqref{eq:4.5} holds with the same number $\mathtt{E}$ when $\eta$ is \emph{not} $G$-invariant.  But this follows because any given section of $\mathbb{L}$ can be written as the sum of a $G$-invariant part (obtained by averaging over the $G$-action) and a part that is orthogonal to the subspace of $G$-invariant elements.  In particular, if $\eta$ is orthogonal to the $G$-invariant subspace in $\mathbb{L}$, then it will be $\mathbb{H}$–orthogonal to that subspace; and this implies that both sides of \eqref{eq:4.5} are zero for any such element $\eta$.

\medskip

What follows describes a relevant example for Lemma~5.2.  Take $\Gamma$ as in the previous example (the intersection of the $x_3 = x_4 = 0$ plane with a ball in $S^3$ centered on that plane).  Likewise, take $G = \mathbb{Z}/m\mathbb{Z}$ with $m$ being an odd integer.  Introduce coordinates $z$ (a complex coordinate) and $t$ (a real coordinate) on that ball with $z = x_3 + \mathrm{i}x_4$ as before and with $t$ defined by writing $x_1 + \mathrm{i}x_2$ as $(1 - |z|^2)^{1/2} e^{\mathrm{i}t}$ and such that $t = 0$ at the ball's center.  Take the vector space $\mathbb{V}$ to be $T^*S^3$.  The differential forms $\mathrm{d}t$ and the real and imaginary parts of $\mathrm{d}z$ span $T^*S^3$ on that ball.

The $\mathbb{Z}/m\mathbb{Z}$ action on $S^3$ lifts canonically to an action on $T^*S^3$ via pull–back; and the generator $g = e^{2\pi\mathrm{i}/m}$ acts according to the rule whereby $g^*\mathrm{d}t = \mathrm{d}t$ and $g^*\mathrm{d}z = e^{2\pi\mathrm{i}/m}\,\mathrm{d}z$. As noted previously, the $\mathbb{Z}/m\mathbb{Z}$ action lifts to an action on $\mathcal{I}$ if $m$ is odd.  Thus, if $m$ is odd, then the action lifts to an action on $T^*S^3 \otimes \mathcal{I}$.  As for invariant sections:  A section of $T^*S^3 \otimes \mathcal{I}$ can be written as
\begin{align}\label{eq:5.4}
w_t\,\mathrm{d}t
\;+\;
\frac{1}{2}\,(w_1 - \mathrm{i}w_2)\,\mathrm{d}z
\;+\;
\frac{1}{2}\,(w_1 + \mathrm{i}w_2)\,\mathrm{d}\bar{z} ,
\end{align}
where $w_t$, $w_1$ and $w_2$ are sections of $\mathcal{I}$ (they must pointwise obey \eqref{eq:5.1} where $z \neq 0$).  If this section is sent to itself via the $\mathbb{Z}/m\mathbb{Z}$ action, then this must be the case for $w_t$, whereas
$w_1$ and $w_2$ must be such that
\begin{align}\label{eq:5.5}
\delta_g\bigl(g^*(w_1 - \mathrm{i}w_2)\bigr)\,e^{2\pi\mathrm{i}/m}
\;=\;
w_1 - \mathrm{i}w_2 .
\end{align}

To find sections $w_1$ and $w_2$ obeying \eqref{eq:5.5}, consider the case where
\begin{align}\label{eq:5.6}
w_1
=
a_1\,\frac{1}{\sqrt{|z|}}\,z^{k}
\;+\;
\bar{a}_1\,\bar{z}^{\,k-1}\sqrt{|z|}
\qquad\text{and}\qquad
w_2
=
a_2\,\frac{1}{\sqrt{|z|}}\,z^{k'}
\;+\;
\bar{a}_2\,\bar{z}^{\,k'-1}\sqrt{|z|} ,
\end{align}
where $a_1$ and $a_2$ are complex numbers, and where $k$ and $k'$ are positive integers.  This choice for $w_1$ and $w_2$ obeys \eqref{eq:5.5} if and only if $k = k'$ and $a_1$ and $a_2$ are such that
\begin{align}
&-\,e^{\pi\mathrm{i}(2k-1)/m}\,e^{2\pi\mathrm{i}/m}\,(a_1 - \mathrm{i}a_2)
\;=\;
(a_1 - \mathrm{i}a_2)\nonumber\\
\text{and}
\ \ \ &-\,e^{-\pi\mathrm{i}(2k-1)/m}\,e^{2\pi\mathrm{i}/m}\,(\bar{a}_1 - \mathrm{i}\bar{a}_2)
\;=\;
(\bar{a}_1 - \mathrm{i}\bar{a}_2) .
\end{align}

These preceding two equations for $a_1$ and $a_2$ are solvable if and only if either $k \equiv \frac{1}{2}(m-1) \pmod{m}$ or $k \equiv \frac{1}{2}(m+3) \pmod{m}$.  In the first instance, $a_1 = -\,\mathrm{i}a_2$; and in the second instance, $a_1 = \mathrm{i}a_2$.  This is to say the following:

\begin{itemize}
\item When $k \equiv \tfrac{1}{2}(m-1) \pmod{m}$, then
      $
      w_1 - \mathrm{i}w_2 \;=\; -2\mathrm{i}a_2\,\frac{1}{\sqrt{|z|}}\,z^{k}.
      $
\item When $k \equiv \tfrac{1}{2}(m+3) \pmod{m}$, then
      $
      w_1 - \mathrm{i}w_2 \;=\; -2\mathrm{i}\bar{a}_2\,\bar{z}^{\,k-1}\sqrt{|z|}.
      $
\listeqno\label{eq:5.8}
\end{itemize}

When considering derivatives of sections of $\mathcal{I}$:  Remember that derivatives of sections of $\mathcal{I}$ are defined using isometries between $\mathcal{I}$ and the product $\mathbb{R}$–bundle on balls in $S^3-\Gamma$. However, when $\mathcal{I}$ is depicted as in \eqref{eq:5.1} and sections of $\mathcal{I}$ appear as a map from $S^3-\Gamma$ to $\mathbb{C}$ obeying \eqref{eq:5.1}, then the derivative of that section appears as the covariant derivative of the corresponding map to $\mathbb{C}$ with respect to a particular flat connection on $S^3-\Gamma$.  To be explicit, if $u$ denotes the map to $\mathbb{C}$ obeying \eqref{eq:5.1}, then the relevant covariant derivative is 
\begin{align}\label{eq:5.9}
\nabla u = \mathrm{d}u - \frac{\mathrm{i}}{2}\,\mathrm{d}\theta\,u
\end{align}
with $\mathrm{d}$ being the usual derivative on maps to $\mathbb{C}$ and with $\theta$ denoting the argument of the coordinate $z$ (write $z$ in polar form as $|z|\mathrm{e}^{\mathrm{i}\theta}$).  For example, $u_0 \equiv \sqrt{z}\,\frac{1}{\sqrt{|z|}}$ is a unit length, \textit{covariantly constant} section of $\mathcal{I}$ on the complement of the negative real axis in $\mathbb{C}$ (it satisfies \eqref{eq:5.1} there and it is annihilated by the covariant derivative in \eqref{eq:5.9}. What follows is a consequence:\\

\noindent
\emph{The $\mathcal{I}$–valued $1$–form}
$
a \;=\;
\frac{1}{2}(w_1 - \mathrm{i}w_2)\,\mathrm{d}z
\;+\;
\frac{1}{2}(w_1 + \mathrm{i}w_2)\,\mathrm{d}\bar{z}
$
\emph{from  \eqref{eq:5.5} is closed when $k = k'$ with both being $\tfrac{1}{2}(m-1) \pmod{m}$.  This is to say that $\mathrm{d}a = 0$. Moreover, although $\mathrm{d}^*a$ is not zero, its norm is bounded by a constant multiple of $|a|$.}\hfill\listeqno\label{eq:5.10}\\

As explained subsequently, this last observation plays a key role in the upcoming proof of the assertions in Theorem~1.2 regarding $\mathbb{Z}/2$ harmonic, self-dual harmonic $2$-forms.  By way of comparison, if $k = k'$ and both equal $\tfrac{1}{2}(m+3) \bmod(m)$, then $\mathrm{d}a$ is not zero; in fact, the norm of $\mathrm{d}a$ and that of $\mathrm{d}^*a$ are both $O(|z|^{-1}|a|)$.

\section{H\"older Continuity Across the Interior of the Edges of $\Gamma$}

This section states and then proves preliminary results that are needed for the proof of Lemma~4.2 and for later use.  These preliminary results are summarized by the upcoming Lemmas~6.1 and~6.2.  To set the background for what is to come:  A graph in $S^3$ has been specified (it is again denoted by $\Gamma$); and it is assumed implicitly that this graph has even valencies so that there is a corresponding real line bundle $\mathcal{I}$ on $S^3 - \Gamma$.  A vector bundle $\mathbb{V}$ on $S^3$ has also been specified along with fiber metric and metric compatible connection.  As before, $\mathbb{H}$ denotes the Hilbert space completion of the space of smooth sections of $\mathbb{V}\otimes\mathcal{I}$ with compact support in $S^3 - \mathcal{I}$ using the norm whose square is depicted in  \eqref{eq:4.1}.  Also as before, use $\mathbb{L}$ to denote the completion of that same space of sections using the norm whose square is depicted in  \eqref{eq:4.3}.  Supposing that $\mathcal{R}$ denotes a given endomorphism of $\mathbb{V}$ (as in \eqref{eq:4.4}), the promised Lemmas~6.1 and~6.2 concern $(\nabla^{\dagger}\nabla + \mathcal{R})$-eigensections from the Hilbert space $\mathbb{H}$.

By way of a look ahead:  These lemmas concern $(\nabla^{\dagger}\nabla + \mathcal{R})$-eigensections near interior points of edges from the graph $\Gamma$.  Lemma~6.1 concerns the pointwise norm of an eigensection from $\mathbb{H}$ near an interior point of an edge; and Lemma~6.2 concerns the directional derivative of an eigensection from $\mathbb{H}$ in the direction tangent to that nearby edge.  Both Lemma~6.1 and~6.2 assume implicitly that the edges of $\Gamma$ are geodesic arcs.  (Both lemmas are true when the geodesic arc assumption is not made; but the proofs in that generality are not needed for this paper.)

To set the notation for what follows momentarily:  Supposing that $p$ denotes a given point in $S^3 - \Gamma$, the notation has $r_p$ denoting the distance from $p$ to $\Gamma$.  When $p$ is in $\Gamma$ but not a vertex of $\Gamma$, the number $r_p$ denotes the distance from $p$ to the nearest vertex of $\Gamma$.  When $p$ is a vertex of $\Gamma$, then $r_p$ denotes the distance from $p$ to the nearest of the edges of $\Gamma$ that don't have an endpoint at $p$.  An important geometric point to keep in mind regarding the various $p \in \Gamma$ versions of $r_p$ is the following:

\begin{center}\it
\textit{There exists a number to be denoted by $c$ (it is greater than $100$) such that when $p$ is a point from $\Gamma$, then the only edges of $\Gamma$ that intersect the radius $\frac{1}{c} r_p$ ball centered at $p$ are the edges that contain $p$, the intersection of that radius $\frac{1}{c} r_p$ ball centered at $p$ with $\Gamma$ being a union of half-open arcs with $p$ being their common starting point.}\\
\hfill
\listeqno\label{eq:6.1}
\end{center}

Since the geometry is a bit subtle, the subsequent digression to describe it might prove illuminating.  The digression refers to the following diagram, which sketches a 3-dimensional cartoon version of part of a graph in $S^3$, as illustrated in Fig.~\ref{eq:6.2}.

What should first catch your eye in this sketch is a sphere centered at a vertex of \(\Gamma\).
Supposing that \(p\) denotes that vertex, then the radius of this sphere is the corresponding
\(r_p\), where \(r_p\) is the distance from \(p\) to the nearest edge of \(\Gamma\) that is not
incident to \(p\). The interior of this sphere corresponds to the ball in \(S^3\) with distance
less than \(r_p\) from \(p\). This vertex \(p\) has valency 4; its incident edges are the black
lines that have \(p\) as an endpoint. 
The colored cones indicate the regions in \(S^3\) whose
closest point in \(\Gamma\) lies on the corresponding edge, with the distance from a given
point to that closest point \(q\) bounded by \(\frac{1}{8c}r_q\). Note that the solid-angle
cross-sections of these cones are very nearly \(\frac{1}{8c}\). The idea is that, after shrinking
to the ball of radius \(c^{-1}r_p\), only the incident edges remain in view and these conical
regions are uniformly separated.

\begin{figure}[htp]
\centering
\includegraphics[width=0.45\linewidth]{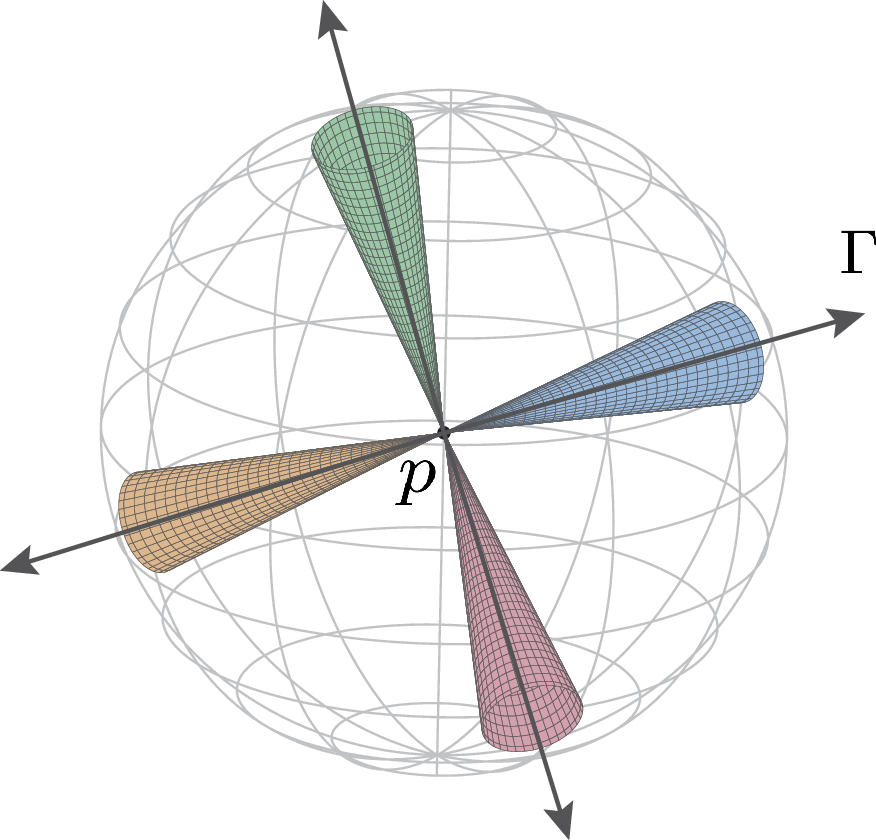}
\caption{Local geometry near a vertex \(p\in V(\Gamma)\). The sphere represents \(\partial B_{r_p}(p)\) and the colored conical regions represent uniformly separated distance cones around the incident edges. The picture shows the valency--4 case.}
\label{eq:6.2}
\end{figure}

\medskip
What follows is the first of the promised lemmas from this section.

\begin{lemma}
\emph{Let $\psi$ denote a $(\nabla^{\dagger}\nabla + \mathcal{R})$-eigensection from $\mathbb{H}$ with $\mathbb{L}$-norm equal to $1$.  If $q$ signifies an interior point of an edge from $\Gamma$, then the function $|\psi|$ on $S^3-\Gamma$ extends as a continuous function across $\Gamma$ in the radius $\frac{1}{4c}\, r_q$ ball centered at $q$; and $|\psi|$ on this ball obeys
\begin{align}
|\psi| \;\leq\; \kappa\,\operatorname{dist}(\,\cdot\,,\Gamma)^{1/2}
\end{align}
with $\kappa$ being independent of the chosen point $q$.}
\end{lemma}
This lemma can be proved using the technology that was introduced in \cite{HMT2}, and also \cite{franceschini2026minimal}.  A proof using a different approach is given below.

The second lemma introduces additional notation:  To define this new notation, let $q$ denote a point in the interior of an edge of $\Gamma$.  The part of $\Gamma$ in the radius $\tfrac{1}{4c} r_q$ ball centered at $q$ is then a geodesic arc which is an arc in a great circle.  Denote that great circle by $\gamma$.  Fix Euclidean coordinates $x = (x_1,x_2,x_3,x_4)$ for $\mathbb{R}^4$ so that $\gamma$ appears as the radius $1$ circle in the $x_1,x_2$ plane, i.e.,~$x_3 = x_4 = 0$. The $\mathbb{R}^4$ vector field $x_1 \frac{\partial}{\partial x_2} - x_2 \frac{\partial}{\partial x_1}$ is tangent to $S^3$ along $S^3$ (which is the $|x| = 1$ locus in $\mathbb{R}^4$) and it is tangent to $\gamma$ along $\gamma$.  In addition, this vector field along $S^3$ generates an isometry of $S^3$ which maps $\gamma$ to itself as the group of constant rotations of $\gamma$.  With respect to the coordinates $(z,t)$ from Section~5 for a ball centered at a point in $\gamma$, the vector field $x_1 \frac{\partial}{\partial x_2} - x_2 \frac{\partial}{\partial x_1}$ is just $\frac{\partial}{\partial t}$.  (By way of a reminder:  The complex coordinate $z$ is the restriction of $x_3 + \mathrm{i}x_4$ to $S^3$, and $t$ comes by writing the complex coordinate $x_1 + \mathrm{i}x_2$ on that ball as $(1 - |z|^2)^{1/2} \mathrm{e}^{\mathrm{i}t}$ with $t = 0$ at ball's center point.)

More notation for the second lemma:  Given $q$ as above, let $\chi_q$ denote a favorite non-increasing, smooth function with values in $[0,1]$ that is equal to $1$ where the distance to $q$ is less than $\tfrac{1}{8c} r_q$, equal to $0$ where the distance to $q$ is greater than $\tfrac{1}{4c} r_q$, and such that $|d\chi_q| \le 100c\,\tfrac{1}{r_q}$.
\begin{lemma}\label{lem:6.2}
\emph{Let $\psi$ denote a $(\nabla^{\dagger}\nabla + \mathcal{R})$-eigensection from $\mathbb{H}$ with $\mathbb{L}$-norm equal to $1$.  Restricting to the radius $\tfrac{1}{8c} r_q$ ball centered at $q$, let $\nabla_t\psi$ denote the directional derivative of $\psi$ along the vector field $\tfrac{\partial}{\partial t}$. Then $\chi_q \nabla_t\psi$ is in the Hilbert space $\mathbb{H}$; and its $\mathbb{H}$-norm is at most $\kappa\,\tfrac{1}{\sqrt{r_q}}$ with $\kappa$ being independent of $q$.}
\end{lemma}
The rest of this subsection has five parts with Part~4 containing the proof of Lemma~6.1 and Part~5 containing the proof of Lemma~6.2.  Part~1 states some direct \emph{a priori} bounds for $(\nabla^{\dagger}\nabla + \mathcal{R})$-eigensections that are used later.  The intervening Parts~2 and~3 then introduce and analyze a Dirichlet Green's function for the operator $\nabla^{\dagger}\nabla$ acting on sections of \(\mathbb{V}\otimes\mathcal{I}\) over a given small radius ball centered on an interior point of an edge of $\Gamma$.  This Green's function is then used in Part~4 to analyze the behavior of a given $(\nabla^{\dagger}\nabla + \mathcal{R})$-eigensection from $\mathbb{H}$ near that interior point of the given edge of $\Gamma$.  Using the Green's function in this way requires some specific bounds for that Green's function near $\Gamma$; and those bounds are relatively straightforward to come by when $\Gamma$ is a geodesic arc because symmetry considerations can be exploited.  That is why the geodesic arc assumption is made in Lemma~6.1.  As for the geodesic arcs in Lemma~\ref{lem:6.2}:  The proof of that lemma uses the bounds from Lemma~6.1, but it also uses the fact that there are isometries of $S^3$ that restrict to any given geodesic (when parametrized by arc-length) as a rigid rotation.

\medskip
\noindent\textit{Part 1:}  Let $\psi$ denote a $(\nabla^{\dagger}\nabla + \mathcal{R})$-eigensection from $\mathbb{H}$ with $\mathbb{L}$-norm equal to $1$, and let \(\mathtt{E}\) denote its eigenvalue. The first observations in this part of Section~6 are the inequalities below for the $\mathbb{L}$-norm of $\nabla\psi$, for a certain weighted version of the $\mathbb{L}$-norm of $\psi$, and for the pointwise norm of $|\psi|$:
\begin{itemize}
\item $\displaystyle \int |\nabla\psi|^2 \;\leq\; c_0(1+\mathtt{E})\,.$
\item $\displaystyle \int \frac{1}{\operatorname{dist}(\,\cdot\,,\Gamma)^2}\,|\psi|^2
       \;\leq\; c_0(1+\mathtt{E})\,.$
\item \it $|\psi|$ extends to $S^3$ as an $L^\infty$ function obeying
      $\;|\psi| \leq c_0(1+\mathtt{E})\,.$
\item \emph{For any point $x \in S^3$:}
      $\displaystyle \int \frac{1}{\operatorname{dist}(\,\cdot\;;x)}\,|\nabla\psi|^2
      \;\leq\; c_0(1+\mathtt{E})\,.$\hfill \listeqno\label{eq:6.3}
\end{itemize}

By way of a proof:  The first of these inequalities follows directly from \eqref{eq:4.5} by taking $\eta$ to be the section $\psi$.  The second inequality in  \eqref{eq:6.3} follows from \eqref{eq:4.2} and the first inequality.  The third and forth inequalities in  \eqref{eq:6.3} are proved using the fact that $\psi$ obeys the equation below on $S^3-\Gamma$:
\begin{equation}\label{eq:6.4}
\nabla^{\dagger}\nabla\psi + \mathcal{R}\psi = \mathtt{E}\,\psi.
\end{equation}
To say more: The identity in \eqref{eq:6.4} implies that $|\psi|^2$ obeys the differential inequality below:
\begin{equation}\label{eq:6.5}
-\tfrac12\,\Delta^{\perp}|\psi|^2 + |\nabla\psi|^2
\;\leq\; c_0(1+\mathtt{E})\,|\psi|^2.
\end{equation}

Now, having chosen a point $x \in S^3 - \Gamma$, multiply both sides of \eqref{eq:6.5}’s inequality by the Green’s function with pole at $x$ for the operator $-\Delta^{\perp} + 1$ acting on $\mathbb{R}$–valued functions on $S^3$. Integrate the result over $S^3$ and then integrate by parts to see that
\begin{align}\label{eq:6.6}
|\psi|^2(x)
\;+\;
\int \frac{1}{\operatorname{dist}(\,\cdot\;;x)}\,|\nabla \psi|^2
\;\leq\;
c_0(1+\mathtt{E}) 
\int \frac{1}{\operatorname{dist}(\,\cdot\;;x)}\,|\psi|^2 .
\end{align}

The integration by parts doesn’t get hung up on $\Gamma$ even when $x$ is from $\Gamma$ because $\psi$ is from $\mathbb{H}$.  This can be justified as done in \cite{T4} using a sequence of bounded approximations to the Green’s function, for example the sequence $\{ G/(1+\tfrac{1}{n}G) \}_{n=1,2,\dots}$ with $G$ denoting just here the $-\Delta^{\perp}+1$ Green’s function.  Use of the Cauchy–Schwarz inequality with \eqref{eq:6.6} and the first bullet in \eqref{eq:6.3}, and with a dimension–$3$ Sobolev inequality leads to the assertions of the third and fourth assertions in \eqref{eq:6.3}.  (That Sobolev inequality says in effect that the fourth root of the integral of the fourth power of an $L^2_1$ function over $S^3$ is no larger than $c_0$ times that function’s $L^2_1$ norm.  The function in this case is $|\psi|$.  Remember in this regard that Kato’s inequality bounds at each point the norm of $d|\psi|$ by the norm of $\nabla\psi$.)

\medskip

\noindent\textit{Part 2:}  
Suppose in what follows that $\mathtt{B}$ is a ball in $S^3$ with radius at most $\frac{\pi}{4}$ (and thus it sits inside a hemisphere).  The set of elements in $\mathbb{H}$ that vanish on $S^3 - \mathtt{B}$ is the completion using the $\mathbb{H}$–norm of the space of smooth sections of $\mathbb{V}\otimes \mathcal{I}$ that have compact support on $\mathtt{B} - (\mathtt{B}\cap\Gamma)$.  This subspace is denoted by $\mathbb{H}_{\mathtt{B}}$. The completion of the space of smooth sections of $\mathbb{V}\otimes \mathcal{I}$ over $\mathtt{B} - (\mathtt{B}\cap\Gamma)$ using the $\mathbb{L}$–norm is denoted by $\mathbb{L}_{\mathtt{B}}$; it is a subspace of the Hilbert space $\mathbb{L}$.

\medskip

The lemma below is the centerpiece of Part~2.

\begin{lemma}\label{lem:6.3}
\emph{There exists $\kappa > 1$ with the following significance:  Let $\mathtt{B}$ denote an open ball in $S^3$ with radius $\rho < \tfrac{\pi}{4}$. Given $\varsigma \in \mathbb{L}_{\mathtt{B}}$, there is a unique element $\psi \in \mathbb{H}_{\mathtt{B}}$ which obeys 
\begin{align}
\nabla^{\dagger}\nabla \psi = \varsigma \quad \text{on } \mathtt{B} - (\mathtt{B}\cap\Gamma),
\end{align}
and is such that the identity
\begin{align}
\int_{\mathtt{B}} \langle \nabla\eta,\nabla\psi\rangle
\;=\;
\int_{\mathtt{B}} \langle \eta,\varsigma\rangle
\end{align}
holds for all elements $\eta \in \mathbb{H}_{\mathtt{B}}$.  This $\psi$ obeys the bound}
\begin{align}
\int_{\mathtt{B}} \Bigl(|\nabla\psi|^2 + \tfrac{1}{\rho^2}|\psi|^2\Bigr)
\;\leq\; \kappa\,\rho^2 \int_{\mathtt{B}} |\varsigma|^2 .
\end{align}
\end{lemma}

\medskip
\noindent\textit{Proof of Lemma 6.3.}
The desired section $\psi$ can be obtained by minimizing the functional on $\mathbb{H}_{\mathtt{B}}$ given by the rule
\begin{align}
\psi \longmapsto \tfrac12 \int_{\mathtt{B}} |\nabla\psi|^2
\;-\; \int_{\mathtt{B}} \langle \psi,\varsigma\rangle.
\end{align}
The existence of a unique minimizer in $\mathbb{H}_{\mathtt{B}}$ can be readily proved using \eqref{eq:4.6} with input from the fact that the inequality
\begin{align}\label{eq:6.7}
\int_{\mathtt{B}} |\nabla\psi|^2
\;\geq\; c_0^{-1}\rho^{-2} \int_{\mathtt{B}} |\psi|^2
\end{align}
holds for all $\psi \in \mathbb{H}_{\mathtt{B}}$.  (This follows from Kato’s inequality $|\nabla\psi| \leq |d|\psi||$ and the fact that \eqref{eq:6.7} holds when $\psi$ is an ordinary function with compact support in $\mathtt{B}$.)  The existence of the minimizer also invokes the Rellich lemma to conclude that bounded $\mathbb{H}_{\mathtt{B}}$-norm sequences
have $\mathbb{L}_{\mathtt{B}}$-norm Cauchy subsequences.

\medskip
\noindent\textit{Part 3:}  
Let $\mathtt{B}$ denote an (open) ball of radius less than $\tfrac{\pi}{4}$ centered at an interior point of an edge in $\Gamma$ whose closure has distance at least twice $\mathtt{B}$'s radius from any other edge in $\Gamma$.  This step states and then proves a lemma regarding a Dirichlet Green's function for the operator $\nabla^{\dagger}\nabla$.  By way of terminology, a Green's function for $\nabla^{\dagger}\nabla$ is an example of a singular integral operator, that being an object which is defined for the purposes of the upcoming lemma as follows:  Let $(\mathbb{V}\otimes\mathcal{I})_R$ and $(\mathbb{V}\otimes\mathcal{I})_L$ denote the respective pull-backs to 
\begin{align}\label{eq:set-off}
\bigl(\mathtt{B} - (\mathtt{B}\cap\Gamma)\bigr) \times
\bigl(\mathtt{B} - (\mathtt{B}\cap\Gamma)\bigr)
\end{align}
of the bundle $\mathbb{V}\otimes\mathcal{I}$ by the projections to the right- and left-hand factors of $\mathtt{B} - (\mathtt{B}\cap\Gamma)$.  A singular integral operator is a smooth section of the vector bundle
\[
\operatorname{Hom}\bigl((\mathbb{V}\otimes\mathcal{I})_R, \;(\mathbb{V}\otimes\mathcal{I})_L\bigr)
\]
over the complement of the diagonal in the product space depicted in \eqref{eq:set-off}, which is characterized by the following property:  Let $\varsigma$ denote any given smooth section of $\mathbb{V}\otimes\mathcal{I}$ over $\mathtt{B} - (\mathtt{B}\cap\Gamma)$.  The section of $\mathbb{V}\otimes\mathcal{I}$ over $\mathtt{B} - (\mathtt{B}\cap\Gamma)$ that is obtained by first acting by that section of $\operatorname{Hom}\bigl((\mathbb{V}\otimes\mathcal{I})_R, \,(\mathbb{V}\otimes\mathcal{I})_L\bigr)$ on the pull-back of $\varsigma$ to the product space in \eqref{eq:set-off} via the right-hand projection and then integrating the result over that same right-hand factor yields a section of $\mathbb{V}\otimes\mathcal{I}$ in the Hilbert space
$\mathbb{H}_{\mathtt{B}}$.

\smallskip

The lemma below is the centerpiece for Part~3.

\begin{lemma}\label{lem:6.4}
\emph{Suppose in what follows that $\Gamma$ is a graph in $S^3$ whose vertices have even valency and whose edges are geodesic arcs.  Given this graph $\Gamma$, there exists $\kappa > 1$ with the following significance:  If the radius of a ball $\mathtt{B} \subset S^3$ is sufficiently small (less than $\tfrac1\kappa$), then there is a linear operator (denoted by $G_{\mathtt{B}}$) mapping $\mathbb{L}_{\mathtt{B}}$ to $\mathbb{H}_{\mathtt{B}}$ with operator norm at most $\kappa$, and which has the properties listed in the subsequent bullets.}
\begin{itemize}[leftmargin=12pt]
\item \emph{If $\varsigma \in \mathbb{L}_{\mathtt{B}}$, then $\psi = G_{\mathtt{B}}(\varsigma)$ is the section from Lemma~6.3 that obeys
$\nabla^{\dagger}\nabla\psi = \varsigma$ on $\mathtt{B} - (\mathtt{B}\cap\Gamma)$, 
and is such that}
\begin{align}
\int_{\mathtt{B}} \langle \nabla\eta,\nabla\psi\rangle
\;=\;
\int_{\mathtt{B}} \langle \eta,\varsigma\rangle
\qquad\text{\emph{for all $\eta \in \mathbb{H}_{\mathtt{B}}$.}}
\end{align}

\item \emph{The operator $G_{\mathtt{B}}$ has a unique representation as a Green's function, and in particular, as a singular integral operator for which the bounds below hold for any given pair of distinct points $(x,y)$ in $\mathtt{B} - (\mathtt{B}\cap\Gamma)$:}
\begin{itemize}[leftmargin=14pt]
\rm
\item[a)] \it $\displaystyle |G_{\mathtt{B}}(x,y)|
   \;\leq\; \frac{\kappa}{\operatorname{dist}(x,y)}$. \rm
\item[b)] \it If $\displaystyle \operatorname{dist}(x,\Gamma) \le \tfrac12\,\operatorname{dist}(y,\Gamma)$, then $\displaystyle|G_{\mathtt{B}}(x,y)|
   \;\leq\;
   \kappa\,\frac{\operatorname{dist}(x,\Gamma)^{1/2}}
              {\operatorname{dist}(y,\Gamma)^{3/2}}$. \rm
\item[c)] \it For any fixed $y \in \mathtt{B} - (\mathtt{B}\cap\Gamma)$, the function $|\nabla G_{\mathtt{B}}(\,\cdot\,,y)|$ has finite integral on $\mathtt{B} - (\mathtt{B}\cap\Gamma)$; this integral being at most $\kappa$. Meanwhile, the integral of $|\nabla G_{\mathtt{B}}(\,\cdot\,,y)|^2$ over the complement in $\mathtt{B} - (\mathtt{B}\cap\Gamma)$ of any open ball centered at $y$ is finite,that integral being at most $\tfrac{\kappa}{\delta}$, with $\delta$ denoting the radius of that ball.
\end{itemize}
\end{itemize}
\end{lemma}

\noindent\textit{Proof of Lemma 6.4.}
The first bullet simply restates conclusions from Lemma~6.3.  As for the second bullet:  The existence of the Green's function can be proved using standard arguments (see, e.g., \cite{GT}) given that $\mathcal{I}$ is isometric to the product $\mathbb{R}$-bundle on any ball in $\mathtt{B} - (\mathtt{B}\cap\Gamma)$.  In any event, some remarks about a construction of the Green's function are given below in Step~3 of the upcoming six-step proof.

\medskip

With regards to the assumption that the edges of $\Gamma$ are geodesic arcs:  This is the case for the graphs that are defined by the radial projections to $S^3$ of the polytopes from Theorem~1.2.  In any event, the proofs that follow of Items a) and c) of Lemma~6.4's second bullet don't make the geodesic arc assumption; and Item b) is true without that assumption but some extra (relatively straightforward) steps are needed to deal honestly with that level of generality.

The six steps of the proof explain how to obtain the bounds that are asserted in Items a), b) and c) of the lemma's second bullet.  What follows directly are two facts about the Dirichlet Green's function for the Laplacian acting on $\mathbb{R}$-valued functions on the ball $\mathtt{B}$.  This is a function on the complement of the diagonal in $\mathtt{B}\times\mathtt{B}$ that is harmonic with respect to $x$ for fixed $y$ and vice versa.  (It is symmetric with respect to the map sending $(x,y)$ to $(y,x)$.)  Of particular note are the properties listed below:
\begin{itemize}
\item \textit{This Green's function is positive, and it is bounded by $\displaystyle \frac{c_0}{\operatorname{dist}(\,\cdot, y)}$.}
\item \textit{The norm of this Green's function's differential (which is a $1$-form on the complement of the diagonal in $\mathtt{B}\times\mathtt{B}$) is bounded by
$\displaystyle \frac{c_0}{\operatorname{dist}(x,y)^2}$.}
\listeqno\label{eq:6.8}
\end{itemize}

With regards to the steps in the proof:  Step~1 proves Item~a) of the lemma's second bullet and Step~2 proves Item~c) of that bullet.  Step~3 is an aside regarding the existence of the Green's function $G_{\mathtt{B}}$.  Steps~4 and~5 prove Item~b) of the lemma's second bullet for the case when $\mathbb{V}$ is the product $\mathbb{R}$-bundle and thus $\mathbb{V}\otimes\mathcal{I} \cong \mathcal{I}$ and when the edges of the graph are geodesic arcs.  The sixth step proves Item~b) for any given $\mathbb{V}$ under that same assumption about the edges.

\medskip

\noindent\underline{\textit{Step 1:}}
Fix a point $y \in \mathtt{B} - (\mathtt{B}\cap\Gamma)$ and then an isometric identification between $\mathcal{I}|_y$ and $\mathbb{R}$. Transport this product structure at $y$ along the radial geodesics from $y$ to identify $\mathcal{I}$ with the product $\mathbb{R}$–bundle on any ball in $\mathtt{B} - (\mathtt{B}\cap\Gamma)$ centered at $y$.  Given any positive number $r$ that is less than one fourth of the distance from $y$ to $\Gamma$, let $u_{y,r}$ denote the characteristic function for the radius $r$ ball centered at $y$.  Fix a unit length section of $\mathbb{V}$ over this radius $r$ ball to be denoted by $t_y$ and use the product structure for $\mathcal{I}$ on this radius $r$ ball to view $u_{y,r} t_y$ as a section of $\mathbb{V}\otimes\mathcal{I}$ from the Hilbert space $\mathbb{L}_{\mathtt{B}}$.  Let $\psi_{y,r}$ denote the corresponding section from $\mathbb{H}_{\mathtt{B}}$ given by Lemma~6.3. The identity
\begin{align}
-\Delta^{\perp}\psi_{y,r} = u_{y,r} t_y
\end{align}
implies a differential inequality below for the norm of $\psi_{y,r}$ which is that
\begin{align}
-\Delta^{\perp}|\psi_{y,r}| \;\leq\; \frac{1}{r^3}\,u_{y,r}.
\end{align}

Granted this inequality, then the comparison/maximum principle can be brought to bear to compare $|\psi_{y,r}|$ with the function $f_{y,r}$ on $\mathtt{B}$ that obeys the equation $-\Delta^{\perp} f_{y,r} = u_{y,r}$ and vanishes on $\mathtt{B}$'s boundary. The comparison finds that
\begin{align}\label{eq:6.10}
|\psi_{y,r}(x)| \;\leq\; c_0\,\frac{1}{r^3}
   \int_{\mathtt{B}} \frac{1}{\operatorname{dist}(\,\cdot\;;x)}\,u_{y,r}.
\end{align}
(The function on the right comes via the top bullet of \eqref{eq:6.8} since $f_{y,r}$ at any given point $x$ is given by the integral over $\mathtt{B}$ of the product of $u_{y,r}$ with Dirichlet Green's function on $\mathtt{B}$ with pole at $x$.)  In particular, if $\operatorname{dist}(x,y) < 2r$, then the bound in \eqref{eq:6.10}
implies in turn that
\begin{align}\label{eq:6.11}
|\psi_{y,r}(x)| \;\leq\; c_0\,\frac{1}{\operatorname{dist}(x,y)} .
\end{align}
Letting $r \to 0$ gives the bound in Item~a) of the lemma's second bullet.

\medskip
\noindent\underline{\textit{Step 2:}}
Take that same $\psi_{y,r}$ with $r$ still positive.  Fix $\delta \in \bigl(0,\tfrac{1}{100}\operatorname{dist}(y,\Gamma)\bigr)$ and let $\beta_{y,\delta}$ denote a smooth, non‑negative function mapping to $[0,1]$ which is zero on the radius $\tfrac{1}{2}\delta$ ball centered at $y$ and one on the complement of the radius $\delta$ ball centered at $y$.  This function can and should be chosen so that $|d\beta_{y,\delta}| \leq c_0\,\delta^{-1}$.  Now take the inner product of both sides of the equation $-\Delta^{\perp}\psi_{y,r} = u_{y,r} t_y$ with $\beta_{y,\delta}^2 \psi_{y,r}$ and integrate the result over $\mathtt{B} - (\mathtt{B}\cap\Gamma)$.  Supposing that $r < \tfrac{1}{4}\delta$, then an integration by parts (which is allowed because $\psi_{y,r} \in \mathbb{H}_{\mathtt{B}}$) together with \eqref{eq:6.11}'s bound leads to this inequality:
\begin{align}\label{eq:6.12}
\int_{\mathtt{B}} \bigl|\nabla(\beta_{y,\delta}\psi_{y,r})\bigr|^2
\;\leq\; c_0\,\frac{1}{\delta} .
\end{align}
Noting that the right hand side is independent of $r$, taking $r \to 0$ proves the second assertion of Item~c) of the lemma's second bullet.

To prove that $|\nabla G_{\mathtt{B}}(\,\cdot\,,y)|$ has finite integral, set $\delta_0 = \tfrac{1}{1000}\operatorname{dist}(y,\Gamma)$, and for each positive integer $n$, let $\delta_n$ denote $\tfrac{1}{2^n}\delta_0$.  It follows from \eqref{eq:6.6} that the integral of $|\nabla\psi_{y,r}|$ over the complement of the radius $\delta_0$ ball centered at $y$ is at most $c_0\,\delta_0^{-1/2}$.  By the same token, for any positive integer $n$, the integral of $|\nabla\psi_{y,r}|$ over the spherical annulus centered at $y$ with inner radius $\delta_n$ and outer radius $\delta_{n-1}$ is at most $c_0\,\delta_n^{-1/2}\,(\delta_{n-1})^{3/2},$ which is at most $c_0\,2^{-n}\delta_0$.  
Thus, the sum of these annular contributions to the integral of $|\nabla G_{\mathtt{B}}(\,\cdot\,,y)|$ is finite and at most $c_0\,\delta_0$, as 
\begin{align}
\sum_{n=1}^\infty \int_{A_n} |\nabla \psi_{y,r}|
\;\le\;
c_0\,\delta_0 \sum_{n=1}^\infty 2^{-n}
=
c_0\,\delta_0.
\end{align}

\medskip
\noindent\underline{\textit{Step 3:}}
This step is an aside for a remark regarding the construction of the Green's function $G_{\mathtt{B}}$.  The remark is that the existence of the $r \to 0$ limits for \eqref{eq:6.11} and \eqref{eq:6.12} can be used to prove the existence of $G_{\mathtt{B}}$. This is because sections of the form $u_{y,r} t_y$ are dense in $\mathbb{L}_{\mathtt{B}}$; and hence any given section $\varsigma$ of $\mathbb{L}_{\mathtt{B}}$ can be approximated to any given accuracy (as measured by the $\mathbb{L}_{\mathtt{B}}$-norm) by a finite linear combination of these sorts of sections with disjoint support and with any given positive upper bound for the $r$-values.  As a consequence, the corresponding $\psi$ from Lemma~6.3 can be approximated to any given accuracy (as measured by the $\mathbb{H}_{\mathtt{B}}$-norm) by the corresponding finite linear combination of the $\psi_{y,r}$ sections.

\medskip
\noindent\underline{\textit{Step 4:}}
This step and the next step prove Item~b) of the lemma's second bullet for the case when $\mathbb{V}$ is the product $\mathbb{R}$–bundle and thus $\mathbb{V}\otimes\mathcal{I} \cong \mathcal{I}$. This step, the next one, and Step~6 all assume that the edges of $\Gamma$ are geodesic arcs.

To describe $G_{\mathtt{B}}$ when $\mathbb{V}$ is the product $\mathbb{R}$–bundle, it proves convenient to introduce Gaussian normal coordinates (denoted by $x_1,x_2,x_3$) centered at the point $q$ so as to identify the ball $\mathtt{B}$ with the ball of radius $r_{\mathtt{B}}$ in $\mathbb{R}^3$ centered at the origin and with the intersection of $\mathtt{B}$ with the edge containing $q$ being the $-r_{\mathtt{B}}<x_3<r_{\mathtt{B}}$ part of the $x_3$-axis.  In these coordinates, the spherical Riemannian metric is invariant with respect to rotations about the origin, and both the metric and $\mathtt{B}$'s part of the chosen edge from $\Gamma$ are invariant with respect to rotations of the $(x_1,x_2)$ plane.  This being the case, it is sufficient for what follows to take $y$ to be a point where $x_2$ is zero and $x_1$ is positive.  Let $\Pi$ denote the part of $\mathtt{B}$ where $x_1\le 0$ and $x_2=0$ (this is $\mathtt{B}$'s intersection with a half-plane in $\mathbb{R}^3$).

As explained in the following, the endomorphism $G_{\mathtt{B}}(\cdot,y)$ in this case (when $\mathbb{V}=\mathbb{R}$) must vanish on $\Pi$.  To see why, note first that $\mathcal{I}$ can be identified with the product line bundle on $\mathtt{B}-\Pi$ with the identification such that $G_{\mathtt{B}}>0$ near $y$.  Now take $u_{y,r}$ as before and let $g_{y,r}$ denote the function on $\mathtt{B}-\Pi$ that obeys the conditions below:
\begin{itemize}[leftmargin=12pt]
\item $\quad -\Delta^{\perp} g_{y,r} = u_{y,r},$
\item $\quad g_{y,r}\big|_{\partial\mathtt{B}} = 0 \quad\text{and}\quad g_{y,r}\big|_{\Pi}=0.$
\listeqno\label{eq:6.13}
\end{itemize}
This function is positive on $\mathtt{B}-\Pi$ and, by symmetry (the symmetry of the Laplacian when $x_2$ is replaced by $-x_2$), the $x_2$-derivative of $g_{y,r}$ at any given point on the $x_1<0$ part of $\Pi$ coming from positive $x_2$ side must be $-1$ times the $x_2$-derivative at that same point coming from the negative $x_2$ side.  It follows from this that $g_{y,r}$ can be viewed as a smooth section of $\mathcal{I}$ on $\mathtt{B}-(\mathtt{B}\cap\Gamma)$.  And, as such, this section of $\mathcal{I}$ obeys the same equation and same Dirichlet boundary conditions on $\partial\mathtt{B}$ as $G_{\mathtt{B}}(\varsigma)$ with $\varsigma = u_{y,r}t_y$ where $t_y$ here denotes the unit length section of $\mathcal{I}$ on $\mathtt{B}-\Pi$ that is mapped to $1\in\mathbb{R}$ by the aforementioned isometry between $\mathcal{I}$ and the product $\mathbb{R}$-bundle. Thus, \(g_{y,r}=G_{\mathtt{B}}(\varsigma)\). Since this holds for all \(r\), it follows that
\begin{equation}\label{eq:whatever-label}
\lim_{r\to 0}\frac{1}{r^{3}}\,g_{y,r}=G_{\mathtt{B}}(\cdot,y).
\end{equation}
Thus, $G_{\mathtt{B}}(\cdot,y)$ vanishes on $\Pi$ as claimed (and nowhere else in $\mathtt{B}-\Gamma)$.  (The preceding argument for the vanishing of $G_{\mathtt{B}}(\cdot,y)$ on $\Pi$ is a simplistic version of a very clever construction in \cite{FS}.)

\smallskip
\noindent\underline{\textit{Step 5:}}
Continue here with the assumptions and notation of the previous step.  Let $\mathtt{B}' \subset \mathtt{B}$ denote the ball centered at the origin in $\mathbb{R}^3$ with radius equal to $\frac{1}{100}\,\operatorname{dist}(y,\Gamma)$.  Let $\rho$ and $\theta$ denote the standard polar coordinates on the $(x_1,x_2)$ plane in $\mathbb{R}^3$ with the $\theta=0$ ray being the positive $x_1$ axis. Let $u=\sqrt{\rho}\,\cos(\tfrac12\theta)$.  This is a positive function on $\mathtt{B}-\Pi$ that extends continuously over $\Pi$ so as to be zero on $\Pi$.  It is also harmonic on $\mathtt{B}-\Pi$ with respect to the Euclidean metric on $\mathtt{B}$.  With respect to the metric from the sphere (which is conformal to the Euclidean metric), this function is such that
\begin{equation}\label{eq:6.14}
\Delta^{\perp}u \le 0 \, .
\end{equation}
Hold onto $u$ for the moment.  In the meantime, fix a very small but positive number to be denoted by $\varepsilon$ and let $w_{\varepsilon}$ denote the $\Delta^{\perp}$-harmonic, $\mathbb{R}$-valued function on the spherical annulus in $\mathtt{B}$ where $\rho$ is between $\varepsilon$ and $\frac{1}{100}\operatorname{dist}(y,\Gamma)$ subject to the boundary conditions that $w_{\varepsilon}$ vanish both on $\partial\mathtt{B}$ and where $\rho=\frac{1}{100}\operatorname{dist}(y,\Gamma)$, and that $w_{\varepsilon}$ is equal $1$ where $\rho=\varepsilon$.  This function $w$ is non-negative and it obeys the bound
\begin{equation}\label{eq:6.15}
w_{\varepsilon}(x) \;\le\;
c_0\,\ln\!\biggl(\frac{\operatorname{dist}(y,\Gamma)}{100\rho}\biggr)\bigg/
\ln\!\biggl(\frac{\operatorname{dist}(y,\Gamma)}{100\varepsilon}\biggr) \, .
\end{equation}
(The bounds in \eqref{eq:6.8} can be used to prove the preceding bound by bounding $w_{\varepsilon}$ by a weighted integral of \eqref{eq:6.8}'s Green function.)

By virtue of what is said above about $u$ and $w$, and by virtue of the fact that $G_{\mathtt{B}}(x,y)$ is bounded by $c_0\,\frac{1}{\operatorname{dist}(x,y)}$, and by virtue of the fact that $G_{\mathtt{B}}(\,\cdot\,,y)$ vanishes on $\Pi$, there exists $c \in (1,c_0)$ such that the function
\begin{equation}\label{eq:6.16}
h_{\varepsilon}
= c\,\frac{1}{\operatorname{dist}(y,\Gamma)}\,(u+w_{\varepsilon}) .
\end{equation}
is larger than $G_{\mathtt{B}}(\,\cdot\,,y)$ on the boundary of the complement in $\mathtt{B}^\prime-(\mathtt{B}^\prime\cap\Pi)$ to the cylinder where $\rho=\varepsilon$.  This being the case, the maximum/comparison principle can be brought to bear using \eqref{eq:6.14} and the fact that $w_{\varepsilon}$ is harmonic to see that $|G_{\mathtt{B}}(\,\cdot\,,y)| \le c_0 h_{\varepsilon}$ on the complement of that same cylinder in $\mathtt{B}'-(\mathtt{B}'\cap\Pi)$.  The claim in Item~a) of the second bullet of the lemma for the case where $\mathbb{V}$ is the product $\mathbb{R}$-bundle
and $\Gamma$'s edges are geodesic arcs follows directly from the latter fact by taking $\varepsilon \to 0$ while noting that the right hand side of \eqref{eq:6.15} limits to zero at fixed $\rho$ as $\varepsilon \to 0$.\\

\noindent\underline{\textit{Step 6:}}
This step finishes the proof of Item~b) of the lemma's second bullet under the assumption that the edges of $\Gamma$ are geodesic arcs.  To this end, first fix an identification between $\mathcal{I}|_y$ and the product $\mathbb{R}$–bundle as before.  Next, choose an orthonormal basis for $\mathbb{V}$ at the point $q$ and parallel transport that chosen basis using the metric's Levi-Civita connection along the radial geodesic arcs from $q$ so as to obtain a basis for $\mathbb{V}|_{\mathtt{B}}$.  Denote this basis as $\{t_1,\dots,t_{\dim(\mathbb{V})}\}$.  Now use this basis to write $G_{\mathtt{B}}(x,y)$ as
\begin{align}
\sum G_{\alpha\beta}(x,y)\, t_\alpha|_x \otimes t_\beta|_y
\end{align}
with it understood that the summation is over the labels $\alpha,\beta$, which come from the index set for $\mathbb{V}$'s orthonormal basis.  Each $G_{\alpha\beta}(\,\cdot\,,y)$ is a section of $\mathcal{I}$ over $\mathtt{B}$ that vanishes on $\partial\mathtt{B}$. As such, it obeys an equation on the complement of $y$ that has the schematic form
\begin{equation}\label{eq:6.17}
-\Delta^{\perp} G_{\alpha\beta}
+ Q_{\alpha\gamma}\cdot \nabla G_{\gamma\beta}
+ P_{\alpha\gamma} G_{\gamma\beta}
= 0
\end{equation}
where the repeated index ($\gamma$ in this case) is implicitly summed over that same index set and where $y$ is fixed with differentiation on the $x$ coordinate. (The tensor valued functions $Q$ and $P$ are smooth on a neighborhood of $\mathtt{B}$ in $S^3$ with uniform derivative bounds to any given order.) It follows as a consequence of \eqref{eq:6.17} that $G_{\alpha\beta}(x,y)$ with $x\ne y$ can be written using the Green's function for $-\Delta^{\perp}$ acting on sections of the line bundle $\mathcal{I}$ (here denoted by $G$) as \begin{equation}\label{eq:6.18}
G_{\alpha\beta}(x,y)
= G(x,y)\,\delta_{\alpha\beta}
-\int_{\mathtt{B}}
\left.\Bigl( G(x,\cdot)\bigl(Q_{\alpha\gamma}\cdot\nabla G_{\gamma\beta}
+ P_{\alpha\gamma}G_{\gamma\beta}\bigr)\Bigr)\right|_{(\cdot,y)}.
\end{equation}
Hold on to this identity for the moment.

Let $\mathtt{B}'\subset\mathtt{B}$ again denote the ball centered at the origin in $\mathbb{R}^3$ with radius equal to $\frac{1}{100}\operatorname{dist}(y,\Gamma)$. Let $\mathtt{B}''$ denote the ball concentric to $\mathtt{B}$ with radius $\frac{1}{500}\operatorname{dist}(y,\Gamma)$. Supposing that $x$ is from $\mathtt{B}''$, an appeal to the $\mathbb{V}=\mathbb{R}$ version of Lemma~6.3 (whose proof was just finished) implies the inequality below:
\begin{equation}\label{eq:6.19}
\begin{aligned}
|G_{\alpha\beta}(x,y)|
\le\;&
c_0\,\frac{\operatorname{dist}(x,\Gamma)^{1/2}}{\operatorname{dist}(y,\Gamma)^{3/2}}
\left(
1+\int_{\mathtt{B}-\mathtt{B}''}
\bigl(|\nabla G_{\mathtt{B}}|_{(\cdot,y)} + |G_{\mathtt{B}}|_{(\cdot,y)}\bigr)
\right) \\
&\quad
+\,c_0\int_{\mathtt{B}''} |G|_{(x,\cdot)}
\bigl(|\nabla G_{\mathtt{B}}|_{(\cdot,y)} + |G_{\mathtt{B}}|_{(\cdot,y)}\bigr) .
\end{aligned}
\end{equation}

As for these integrals:  An appeal to Items~a) and~c) of the lemma's second bullet for the case when $\mathbb{V} = \mathbb{R}$ (and thus $\mathbb{V}\otimes \mathcal{I} \cong \mathcal{I}$) supply the bounds asserted below:
\begin{itemize}
\item $\int_{\mathtt{B}-\mathtt{B}''}
\bigl(|\nabla G_{\mathtt{B}}|_{(\cdot,y)} + |G_{\mathtt{B}}|_{(\cdot,y)}\bigr)
\le c_0,$
\smallskip
\item $\int_{\mathtt{B}'} |G|_{(x,\cdot)}
\bigl(|\nabla G_{\mathtt{B}}|_{(\cdot,y)} + |G_{\mathtt{B}}|_{(\cdot,y)}\bigr)
\le c_0\,\operatorname{dist}(x,\Gamma)^{1/2}.$
\listeqno\label{eq:6.20}
\end{itemize}

And, using these bounds in \((6.19)\) finishes the proof of Item~2 of the lemma's second bullet for the case when $\Gamma$'s edges are geodesic arcs.

\medskip
\noindent\textit{Part 4:} This part of the section proves Lemma~6.1.

\medskip
\noindent\textbf{Proof of Lemma 6.1:}
To start the proof, let $q$ denote a given point in the interior of a given edge of $\Gamma$.  Use $q$ to define the number $r_q$ as done at the beginning of this section and let $\mathtt{B}$ denote the ball of radius $\frac{1}{2c}\,r_q$ centered at $q$. Use $\chi_{1q}$ to denote (for the purposes of this proof) a smooth function with values in $[0,1]$ that equals $1$ in the radius $\frac{3}{8c}\,r_q$ ball centered at $q$ and equals $0$ near the boundary of $\mathtt{B}$. This function can and should be constructed so that the norm of its differential and Hessian are bounded respectively by $c_0\,\frac{1}{r_q}$ and by $c_0\,\frac{1}{r_q^2}$.

Let $\psi_q$ denote $\chi_q\psi$, this being a section of $\mathbb{V}\otimes\mathcal{I}$ on $\mathtt{B}-(\mathtt{B}\cap\Gamma)$ with compact support on $\mathtt{B}$.  It is in the Hilbert space $\mathbb{H}_{\mathtt{B}}$ (this subspace of $\mathbb{H}$ is defined at the start of Part~2).  By virtue of \eqref{eq:6.4}, the section $\psi_q$ obeys an equation that can be written schematically as
\begin{equation}\label{eq:6.21}
\nabla^\dagger\nabla \psi_q
=
\mathtt{E}\,\psi_q
-\mathcal{R}\cdot\psi_q
+2\,d\chi_q\cdot\nabla\psi
+(\Delta^{\perp}\chi_q)\,\psi \, .
\end{equation}
where $d\chi_q\cdot\nabla\psi$ denotes the directional covariant derivative of $\psi$ along the gradient of the function $\chi_q$.

With \eqref{eq:6.21} in hand, use the Green's function $G_{\mathtt{B}}$ from Lemma~6.3 to depict $\psi_q$ at a point $x$ in $\mathtt{B}$ as done below:
\begin{equation}\label{eq:6.22}
\psi_q(x)
=
\int_{\mathtt{B}} G_{\mathtt{B}}(x,\cdot)\Bigl(
\mathtt{E}\,\chi_q\psi
-\mathcal{R}\cdot(\chi_q\psi)
+2\,d\chi_q\cdot\nabla\psi
+(\Delta^{\perp}\chi_q)\psi
\Bigr).
\end{equation}
To exploit the preceding identity, suppose that $x$ is in the radius $\frac{1}{4c}\,r_q$ ball centered at $q$.  This is inside the set where $\psi_q=\psi$.  Moreover, the distance from $x$ to the support of $|d\chi_q|$ is no less than $\frac{1}{8c}\,r_q$.  (This last fact is important when it comes to analyzing the contributions to $|\psi(x)|$ by the right-most two terms in
\eqref{eq:6.22}.)  Introducing $r$ to denote the distance from $x$ to $\Gamma$, let $\mathtt{B}'$ denote the part of $\mathtt{B}$ where the distance to $\Gamma$ is less than $4r$, and let $\mathtt{B}''$ denote the complementary part of $\mathtt{B}$.  By virtue of the inequality from Item~a) of the second bullet in Lemma~6.3, the contribution to $|\psi(x)|$ from $\mathtt{B}'$ is at most
\begin{equation}\label{eq:6.23}
c_0
\left(\int_{\mathtt{B}'} \frac{1}{\operatorname{dist}(x,\cdot)^2}\right)^{1/2}
\left(\left(1+\mathtt{E}+\frac{1}{r_q}\right)\|\psi\|_{\mathbb{H}}
+\frac{1}{r_q^2}\,\|\psi\|_{\mathbb{L}}\right).
\end{equation}
This in turn is bounded courtesy of \eqref{eq:6.3} by
$c_0\left(1+\frac{1}{r_q}\right)(1+\mathtt{E})\sqrt{r}$.  Meanwhile, by virtue of
the inequality from Item~b) of the second bullet in Lemma~6.4 and \eqref{eq:6.3},
the contribution to $|\psi(x)|$ from $\mathtt{B}''$ is at most
\begin{equation}\label{eq:6.24}
c_0\sqrt{r}\left(\int_{\mathtt{B}''}
\frac{1}{\operatorname{dist}(\Gamma,\cdot)^{3/2}}(1+\mathtt{E})|\psi|
+\frac{1}{r_q^{1/2}}(1+\mathtt{E})\right).
\end{equation}
To elaborate: The contribution to \eqref{eq:6.24} from the terms in \eqref{eq:6.22}'s integral with derivatives on $\chi_q$ is accounted for by the right most term in \eqref{eq:6.24} (the term with the factor of $r_q^{-1/2}$). The contribution to \eqref{eq:6.24} from \eqref{eq:6.22}'s integral from its two $\chi_q\psi$ terms is accounted for by the term with the explicit integral in \eqref{eq:6.24}.  With regards to that integral: The third and fourth bullets in \eqref{eq:6.3} bounds that integral by $c_0(1+\mathtt{E})$.  As a consequence, the expression in \eqref{eq:6.24} is bounded in turn by $c_0\sqrt{r}\left(1+\frac{1}{\sqrt{r_q}}\right)(1+\mathtt{E})$.

Since $r=\operatorname{dist}(x,\Gamma)$, the $c_0\sqrt{r}\left(1+\frac{1}{r_q}\right)(1+\mathtt{E})$ bound for both the $\mathtt{B}'$ and $\mathtt{B}''$ contributions to the integral on the right hand side of \eqref{eq:6.24} prove Lemma~6.1.\\

\noindent\textit{Part 5:} This last part of the section justifies the assertion of Lemma~6.2.\\

\noindent\textbf{Proof of Lemma 6.2:}
When $\varepsilon$ is from the interval $\left(0,\frac{1}{256c}\,r_q\right)$, let $u_{\varepsilon}$ denote the isometry of $S^3$ that is obtained by integrating the vector field $\frac{\partial}{\partial t}$ for time $\varepsilon$.  Then, let $\psi_{\varepsilon}$ denote the pull-back of $\psi$ by this isometry.  This $\psi_{\varepsilon}$ is a section of the pull-back bundle $u_{\varepsilon}^*(\mathbb{V}\otimes\mathcal{I})$.  Use parallel transport along the integral curves of $\frac{\partial}{\partial t}$ to identify this bundle over the radius $\frac{1}{4c}\,r_q$ ball centered at $q$ with $\mathbb{V}\otimes\mathcal{I}$. Both $\psi_{\varepsilon}$ and $\psi$ are sections of bundle $\mathbb{V}\otimes\mathcal{I}$ over that same ball, and as such, they both obey $(\nabla^\dagger\nabla+\mathcal{R})(\cdot)=\mathtt{E}(\cdot)$.  Also, the integrals of $|\nabla\psi|^2$ and $|\nabla\psi_{\varepsilon}|^2$ over that ball are both bounded by $c_0\rho^{-2}$ times the integral of $|\psi|^2$ over $\mathtt{B}$.  That in turn is bounded by $c_0\rho^{3}$ (see Lemma~6.1).  Thus, the integrals over $\mathtt{B}$ of both $|\nabla\psi|^2$ and $|\nabla\psi_{\varepsilon}|^2$ are bounded by $c_0\rho$.

The plan for what follows is to examine the $\varepsilon\to 0$ behavior of $(\psi_{\varepsilon}-\psi)$ inside the ball $\mathtt{B}$.  This section is denoted by $\eta_{\varepsilon}$ in what follows; and it obeys $(\nabla^\dagger\nabla+\mathcal{R})\eta_{\varepsilon}=\mathtt{E}\eta_{\varepsilon}$. To localize the subsequent analysis to $\mathtt{B}$, reintroduce the function $\chi_q$ and take the inner product of the equation $(\nabla^\dagger\nabla+\mathcal{R})\eta_{\varepsilon}=\mathtt{E}\eta_{\varepsilon}$ with $\chi_q^{2}\eta_{\varepsilon}$; then integrate the resulting function over $\mathtt{B}$ (where it has compact support).  An instance of integration by parts with \eqref{eq:6.3} leads to the inequality below
\begin{equation}\label{eq:6.25}
\int_{\mathtt{B}} \bigl|\nabla(\chi_q\eta_{\varepsilon})\bigr|^{2}
\le c_0\,r_q^{-2}\int_{\mathtt{B}} |\eta_{\varepsilon}|^{2}\,.
\end{equation}

The key observation now is that $\varepsilon^{-1}(\psi_{\varepsilon}-\psi)$ converges in the $\mathbb{L}$-topology to $\nabla_t\psi$. Hence, the product of $\varepsilon^{-2}$ times the right hand side of \eqref{eq:6.25} converges as $\varepsilon\to0$ to $c_0 r_q^{-2}$ times the integral over $\mathtt{B}$ of $|\nabla_t\psi|^2$. It follows as a consequence this (and the dominated convergence theorem) and that $\chi_q\nabla_t\psi$ is a limit of a uniformly $\mathbb{H}$-norm bounded sequence of elements in $\mathbb{H}$ thus that it is in $\mathbb{H}$ also. Moreover, given the bound in the top bullet of \eqref{eq:6.3}, taking $\varepsilon\to0$ in \eqref{eq:6.25} leads to an \emph{a priori} bound for the $\mathbb{H}$-norm of $\chi_q\nabla_t\psi$, this being $c_0 r_q^{-1}$ times the square root of the integral over $\mathtt{B}$ of $|\nabla_t\psi|^2$. This last fact implies Lemma~6.2's asserted bound because the square root of
the integral over $\mathtt{B}$ of $|\nabla_t\psi|^2$ has a $c_0 r_q^{1/2}$ upper bound courtesy of \eqref{eq:6.3} and Lemma~6.1.

\section{Almgren's Frequency Function}

This section presents various technical observations and constructions that are also needed for the proofs of Lemmas~4.2 and~4.4.  Some of these technical observations and constructions are used again in later parts of the proof of Theorem~1.2.

Assume at the outset that a graph in $S^3$ has been specified (it is again denoted by $\Gamma$); and that it has even valencies so that there is a corresponding real line bundle $\mathcal{I}$ on $S^3-\Gamma$.  Likewise, assume that a vector bundle $\mathbb{V}$ on $S^3$ has been specified with fiber metric and metric compatible connection.  These are fixed throughout the subsequent discussion.  As before, $\mathbb{H}$ denotes the Hilbert space completion of the space of smooth sections of $\mathbb{V}\otimes\mathcal{I}$ with compact support in $S^3-\Gamma$ using the norm whose square is depicted in \eqref{eq:4.1}.  Also as before, use $\mathbb{L}$ to denote the completion of that same space of sections using the norm whose square is depicted in \eqref{eq:4.3}.  Supposing that $\mathcal{R}$ denotes a given endomorphism of $\mathbb{V}$ (as in \eqref{eq:4.4}), the constructions in what follow concern $(\nabla^\dagger\nabla+\mathcal{R})$-eigensections from the Hilbert space $\mathbb{H}$.

What follows next is a digression to further set the stage and notation for the definition of the frequency function.  To start the digression, suppose that $\psi$ again denotes a $(\nabla^\dagger\nabla+\mathcal{R})$-eigensection from $\mathbb{H}$ and let $\mathtt{E}$ again denote its eigenvalue.  For any given point $p\in S^3$ and $r\in \bigl(0,\frac{1}{1000}\pi\bigr]$, use $\mathtt{B}_r$ to denote the radius $r$ ball centered at $p$ and use $\partial\mathtt{B}_r$ to denote $\mathtt{B}_r$'s boundary. With $p$ fixed in advance, a positive function on $\bigl(0,\frac{1}{1000}\pi\bigr]$ to be denoted by $\K$ is defined by the rule below whereby
\begin{equation}\label{eq:7.1}
\K^2(r) \;=\; \frac{1}{A(r)}\int_{\partial\mathtt{B}_r} |\psi|^2,
\end{equation}
where $A(r)=4\pi\sin^2(r)$ is the area $\partial\mathtt{B}_r$.  Thus, $\K^2$ is the average of $|\psi|^2$ over the radius $r$ ball centered at $p$.  This function $\K$ is a differentiable function on $\bigl(0,\frac{1}{1000}\pi\bigr]$ as an integration by parts (which is allowed if $\psi\in\mathbb{H}$) can be used to write its derivative as below:
\begin{equation}\label{eq:7.2}
\frac{d\K}{dr}
\;=\;
\frac{1}{A(r)}\int_{\mathtt{B}_r}
\Bigl(|\nabla\psi|^2+\langle\psi,\mathcal{R}\psi\rangle-\mathtt{E}|\psi|^2\Bigr)\,.
\end{equation}

As for the $r\to 0$ limit of $\K$: The function $\K$ extends continuously to $r=0$ except possibly in the cases where the chosen point $p$ is in $\Gamma$.  In any event, if $|\psi|$ is continuous at a chosen point $p$, then $\K_p(0)=|\psi|(p)$.

Fix a point $p\in S^3$.  Almgren's frequency function is defined from the corresponding version of $\K$:  This frequency function is denoted by $\N$; it is defined (initially) on $\left(0,\frac{1}{1000}\pi\right]$ by writing the derivative of the function $\K$ as
\begin{equation}\label{eq:7.3}
\frac{d\K}{dr} \;=\; \frac{\N(r)}{r}\,\K(r)\,.
\end{equation}

It follows from \eqref{eq:7.2} that $\N$ is a continuous function on
$\left(0,\frac{1}{1000}\pi\right]$ that can be depicted as the two bullets
below:
\begin{itemize}
    \item 
$\N(r)
= \frac{r}{2A \K^2}\int_{\partial\mathtt{B}_r}
\Bigl(\langle \psi,\nabla_r\psi\rangle + \langle \nabla_r\psi,\psi\rangle\Bigr),$
    \item 
$\N(r)
= \frac{r}{A \K^2}\int_{\mathtt{B}_r}
\Bigl(|\nabla\psi|^2+\langle\psi,\mathcal{R}\psi\rangle-\mathtt{E}|\psi|^2\Bigr)$.
\listeqno\label{eq:7.4}
\end{itemize}
The first line is obtained by directly differentiating K and the second follows from the first using the divergence theorem (integration by parts). What follows are two lemmas that states key properties of the function $\N$ and then a third lemma that states a key property of the function $\K$.
\begin{lemma}\label{lem:7.1}
There exists $\kappa>1$ such that the assertions below hold for any $(\nabla^\dagger\nabla+\mathcal{R})$-eigensection from $\mathbb{H}$. Fix a point in $S^3$ to define the functions $\K$ and then $\N$ from the given eigensection.  The chosen point's version of function $\N$ obeys
\begin{itemize}
\item $\N(r)\ge -\kappa(1+\mathtt{E})\,r^2$.

\item The function $\N$ has a unique $r\to 0$ limit; and with regards to that limit,
  \begin{enumerate}
  \item[(a)] $\lim_{r\to 0}\N(r)=0$ if $\psi$ is non-zero at the chosen point.
  \item[(b)] $\lim_{r\to 0}\N(r)>0$ if the chosen point is in a vertex of $\Gamma$.
  \item[(c)] $\lim_{r\to 0}\N(r)\ge \frac{1}{2}$ if the chosen point is in the interior
            of an edge of $\Gamma$.
  \end{enumerate}

\item Supposing that $\{p_n\}\subset S^3$ converges to a chosen point in the interior of an edge of $\Gamma$.  Then that chosen point's version of $\N(0)$ is no smaller than the $\limsup$ of the respective $p_n$ versions of $\N(0)$.

\item Suppose that $p$ denotes a vertex of $\Gamma$ and suppose that the versions of $\N(0)$ for points on the interiors of the edges of $\Gamma$ in a neighborhood of $p$ are all strictly greater than $\frac{1}{2}$.  Then $p$'s version of $\N(0)$ is not less than the $\limsup$ of the versions of $\N(0)$ that are defined by any sequence of edge points from $\Gamma$ that converges to $p$.
\end{itemize}
\end{lemma}
Fix a basepoint $p\in S^3$.  Let $\N_p(r)$ denote the frequency function defined by \eqref{eq:7.3}, with the integrals taken over the geodesic ball $B_r(p)$ and its boundary $\partial B_r(p)$.  By Lemma~7.1, $\N_p(r)$ has a limit as $r\to 0$; we set
\[
\N_p(0):=\lim_{r\to 0} \N_p(r).
\]
When $p$ is fixed, we often suppress the subscript and write $\N(r)$ and $\N(0)$. We write $\N_{(\cdot)}(0)$ for the resulting function $p\mapsto \N_p(0)$.

The proof of this lemma is given momentarily.

\medskip

The second lemma refers to numbers from the set $\{r_p: p\in S^3\}$ and the positive number $c$ that are described just prior to \eqref{eq:6.1}, and then in \eqref{eq:6.1} and Figure \ref{eq:6.2}.  What follows is the promised second lemma about the function $\N$.

\begin{lemma}\label{lem:7.2}
\emph{Let $\psi$ denote a $(\nabla^\dagger\nabla+\mathcal{R})$-eigensection from $\mathbb{H}$ with $\mathbb{L}$-norm equal to $1$.  There exists $\kappa>0$ with the following significance: Supposing that $p$ denotes a given point in $S^3$ and supposing that $r\in\bigl(0,\frac{1}{8c}\,r_p\bigr)$, then the corresponding version of $\N$ is almost everywhere differentiable on $\bigl(0,\frac{1}{8c}\,r_p\bigr)$; and its derivative (where $\N$ is differentiable) obeys the inequality below:}
\[
\frac{d\N}{dr}
\;\ge\;
\frac{1}{A \K^2}\int_{\partial \mathtt{B}_r}
\left|\nabla_r\psi-\frac{\N}{r}\psi\right|^2
\;-\;
\kappa(1+\N)\,r \, .
\]
\emph{Moreover, if $\rho>0$, and if $\N_{(\cdot)}(0)>\frac{1}{2}$ at each point from the part of $\Gamma$ inside the radius $\rho$ ball centered at a vertex of $\Gamma$, and if $p$ has distance at most $\frac{1}{2}\rho$ from that vertex, then the preceding inequality holds at any differentiable value for the function $\N$ from the whole interval $(0,\kappa^{-1}\rho)$.}
\end{lemma}
The proof of this lemma is also given momentarily.

\medskip

The next lemma states some facts about the function $\K$.

\begin{lemma}\label{lem:7.3}
\emph{There exists $\kappa>1$ with the following significance: Let $\psi$ denote a $(\nabla^\dagger\nabla+\mathcal{R})$-eigensection from $\mathbb{H}$ with $\mathbb{L}$-norm equal to $1$.  Fix a point from $\Gamma$ (denoted by $p$) to define the corresponding functions $\K$ and $\N$.  If $r\in\bigl(0,\frac{1}{4c}\,r_p\bigr)$, then
\[
\K(r)\le \kappa\, z_p\, r^{\N(0)}
\]
with $z_p$ being independent of $r$.  As a consequence,
\[
\int_{\partial\mathtt{B}_r} |\psi|^2 \le 4\pi\kappa\, z_p\, r^{2+2\N(0)}
\qquad\text{and}\qquad
\int_{\mathtt{B}_r} |\psi|^2 \le \frac{4\pi}{3}\,\kappa\, z_p\, r^{3+2\N(0)}.
\]
With regards to $z_p$: Suppose that $p$ is in the interior of an edge from $\Gamma$. Let $q$ denote the nearest vertex in $\Gamma$ to $p$ and let $\N_q(0)$ denote $q$'s version of the number $\N(0)$.  Define $\bar \N$ to be the maximum of $0$ and $\N(0)-\N_q(0)$.  Then, $z_p \le \kappa\, r_p^{-\bar \N}$.}
\end{lemma}
\noindent\textbf{Proof of Lemma 7.3.}
Fix $p\in\Gamma$ and integrate \eqref{eq:7.3} between $r$ and $\frac{1}{4c}\,r_p$. Regarding that integral, Lemma~7.2 lead to the inequality
\begin{equation}\label{eq:7.5}
\K(r)\le c_0\,\K\!\left(\frac{1}{4c}\,r_p\right)\left(\frac{1}{r_p}\right)^{\N(0)} r^{\N(0)}.
\end{equation}
Let $z_p$ now denote $\K\!\left(\frac{1}{4c}\,r_p\right)\left(\frac{1}{r_p}\right)^{\N(0)}$. The lemma's integral inequalities follow from \eqref{eq:7.5}.

\smallskip
To see about a bound for $z_p$: Integrate \eqref{eq:7.3} again to see that $\K(\rho)\ge c_0^{-1} z_p\,\rho^{\N(0)}$ for any given $\rho\in\bigl[\frac{1}{4c}\,r_p,\frac{1}{2c}\,r_p\bigr)$.  This implies in turn that
\begin{equation}\label{eq:7.6}
\int_{\mathtt{B}_{\frac{r_p}{2c}}-\mathtt{B}_{\frac{r_p}{4c}}} |\psi|^2
\ge c_0^{-1}\,z_p\, r_p^{3+2\N(0)}.
\end{equation}
If $\frac{1}{2c}r_p$ is greater than $\frac{1}{1000\,c}$ times the shortest edge length, then \eqref{eq:7.6} will run afoul of the $\|\psi\|_{\mathbb{L}}=1$ constraint in the event that $z_p\ge c_0$.

To see about the case when $r_p$ is relatively small, let $q$ denote the closest vertex to $p$.  For any given positive number $\rho$, let $\mathtt{B}'_{\rho}$ denote the radius $\rho$ ball centered at $q$.  It then follows from \eqref{eq:7.6} using \eqref{eq:7.3} that the inequality below holds for $\rho=\frac{1}{c}r_p$:
\begin{equation}\label{eq:7.7}
\int_{\mathtt{B}'_{\rho}} |\psi|^2
\ge c_0^{-1}\,z_p\,\rho^{3+2\N(0)}.
\end{equation}

Let $\N'(0)$ denote the vertex $q$'s version of the number $\N(0)$.  Granted \eqref{eq:7.7}, use \eqref{eq:7.3} with $q$ replacing $p$ to see that
\begin{equation}\label{eq:7.8}
\int_{\mathtt{B}'_{r_q}} |\psi|^2
\ge c_0^{-1}\,z_p\left(\frac{r_q}{r_p}\right)^{3+2\N'(0)} r_p^{3+2\N(0)}.
\end{equation}

Since $r_q\ge c_0^{-1}$, this bound runs afoul of the $\|\psi\|_{\mathbb{L}}=1$ constraint if $z_p>c_0/r_p^{\bar \N}$ with $\bar \N$ denoting the maximum of $0$ and $\N(0)-\N'(0)$.

\medskip
\noindent\textbf{Proof of Lemma 7.1.}
The assertion made by the lemma's first bullet follows from the uniform $L^\infty$ bound for $|\psi|$ from the third bullet of \eqref{eq:6.3}.  Since other assertions of the lemma invoke the conclusions of Lemma~7.2, assume that lemma is true for now.

With regards to the assertion of the second bullet of Lemma~7.1 regarding a $r\to 0$ limit for $\N$: Lemma~7.2 says in effect that the derivative of $\N$ is no smaller than $-c_0(1+\N)r$ on an interval of the form $\bigl(0,\frac{1}{2c}r_p\bigr)$ or $\bigl(0,c_0^{-1}\rho\bigr)$ as the case may be.  Integration of this derivative inequality implies that a function of the form
\begin{equation}\label{eq:7.9}
r \longmapsto e^{c_0 r}\N(r)+c_0 r^2
\end{equation}
is non-decreasing on that same interval.  This non-decreasing property implies in turn the existence of a unique $r\to 0$ limit for $\N$ since $\N$ is bounded from below.

Granted that there is a unique $r\to 0$ limit for $\N$, then Item~a) of Lemma~7.1's second bullet follows by integrating~\eqref{eq:7.3}. Given what is said by Lemma~6.1, a completely analogous application of~\eqref{eq:7.3} leads directly to the conclusions of Item~c) of Lemma~7.1.  To see about Item~b) of that bullet, note that the bound in~\eqref{eq:4.2} holds with the integrals on both sides restricted to the ball of radius $r$ centered on the given vertex of $\Gamma$.  Given the formula in~\eqref{eq:7.4}, that bound leads to the lower bound below which holds when $r$ is sufficiently small:
\begin{equation}\label{eq:7.10}
\N(r)\ge c_0^{-1}\,\frac{r}{A\,\K^2}\int_0^r \K^2(s)\,ds.
\end{equation}
If the $r\to 0$ limit of $\N(r)$ is greater than $1$, then there is nothing to prove.  If that limit is less than $1$, then~\eqref{eq:7.3} implies that $\K(s)\ge \tfrac12\,\K(r)$ for $s$ between $\tfrac12 r$ and $r$; and then \eqref{eq:7.10} leads directly to a $c_0^{-1}$ lower bound for $\N(r)$ when $r$ is small.

With regards to the third bullet of Lemma~7.1: Let $p$ denote the limit point inside that edge of $\Gamma$.  Fix small value for $r$ and let $\mathtt{B}_r$ denote the ball of radius $r$ centered at $p$.  Given a point $q$ from that same edge of $\Gamma$, let $\mathtt{B}'_r$ denote the ball of radius $r$ centered at $q$.  A standard measure-theoretic fact is this: Given $\varepsilon>0$ but very small, there exists $\delta>0$ such that if $q$ is in that same edge and $\operatorname{dist}(p,q)<\delta$, then
\begin{equation}\label{eq:7.11}
\begin{aligned}
&\text{\textbullet}\quad
\int_{\mathtt{B}_r}\Bigl(|\nabla\psi|^2+\langle\psi,\mathcal{R}\psi\rangle
-\mathtt{E}|\psi|^2\Bigr)
\ \text{and}\ 
\int_{\mathtt{B}'_r}\Bigl(|\nabla\psi|^2+\langle\psi,\mathcal{R}\psi\rangle
-\mathtt{E}|\psi|^2\Bigr)
\ \textit{differ by at most }\varepsilon,
\\
&\text{\textbullet}\quad
\frac{1}{A(r)}\int_{\partial\mathtt{B}_r}|\psi|^2
\ \text{and}\ 
\frac{1}{A(r)}\int_{\partial\mathtt{B}'_r}|\psi|^2
\ \textit{differ by at most }\varepsilon.
\end{aligned}
\end{equation}

For that same $q$ as above, let $\N'(r)$ denote $q$'s version of $\N$.  It follows from the formula in the second bullet of~\eqref{eq:7.4} using~\eqref{eq:7.11} that if $\delta$ is sufficiently small (with $r$ fixed), then $\N(r)$ and $\N'(r)$ differ by at most $\varepsilon$.  Keeping this in mind, suppose that $r$ is small enough so that $\N(r)$ and $\N(0)$ also differ by at most $\varepsilon$.  Meanwhile, $\N'(r)$ is no smaller than $e^{-c_0\varepsilon}\N'(0)-c_0\varepsilon$ if $r$ is small and $q$ is very close to $p$.  (See what is said regarding \eqref{eq:7.9} in the proof above for the lemma's second bullet.  Note in this regard, that if $p$ is in the interior of an edge, and if $q$ is very close to $p$, then the distance from $p$ to $q$ will be very much smaller than $\frac{1}{2c}\,r_q$.)  Thus, if $r$ small to begin with and then $q$ is sufficiently close to $p$, then it follows that $\N(0)$ is no smaller than $e^{-c_0\varepsilon}\N'(0)-c_0\varepsilon$.  Since $\varepsilon$ can be any small positive number, this last inequality establishes the assertion from the lemma's third bullet.

With regards to the fourth bullet's assertion: A key input for the argument given above for the third bullet was the assertion that the function depicted in~\eqref{eq:7.9} is non-decreasing as long as $r$ is less than $\frac{1}{2c}\,r_p$.  That upper bound for $r$ can be taken to be independent of the point $p$ (thus $c_0^{-1}$) for any point sufficiently close to the given vertex on an edge that ends at that vertex if the conditions for the fourth bullet of Lemma~7.1 are met.  Granted this uniform upper bound for $r$ at those same edge points, then the argument used for the third bullet of Lemma~7.1 can be repeated almost verbatim.

\medskip
\noindent\emph{Proof of Lemma 7.2:}
The proof of the lemma has $11$ steps.  The first $10$ steps prove the assertion of Lemma~7.2 regarding the case where $r<\frac{1}{8c}\,r_p$.  The final step proves the assertion regarding the $\N_{(\cdot)}(0)>\frac{1}{2}$ assumption for points near a given vertex.

\medskip

\noindent\underline{\emph{Step 1:}}
To start the proof of Lemma~7.2, note first that the depiction of the function $\N$ using the second bullet in~\eqref{eq:7.4} implies that it is an almost everywhere differentiable function on $\bigl(0,\frac{1}{1000}\pi\bigr]$, and that its derivative at its differentiable points has the form
\begin{equation}\label{eq:7.12}
\frac{d\N}{dr}
=
-\frac{(1+2\N)\N}{r}
+\frac{r}{A\,\K^2}\int_{\partial\mathtt{B}_r} |\nabla\psi|^2
+r\,z(r),
\end{equation}
with $z(r)$ denoting a continuous function of $r$ with norm bounded by $c_0(1+\mathtt{E})$. (This bound for $z(r)$ follows directly from definition of $\K^2(r)$ as the average of $|\psi|^2$ on $\partial\mathtt{B}_r$.)

\medskip
\noindent\underline{\emph{Step 2:}}
Use of a certain symmetric tensor on $S^3-\Gamma$ (a section of $T^*S^3\otimes T^*S^3$) is the key idea for analyzing the right-hand side of \eqref{eq:7.12}. (The introduction of this tensor and its application in very similar contexts is ultimately due Almgren, but see also \cite{DF, H, HHL}.) This tensor is denoted by $T$ and it is defined via its components with respect to a chosen orthonormal frame for $T^*S^3$. These are the functions $\{T_{ab}\}_{a,b=1,2,3}$ with any given $T_{ab}$ defined below:
\begin{equation}\label{eq:7.13}
T_{ab}
=
\langle \nabla_a\psi,\nabla_b\psi\rangle
+\langle \nabla_b\psi,\nabla_a\psi\rangle
-\delta_{ab}|\nabla\psi|^2.
\end{equation}
Here, $\nabla_a$ is the directional derivative along the dual basis vector to the $1$-form labeled by the subscript $a$. The key observation for what follows is that the covariant divergence of $T$ has the schematic form below (repeated indices are summed over the index set $\{1,2,3\}$)
\begin{equation}\label{eq:7.14}
\nabla_a T_{ab}
=
\langle \psi,\, \mathcal{Z}_{ab}\nabla_a\psi\rangle
+\langle \mathcal{Z}_{ab}\psi,\, \nabla_a\psi\rangle.
\end{equation}
The notation here uses $\{\mathcal{Z}_{ab}\}_{a,b=1,2,3}$ to denote a collection of endomorphisms of $\mathbb{V}$ whose norms are bounded by $c_0(1+\mathtt{E})$.

\medskip
\noindent\underline{\emph{Step 3:}}
Use a Gaussian coordinate system based at $p$ for the radius $\frac{1}{100}\pi$ ball centered at $q$ to define a Euclidean coordinate chart for this ball. Let $\{x_1,x_2,x_3\}$ denote these Euclidean coordinates. The $1$-forms $\bigl\{\frac{1}{1+|x|^2}\,dx_a\bigr\}_{a=1,2,3}$ are an orthonormal basis for $T^*S^3$ on this ball. Use this basis to define the function $x_b\nabla_a T_{ab}$ on the ball (as in \eqref{eq:7.14}, repeated indices are summed).

Were $p$ in $S^3-\Gamma$ and were $r$ less than the distance from $p$ to $\Gamma$, then integrating the function $x_b\nabla_a T_{ab}$ over the radius $r$ ball centered at $p$ and then integrating by parts leads to the identity
\begin{align}\label{eq:7.15}
r\int_{\partial\mathtt{B}_r}\bigl(|\nabla_r\psi|^2-|\nabla^\partial\psi|^2\bigr)
=
-\int_{\mathtt{B}_r}|\nabla\psi|^2
+\mathfrak{C}(r)
\end{align}
with the notation being as follows: This identity is using $\nabla_r\psi$ to denote the covariant derivative of $\psi$ in the normal direction of $\partial\mathtt{B}_r$ and it is using $\nabla^{\partial}\psi$ to denote the covariant derivative of $\psi$ in the directions tangent to $\partial\mathtt{B}_r$. Meanwhile, $\mathfrak{C}$ in \eqref{eq:7.15} denotes a function of $r$ with the following norm bound:
\begin{equation}\label{eq:7.16}
|\mathfrak{C}(r)|
\le c_0\left(r^2\int_{\mathtt{B}_r} |\nabla\psi| + \int_{\mathtt{B}_r} |\psi|^2\right).
\end{equation}
To be sure: The restriction on $p$ and $r$ in the derivation of \eqref{eq:7.15} is invoked so that the integration by parts that leads to \eqref{eq:7.15} can't get hung up on $\Gamma$.

The explanation for why \eqref{eq:7.15} holds when the closure of $B_r$ intersects $\Gamma$ is given in subsequent steps; but (in some cases) with a bound for the corresponding version of $\mathfrak{C}$ that is not as strong as the bound in~\eqref{eq:7.16}.

\medskip
\noindent\underline{\emph{Step 4:}}
Assuming \eqref{eq:7.15} holds (whether or not $\mathfrak{C}$ obeys \eqref{eq:7.16}), that identity can be used to rewrite \eqref{eq:7.12} as done below:
\begin{equation}\label{eq:7.17}
\frac{d\N}{dr}
=
-\frac{(1+2\N)\N}{r}
+\frac{2r}{A \K^2}\int_{\partial\mathtt{B}_r} |\nabla_r\psi|^2
+\frac{1}{A \K^2}\int_{\mathtt{B}_r} |\nabla\psi|^2
-\frac{1}{A \K^2}\,\mathfrak{C}(r)
+r\,z(r).
\end{equation}
To proceed from here: Write $\nabla_r\psi$ in \eqref{eq:7.17} as $\nabla_r\psi-\frac{\N}{r}\psi+\frac{\N}{r}\psi$ and use the first bullet in \eqref{eq:7.4} to write the identity above as
\begin{equation}\label{eq:7.18}
\frac{d\N}{dr}
=
-\frac{\N}{r}
+\frac{2r}{A \K^2}\int_{\partial\mathtt{B}_r}\left|\nabla_r\psi-\frac{\N}{r}\psi\right|^2
+\frac{1}{A \K^2}\int_{\mathtt{B}_r} |\nabla\psi|^2
-\frac{1}{A \K^2}\,\mathfrak{C}(r)
+r\,z(r).
\end{equation}
With \eqref{eq:7.18} in hand, use the depiction of $\N$ from the second bullet in \eqref{eq:7.4} to write the terms in \eqref{eq:7.18} with the $\mathtt{B}_r$ integral of $|\nabla\psi|^2$ in terms of $\N$. Doing so leads to this inequality:
\begin{equation}\label{eq:7.19}
\frac{d\N}{dr}
=
\frac{2r}{A \K^2}\int_{\partial\mathtt{B}_r}\left|\nabla_r\psi-\frac{\N}{r}\psi\right|^2
-\frac{1}{A \K^2}\,\mathfrak{C}(r)
+r\,w(r)
\end{equation}
with $w(\cdot)$ denoting a continuous function of $r$ with a $c_0(1+\mathtt{E})$ norm bound.

\medskip
\noindent\underline{\emph{Step 5:}}
Suppose in this step that $\mathfrak{C}$ does obey the bound in \eqref{eq:7.16}. Noting that the function $\K$ is almost non-decreasing (this follows from \eqref{eq:7.3} and from the first bullet in Lemma~7.1), it follows that
\begin{equation}\label{eq:7.20}
\int_{\mathtt{B}_r} |\psi|^2 \le c_0\, r^2\, \K^2(r).
\end{equation}
Use this last bound for the $|\psi|^2$ integral on the right-hand side of \eqref{eq:7.16}, and use the second bullet in \eqref{eq:7.4} to write the $|\nabla\psi|^2$ integral on the right-hand side of \eqref{eq:7.16} in terms of $\N$. Doing both of those and then using the resulting bound for $\mathfrak{C}$ in \eqref{eq:7.19} leads directly to the inequality for the derivative of $\N$ that is asserted by Lemma~7.2.

With the preceding understood: As noted in Step~3, the inequality in \eqref{eq:7.16} holds when $p$ is in $S^3-\Gamma$ and when $r\le \frac{1}{2c}\,r_p$. The next step explains why this is also the case when $p$ is from the interior of an edge of $\Gamma$. Step~7 explains why this is the case when $p$ is a vertex in $\Gamma$. The last step in the proof (Step~11) explains why \eqref{eq:7.16} holds for the case under consideration in the final assertion of Lemma~7.2.

\medskip
\noindent\underline{\emph{Step 6:}}
To see about the cases when $\mathtt{B}_r$ intersects $\Gamma$ (which is always the case if the point $p$ is in $\Gamma$), fix $\varepsilon>0$ but very small (the $\varepsilon\to 0$ limit will be taken eventually) and for the purposes of this step, let
\begin{align}
\alpha_\varepsilon
:=\chi\!\left(1-\frac{1}{\varepsilon}\,\operatorname{dist}(\cdot,\Gamma)\right).
\end{align}
To be sure, this function is equal to $1$ where the distance to $\Gamma$ is greater than $\varepsilon$; and it is equal to zero where the distance to $\Gamma$ is less than $\frac14\,\varepsilon$. Also fix $\delta>100\varepsilon$ but very small (the $\delta\to 0$ limit will be taken also), let $\mathcal{V}$ denote the set of vertices of the graph $\Gamma$, and 
\begin{align}
\beta_\delta:=\chi\!\left(1-\frac{1}{\delta}\,\operatorname{dist}(\cdot,\mathcal{V})\right),
\end{align}
where
\begin{align}
\operatorname{dist}(\cdot,\mathcal{V}):=\min_{v\in\mathcal{V}}\dist(\cdot, v).
\end{align}
To be sure, this function is equal to $1$ where the distance to each vertex of $\Gamma$ is greater than $\delta$ and equal to $0$ where the distance to some vertex is less than $\frac14\,\delta$.

Now return to the Step~3 and instead of integrating $x_b\nabla_aT_{ab}$ over $\mathtt{B}_r$, integrate the function $\alpha_\varepsilon\beta_\delta\, x_b\nabla_aT_{ab}$ over $\mathtt{B}_r$. Then integrate by parts. (The function $\alpha_\varepsilon$ is used so that the support of the integrand has distance greater than $\frac14\,\varepsilon$ from $\Gamma$; and the function $\beta_\delta$ is used to further shrink the support so that its distance to any vertex of $\Gamma$ is greater than $\frac14\,\delta$).  The result of this integration by parts is an $\varepsilon$-dependent version of \eqref{eq:7.15} where $\mathfrak{C}(\cdot)$ can be written as a sum of three terms,
\begin{equation}\label{eq:7.21}
\mathfrak{C}(\cdot)
=
\mathfrak{C}_{1\varepsilon,\delta}(\cdot)
+\mathfrak{C}_{2\varepsilon,\delta}(\cdot)
+\mathfrak{C}_{3\varepsilon,\delta}(\cdot).
\end{equation}
with the distinction between these terms as follow: The term $\mathfrak{C}_{1\varepsilon,\delta}$ denotes a function of $r$ that is bounded by what is written on the right-hand side of \eqref{eq:7.16}; it is short-hand for an integral over $\mathtt{B}_r$ of terms with no derivatives on $\alpha_\varepsilon$ or $\beta_\delta$.  Meanwhile, $\mathfrak{C}_{2\varepsilon,\delta}$ signifies the function of $r$ that contains the integration by parts terms that do have derivatives of $\alpha_\varepsilon$ but none of $\beta_\delta$. And,
$\mathfrak{C}_{3\varepsilon,\delta}$ contains the integration by parts terms that have derivatives of $\beta_\delta$ but none of $\alpha_\varepsilon$. The terms $\mathfrak{C}_{2\varepsilon,\delta}$ and $\mathfrak{C}_{3\varepsilon,\delta}$ are depicted below in \eqref{eq:7.22} with the notation as follows: The first bullet's depiction of $\mathfrak{C}_{2\varepsilon,\delta}$ uses $\mathfrak{d}_\Gamma$ for the function $\operatorname{dist}(\cdot,\Gamma)$, and it uses $\alpha'_\varepsilon$ to denote the function on $S^3$ that is obtained by evaluating the derivative of $\chi$ at the value of the function $\left(1-\frac{1}{\varepsilon}\operatorname{dist}(\cdot,\Gamma)\right)$.  The second bullet has $\mathfrak{d}_{\mathcal{V}}$ denoting the function $\operatorname{dist}(\cdot,\mathcal{V})$ and it defines $\beta'_\delta$ analogously to $\alpha'_\varepsilon$.
\begin{itemize}
    \item 
$\mathfrak{C}_{2\varepsilon,\delta}(r)
= \frac{1}{\varepsilon}\int_{\mathtt{B}_r}
\beta_\delta\,\alpha'_\varepsilon\,\bigl(\nabla_a\mathfrak{d}_\Gamma\bigr)\,x_b T_{ab}.$
\smallskip
    \item 
$\mathfrak{C}_{3\varepsilon,\delta}(r)
= \frac{1}{\delta}\int_{\mathtt{B}_r}
\alpha_\varepsilon\,\beta'_\delta\,\bigl(\nabla_a\mathfrak{d}_{\mathcal{V}}\bigr)\,x_b T_{ab}.$
\listeqno
\label{eq:7.22}
\end{itemize}

The key concern with regard to $\mathfrak{C}_{2\varepsilon,\delta}$ and $\mathfrak{C}_{3\varepsilon,\delta}$ is whether they have respective limits as $\varepsilon$ and $\delta$ are taken to zero, and if so, whether those limits are zero.  Because the $\mathfrak{C}_{3\varepsilon,\delta}$ analysis is relatively straightforward, it is considered first below.

\medskip
\noindent\underline{\emph{Step 7:}}
This step explains why the contribution to $\mathfrak{C}_{3\varepsilon,\delta}$ has limit zero as $\delta\to 0$.  To see that this is so, note first that the integrand in the second bullet of \eqref{eq:7.22} is supported in the union of the radius $\delta$ balls centered at the vertices of $\Gamma$.  With that understood, let $q$ denote a given vertex and let $\mathtt{B}'_\delta$ denote the radius $\delta$ ball centered at $q$.  The contribution from $\mathtt{B}'_\delta$ is at most
\begin{equation}\label{eq:7.23}
c_0\, r_p \int_{\mathtt{B}'_\delta} \frac{1}{\operatorname{dist}(\cdot,q)}\,
|\nabla\psi|^2 \, .
\end{equation}
The inequality depicted by the fourth bullet in \eqref{eq:6.3} implies that \eqref{eq:7.23} is no larger than $c_0 r_p$, and then it follows from this last fact (by the dominated convergence theorem) that the $\delta\to 0$ limit of \eqref{eq:7.23} is zero.  (Note that $\varepsilon$ plays no role in this analysis.)

\medskip
\noindent\underline{\emph{Step 8:}}
The analysis of the integral in \eqref{eq:7.22} for $\mathfrak{C}_{2\varepsilon,\delta}$ will exploit the fact that the integrand is supported where the distance to the vertices of $\Gamma$ are greater than $\frac14\,\delta$ but where the distance to some edge of $\Gamma$ is less than $\varepsilon$.  In particular, if $\varepsilon$ is much smaller than $\delta$, then contribution to the integral for $\mathfrak{C}_{2\varepsilon,\delta}$ in \eqref{eq:7.22} comes from points that are much closer to some given edge of $\Gamma$ than they are to any vertex.

The subsequent analysis of the contribution of a given edge from $\Gamma$ to the integral for $\mathfrak{C}_{2\varepsilon,\delta}$ in \eqref{eq:7.22} depends whether the point $p$ is on that edge or not.  Suppose henceforth in this step that $p$ is not on a given edge that intersects $\mathtt{B}_r$.  (Note that if $\mathtt{B}_r$ intersects that edge, then $r\ge c_0^{-1} r_p$.)  Since the length of $\mathtt{B}_r$'s intersection with that edge is at most $c_0 r$, the part of that edge where $\beta_\delta$ is non-zero can be covered by $c_0 r\,\varepsilon^{-1}$
balls of radius $\varepsilon$ centered at points on this edge.  Let $q$ now denote the center point of one of those balls and let $\mathtt{B}'_\varepsilon$ denote
the radius $\varepsilon$ ball centered at $q$.  The contribution to $\mathfrak{C}_{2\varepsilon,\delta}$ in \eqref{eq:7.22} from the ball $\mathtt{B}'_\varepsilon$ is at most
\begin{equation}\label{eq:7.23b}
c_0\, r\,\frac{1}{\varepsilon}\int_{\mathtt{B}'_\varepsilon} |\nabla\psi|^2 \, .
\end{equation}

To see how big this can be, multiply both sides of \eqref{eq:6.5} by
\[
\chi\!\left(\frac{1}{\varepsilon}\,\operatorname{dist}(\cdot,q)-1\right),
\]
which is a function with compact support in $\mathtt{B}'_{2\varepsilon}$ that is equal to $1$ on $\mathtt{B}'_{\varepsilon}$.  Then integrate the resulting identity and after doing that, integrate by parts twice to remove derivatives from $|\psi|^2$ and put them on the cut--off function $\chi\!\left(\frac{1}{\varepsilon}\,\operatorname{dist}(\cdot,q)-1\right)$.  The
result of doing that leads to the bound below when $\varepsilon$ is very small:
\begin{equation}\label{eq:7.24}
\int_{\mathtt{B}'_{\varepsilon}} |\nabla\psi|^2
\le c_0\,\frac{1}{\varepsilon^2}\int_{\mathtt{B}'_{2\varepsilon}} |\psi|^2 \, .
\end{equation}

Crucially, Lemma~6.1 can be brought to bear if $\varepsilon<\frac{1}{4c}\,r_q$; hence the right-hand side of \eqref{eq:7.24} in this case is bounded by $c_0\varepsilon$.  Note that if
$\varepsilon<\frac{1}{1000c}\,\delta$, then $\varepsilon$ will be less than $\frac{1}{4c}\,r_q$ if the distance of $q$ to the nearest vertex of $\Gamma$ is greater than $\delta$.

Supposing henceforth that $\varepsilon$ is less than $\frac{1}{1000c}\,\delta$, then this $c\varepsilon$ bound for \eqref{eq:7.23b} implies a $c_0r$ bound for the contribution to the integral for $\mathfrak{C}_{2\varepsilon,\delta}$ in \eqref{eq:7.22} from an edge of $\Gamma$ that does not contain the point $p$. Note in particular that this bound is independent of both $\varepsilon$ and $\delta$.

\smallskip

\noindent\underline{\emph{Step 9:}}
This step considers the contribution to the integral for $\mathfrak{C}_{2\varepsilon,\delta}$ in \eqref{eq:7.22} from an edge that contains $p$ but with the additional assumption that $p$ is an interior point of that edge. It is also assumed in this step that $r$ is less than $\frac{1}{8c}\,r_p$ so as to eventually bring Lemma~6.2 to bear.  The exploitation of Lemma~6.2 uses the fact that if $\varepsilon$ is much less than $\frac{1}{8c}\,r_p$, then on the support of $\alpha'_{\varepsilon}$ inside the ball of radius $\frac{1}{8c}\,r_p$ centered at $p$, the vector $x_b$ is almost tangent to the edge in question.  This means in particular that
\begin{equation}\label{eq:7.25}
|x_b\nabla_b\,\mathfrak{d}_\Gamma| < c_0\varepsilon
\qquad\text{and}\qquad
|x_b-|x|\nu_b| < c_0\varepsilon\,|x| \, ,
\end{equation}
where $\{\nu_b\}_{b=1,2,3}$ denotes the components of the vector field $\frac{\partial}{\partial t}$ that appears in Lemma~6.2.

Let $N_{\varepsilon,\delta}$ denote the part of the radius $\varepsilon$ tubular neighborhood of the edge in question where the distance to $\Gamma$ is greater than $\frac14\,\varepsilon$ and where the distance to the vertex set $\mathcal{V}$ of $\Gamma$ is greater than $\frac14\,\delta$.  With $N_{\varepsilon,\delta}$ understood, it then follows as a consequence of the preceding bounds that the contribution to the integral for $\mathfrak{C}_{2\varepsilon,\delta}$ from the edge containing $p$ is at most
\begin{equation}\label{eq:7.26}
c_0\,r\,\frac{1}{\varepsilon}
\left(\int_{N_{\varepsilon,\delta}} |\nabla\psi|^2\right)^{1/2}
\left(\int_{N_{\varepsilon,\delta}} |\nabla_t\psi|^2\right)^{1/2}
\;+\;
c_0 r\int_{N_{\varepsilon,\delta}} |\nabla\psi|^2 \, .
\end{equation}

The rightmost integral in \eqref{eq:7.26} has limit $0$ as $\varepsilon\to 0$ because $|\nabla\psi|$ is in the Hilbert space $\mathbb{H}$.  With regards to the rightmost integral:  Lemma~6.2 says that $\nabla_t\psi$ is in $\mathbb{H}$, and bounds its $\mathbb{H}$-norm by $c_0\frac{1}{\sqrt{r_p}}$.  It follows from this fact and \eqref{eq:4.2} that
\begin{equation}\label{eq:7.27}
\int_{N_{\varepsilon,\delta}} |\nabla_t\psi|^2 \le c_0\varepsilon^2\,\frac{1}{r_p} \, .
\end{equation}
Use of the latter in the leftmost term of \eqref{eq:7.26} bounds the latter by
\begin{equation}\label{eq:7.28}
c_0\,\frac{r}{\sqrt{r_p}}\left(\int_{N_{\varepsilon,\delta}} |\nabla\psi|^2\right)^{1/2},
\end{equation}
which has limit zero as $\varepsilon\to 0$.  To summarize what was just proved:  If $p$ is in the interior of an edge of $\Gamma$ and if $r<\frac{1}{8c}\,r_p$, then the term $\mathfrak{C}_{2\varepsilon,\delta}$ has limit zero as $\varepsilon\to 0$.

\medskip
\noindent\underline{\emph{Step 10:}}
Now consider the case when $p$ is a vertex of $\Gamma$ with $r<\frac{1}{8c}\,r_p$. In this case, the integral for $\mathfrak{C}_{2\varepsilon,\delta}$ in \eqref{eq:7.22} is the sum of contributions from the part of the radius $\varepsilon$ tubular neighborhood of $p$'s incident edges that have distance between $\frac14\,\delta$ and $r$ from $p$.  If $\varepsilon$ is much smaller than $\delta$, then those tubular neighborhoods are disjoint.  In addition, the bounds in \eqref{eq:7.25} hold on each such neighborhood.  As a consequence, the analysis in Step~9 can be repeated with almost no changes to see that the $\varepsilon\to 0$ limit of $\mathfrak{C}_{2\varepsilon,\delta}$ is zero.

\medskip
\noindent\underline{\emph{Step 11:}}
Suppose in this step that $q$ is a given vertex of $\Gamma$ and that there exists $\rho>0$ such that $\N_{(\cdot)}(0)>\frac12$ at each point in $\Gamma$ with distance $\rho$ or less from $q$.  The plan for this case is to show that if $p\in S^3$ has distance less than $\frac12\,\rho$ from $q$, then the inequality for the derivative of $\N$ holds on $(0,\,c_0^{-1}\rho]$.  This will follow with a demonstration that the integral in \eqref{eq:7.22} for $\mathfrak{C}_{2\varepsilon,\delta}(r)$ has limit zero as $\varepsilon\to 0$ for fixed $\delta$ when $r\le c_0^{-1}\rho$.

To see about this limit, the first observation is the conclusions of Lemma~7.1 can be assumed and, in particular, the conclusions of the third bullet since those conclusions only require that the differential inequality from Lemma~7.2 hold at points in the interior of an edge and for $r<\frac{1}{8c}\,r_{(\cdot)}$.  This being the case, it then follows that for any give value of $\delta$, there exists $\mu>0$ such that if $p$ is a point with distance between $\frac18\,\delta$ and $\rho$ from the vertex $q$, then $\N_p(0)\ge \frac{1+\mu}{2}$.  It then follows from this that using \eqref{eq:7.3} that the integral on the right-hand side of \eqref{eq:7.24} is bounded by $c_0\varepsilon^{2+\mu}\,r_p^{-(1+\mu)}$.  Since $r_p$ is greater than $c_0^{-1}\delta$ by assumption, this in turn is bounded by $c_0\varepsilon^{2+\mu}\delta^{-(1+\mu)}$.

Since the domain of integration for \eqref{eq:7.22}'s integral defining $\mathfrak{C}_{2\varepsilon,\delta}$ can be covered by $c_0\varepsilon^{-1}$ balls of radius $\varepsilon$ centered at points on edges of $\Gamma$, the bound $\bigl|\mathfrak{C}_{2\varepsilon,\delta}\bigr|\le c_0\varepsilon^{1+\mu}\delta^{-(1+\mu)}$ follows directly from the preceding analysis.  The $\varepsilon\to 0$ limit of $\mathfrak{C}_{2\varepsilon,\delta}$ is therefore zero for any fixed choice of $\delta$ as required.

\section{Proofs for Lemmas 4.2 and 4.4}

This section uses parts of Lemmas~7.1--7.3 to prove Lemmas~4.2 and 4.4.

\medskip
\noindent\emph{Proof of Lemma 4.2:}
The proof explains why $|\psi|$ is H\"older continuous near $\Gamma$ with H\"older exponent that is no less than the minimum of $\frac12$ and the versions of the number $\N(0)$ that are defined by the vertices of $\Gamma$.  (The second bullet of Lemma~7.1 says, in part, that this minimum is positive.)  This explanation has $5$ steps, most riffing on discourses in \cite{T4}.
Even so, the details are written.  With regards to notation: Given a positive number $r$ and a point in $S^3$, the notation is such that $\mathtt{B}_r$ again denotes the radius $r$ ball centered at that point.

\medskip
\noindent\underline{\emph{Step 1:}}
Suppose that $d>0$ and that $q$ is a point in $S^3$ with $\operatorname{dist}(q,\Gamma)=d$.  As explained directly, the equation in \eqref{eq:6.4} on the radius $d$ ball centered at $q$ when written as
\begin{align}
\nabla^\dagger\nabla\psi = (\mathtt{E}-\mathcal{R})\psi
\end{align}
implies the inequalities below for the integral of $|\nabla\psi|^2$ on the radius $\frac{31}{32}d$ ball centered at $q$, and for the integral of $|\nabla^{\otimes 2}\psi|^2$ on the radius $\frac{30}{32}d$ ball centered at $q$:
\begin{itemize}
    \item \smallskip
$\int_{\mathtt{B}_{31d/32}} |\nabla\psi|^2
\le c_0(1+\mathtt{E})^2\,\frac{1}{d^2}\int_{\mathtt{B}_d} |\psi|^2$,
\item \smallskip
$\int_{\mathtt{B}_{30d/32}} |\nabla^{\otimes 2}\psi|^2
\le c_0(1+\mathtt{E})^2\,\frac{1}{d^4}\int_{\mathtt{B}_d} |\psi|^2$.
\listeqno
\label{eq:8.1}
\end{itemize}
\smallskip
To obtain the first bullet of \eqref{eq:8.1}: Construct a bump function (to be denoted by $\varsigma_d$) with values in $[0,1]$ that is equal to $0$ near the boundary of the radius $d$ ball centered at $q$ and equal to $1$ on the concentric radius $\frac{31}{32}d$ ball with respective $c_0\frac{1}{d}$ and $c_0\frac{1}{d^2}$ bounds on its differential and Hessian.  Write the identity in \eqref{eq:6.4} as $\nabla^\dagger\nabla\psi = (\mathtt{E}-\mathcal{R})\psi$, take the inner product of both sides with $\varsigma_d\psi$ and then integrate over $\mathtt{B}_r$.  Having done that, then integrate by parts on the left-hand side.  To obtain the second bullet of \eqref{eq:8.1}: Construct a second version of $\varsigma_d$ that equals $0$ on the complement of that concentric, radius $\frac{31}{32}d$ ball and equals $1$ on the concentric, radius $\frac{30}{32}d$ ball.  Take the square of the norm of both sides of the identity $\nabla^\dagger\nabla\psi = (\mathtt{E}-\mathcal{R})\psi$, multiply the result by this new version of $\varsigma_d$, and then integrate over $\mathtt{B}_d$. Two applications of integration by parts on the left-hand side lead to the second bullet of \eqref{eq:8.1}, given what is said by \eqref{eq:8.1}'s first bullet.

\medskip
\noindent\underline{\emph{Step 2:}}
What follows directly is an application of \eqref{eq:8.1}: Suppose that $x$ and $y$ are distinct points in $S^3$ with $\operatorname{dist}(x,y)<\frac{1}{100}\pi$. Given $x$ and $y$, let $d$ denote a number that obeys the bound $\frac{30}{32}d>\frac{1}{2}\operatorname{dist}(x,y)$ which is imposed so that both $x$ and $y$ are in the radius $\frac{30}{32}d$ ball centered at the midpoint of the geodesic arc between them.  Let $\psi$ denote for the moment a section of $\mathbb{V}\otimes\mathcal{I}$ over the radius $d$ ball centered at the midpoint on the short geodesic between $x$ and $y$ with $|\nabla^{\otimes 2}\psi|^2$ having finite integral on the concentric radius $\frac{30}{32}d$ ball. Then the H\"older norm bound below holds
\begin{equation}\label{eq:8.2}
\bigl||\psi(x)|-|\psi(y)|\bigr|
\le c_0\,\operatorname{dist}(x,y)^{1/2}\,
\left(
\int_{\mathtt{B}_{30d/32}} |\nabla^{\otimes 2}\psi|^2
+\frac{1}{d^4}\int_{\mathtt{B}_{30d/32}} |\psi|^2
\right)^{1/2},
\end{equation}
this being a consequence of dimension $3$ Sobolev inequalities.  (See e.g.\ \cite{GT} for the latter).

Suppose now that $\psi$ obeys \eqref{eq:6.4} and that the radius $d$ ball centered at the midpoint between $x$ and $y$ is disjoint from $\Gamma$.  According to \eqref{eq:8.1}, the function $|\nabla^{\otimes 2}\psi|^2$ does have finite integral over $\mathtt{B}_{30d/32}$; and then \eqref{eq:8.2} with \eqref{eq:8.1} lead to the next bound:
\begin{equation}\label{eq:8.3}
\bigl||\psi(x)|-|\psi(y)|\bigr|
\le c_0\,\operatorname{dist}(x,y)^{1/2}\,\frac{1}{d^2}
\left(\int_{\mathtt{B}_{30d/32}} |\psi|^2\right)^{1/2}.
\end{equation}
Hold on to the preceding inequality for the moment.

\medskip
\noindent\underline{\emph{Step 3:}}
Supposing that $p$ is a given point in $S^3-\Gamma$, let $p'$ denote a closest point to $p$ in $\Gamma$, the distance from $p$ to $p'$ being the number $r_p$. Suppose in this step and the next step that $r_p$ (which is $\operatorname{dist}(p,p')$) is less than $\frac{1}{8c}\,r_{p'}$ with $c$ defined by \eqref{eq:6.1}.  Thus the point $p$ is in one of the cones depicted in Figure~1 and $p'$ is from the interior of one of $\Gamma$'s edges.

Letting $\gamma$ denote the short geodesic between $p$ and $p'$, set $x_1$ to denote the point on $\gamma$ half-way between $p$ and $p'$. Thus, $\operatorname{dist}(p,x_1)=\frac12\,\operatorname{dist}(p,p')$.  Set $d=\frac12\,r_p$.  Of particular note is that $\frac{30}{32}d>\frac12\,
\operatorname{dist}(p,x_1)$ and that the radius $d$ ball centered at the midpoint between $x_1$ and $p$ is disjoint from $\Gamma$.  As a consequence, \eqref{eq:8.3} can be invoked with $x=x_1$ and $y=p$.

To say more about that version of \eqref{eq:8.3}, note also that the radius $d$ ball centered on the midpoint between $x_1$ and $p$ is contained in the radius $\frac32 r_p = 3d$ ball centered at $p'$.  Noting that the latter ball is inside the radius $\frac{1}{4c}\,r_{p'}$ ball centered at $p'$, Lemma~7.3 can be brought to bear to bound the integral that appears in this version of \eqref{eq:8.3}.  Bring that lemma to bear results in a bound of the form
\begin{equation}\label{eq:8.4}
\bigl||\psi(x_1)|-|\psi(p)|\bigr|
\le c_0 z_{p'}\,\operatorname{dist}(x_1,p)^{\N(0)}
= c_0 z_{p'}\,\frac{1}{\sqrt{2}}\operatorname{dist}(q,p)^{\N(0)}.
\end{equation}
with $z_{p'}$ coming from the $p'$ version of Lemma~7.3 and with $\N(0)$ coming from the $p'$ version of Lemma~7.1.  (Lemma~7.1 says in part that this version of $\N(0)$ is no smaller than $\frac12$).

With regards to $z_{p'}$:  Let $q$ denote the closest vertex to $p'$ so as to define Lemma~7.3's number $\overline{\N}=\max(0,\N(0)-\N_q(0))$ and then bound $z_{p'}$ by $c_0/r_{p'}^{\overline{\N}}$.  Granted the latter bound, and granted that $\operatorname{dist}(p,q)\le c_0\,r_{p'}$, the bound in \eqref{eq:8.4} leads directly to the H\"older bound:
\begin{equation}\label{eq:8.5}
\bigl||\psi(x_1)|-|\psi(p)|\bigr|
\le c_0\,\frac{1}{\sqrt{2}}\,\operatorname{dist}(q,p)^{\mu},
\end{equation}
with $\mu$ denoting the lesser of the numbers $\frac{1}{2}$ and the versions of $\N(0)$ that are defined by the vertices of $\Gamma$. For use below: Introduce $c_\ast$ to denote the version of the number $c_0$ that appears in \eqref{eq:8.5}.

To continue, let $x_2$ denote the point along $\gamma$ half way between $x_1$ and $p$. The analysis that lead to \eqref{eq:8.5} can be redone with $x_1$ replaced by $x_2$ and $p$ replaced by $x_1$ to see that
\begin{equation}\label{eq:8.6}
\bigl||\psi(x_2)|-|\psi(x_1)|\bigr|
\le c_\ast\,\operatorname{dist}(x_2,x_1)^{\mu}
= c_\ast\,\frac{1}{2}\,\operatorname{dist}(q,p)^{\mu}\,.
\end{equation}
Continue sequentially in this same vein with $x_n$ for $n\ge 3$ denoting the point on $\gamma$ with distance to $q$ being $2^{-n}\operatorname{dist}(p,q)$ to see that
\begin{equation}\label{eq:8.7}
\bigl||\psi(x_n)|-|\psi(x_{n-1})|\bigr|
\le c_\ast\,\operatorname{dist}(x_n,x_{n-1})^{\mu}
= c_\ast\,\frac{1}{2^{n/2}}\,\operatorname{dist}(q,p)^{\mu}\,.
\end{equation}

Since the sequence $\{x_n\}_{n=1,2,\dots}$ converges to $p'$, adding the corresponding versions of \eqref{eq:8.7} leads directly to the H\"older bound
\begin{equation}\label{eq:8.8}
\bigl||\psi(q)|-|\psi(p)|\bigr|
\le c_\ast\,\operatorname{dist}(q,p)^{\mu}\,.
\end{equation}
Note that \eqref{eq:8.8} implies that the function $|\psi|$ is H\"older continuous with exponent $\mu$ along the normal direction at any point from the interior of any edge of $\Gamma$. (Of course, this also follows directly from Lemma~6.1.)

\medskip
\noindent\underline{\emph{Step 4:}}
Lemma~6.1 says that $|\psi|$ extends across the interiors of the edges of $\Gamma$ to be zero on $\Gamma$.  This step explains why $|\psi|$ is H\"older continuous with exponent $\mu$ on a neighborhood of any point in the interior of an edge of $\Gamma$. By way of background for this explanation: Suppose that $q$ is an interior point on an edge of $\Gamma$ and that points $x$ and $y$ are in the radius $\frac{1}{128c}\,r_q$ ball centered at $q$.  Let $q_x$ and $q_y$ denote the respective closest points to $x$ and $y$ in $\Gamma$.  These will be in the interior
of the same edge as $q$ and with their corresponding numbers $r_{q_x}$ and $r_{q_y}$ being greater than $\frac14\,r_q$.  It follows as a consequence of what is said by \eqref{eq:8.8} that
\begin{equation}\label{eq:8.9}
\bigl||\psi(x)|-|\psi(q_x)|\bigr|
\le c_0\,\operatorname{dist}(x,q_x)^{\mu}
\qquad\text{and}\qquad
\bigl||\psi(y)|-|\psi(q_y)|\bigr|
\le c_0\,\operatorname{dist}(y,q_y)^{\mu}\,.
\end{equation}

By way of more background:  Let $\rho$ denote the smaller of $\operatorname{dist}(x,q_x)$ and $\operatorname{dist}(y,q_y)$.  If $\operatorname{dist}(x,y)<\frac{1}{1000}\,\rho$, then \eqref{eq:8.2} can be brought to bear with $d=\frac{1}{100}\,\rho$; and that inequality with the $r=\frac{1}{10}\rho$ version of Lemma~7.3 leads to the H\"older bound
\begin{equation}\label{eq:8.10}
\bigl||\psi(x)|-|\psi(y)|\bigr|
\le c_0\,\operatorname{dist}(x,y)^{\mu}\,.
\end{equation}

The inequalities in \eqref{eq:8.9} and \eqref{eq:8.10} can now be used to obtain a uniform, exponent $\mu$ H\"older bound for $|\psi|$ near the given point $p$. Indeed, this comes directly from \eqref{eq:8.10} in the event that $\operatorname{dist}(x,y)<\frac{1}{1000}\,\rho$ where $\rho$ again denotes the minimum of $\operatorname{dist}(x,q_x)$ and $\operatorname{dist}(y,q_y)$.  (Implicit here and in what follows in this step is that both $x$ and $y$ are in the radius $\frac{1}{128c}\,r_q$ ball centered at some given point $q$ in the interior of one
of $\Gamma$'s edges.)  In the event that $\operatorname{dist}(x,y)\ge \frac{1}{1000}\,\rho$, the desired H\"older bound is obtained from the inequality
\begin{equation}\label{eq:8.11}
\bigl||\psi(x)|-|\psi(y)|\bigr|
\le
|\psi(x)|+|\psi(y)|.
\end{equation}

To see how to the desired H\"older bound follows from \eqref{eq:8.11}, let $q_x$ and $q_y$ again denote the respective closest points to $x$ and $y$ in $\Gamma$.  Since both $|\psi(q_x)|$ and $|\psi(q_y)|$ are zero, the inequality in \eqref{eq:8.11} leads via \eqref{eq:8.9} to the one below:
\begin{equation}\label{eq:8.12}
\bigl||\psi(x)|-|\psi(y)|\bigr|
\le
c_0\,\operatorname{dist}(x,q_x)^\mu
+
c_0\,\operatorname{dist}(y,q_y)^\mu .
\end{equation}

To proceed from \eqref{eq:8.12} to the desired H\"older bound, assume (without loss of generality) that
\begin{align}
\operatorname{dist}(x,q_x)\le \operatorname{dist}(y,q_y).
\end{align}

Noting that $\operatorname{dist}(y,q_y)\le \operatorname{dist}(y,q_x)$ because $q_y$ is the closest point to $y$ in $\Gamma$, and noting also that
\begin{align}
\operatorname{dist}(y,q_x)\le \operatorname{dist}(x,y)+\operatorname{dist}(x,q_x)
\end{align}
which is less than $1001\,\operatorname{dist}(x,y)$, it then follows from
\eqref{eq:8.12} that
\begin{equation}\label{eq:8.13}
\bigl||\psi(x)|-|\psi(y)|\bigr|
\le
2\sqrt{1001}\,c_0\,\operatorname{dist}(x,y)^\mu .
\end{equation}
The bound above is the promised bound for when $\operatorname{dist}(x,y)$ is not less than $\frac{1}{1000}$ times the minimum of $\operatorname{dist}(x,\Gamma)$ and $\operatorname{dist}(y,\Gamma)$.

\medskip
\noindent\underline{\emph{Step 5:}}
Suppose in this step that $p$ is from $S^3-\Gamma$ and that $r_p = \operatorname{dist}(p,\Gamma)$ is not less than any version of $\frac{1}{8c}\,r_q$ when $q$ is a closest point to $p$ in $\Gamma$.  In this event, let $q$ henceforth denote a chosen, closest vertex to $p$.  If $r_p$ is greater than this vertex $q$'s version of $\frac{1}{8c}\,r_q$, then the distance from $p$ to $\Gamma$ is no less than $c_0^{-1}$ which is to say that $p$ is quite far from $\Gamma$ in which case $p$ is of no concern for purposes at hand.  Therefore, suppose in what follows that $r_p < \frac{1}{8c}\,r_q$ with $q$ denoting the chosen closest vertex from $\Gamma$.  This assumption and the initial assumption in this step imply that $p$ lies in the ball depicted in Figure~1, but that it is not inside the cones centered on the edges from the vertex $q$.  An important point to keep in mind for what follows is that $\operatorname{dist}(p,\Gamma)\ge c_0^{-1}\operatorname{dist}(p,q)$, this being a consequence of the positive lower bound for the solid angle cross sections of the cones that are depicted in Figure~1 and it is closer to the vertex center of that sphere than it is to the edge that is tangent to the disk.  Let $c_\ast$ denote the larger of this last version of $c_0$ and the number $c$ from \eqref{eq:6.1}.

To proceed in this case, let $\gamma$ denote the short geodesic from $q$ to $p$ and let $x_1$ denote the point on $\gamma$ that is halfway between that vertex $q$ and $p$.  Let
\begin{align}
d \;=\; \frac{1}{1000\,c_\ast}\,\operatorname{dist}(p,q).\label{eq:new8.17}
\end{align}

Fix equally spaced points along $\Gamma$ with distance $d/4$ apart, starting with $p$ such that the last one has distance $d/4$ or less from $x_1$.  Let $N$ denote the number of such points noting in particular that $N \leq 1000 c_*$.  Denote these points in order along $\Gamma$ moving from $p$ to $x_1$ by $y_1, \ldots, y_N$ and then set $y_{N+1}$ to be $x_1$.  Now, for each integer $k$ starting with $k = 1$ and ending with $N$, bring \eqref{eq:8.3} to bear sequentially with $x = y_k$ and $y = y_{k+1}$ keeping in mind that Lemma 7.3 can be invoked with p the given vertex because of~\eqref{eq:new8.17}. The result is the bound below:
\begin{equation}\label{eq:8.14}
\bigl||\psi(y_k)|-|\psi(x_{k+1})|\bigr|
\;\le\;
c_0\,\operatorname{dist}(p,q)^\mu \,.
\end{equation}
This implies in turn that
\begin{equation}\label{eq:8.15}
\bigl||\psi(x_1)|-|\psi(p)|\bigr|
\;\le\;
c_0\,\operatorname{dist}(p,q)^\mu \,.
\end{equation}

Let $x_2$ denote the point half way along $\gamma$ between $x_1$ and $q$. Repeat the preceding operation with $x_1$ replaced by $x_2$ and with $p$ replaced by $x_1$.  Having done that, repeat again with $x_2$ replaced by the point half way along $\gamma$ between $x_2$ and $q$, and with $x_1$ replaced by $x_2$.  Continuing sequentially in this way leads to a sequence $\{p,x_1,x_2,\dots\}$ that converges to $q$ and such that the following holds for each positive integer:
\begin{equation}\label{eq:8.16}
\bigl||\psi(x_n)|-|\psi(x_{n+1})|\bigr|
\;\le\;
c_0\,\frac{1}{2^{\mu n}}\,\operatorname{dist}(p,q)^\mu \,.
\end{equation}
Of course, these imply that
\begin{equation}\label{eq:8.17}
\bigl||\psi(p)|-|\psi(q)|\bigr|
\;\le\;
c_0\,\operatorname{dist}(p,q)^\mu \,.
\end{equation}

Coupled with the results of the previous step, this last inequality implies that $|\psi(q)|=0$.  Thus, $|\psi|$ vanishes at the vertices $\Gamma$ also. The arguments from the previous step can also be repeated with minor changes to see that $|\psi|$ is uniformly H\"older continuous with exponent $\mu$ on a neighborhood of each vertex.

\medskip

\noindent\emph{Proof of Lemma~4.4:} The proof has four steps.

\medskip
\noindent\underline{\emph{Step 1:}}
It follows from the assumptions and from Lemmas~7.1 and~7.2 that any $q\in\Gamma$ version of the number $\N(0)$ is no smaller than $1$. Indeed, if $\N(r)$ were less than $1-\varepsilon$ for any given sufficiently small but positive $r$, then Lemma~7.1's second bullet would imply that
\begin{align}
\N(r) < 1 - c_0^{-1}\varepsilon
\end{align}
for all smaller values of $r$. Then, by virtue of the definition of $\N$ via \eqref{eq:7.3}, this would run afoul of the assumption that the function $r \to \frac{1}{r}\,\K(r)$ is bounded as $r\to 0$. Thus, since $\N(\cdot)$ is (almost) a non-decreasing function of $r$ (see Lemma~7.2), it follows (again from
\eqref{eq:7.3}) that the inequalities
\begin{equation}\label{eq:8.18}
\int_{\partial\mathtt{B}_r} |\psi|^2 \le c_0\, r^{4}
\qquad \text{and} \qquad
\int_{\mathtt{B}_r} |\psi|^2 \le c_0\, r^{5}
\end{equation}
hold for all $r\in (0,c_0^{-1}]$. Since the rightmost bound is uniform with respect to the choice of the point in $\Gamma$ (see Lemma~7.3), that bound implies in turn the following bound:
\begin{equation}\label{eq:8.19}
\int_{\operatorname{dist}(\cdot,\Gamma)\le r} |\psi|^2 \le c_0\, r^{4}.
\end{equation}
(The exponent of $r$ drops from $5$ to $4$ because $\Gamma$ has non-zero, finite $1$-dimensional Hausdorff measure.) Hold onto this inequality for the moment.

\medskip
\noindent\underline{\emph{Step 2:}}
Let $\sigma$ denote a favorite, non-decreasing $[0,1]$-valued function on $\mathbb{R}$ that is equal to $0$ on $(-\infty,\tfrac{1}{4}]$ and equal to $1$ on $[\tfrac{3}{4},\infty)$. Given a positive number $r$, use this function to define a function on $S^3$ (denoted in what follows by $\sigma_r$) by the rule
\begin{align}
\sigma_r(\cdot)=\sigma\!\left(\frac{1}{r}\,\operatorname{dist}(\cdot,\Gamma)\right).
\end{align}

Write the identity in \eqref{eq:6.4} as $\nabla^{\dagger}\nabla\psi=(\mathtt{E}-\mathcal{R})\psi$, multiply each side by the function $\sigma_r$, and note that it implies in turn the identity
\begin{equation}\label{eq:8.20}
\nabla^{\dagger}(\sigma_r\nabla\psi)
=
(\mathtt{E}-\mathcal{R})\,\sigma_r\psi
-\nabla\sigma_r\cdot\nabla\psi \, .
\end{equation}
Of particular import is that $\sigma_r\nabla\psi$ has compact support on $S^3-\Gamma$ (as does each term on the right-hand side of \eqref{eq:8.20}).

Take the square of the norm of both sides of the identity in \eqref{eq:8.20} and then integrate the result over $S^3$. An integration by parts with an appeal to the top bullet in \eqref{eq:6.3} leads to the following inequality:
\begin{equation}\label{eq:8.21}
\int_{S^3}\bigl|\nabla(\sigma_r\nabla\psi)\bigr|^2
\le
c_0(1+\mathtt{E})^2
+
c_0\,\frac{1}{r^2}
\int_{\frac{r}{4}<\operatorname{dist}(\cdot,\Gamma)\le r} |\nabla\psi|^2 \, .
\end{equation}

\medskip
\noindent\underline{\emph{Step 3:}}
To see about the integral on the right-hand side of \eqref{eq:8.21}, do as before and write \eqref{eq:6.4} as $\nabla^{\dagger}\nabla\psi=(\mathtt{E}-\mathcal{R})\psi$, then take the inner product of both sides of that identity with $\psi$, then multiply the result by $\sigma_{r/8}(1-\sigma_{8r})$. Having done all of that, integrate the result over $S^3$. An integration by parts then results in the inequality below:
\begin{equation}\label{eq:8.22}
\int_{\,\frac{r}{4}<\operatorname{dist}(\cdot,\Gamma)\le r} |\nabla\psi|^2
\;\le\;
c_0\,\frac{1}{r^2}
\int_{\,\frac{r}{8}<\operatorname{dist}(\cdot,\Gamma)\le 2r} |\psi|^2 \,.
\end{equation}

\medskip
\noindent\underline{\emph{Step 4:}}
Use \eqref{eq:8.22} in \eqref{eq:8.21} to bound the latter's right-hand side integral, and then invoke \eqref{eq:8.19} to obtain the $r$-independent bound below for the $\mathbb{L}$-norm of $\nabla(\sigma_r\nabla\psi)$:
\begin{equation}\label{eq:8.23}
\int \bigl|\nabla(\sigma_r\nabla\psi)\bigr|^2
\;\le\;
c_0\,(1+\mathtt{E})^2 \,.
\end{equation}

It now follows by taking $r$ to zero in \eqref{eq:8.23} 
that $\nabla(\nabla\psi)$, the covariant derivative of $\nabla\psi$ is in the Hilbert space $\mathbb{L}$. This implies in turn that $\nabla\psi$ is in $\mathbb{H}$.

\section{Fourier Modes of $\psi$ Near an Edge of $\Gamma$}

Fix the vector bundle $\mathbb{V}$ and its connection; and then let $\psi$ denote a $(\nabla^{\dagger}\nabla + \mathcal{R})$--eigensection on $S^{3}-\Gamma$ with $\mathbb{L}$--norm equal to $1$. There is a standard way to use the properties of the frequency function $\N$ to `zoom' in on the behavior of $\psi$ near any given point in $\Gamma$. This zoom in process for a point in the interior of a given edge of $\Gamma$ is described in this section.

To do this zoom-in, let $p$ denote a point on the interior of an edge. Choose a Gaussian coordinate chart centered at $p$ to identify the ball of radius $\frac{1}{8c}\,r_{p}$ centered at $p$ with the same radius ball in $\mathbb{R}^{3}$. To be more specific: This chart should be chosen so that the edge containing $p$ corresponds to the part of the $x_{3}$--axis where $x_{3}$ is between $-\frac{1}{8c}\,r_{p}$ and $\frac{1}{8c}\,r_{p}$. The restriction of the line bundle $\mathcal{I}$ to the radius $\frac{1}{8c}\,r_{p}$ ball centered at $p$ is identified by these coordinates with the real line bundle that is
defined on the complement in this ball of the $x_{3}$ axis with monodromy $-1$ on linking circles with the $x_{3}$--axis.

With $\mathbb{V}$ now viewed as a vector bundle with connection over the radius $\frac{1}{8c}\,r_{p}$ ball centered at the origin in $\mathbb{R}^{3}$, pick an orthonormal frame for $\mathbb{V}$ at the origin and use parallel transport out along the rays in $\mathbb{R}^{3}$ from the origin to write $\mathbb{V}$ as a product bundle. The fiber of that product vector bundle is denoted by $V$ in what follows. Doing that identifies a section of $\mathbb{V}$ over that ball with a map from the $|x|<\frac{1}{8c}\,r_{p}$ ball to $V$. Likewise, a section of $\mathbb{V}\otimes\mathcal{I}$ over the complement of the $x_{3}$--axis in the $\frac{1}{8c}\,r_{p}$ ball centered at $p$ will be viewed as a section of the vector bundle $\mathcal{I}\times V$.

Supposing that $\lambda\in(0,1]$, define the rescaling map $\phi_{\lambda}:\mathbb{R}^{3}\to\mathbb{R}^{3}$ by the rule whereby $x=(x_{1},x_{2},x_{3})\mapsto \lambda x$. Parallel transport on radial geodesics from the points on the $|x|=1$ sphere in $\mathbb{R}^{3}$ identifies the line bundle $\mathcal{I}$ with its pull-back via $\phi_{\lambda}$. This identification is taken for granted henceforth. If $\psi$ denotes a section of the bundle $\mathcal{I}\times V$ on the complement of the $x_{3}$--axis in the $|x|\le \frac{1}{8c}\,r_{p}$ ball, then $\phi_{\lambda}^{*}\psi$ (which is the composition of first $\phi_{\lambda}$ and then $\psi$) defines a section of $\mathcal{I}\times V$ on the complement of the $x_3$--axis in the $|x|\le \frac{1}{8c}\,\lambda^{-1}r_p$ ball.

Given $\lambda\in\bigl(0,\frac{1}{16c}\,r_p\bigr]$, define the section $\psi_\lambda$ on this same domain by the rule whereby
\begin{equation}\label{eq:9.1}
x \longmapsto \psi_\lambda(x) \;=\; \frac{1}{\K(\lambda)}\,\psi(\lambda x)\,.
\end{equation}
The square of what is denoted by $\K(\lambda)$ in \eqref{eq:9.1} is the $r=\lambda$ version of what is depicted on the right-hand side of \eqref{eq:7.1}.  The division above by $\K(\lambda)$ is done so that the average of $|\psi_\lambda|^2$ over the $|x|=1$ sphere in $\mathbb{R}^3$ equals $1$.

Define a $\psi_\lambda$ version of the function $\K$ with domain $\bigl(0,\frac{1}{8c}\,\lambda^{-1}r_p\bigr)$ to be the positive square root of the function depicted below:
\begin{equation}\label{eq:9.2}
r \longmapsto \K_\lambda(r)^2 = \frac{1}{A(r)} \int_{\partial\mathtt{B}_r} |\psi_\lambda|^2\,,
\end{equation}
where $\mathtt{B}_r$ denotes here the $|x|\le r$ ball in $\mathbb{R}^3$ and $\partial\mathtt{B}_r$ denotes the $|x|=r$ boundary sphere.  Meanwhile, define the function $\N_\lambda$ on the same domain as $\K_\lambda$ by writing
\begin{equation}\label{eq:9.3}
\frac{d\K_\lambda}{dr} \;=\; \frac{\N_\lambda(r)}{r}\,\K_\lambda(r)\,.
\end{equation}

Of particular note is the fact that $\K_\lambda$ and $\psi$'s version of the function $\K$ (defined using the point $p$), and $\N_\lambda$ and $\psi$'s version of the function $\N$ (also defined using $p$) are related as follows:
\begin{equation}\label{eq:9.4}
\K_\lambda(r) \;=\; \K(\lambda)^{-1}\,\K(\lambda r)
\qquad\text{and}\qquad
\N_\lambda(r) \;=\; \N(\lambda r)\,.
\end{equation}

Granted \eqref{eq:9.4}, then Lemma~7.1's second bullet can be brought to bear to justify the following observation:
\begin{center}
\textit{Having fixed $R\ge 2$ and $\varepsilon>0$, there exists a positive number $\lambda_{R,\varepsilon}$ such that $|\N_\lambda(\cdot)-\N(0)|\le \varepsilon$ on $[0,R]$ when $\lambda\le \lambda_{R,\varepsilon}$.}
\end{center}
\listeqno
\label{eq:9.5}

This is to say that as $\lambda$ limits to zero, the values of $\N_\lambda$ converge to $\N(0)$ uniformly on any \emph{a priori} bounded subset of $[0,\infty)$. This convergence implies in turn (via \eqref{eq:9.3}) the following convergence assertion for $\K_\lambda$:

\begin{center}
\emph{Having fixed $R\ge 2$ and $\varepsilon>0$, there exists a positive number, $\lambda'_{R,\varepsilon}$ such that}
$\left|\frac{1}{r^{N(0)}}\,\kappa_\lambda(r)-1\right|
\le \varepsilon
\ \ \emph{for all } r\in\Bigl[\frac{1}{R},R\Bigr]
\emph{ when } \lambda\le \lambda'_{R,\varepsilon}.$
\end{center}
\listeqno
\label{eq:9.6}

The convergence of $\N_\lambda$ and $\K_\lambda$ have implications for the integral of $|\nabla\psi_\lambda|^2$, the following is the first (it comes from \eqref{eq:9.4}--\eqref{eq:9.6} and \eqref{eq:7.4}):
\begin{center}
\emph{Having fixed $R\in(1,\infty)$ and $\varepsilon>0$, there exists a positive number,
$\lambda''_{R,\varepsilon}$ such that when $\lambda\le \lambda''_{R,\varepsilon}$ and
$r\in\bigl[\frac{1}{R},R\bigr]$, then}
$\left|\frac{1}{r^{2N(0)+1}}\int_{\mathtt{B}_r}|\nabla\psi_\lambda|^2 - N(0)\right|
\le \varepsilon.$
\end{center}
\listeqno
\label{eq:9.7}
\\
In this last equation and subsequently, the connection that defines the covariant derivative $\nabla$ is the product connection on the product $V$ bundle.

The inequality in \eqref{eq:8.1} leads (via coordinate rescaling) to \emph{a priori} bounds on the second derivatives of $\psi_\lambda$ in an annulus around the $x_3$-axis. To write this precisely, fix $R>2$ and let $A_R$ denote the subset of $\mathbb{R}^3$ where the Euclidean coordinates $(x_1,x_2,x_3)$ are such that
\begin{align}
\frac{1}{R}< (x_1^2+x_2^2)^{1/2} < R
\quad\text{and}\quad
|x_3|<R.
\end{align}
The bound below holds for any given $R\ge 2$ when $\lambda$ is sufficiently small (the number $c_R$ below depends on $R$ but not on $\lambda$):
\begin{equation}\label{eq:9.8}
\int_{A_R}\bigl|\nabla^{\otimes 2}\psi_\lambda\bigr|^2 \le c_R\,R^{2\N(0)}.
\end{equation}

Moreover, given $R>2$, there are also $\lambda$-independent bounds for sufficiently small $\lambda$ on the integral over $A_R$ of the square of the norms of the derivatives of $\psi_\lambda$ to any given order:
\begin{equation}\label{eq:9.9}
\int_{A_R} \bigl|\nabla^{\otimes k}\psi_\lambda\bigr|^2 \le c_{R,k}\,R^{2\N(0)} .
\end{equation}

\noindent where $c_{R,k}$ depends on $R$ and $k$, but not $\lambda$.  Indeed, these follow using standard elliptic regularity tools given that $\psi_\lambda$ where $|x|\le \frac{1}{8c}\,\lambda^{-1}r_p$ obeys the rescaled version of \eqref{eq:6.4}'s eigenvalue equation, this being an equation that has the schematic form depicted below:
\begin{equation}\label{eq:9.10}
\nabla^{\dagger}\nabla\psi_\lambda
+\lambda\bigl(a_\lambda\cdot\nabla\psi_\lambda\bigr)
+\lambda^2\mathcal{T}_\lambda\psi_\lambda
=\lambda^2\mathtt{E}\,\psi_\lambda
=0 .
\end{equation}

\noindent Here, the norms of $a_\lambda$ and $\mathcal{T}_\lambda$ obey $|a_\lambda|\le c_0\lambda|x|$ and $|\mathcal{T}_\lambda|\le c_0$; and the norms of their derivatives obey $\bigl|\nabla^{\otimes k}a_\lambda\bigr|\le c_{\ast,k}\lambda^{k+1}(1+|x|)$ and $\bigl|\nabla^{\otimes k}\mathcal{T}_\lambda\bigr|\le c_{\ast,k}\lambda^{k}$ with $c_{\ast,k}$ signifying a constant that depends on $k$ but not on $\lambda$.

Another critically important bound follows from Lemma~7.2 which is the one below:

\begin{center}
\emph{Having fixed $R>2$ and $\varepsilon>0$, there exists a positive number, $\lambda''_{R,\varepsilon}$ such that} 

\emph{when $\lambda\le \lambda''_{R,\varepsilon}$ then $\int_{\B_R-\B_{1/R}}
\left|\nabla_r\psi_\lambda - \frac{\N(0)}{r}\,\psi_\lambda\right|^2
\le \varepsilon.$}
\end{center}
\listeqno\label{eq:9.11}

\noindent And, there is also the critically important bound which follows from Lemma~6.2 regarding the derivative of $\psi_\lambda$ in the $x_3$ direction:

\begin{center}
\emph{If $R>2$, and assuming that $\lambda$ is sufficiently small given $R$, then}
\begin{equation}\label{eq:9.12}
\int_{\mathtt{B}_R} |\nabla_3 \psi_\lambda|^2 \le c_0 \lambda .
\end{equation}
\end{center}

By way of an explanation for this last bound: The rescaling that defines $\psi_\lambda$ from $\psi$ identifies the integral in \eqref{eq:9.12} with an integral of $\frac{1}{\lambda}|\nabla_\nu \psi|^2$ over the radius $\lambda R$ ball in $S^3$ centered at the original point $p$ up to an $\mathcal{O}(\lambda R^2 \kappa^2(\lambda R))$ error which is due to the choice of the connection that defines the covariant derivative in \eqref{eq:9.12}. This identification results initially in the bound below (the notation has $\nu$ denoting the vector field in Lemma~6.2):
\begin{equation}\label{eq:9.13}
\int_{\mathtt{B}_R} |\nabla_3 \psi_\lambda|^2
\le
c_0\,\frac{1}{\lambda}\int_{\mathtt{B}_{\lambda R}} |\nabla_\nu \psi|^2
+ c_1 \lambda R^2 \kappa^2(\lambda R).
\end{equation}

Meanwhile, Lemma~6.2 gives an \emph{a priori} bound for the integral of the square of the covariant derivative of $\nabla_\nu\psi$ over any small radius ball centered at $p$; and that bound with the inequality in \eqref{eq:6.7} which holds now when $\psi$ is replaced by $\nabla_\nu\psi$ implies (via a dimension $3$ Sobolev inequality for $|\nabla_\nu\psi|$) that the integral of $|\nabla_\nu\psi|^6$ over $\mathtt{B}_{\lambda R}$ has an $R$ and $\lambda$ independent upper bound. Since
\begin{equation}\label{eq:9.14}
\int_{\mathtt{B}_{\lambda R}} |\nabla_\nu\psi|^2
\le
c_0(\lambda^3 R^3)^{2/3}
\left(\int_{\mathtt{B}_{\lambda R}} |\nabla_\nu\psi|^6\right)^{1/3},
\end{equation}
that uniform upper bound on the integral of $|\nabla_\nu\psi|^6$ over $\mathtt{B}_{\lambda R}$ with \eqref{eq:9.14} leads directly from \eqref{eq:9.13} to the asserted bound in \eqref{eq:9.12}.

The preceding \emph{a priori} bounds imply convergence of subsequences of the various $\psi_\lambda$'s as $\lambda\to 0$. The lemma below makes a precise statement and describes the limit.
\begin{lemma}\label{lem:9.1}
\emph{Suppose that $\{\lambda_j\}_{j=1,\ldots}$ is a decreasing sequence of $\lambda$-values with limit zero. There is a subsequence whose corresponding sequence of $\psi_\lambda$'s converges to a non-zero, section of $\mathcal{I}\times V$ in the $C^\infty$ topology on compact subsets of the complement in
$\mathbb{R}^3$ of the $x_3$-axis. This subsequence also converges in both the $\mathbb{L}$ and $\mathbb{H}$ topologies for the space of sections of $\mathcal{I}\times V$ over any ball in $\mathbb{R}^3$. Moreover, the limit (denoted by $\psi_{\ddagger}$) is harmonic in the sense that $\nabla^\dagger\nabla \psi_{\ddagger}=0$. It is also invariant with respect to translations in the $x_3$ direction; and it is equivariant under coordinate rescalings in the sense that $\phi_\lambda^*\psi_{\ddagger}=\lambda^{\N(0)}\psi_{\ddagger}$.}
\end{lemma}
\noindent\textbf{Proof of Lemma 9.1.}
Fix $R>1$. It follows from \eqref{eq:9.7} that the set $\{\psi_\lambda\}_{\lambda\in(0,1)}$ has an \emph{a priori} $\mathbb{H}$-norm bound on the $|x|\leq R$ ball in $\mathbb{R}^3$. This implies that any sequence from this set with $\lambda$ converging to zero along the sequence has a subsequence (hence denoted by $\{\psi_i\}_{i=1,2,\ldots}$) that converges weakly in the $\mathbb{H}$ topology and strongly in the $\mathbb{L}$-topology (the Rellich Lemma) on this ball. By taking diagonal subsequences for $R=2,3,\ldots$, no generality is lost by requiring that $\{\psi_i\}_{i=1,2,\ldots}$ converge weakly in the $\mathbb{H}$-topology on any $R>1$ version of the $|x|<R$ ball and strongly in the $\mathbb{L}$-topology on any such ball. What with \eqref{eq:9.6}, the Rellich lemma also implies that the limit section $\psi_{\ddagger}$ is non-zero and, more to the point, that
\begin{equation}\label{eq:9.15}
\int_{B_r} |\psi_{\ddagger}|^2 \;=\; 4\pi\, r^{\N(0)} \, .
\end{equation}

The convergence in \eqref{eq:9.6} with \eqref{eq:9.10} also implies that $\{\psi_i\}_{i=1,2,\ldots}$ converges to $\psi_{\ddagger}$ on any $|x|<R$ ball in the strong $\mathbb{H}$-topology.

\smallskip

In addition: The bounds supplied by \eqref{eq:9.8} and \eqref{eq:9.9} imply that $\{\psi_i\}_{i=1,2,\ldots}$ converges in the $C^\infty$ topology on any $R>1$ version of the set $A_R$. That fact coupled with \eqref{eq:9.10} implies that $\psi_{\ddagger}$ is a smooth section of $\mathcal{I}\times V$ on the complement of the $x_3$-axis in $\mathbb{R}^3$ and that it obeys $\nabla^\dagger\nabla\psi_{\ddagger}=0$. Meanwhile \eqref{eq:9.11} implies that $\nabla_r\psi_{\ddagger}=\frac{\N(0)}{r}\psi_{\ddagger}$ and \eqref{eq:9.12} implies that $\nabla_3\psi_{\ddagger}=0$.

\medskip

Let $\psi_{\ddagger}$ denote a $\lambda\to 0$ limit from Lemma~7.1 of a sequence from the set $\{\psi_\lambda\}_{\lambda\in(0,1]}$. The fact that $\psi_{\ddagger}$ is $\nabla^\dagger\nabla$ harmonic and that it is equivariant under coordinate rescalings and that it is annihilated by $\nabla_3$ determines any given $\mathcal{I}\times V$ component of $\psi_{\ddagger}$. To say more, let $\psi_{\ast}$ denote a component of $\psi_{\ddagger}$, thus a section of $\mathcal{I}$ over the complement in $\mathbb{R}^3$ of the $x_3$-axis. Write $\mathbb{R}^3$ as $\mathbb{C}\times\mathbb{R}$ with the $\mathbb{R}$ factor corresponding to the $x_3$ coordinate and with the $\mathbb{C}$ factor having the complex coordinate $z=x_1+i x_2$. Let $u_0$ denote a unit length, constant section of $\mathcal{I}$ over the complement of the $x_3$-axis and the negative real axis
(where $x_1<0$). Then $\psi_{\ast}$ must have the form
\begin{equation}\label{eq:9.16}
\psi_{\ast} \;=\; \bigl(\alpha z^{\N(0)} + \bar{\alpha}\,\bar{z}^{\N(0)}\bigr)\,u_0 \, ,
\end{equation}
with $\alpha$ denoting a complex number. This restriction on $\psi_{\ast}$ implies in turn (since $\psi_{\ast}$ is a section of $\mathcal{I}$ and $u_0$ is constant) that $\N(0)$ must come from the set $\bigl\{\tfrac12, \tfrac32, \ldots\bigr\}$.

\section{Proof of Theorem 1.2}

To set some background for the subsequent arguments: Having chosen one of the polytopes listed in Theorem~1.2, center that polytope in $\mathbb{R}^4$ so that its symmetry group is a finite subgroup of the group $\mathrm{SO}(4)$ with the latter acting linearly on $\mathbb{R}^4$. Let $G$ denote that subgroup. Now let $\Gamma$ denote the graph in $S^3$ that is given by the union of the radial projection to $S^3$ of the set of edges and vertices of that polytope. The group $G$ acts on $S^3-\Gamma$ and, according to Proposition~\ref{prop:lift} in Section~\ref{ch:pre}, it lifts in the manner described in Section~5 to an action on the line bundle $\mathcal{I}$. Moreover, these lifts are such that any given edge of $\Gamma$ is fixed by a $\mathbb{Z}/m\mathbb{Z}$ subgroup with $m$ being an odd integer and this $\mathbb{Z}/m\mathbb{Z}$ subgroup induces a linear action on the normal $2$-plane to any interior point of that edge (this is a $2\pi/m$ rotation of that normal plane).

Theorem~1.2 makes claims regarding homogeneous, $\mathbb{Z}/2$ harmonic $1$-forms, self-dual spinors and self-dual $2$-forms. The argument for the claims regarding $\mathbb{Z}/2$ harmonic $1$-forms and self-dual spinors are considered first. To start that argument, invoke Lemma~5.1 to obtain a non-trivial, $G$-invariant Laplace eigensection from $\mathcal{I}$'s version on $S^3$ of the Hilbert space $\mathbb{H}$. This section is denoted by $\psi$ and its eigenvalue is denoted by $\E$. This is to say that $\psi$ obeys the equation $-\Delta^{\perp}\psi=\E\psi$. To continue, let $q$ denote an interior point to an edge of $\Gamma$ and use this point to `zoom in' on the behavior of $\psi$ near $q$ as was done in Section~9. Doing that results in a version of $\psi_{*}$ that has the form depicted in~\eqref{eq:9.16}. Keeping in mind that the edge is fixed by a $\mathbb{Z}/m\mathbb{Z}$ subgroup of $G$, what is said in Section~5 (see the paragraph prior to Lemma~5.1) implies that the version of $\N(0)$ that appears in~\eqref{eq:9.16} must come from the set $\{\tfrac{m}{2},\tfrac{3m}{2},\dots\}$. Since $m$ is at least $3$, it then follows from the manner of convergence dictated in Section~9 that the conditions for bringing Lemma~4.4 to bear are met by this section $\psi$ which implies in particular that $d\psi$ is in the Hilbert space $\mathbb{H}$ and as such, it is a $\nabla^{\dagger}\nabla+\mathrm{Ric}$ eigensection from the $T^{*}S^{3}\otimes\mathcal{I}$ version of the Hilbert space $\mathbb{H}$. (Ric here denotes the Ricci curvature when viewed as an endomorphism of $T^{*}S^{3}$.) In particular, $|d\psi|$ extends across $\Gamma$ as a H\"older continuous function that vanishes on $\Gamma$.

Now let $r$ denote the radial coordinate on $\mathbb{R}^{4}$ and let $\pi$ denote the radial projection map from $\mathbb{R}^{4}\setminus\{0\}$ to $S^{3}$. Let $\kappa\equiv(1+\E)^{1/2}-1$. Then $\phi=r^{\kappa}\pi^{*}\psi$ is a homogeneous, $\mathbb{Z}/2$ harmonic function on $\mathbb{R}^{4}$ whose corresponding version of the set $Z$ is the closure of $\pi^{-1}\Gamma$ and whose corresponding real line bundle is the $\pi$-pull-back of the bundle $\mathcal{I}$ on $S^{3}$. (That real line bundle is subsequently denoted by $\mathcal{I}$ also.) The exterior derivative of $\phi$ (thus $d\phi$) is the promised homogeneous, harmonic, $\mathcal{I}$-valued $1$-form on $\mathbb{R}^{4}$ if $\E$ is greater than $3$ (because then $|d\phi|$ limits to zero as $r\to0$). With regards to this lower bound for $\E$: Lemmas~5.1 and~4.1 guarantee an infinite set of $\E$ in $(3,\infty)$.

To obtain a homogeneous, $\mathcal{I}$-valued, harmonic self-dual spinor, let $\eta$ denote a given, non-zero, constant anti-self-dual spinor on $\mathbb{R}^{4}$. Supposing that $\mathfrak{D}$ denotes the standard Dirac operator on $\mathbb{R}^{4}$ mapping anti-self-dual spinors to self-dual spinors (and vice-versa), then $\mathfrak{D}(\phi\eta)$ is a homogeneous, harmonic, $\mathcal{I}$ self-dual spinor on
$\mathbb{R}^{4}$ if $\E$ is greater than $3$.

The construction of Theorem~1.2's homogeneous $\mathbb{Z}/2\mathbb{Z}$ harmonic, self-dual $2$-forms requires more work. The following seven parts of this section exhibit that construction.

\medskip
\noindent\emph{Part 1:} The plan for what follows is to construct a version of what is denoted by $\nu^{\perp}$ in \eqref{eq:2.5} for certain positive values of $\kappa$. This $\nu^{\perp}$ will be a section of the $T^{*}S^{3}\otimes\mathcal{I}$ version of the Hilbert space $\mathbb{H}$ whose norm extends across $\Gamma$ as a H\"older continuous function that vanishes on $\Gamma$. Granted $\nu^{\perp}$, then the self-dual, $\mathcal{I}$ valued $2$-form that is depicted in \eqref{eq:2.4} is the desired, homogeneous, $\mathcal{I}$-valued harmonic self-dual $2$-form for Theorem~1.2.

The construction of $\nu^\perp$ starts by invoking Lemma~5.2 to obtain $G$-invariant, $(\nabla^\dagger\nabla + \mathrm{Ric})$-eigensections in the $T^*S^3\otimes\mathcal{I}$ version of the Hilbert space $\mathbb{H}$. Let $\psi$ denote such an eigensection and let $\E$ denote its eigenvalue. Let $d$ denote for the purposes of this proof the exterior derivative on $S^3$ and let $*$ denote the Hodge star on $S^3$. Granted this notation, set $\phi = *d*\psi$, this being a section of $\mathcal{I}$ which is a priori only in the Hilbert space $\mathbb{L}$. But note in this regard that if each edge of $\Gamma$ is fixed by a $\mathbb{Z}/m\mathbb{Z}$ subgroup of $G$ with $m$ odd and at least $5$, then (by virtue of \eqref{eq:5.8} and what is said in Section~9), the assumptions required by Lemma~4.4 will be
met and as a consequence, $\phi$ will come from $\mathcal{I}$'s version of $\mathbb{H}$.)

Anyway, with regards to $\phi$, this obeys the equation $-\Delta^\perp\phi = \E\phi$ because the eigenvalue equation $(\nabla^\dagger\nabla + \mathrm{Ric})\psi = \E\psi$ can be written as
\begin{equation}\label{eq:10.1}
-\bigl(*d(*d\psi) + d(*d\psi)\bigr) = \E\psi ,
\end{equation}
and because $\Delta^\perp\phi = *d(*d\phi)$. Note also that $\phi$ is a $G$-invariant section of $\mathcal{I}$.

If it is the case that $\phi$ is from $\mathbb{H}$ (which is guaranteed if $m>3$), then $d\phi$ is also in $\mathbb{H}$. This is a consequence of Lemma~4.4 whose requirements are met by $\phi$ because $\phi$ is a $G$-invariant, $(-\Delta^\perp)$-eigensection from $\mathbb{H}$. Indeed, being a $-\Delta^\perp$-eigensection from $\mathbb{H}$ implies that $\phi$ has its versions of \eqref{eq:9.16}, and being $G$-invariant implies that the values of $\N(0)$ that appear there are no smaller than $\tfrac{3}{2}$ (see Section~5).

Supposing that $\phi$ is in $\mathbb{H}$ and thus $d\phi$ too, it then follows that $y = \psi - d\phi$ is a coclosed section of $\mathbb{H}$ that obeys
\begin{equation}\label{eq:10.2}
*d(*dy) = \E y .
\end{equation}
Since $y$ is coclosed, it is a $G$-invariant, $(\nabla^\dagger\nabla + \mathrm{Ric})$-eigensection from the $T^*S^3\otimes\mathcal{I}$ version of $\mathbb{H}$. As explained in subsequent parts of this proof, $*dy$ is also from $\mathbb{H}$ (this follows directly from \eqref{eq:9.16} and \eqref{eq:5.8} if $m>3$). Granted that $*dy$ is from $\mathbb{H}$, then so is $y + \frac{1}{\sqrt{\E}}*dy$. This is the desired section $\nu^\perp$ for use in \eqref{eq:2.5}. (The number $\kappa$ is $\sqrt{\E}$.)

With regards to the situation when $m=3$: In that case, the top bullet of \eqref{eq:5.8} allows for the case $k=1$ which implies that $\psi$'s version of $\N(0)$ can be equal to $\tfrac{1}{2}$ which is too small to guarantee \emph{a priori} that $\phi$ (which is $*d*\psi$) is from $\mathbb{H}$. Even so, the manner of convergence to the rescaling limits of $\psi$ near the interior points of the edges of $\Gamma$ as described in Section~9 with what is said by \eqref{eq:5.10} suggests that $\phi$ is indeed in $\mathbb{H}$ notwithstanding the fact that $\nabla\psi$ is not in $\mathbb{H}$. This same manner of convergence to rescaling limits with \eqref{eq:5.10} also suggests that $*dy$ is in $\mathbb{H}$ notwithstanding the fact that $\nabla\psi$ is not.

\medskip
\noindent\emph{Part 2:} This part of the proof and Parts 3 and 4 explain why $\phi$ is indeed in $\mathbb{H}$. To start the explanation, bring Lemma~5.1 to bear so as to obtain a $G$-invariant section from $\mathcal{I}$'s version of the Hilbert space $\mathbb{H}$ (to be denoted by $\mathfrak{\upupsilon}$) that obeys the equation $-\Delta^{\perp}\upupsilon=\E\phi$. Let $\varsigma=\phi-\upupsilon$. This section $\varsigma$ is a $\Delta^{\perp}$-harmonic section of $\mathcal{I}$ from the Hilbert space $\mathbb{L}$ since $\Delta^{\perp}\phi$ and $\Delta^{\perp}\upupsilon$ are both $\E\phi$. A key point now is that $\varsigma\equiv 0$ if $\varsigma$ is in $\mathbb{H}$; and if $\varsigma\equiv 0$, then $\phi=\upupsilon$ and so $\phi$ is from $\mathbb{H}$ as claimed.

The question of whether or not $\varsigma$ is from $\mathbb{H}$ concerns its behavior near $\Gamma$ since it is smooth on $S^3-\Gamma$. To see about this, the rest of this part of the proof and all of Part~3 focus attention on some given point in the interior of an edge of $\Gamma$. This point is denoted subsequently by $p$. Rotate the $3$-sphere in $\mathbb{R}^4$ if needed so that the edge containing $p$ is an arc in the $|x|=1$ part of the $x_3=x_4$ plane in $\mathbb{R}^4$. Let $z=(x_3+i x_4)$ and then write $x_1+i x_2$ as $(1-|z|^2)^{1/2}e^{it}$. The $\mathbb{R}/2\pi\mathbb{Z}$ valued coordinate $t$ is chosen so that $p$ is the $t=0$ point on the $|z|=0$ great circle. It proves convenient to write $z$ in polar form as $z=\rho e^{i\theta}$ with $\rho>0$ and $\theta$ being $\mathbb{R}/2\pi\mathbb{Z}$ valued. The metric on a ball of radius less than $\pi$ in $S^3$ when written using the coordinates $\rho$, $\theta$ and $t$ has the form
\begin{equation}\label{eq:10.3}
\frac{1}{(1-\rho^2)}\,d\rho^2 + \rho^2\,d\theta^2 + (1-\rho^2)\,dt^2 \, .
\end{equation}

As noted in the proof of Lemma~6.2, the constant translations of $t$ are isometries. (These identify the normal disks to the given edge of $\Gamma$. The constant rotations of these normal disks (thus $\theta\to\theta+\mu$ with $\mu$ constant) are also isometries and they commute with those given by constant translations of $t$. These isometries will be exploited momentarily.)

Fix a constant, unit length section of $\mathcal{I}$ in the part of the $\rho<\frac{1}{1024c}\,r_p$ and $|t|<\frac{1}{1024c}\,r_p$ polydisc where $\theta\in(-\pi,\pi)$. Use that section to depict $\varsigma$ on that part of the polydisc as an $\mathbb{R}$-valued function. Fourier series can then be used to write $\varsigma$ in this polydisc as a sum,
\begin{equation}\label{eq:10.4}
\varsigma = \sum a_{k,n}(\rho)\, e^{i\left(k+{1\over 2}\right)\theta}\, e^{i1024\pi c\, n\, t/r_p} \, ,
\end{equation}
this sum being a double sum over the values of $k$ from $\mathbb{Z}$ and the values of $n$ from $\mathbb{Z}$. With regards to this Fourier decomposition: Keep in mind that $\varsigma$ is smooth on the complement of the $\rho=0$ locus and it is also in $\mathbb{L}$ (which is to say that $|\varsigma|^2$ has finite integral on $S^3$). It follows as a consequence that the sum in \eqref{eq:10.4} converges pointwise assuming $\rho$ is positive and both $\rho$ and $|t|$ are less than $\frac{1}{1024c}\,r_p$.

To say more about the functions $a_{k,n}(\cdot)$: Note first that $\overline{a}_{k,n}=a_{-k-1,-n}$ since $\varsigma$ is $\mathbb{R}$-valued. A second point of note: Because $\varsigma$ is harmonic, any given $a_{k,n}$ obeys the second order ordinary differential equation
\begin{equation}\label{eq:10.5}
-\frac{1}{\rho}\frac{d}{d\rho}\Bigl(\rho(1-\rho^2)\frac{d}{d\rho}a_{k,n}\Bigr)
+ \frac{\left(k+{1\over2}\right)^2}{\rho^2}\,a_{k,n}
+ \frac{\bigl(\pi n/2r\bigr)^2}{(1-\rho^2)}\,a_{k,n}
= 0,
\end{equation}
with $r$ denoting here $\frac{1}{1024c}\,r_p$. A key point regarding the preceding equation (which can be derived using standard theorems about ODEs) is that any solution must have the form
\begin{equation}\label{eq:10.6}
a_{k,n}(x)
=
\bigl(\alpha_{k,n}+\rho^2 \f_{k,n}\bigr)\,\rho^{k+\frac12}
+
\bigl(\beta_{k,n}+\rho^2 \h_{k,n}\bigr)\,\rho^{-(k+\frac12)},
\end{equation}
with $\alpha_{k,n}$ and $\beta_{k,n}$ being constant, and with $|\f_{k,n}|$ and $|\h_{k,n}|$ being bounded functions of $\rho$. (This is because $\rho=0$ is the only singular point of that equation.) Since $\varsigma$ is in $\mathbb{L}$, the depiction of $a_{k,n}$ in \eqref{eq:10.6} implies in particular (by taking $\rho\ll 1$) that $\alpha_{k,n}=0$ for $k\le -2$ and $\beta_{k,n}=0$ for $k\ge 1$. But there is more: Since $\varsigma$ is $G$ invariant, $\alpha_{k,n}$ and $\beta_{k,n}$ must be zero unless $2k+1 = 0 \bmod m$. These conditions require (in particular) that $\alpha_{-1,n}=\alpha_{0,n}=0$ and that
$\beta_{0,n}=\beta_{-1,n}=0$.

Because of the preceding constraints, the norm of $a_{k,n}$ obeys
\begin{equation}\label{eq:10.7}
|a_{k,n}|\le z_{k,n}\,\rho^{3/2}\,,
\end{equation}
with $z_{k,n}$ being constant. Here is more to keep in mind about the set $\{a_{k,n}\}$: The fact that $|\varsigma|^2$ has finite integral implies that the integral below is finite \begin{equation}\label{eq:10.8}
\sum \int_{0}^{\frac{r_p}{1024c}} |a_{k,n}|^2\,\rho\,d\rho \le c_0(1+\E)\,,
\end{equation}
where the sum is over the same set as in \eqref{eq:10.4}.\\

\noindent\textit{Part 3:} Given $\varepsilon>0$, let $\sigma_\varepsilon$ denote a smooth function of the coordinate $\rho$ that is equal to $1$ where $\rho<\varepsilon$ and equal to $0$ where $\rho>2\varepsilon$; and also with $|d\sigma_\varepsilon|\le c_0\,\varepsilon^{-1}$. The constructions that follow use $\sigma_p$ to denote the $\varepsilon=\frac{1}{2048c}\,r_p$ version of $\sigma_\varepsilon$. Multiply both sides of \eqref{eq:10.6} by $(1-\sigma_\varepsilon)\sigma_p\,\overline{a}_{k,n}$ and then integrate the result. Since the integrand has compact support on $\bigl(0,\frac{1}{1024c}\,r_p\bigr)$, an integration by parts can be done (and \eqref{eq:10.7} can be invoked) to obtain the following inequality:
\begin{equation}\label{eq:10.9}
\int_{0}^{\frac{r_p}{2048c}}
\left(
\left|\frac{d}{d\rho}a_{k,n}\right|^2
+\frac{\left(k+\frac12\right)^2}{\sin^2\rho}\,|a_{k,n}|^2
+\frac{\left(\frac{\pi n}{2r}\right)^2}{\cos^2\rho}\,|a_{k,n}|^2
\right)\rho\,d\rho
\le
c_0\,|z_{n,k}|^2\,\varepsilon^3
+
\frac{c_0}{r_p^2}\int_{0}^{\frac{r_p}{1024c}} |a_{k,n}|^2\,\rho\,d\rho \,.
\end{equation}

With \eqref{eq:10.9} in hand: Having fixed $\delta>0$ and a large integer $N$, it follows from \eqref{eq:10.9} and \eqref{eq:10.8} that there exists $\varepsilon_{\delta,N}>0$ such that if $\varepsilon<\varepsilon_{\delta,N}$, then
\begin{equation}\label{eq:10.10}
\sum^{N}\int_{0}^{\frac{r_p}{2048c}}
\sigma_\varepsilon
\left(
\left|\frac{d}{d\rho}a_{k,n}\right|^2
+\frac{\left(k+\frac12\right)^2}{\sin^2\rho}\,|a_{k,n}|^2
+\frac{\left(\frac{\pi n}{2r}\right)^2}{\cos^2\rho}\,|a_{k,n}|^2
\right)\rho\,d\rho
\le
\delta +\frac{c_0}{r_p^2}\,(1+\E),
\end{equation}
where the symbol $\displaystyle\sum^N$ indicates there is a sum over pairs $(k,n)$ such that $|k|<N$ and $|n|<N$. Another important point with regards to \eqref{eq:10.9}: Having fixed $\delta>0$ then that inequality also implies that there exists positive $\varepsilon'_{k,N}$ such that if $\varepsilon<\varepsilon'_{k,N}$, and $N'>N$, then
\begin{equation}\label{eq:10.11}
\sum^{N,N'} \int_{0}^{r_p/2048c}
\sigma_{\varepsilon}\!\left(
\left|\frac{d}{d\rho}a_{k,n}\right|^{2}
+\frac{\left(k+{1\over2}\right)^{2}}{\sin^{2}\rho}\,|a_{k,n}|^{2}
+\frac{(\pi n/2r)^{2}}{\cos^{2}\rho}\,|a_{k,n}|^{2}
\right)\rho\,d\rho
<\delta,
\end{equation}
where the symbol $\displaystyle\sum^{N,N'}$ indicates that the sum is over all pairs $(k, n)$ with $|k|$ and $|n|$ between $N$ and $N'$. The observations in \eqref{eq:10.10} and \eqref{eq:10.11} imply this: Set $r=\frac{1}{1024c}\,r_p$ and set $\chi_p$ to denote for now the function $\chi\bigl(2^{10}\dist(p,\cdot)/r-1\bigr)$. Then, the sequence
\begin{equation}\label{eq:10.12}
\left\{
\chi_p \sum^N a_{k,n}(\rho)\, e^{\mathrm{i}\left(k+{1\over2}\right)\theta}\, e^{\mathrm{i}\pi 1024 c\, n\, t/r_p}
\;:\; N=1,2,\ldots
\right\}
\end{equation}
is a Cauchy sequence in the $\mathtt{B}_r$ version of the Hilbert space $\mathbb{H}_{\mathtt{B}}$. This last conclusion implies in turn that $\chi_p\varsigma$ is in $\mathbb{H}$ since this sequence converges in the $\mathtt{B}_r$ version of $\mathbb{L}_{\mathtt{B}}$ to $\chi_p\varsigma$.

\medskip
\noindent\emph{Part 4:} To summarize: Part 3 established that $|\nabla\varsigma|^2$ is integrable on the complement in $S^3$ of any open set containing the vertices of $\Gamma$. With this understood, the focus in this part of the proof is on the behavior of $\nabla\varsigma$ near some given vertex of $\Gamma$. Use $p$ now to denote that vertex. Fix $\varepsilon\in \bigl(0,\frac{1}{2048c}\,r_p\bigr)$ but very small. Let $\chi_\varepsilon$ and $\chi_p$ denote the respective $r=\varepsilon$ and $r=\frac{1}{2048c}\,r_p$ versions of the function $\chi\bigl((\dist(\cdot,p)-1)/r\bigr)$. Since $-\Delta^{\perp}\varsigma=0$, it is also true that
\begin{equation}\label{eq:10.13}
-\frac{1}{2}\Delta^{\perp}|\varsigma|^{2}+|\nabla\varsigma|^{2}=0 \, .
\end{equation}

Multiply both sides of this last identity by $(1-\chi_\varepsilon)\chi_p\varsigma$ and then integrate the result over the radius $\frac{1}{1024c}\,r_p$ ball centered at $p$. Two instances of integration by parts leads from that integral identity to the inequality below:
\begin{equation}\label{eq:10.14}
\int (1-\chi_\varepsilon)\chi_p|\nabla\varsigma|^{2}
\le \frac{c_0}{\varepsilon^{2}}\int (1-\chi_{\varepsilon/16})\chi_{2\varepsilon}|\varsigma|^{2}
+ c_{p,\E}\, ,
\end{equation}
where $c_{p,\E}$ is a positive number that is determined by $p$ and $\E$. This last inequality leads directly to an \textit{a priori} bound for the integral of $|\nabla\varsigma|^2$ on the radius $\frac{1}{4096c}\,r_p$ ball centered at $p$ if, for every small but positive $\varepsilon$, there exists a $c_0\varepsilon^2$ bounds for the integral of $\varsigma$ on radius $2\varepsilon$ ball centered at $p$.

Find such a bound: Let $\mathtt{B}_r$ denote for the moment the $r=\frac{1}{1024c}\,r_p$ radius ball in $S^3$ centered at $p$. Let $v$ denote the valency of the vertex $p$. The edges of $\Gamma$ that are incident to $p$ intersect $\partial\mathtt{B}_r$ as a configuration of $v$ distinct points; denote
this set by
\begin{align}
\mathfrak{p} := \Gamma \cap \partial\mathtt{B}_r .
\end{align}
A chosen Gaussian coordinate system centered at $p$ identifies $\partial\mathtt{B}_r$ with the radius $r$ sphere centered at the origin in $\mathbb{R}^3$, and then the radial projection from the origin in $\mathbb{R}^3$ identifies that sphere with the radius $1$ sphere in $\mathbb{R}^3$. That radius $1$ sphere is denoted by $S^2$. Under this identification, the image of $\mathfrak{p}\subset \partial\mathtt{B}_r$ is a configuration of $v$ points in $S^2$, which we also denote by $\mathfrak{p}$. Radial pull-back from $\partial \mathtt{B}_r$ of the line bundle $\mathcal{I}$ defines a real line bundle on $S^2-\mathfrak{p}$ which will be denoted by $\mathcal{I}$ also.

Borrowing from \cite{taubes2024topological}, let $\mathbb{H}_{\mathfrak{p}}$ denote the Hilbert space that is obtained by taking the completion of the space of compactly supported sections of the $S^2-\mathfrak{p}$ version of $\mathcal{I}$ using the norm whose square is the functional
\begin{equation}\label{eq:10.15}
\f \;\longmapsto\; \int_{S^2} |d\f|^2 \, .
\end{equation}
Define the Hilbert space $\mathbb{L}_{\mathfrak{p}}$ to be the completion of that same space of sections of $\mathcal{I}$ on $S^2-\mathfrak{p}$ using the norm whose square sends a section to the $S^2$ integral of the square of its norm.

Let $\Delta_2$ denote the Laplace operator on $S^2$ for the round metric but with it acting on sections of $\mathcal{I}$ over $S^2-\mathfrak{p}$.  The section $\f$ of $\mathcal{I}$ on $S^2-\mathfrak{p}$ is said to be an eigensection in $\mathbb{H}_{\mathfrak{p}}$ for $-\Delta_2$ when it obeys the equation $-\Delta_2 \f = \lambda \f$ with $\lambda$ being a non-negative real number.  (The number $\lambda$ is $\f$'s eigenvalue.) Proposition 1.1 from Taubes-Wu \cite{taubes2024topological} makes an assertion to the effect that the set of eigenvalues of these eigensections from $\mathbb{H}_{\mathfrak{p}}$ is discrete, bounded away from zero and has no accumulation points.  This same theorem also asserts that each such eigenvalue has finite multiplicity. Meanwhile, Proposition 1.1 \cite{taubes2024topological} makes an assertion to the effect that $\mathbb{L}_{\mathfrak{p}}$ has an orthonormal basis of $(-\Delta_2)$-eigensections from $\mathbb{H}_{\mathfrak{p}}$. Let $\{\f_j\}_{j=1,2,\ldots}$ denote this orthonormal basis, labeled using the convention that the eigenvalue of any given version of $\f_j$ is no larger than that of $\f_{j+1}$.

Returning to $\varsigma$: Let $r = \frac{1}{1024c}\,r_p$ and let $\mathtt{B}_r$ denote again the radius $r$ ball centered at $p$.  Let $\rho$ now denote the function on $\mathtt{B}_r$ giving the distance to $p$.  The section $\varsigma$ can be written on $\mathtt{B}_r$ using that orthonormal basis of $(-\Delta_2)$-eigensections from $\mathbb{H}_{\mathfrak{p}}$ as
\begin{equation}\label{eq:10.16}
\varsigma = \sum \aaa_j(\rho)\, \f_j
\end{equation}
where the sum is over the indexing set $j\in\{1,2,\ldots\}$ for the $\mathbb{L}_{\mathfrak{p}}$-orthonormal basis of $(-\Delta_2)$-eigensections.  The coefficient function set $\{\aaa_j(\rho)\}_{j=1,2,\ldots}$ obey the ordinary differential equation
\begin{equation}\label{eq:10.17}
-\frac{1}{\sin^2\rho}\,\frac{d}{d\rho}\!\left(\sin^2\rho\,\frac{d}{d\rho}\aaa_j\right)
+\lambda_j\,\frac{1}{\sin^2\rho}\,\aaa_j
=0 \, .
\end{equation}
Because $|\varsigma|^2$ is integrable on $S^3$, the function set $\{\aaa_j\}$ are constrained to obey the bound
\begin{equation}\label{eq:10.18}
\sum \int_{0}^{r} |\aaa_j|^2 \sin^2\rho\, d\rho \;\leq\; c'_{p,\E}
\end{equation}
with $c'_{p,\E}$ denoting another positive number that depends only on $p$ and $\E$.

Fix a non-negative integer $j$ and set $\aaa\equiv\aaa_j$ and $\lambda\equiv\lambda_j$.  Then, introduce the numbers
\begin{equation}\label{eq:10.19}
\mu \equiv \frac12\Bigl(-1 + (1+4\lambda)^{1/2}\Bigr)
\qquad\text{and}\qquad
\mu' \equiv \frac12\Bigl(1 + (1+4\lambda)^{1/2}\Bigr)\, .
\end{equation}
Standard ordinary differential equation technology can be used to see that the function $\aaa(\cdot)$ must have the form in \eqref{eq:10.20}:
\begin{equation}\label{eq:10.20}
\aaa(\rho)
=
\alpha\bigl(1+\rho^2\mathfrak{s}\bigr)\rho^\mu
+
\beta\bigl(1+\rho^2 \mathfrak{t}\bigr)\rho^{-\mu'},
\end{equation}
where $\alpha$ and $\beta$ are $j$-dependent real numbers, and $\mathfrak{s}$ and $\mathfrak{t}$ are $j$-dependent functions of $\rho$ obeying $|\mathfrak{s}|+|\mathfrak{t}|\le c_j$ with $c_j$ being a positive real number.

Granted \eqref{eq:10.20}, it then follows that the contribution to the small $\varepsilon$ integral in \eqref{eq:10.14} from the $\alpha\rho^\mu$ term in \eqref{eq:10.20} will limit to zero as $\varepsilon$ limits to zero because $\mu$ is positive. Meanwhile, the $\beta\rho^{-\mu'}$ term in \eqref{eq:10.20} will be absent ($\beta=0$) if $\mu'\ge\frac32$ because if that isn't the case, then the presence of the $\beta\rho^{-\mu'}$ term with $\mu'\ge\frac32$ runs afoul of the constraint in \eqref{eq:10.18}. Since $\mu'$ is no smaller than $\frac32$ when $\lambda\ge\frac34$, the issue to address is whether a $(-\Delta_2)$-eigensection from $\mathbb{H}_p$ with eigenvalue less than $\frac34$ can appear in \eqref{eq:10.16}.

To see about a lower bound for those $\lambda_j$: The key input for the analysis of those eigenvalues is the fact that each edge of $\Gamma$ is invariant under the action of an order $3$ subgroup of $G$ that acts freely on the normal disks to interior points of that edge. That subgroup for an incident edge to $p$ acts as an order three rotation on the constant radius spheres centered at the vertex $p$. With this understood, fix a point (to be denoted by $q$) from the configuration $\mathfrak{p}$. The order $3$ subgroup of $G$ that fixed $\mathfrak{p}$ acts to permute the other points in $\mathfrak{p}$. Each $\f_j$ that appears in \eqref{eq:10.16} must be invariant with regards to this $\mathbb{Z}/3\mathbb{Z}$ subgroup which implies in turn that $|\f_j|$ near $q$ is bounded by $c_0\,\operatorname{dist}(\cdot,q)^{3/2}$ (See \cite[Proposition 2.2, Equation 4.6]{taubes2020examples}).

To explore the ramifications of this observation, introduce by way of notation $\nabla_2$ to denote the metric covariant derivative on sections of $T^*S^2$. The section $d\f_j$ of $T^*S^2\otimes\mathcal{I}$ obeys the differential equation
\begin{equation}\label{eq:10.21}
\nabla_2^\dagger\nabla_2(d\f_j) + d\f_j = \lambda_j\,d\f_j \, .
\end{equation}

Having fixed $\varepsilon>0$ but tiny, let 
\begin{align}
\chi_\varepsilon=\chi\bigl(1-\varepsilon^{-1}\operatorname{dist}(\cdot,\mathfrak{p})\bigr),
\end{align}
where\begin{align}
\displaystyle\dist(\cdot,\mathfrak{p}):=\min_{p\in\mathfrak{p}}\dist(\cdot,p).
\end{align}
To be sure, this function has value $1$ where the distance to $\mathfrak{p}$ is greater than $\varepsilon$, and it has value zero where the distance to $\mathfrak{p}$ is less than $\varepsilon/4$. Take the inner product of both sides of the equation in \eqref{eq:10.21} with $\chi_\varepsilon\,d\f_j$ and then integrate the result over $S^2$. Since $\chi_\varepsilon$ is zero near $\mathfrak{p}$, the result after an integration by parts leads to this:
\begin{equation}\label{eq:10.22}
\int_{S^2}\chi_\varepsilon\Bigl(|\nabla_2 d\f_j|^2 + |d\f_j|^2\Bigr)
\le
\lambda_j\int_{S^2}|d\f_j|^2 + \mathfrak{e},
\end{equation}
with $\mathfrak{e}$ being the term from the integration by parts with derivatives on $\chi_\varepsilon$. In particular, the norm of $\mathfrak{e}$ obeys the bound
\begin{equation}\label{eq:10.23}
|\mathfrak{e}|
\leq c_0\,\varepsilon^{-2}\int_{S^2}\chi_{\varepsilon/8}\bigl(1-\chi_{4\varepsilon}\bigr)\,|d \f_j|^2 \, .
\end{equation}

To see about the integral on the right in \eqref{eq:10.23}, multiply both sides of the eigenvalue equation $-\Delta_2 \f_j=\lambda_j \f_j$ by $\chi_{\varepsilon/8}\bigl(1-\chi_{4\varepsilon}\bigr) \f_j$ and then integrate the result. An integration by parts leads in turn to the inequality below:
\begin{equation}\label{eq:10.24}
\int_{S^2}\chi_{\varepsilon/8}\bigl(1-\chi_{4\varepsilon}\bigr)\,|d \f_j|^2
\leq c_0\,\varepsilon^{-2}\int_{S^2}\bigl(1-\chi_{8\varepsilon}\bigr)\,|\f_j|^2 \, .
\end{equation}

This last inequality leads to a $c_0\varepsilon^{3}$ bound for the integral over $S^2$ of $\int_{S^2}\chi_{\varepsilon/8}\bigl(1-\chi_{4\varepsilon}\bigr)|d \f_j|^2$ since $|\f_j|<c_0\varepsilon^{3/2}$ on the domain of the integral on the left hand side of \eqref{eq:10.23}. That bound with \eqref{eq:10.23} implies in turn that $|\mathfrak{e}|\leq c_0\varepsilon$ which has limit zero as $\varepsilon\to 0$. With this understood, take $\varepsilon$ to zero in \eqref{eq:10.22} to see that $\lambda_j$ must be greater than $1$. (Note that in this regards that $d\f_j$ can't be identically zero since $\mathcal{I}$ is not the trivial line bundle.)\\

\noindent\emph{Part 5:} To summarize where things stand at this point: Sections 2--4 have established that the $T^*S^3\otimes\mathcal{I}$ valued $1$-form $y\equiv \psi-d\phi$ is in the $T^*S^3\otimes\mathcal{I}$ version of the Hilbert space $\mathbb{H}$. This $1$-form is coclosed and it obeys \eqref{eq:10.2}. It is also $G$-invariant. This part of the proof and Part~6 explain why $\ast d y$ is also in this same version of $\mathbb{H}$. This part focuses on the behavior $\ast d y$ near interior points of the edges of $\Gamma$. The next part of the proof considers $\ast d y$ near the vertices of $\Gamma$.

Supposing that $p$ denotes a point in the interior of an edge of $\Gamma$, reintroduce the coordinates $(z,t)$ from Part~2 of this section for a neighborhood of $p$ in $S^3$.
(These are defined just prior to \eqref{eq:10.3}.) Thus, $z$ is a complex coordinate and $t$ is $\mathbb{R}$-valued; and both are zero at the point $p$. With regards to \eqref{eq:10.3}, keep in mind that $\rho=|z|$. The metric in \eqref{eq:10.3} when written using $z$ and $t$ has the form below:
\begin{equation}\label{eq:10.25}
dz\otimes d\bar z + dt^2 + \frac{1}{1-|z|^2}\bigl(z\,d\bar z + \bar z\,dz\bigr)^2 + |z|^2\,dt^2 \, .
\end{equation}

Now write $y$ using the coordinate basis as
\begin{equation}\label{eq:10.26}
y = w_t\,dt + w\,dz + \bar w\,d\bar z \, .
\end{equation}
(This is the depiction in \eqref{eq:5.4} if $w$ is written using its real and imaginary parts as $w=w_1-iw_2$.)

Let $\beta_p$ denote for the moment the $r=\frac{1}{8192c}r_p$ version of $\chi\!\left(\frac{\operatorname{dist}(\cdot,p)}{r}-1\right)$ which is equal to $1$ near $p$ and equal to zero where $\operatorname{dist}(\cdot,p)$ is greater than $\frac{1}{4096c}r_p$. Lemma~6.2 guarantees that $\beta_p\,\nabla_t y$ is in $\mathbb{H}$. As will be explained momentarily, the fact that $y$ is $G$-invariant implies that $\beta_p\,dw_t$ is also in $\mathbb{H}$.  Granted these last two points, it then follows that $|\nabla(*dy)|^2$ has finite integral on the radius $\frac{1}{2\times 8192\,c}\,r_p$ ball centered at $p$ if this is so for $|\nabla(\partial w/\partial\bar z)|^2$.  For the latter, this follows algebraically if $\beta_p\,dw_t$ and $\beta_p\,\nabla_t y$ are in $\mathbb{H}$ using the two identities $*d(*dy)=\E\,y$ and $d^{\ast}y=0$.

To see about $dw_t$, the first key point is that $w_t$ in the radius $\frac{1}{1024\,c}\,r_p$ ball centered at $p$ obeys a differential equation that has the schematic form below:
\begin{equation}\label{eq:10.27}
-\widehat{\Delta}\,w_t \;+\; 2z\,\frac{\partial}{\partial t}w \;+\; 2\bar z\,\frac{\partial}{\partial t}\bar w
\;=\; \E\,w_t \,,
\end{equation}
with $\widehat{\Delta}$ denoting the operator depicted below (the notation reintroduces the nonnegative function $\rho$ and the $\mathbb{R}/2\pi\mathbb{Z}$ valued function $\theta$ by writing $z$ as $\rho e^{i\theta}$):
\begin{equation}\label{eq:10.28}
\widehat{\Delta}
=\frac{1-\rho^2}{\rho}\,\frac{\partial}{\partial\rho}\!\left(\rho\,\frac{\partial}{\partial\rho}\right)
+\frac{1}{\rho^2}\,\frac{\partial^2}{\partial\theta^2}
+\frac{1}{1-\rho^2}\,\frac{\partial^2}{\partial t^2}\,.
\end{equation}

Fourier series (mimicking what was done in Part~2) will be used with \eqref{eq:10.27} to analyze the behavior of $w_t$.  To this end, reintroduce from Part~2 the product structure for $\mathcal{I}$ where $\theta\in(-\pi,\pi)$.  Having done that, write
\begin{equation}\label{eq:10.29}
\beta_p\,w_t
=\sum u_{k,n}(\rho)\,e^{i(k+\frac12)\theta}\,e^{\,i\,1024\pi c\,n\,t/r_p}\,.
\end{equation}

Introduce by way of notation,
\begin{align}
\mathfrak{k} \equiv \beta_p\Bigl(2z\,\frac{\partial}{\partial t}w+2\bar z\,\frac{\partial}{\partial t}\bar w\Bigr)
\end{align}
and write $\mathfrak{k}$ similarly as
\begin{equation}\label{eq:10.30}
\mathfrak{k}
=\sum \chi_{k,n}(\rho)\,e^{i(k+\frac12)\theta}\,e^{\,i\,1024\pi c\,n\,t/r_p}\,.
\end{equation}
Note in particular: The fact that $\beta_p\,\nabla_t y$ is in $\mathbb{H}$ implies that
\begin{equation}\label{eq:10.31}
|\chi_{k,n}| \le c_0\,\rho \,.
\end{equation}

Now it follows from \eqref{eq:10.31} and \eqref{eq:10.28} that $u_{k,n}$ can be written as
\begin{equation}\label{eq:10.32}
u_{k,n}
=\bigl(\alpha_{k,n}+\rho^2 \f_{k,n}\bigr)\,\rho^{k+\frac12}
+\bigl(\beta_{k,n}+\rho^2 \mathfrak{h}_{k,n}\bigr)\,\rho^{-(k+\frac12)}
+\mathfrak{e}_{k,n}\,,
\end{equation}
where $|\mathfrak{e}_{k,n}|\le c_0\,\rho^3$ and where $|\f_{k,n}|$ and $|\mathfrak{h}_{k,n}|$ are bounded functions of $\rho$.  Because $w_t$ is in $\mathbb{H}$ and it is $G$-invariant, an argument essentially a verbatim copy of that used subsequent to (10.6) can be used to deduce that $\alpha_{k,n}=0$ for $k\le 0$ and that $\beta_{k,n}=0$ for $k\ge -1$. Thus,
\begin{equation}\label{eq:10.33}
|u_{k,n}| \le c_0 z_{k,n}\,\rho^{3/2}\,.
\end{equation}

With this bound in hand, the arguments from Part~3 can be copied almost verbatim to see
that $\beta_p\nabla w_t$ is in $\mathbb{H}$ as claimed.\\

\noindent\emph{Part 6:}
This part of Section~10 and the final part explain why $|\nabla(\ast dy)|^2$ is integrable near the vertices of $\Gamma$. To keep the notation succinct, let $q$ now denote $\ast dy$. Keep in mind from Part~5 that $|\nabla q|^2$ is integrable on sufficiently small radius balls centered at the interior points to the edges of $\Gamma$.

With the preceding in mind, let $\rho$ now denote the function on the $\frac{1}{10^6c}\,r_p$ ball centered at $p$ that gives the distance to $p$. Introduce as before, a Gaussian coordinate system centered at $p$ to identify this ball centered at $p$ with the same radius ball centered at the origin in $\mathbb{R}^3$. Use $\rho$ now to denote the radial coordinate on $\mathbb{R}^3$. Let $S^2$ again denote the $\rho=1$ sphere in $\mathbb{R}^3$ and let $\mathfrak{p}\in S^2$ denote the intersection between $S^2$ and the radial projection to $S^2$ of the intersection of $\Gamma$ with any constant $\rho\in\bigl(0,\frac{1}{10^6c}\,r_p\bigr)$ sphere. (The graph $\Gamma$ in the $\rho<\frac{1}{10^6c}\,r_p$ ball is the intersection of that ball with the cone on the configuration $\mathfrak{p}$.)

Let $q_\rho$ denote the $d\rho$ component of $q$ on that $\rho<\frac{1}{10^6c}\,r_p$ ball (this is the pairing between $q$ and the dual vector field $\frac{\partial}{\partial\rho}$). The rest of this part of the proof focuses on $q_\rho$ and in so doing, it proves that $|dq_\rho|^2$ is integrable on small $\rho$ balls.

This $q_\rho$ is a section of $\mathcal{I}$ on the complement of $\Gamma$ in that ball. By virtue of $q$ obeying \eqref{eq:10.2}, this section $q_\rho$ obeys the equation below:
\begin{equation}\label{eq:10.34}
-\frac{1}{\sin^2\rho}\frac{d^2}{d\rho^2}\bigl(\sin^2\rho\,q_\rho\bigr)
-\frac{1}{\sin^2\rho}\Delta_2 q_\rho
=
\E\,q_\rho,
\end{equation}
with $\Delta_2$ denoting (as before) the standard Laplacian on the $S^2$. To see what this equation requires of $q_\rho$, expand $q_\rho$ as was done with $\varsigma$ in \eqref{eq:10.16} using the $\mathbb{H}_{\mathfrak{p}}$ eigensections of $\Delta_2$:
\begin{equation}\label{eq:10.35}
q_\rho=\sum \mathfrak{b}_j(\rho)\,\f_j\,.
\end{equation}
(Remember in this regard that $\{\f_1,\f_2,\dots\}$ is an $\mathbb{L}_{\mathfrak{p}}$-orthonormal basis of eigensections of $\Delta_2$ from $\mathbb{H}_{\mathfrak{p}}$. In what follows $\lambda_j$ for $j\in\{1,2,\dots\}$ denotes the eigenvalue of the corresponding $\f_j$.) Of particular note is that the sum in \eqref{eq:10.35} is convergent in the $\mathbb{L}$-topology on the radius $\frac{1}{10^6c}\,r_p$ ball centered at the origin in $\mathbb{R}^3$. This is to say that
\begin{equation}\label{eq:10.36}
\lim_{N\to\infty}\int_{0}^{\frac{r_p}{10^{6}c}}
\left|\,q_\rho-\textstyle\sum^{N} \mathfrak{b}_j \f_j\,\right|^{2}=0.
\end{equation}
where the notation here and below has $\sum^{N}$ signifying that the sum is over the indices $j$ from the set $\{1,\dots,N\}$.

By virtue of \eqref{eq:10.34}, any given $\mathfrak{b}_j$ obeys the equation below:
\begin{equation}\label{eq:10.37}
-\frac{1}{\sin^{2}\rho}\,\frac{d^{2}}{d\rho^{2}}\bigl(\mathfrak{b}_j\sin^{2}\rho\bigr)
+\frac{1}{\sin^{2}\rho}\,\lambda_j \mathfrak{b}_j
=\mathtt{E}\,\mathfrak{b}_j \, .
\end{equation}
This equation requires in turn that $\mathfrak{b}_j$ have the form
\begin{equation}\label{eq:10.38}
\mathfrak{b}_j=\alpha\,\rho^{\mu}\bigl(1+\rho^2\mathfrak{s}\bigr)
+\beta\,\rho^{-\mu'}\bigl(1+\rho^2 \mathfrak{t}\bigr),
\end{equation}
with $\alpha$ and $\beta$ being constant, with $\mu$ and $\mu'$ being the numbers depicted below in \eqref{eq:10.39} and with $\mathfrak{s}$ and $\mathfrak{t}$ being functions of $\rho$ with bounded norm (which can depend on $j$):
\begin{equation}\label{eq:10.39}
\mu=\Bigl(-\frac{3}{2}+\frac{1}{2}\bigl(1+4\lambda_i\bigr)^{1/2}\Bigr)
\quad\text{and}\quad
\mu'=\Bigl(\frac{3}{2}+\frac{1}{2}\bigl(1+4\lambda_i\bigr)^{1/2}\Bigr).
\end{equation}

With regards to $\beta$: The number $\beta$ must be zero because if $\beta$ is non-zero, then the $\psi=y$ and $x=p$ version of the integral that is depicted in the fourth bullet of \eqref{eq:6.3} would diverge, because $\mu'$ is greater than $2$ since $\lambda_j$ is positive.

With regards to $\mu$: As explained directly, $\mu$ is greater than $-0.4$. This is because $q_\rho$ is invariant with respect to the $\mathbb{Z}/3\mathbb{Z}$ subgroup of $G$ that fixes any given edge incident to $p$. Given such an edge, this $\mathbb{Z}/3\mathbb{Z}$ subgroup acts as a rotation of $S^2$ that fixes the corresponding point in $\mathfrak{p}$. The $\f_j$ must be invariant with respect to theses subgroups which implies (as explained in \cite{taubes2020examples}) that $|\f_j|$ is bounded by a constant multiple of $\dist(\cdot,p)^{3/2}$. That implies in turn that $d \f_j$ is from the $T^*S^2\otimes\mathcal{I}$ version of $\mathbb{H}_\mathfrak{p}$ which implies in turn (as noted previously) that $\lambda_j$ is strictly greater than $1$ (see \eqref{eq:10.22} and the subsequent remarks).

Any $r\in(0,\frac{1}{10^{6}c}\,r_p)$ version of the bound below follows directly since $\mu>-0.4$ (this bound has $c_j$ being independent of $r$):
\begin{equation}\label{eq:10.40}
\int_{0}^{r}|\mathfrak{b}_j|^2\,d\rho<c_j\,r^{0.2}\, .
\end{equation}

The preceding bounds with the bounds implied by the $\varepsilon\to 0$ limit of \eqref{eq:10.22} can be used to see that $|\nabla q_\rho|^2$ is integrable near $p$. This is done using \eqref{eq:10.34} by first approximating $q_\rho$ using a sequence of finite sum versions of \eqref{eq:10.35} with the $N$'th element in this sequence (for $N=1,2,\dots$) being $\Sigma^{N} \mathfrak{b}_j \f_j$. Let $q_{\rho,N}$ denote this approximation. A key fact now is that $q_{\rho,N}$ obeys \eqref{eq:10.34} also. Granted that, fix $\varepsilon$ to be positive but very small. Letting $r$ denote the $\frac{1}{4\times 10^6c}\,r_p$, multiply both sides of the $q_{\rho,N}$ version of
\eqref{eq:10.34} by $\chi(1-\rho/\varepsilon)\chi(\rho/r-1)\,q_{\rho,N}\sin^2\rho$. Having done that, then integrate and use integration by parts to see that
\begin{equation}\label{eq:10.41}
\int_{\mathrm{B}_r-\mathrm{B}_\varepsilon}|\nabla q_{\rho,N}|^2
\leq
c_0\,\varepsilon^{-2}\int_{\mathrm{B}_{2\varepsilon}}|q_{\rho,N}|^2
+
c_0\,r^{-2}\int_{\mathrm{B}_{2r}-\mathrm{B}_r}|q_{\rho,N}|^2 \,.
\end{equation}
(Note that the derivation of this bound uses $\varepsilon=0$ version of \eqref{eq:10.22} and the fact that the eigenvalues $\lambda_j$ for the $\f_j$'s that appear in
\eqref{eq:10.35} are greater than $1$.) To finish from here:
The $r=2\varepsilon$ version of \eqref{eq:10.40} implies that the leftmost term on the right-hand side of \eqref{eq:10.41} has an $\varepsilon\to 0$ limit; and also that this limit is $0$. The latter fact and the dominated convergence theorem leads from \eqref{eq:10.41} and \eqref{eq:10.36} to an $N$-independent bound for the integral of $|\nabla q_{\rho,N}|^2$ on $\mathrm{B}_r$. And, much the same manipulations can be used to conclude that $\left\{\chi\left(\frac{\rho}{r}-1\right)\nabla q_{\rho,N}\right\}_{N=1,2,\dots}$ is a Cauchy sequence in the $\mathrm{B}_r$ version of the Hilbert space $\mathbb{H}_{\mathrm{B}}$. That last fact implies that $\chi\left(\frac{\rho}{r}-1\right)q_\rho$ is in $\mathbb{H}_{\mathrm{B}}$ too.\\

\noindent
\emph{Part 7:} The fact that $\chi\left(\frac{\rho}{r}-1\right)q_\rho$ is in $\mathbb{H}_{\mathrm{B}}$ will be used momentarily to prove that the bound below in \eqref{eq:10.42} holds for any given value of $s\leq \frac{1}{10^8c}\,r_p$ with $c_\ast$ independent of $s$.
\begin{equation}\label{eq:10.42}
\int_{\mathrm{B}_s}\frac{1}{\rho^2}\,|q|^2 \leq c_\ast \,.
\end{equation}

Suppose for the moment that \eqref{eq:10.42} holds. As explained directly, that bound implies that the whole of $\chi\left(\frac{\rho}{r}-1\right)q$ is in the $T^\ast S^3\otimes\mathcal{I}$ version of $\mathbb{H}_{\mathrm{B}}$. This is done by first noting that \eqref{eq:6.5} holds with $\psi=q$ since $q$ obeys $\nabla^\dagger\nabla q + \mathrm{Ric}\cdot q=\mathtt{E}q$ (because $d^\ast q=0$ and $\ast d(\ast d q)=\mathtt{E}q$). Keeping the latter fact in mind, fix $\varepsilon>0$ but much less $r$ (which is $\frac{1}{4\times 10^6c}\,r_p$) and multiply both sides of the $\psi=q$ version of \eqref{eq:6.5} by $\chi\left(1-\frac{\rho}{\varepsilon}\right)\chi\left(\frac{\rho}{r}-1\right)$. Then integrate the result over the radius $2r$ ball centered at $q$ and integrate by parts in the resulting identity to obtain the inequality
\begin{equation}\label{eq:10.43}
\int_{\mathrm{B}_r-\mathrm{B}_\varepsilon}|\nabla q|^2
\leq
c_0\,\varepsilon^{-2}\int_{\mathrm{B}_{2\varepsilon}}|q|^2
+
c_0\,r^{-2}\int_{\mathrm{B}_{2r}}|q|^2 \,.
\end{equation}
(The integration by parts does not get hung up on the edges of $\Gamma$ in $\mathrm{B}_{2r}$ because $|\nabla q|^2$ is integrable near interior points of the edges.)

The right-hand side of \eqref{eq:10.43} is bounded as $\varepsilon$ limits to zero by virtue of the $s=2\varepsilon$ version of \eqref{eq:10.42}. The existence of that limit implies that $|\nabla q|^2$ is integrable on $\mathrm{B}_r$ and that in turn implies (with another appeal to \eqref{eq:10.42}) the desired conclusion that $\chi\left(\frac{\rho}{r}-1\right)q$ is in the $\mathrm{B}=\mathrm{B}_r$ version of $\mathbb{H}_{\mathrm{B}}$. (That other appeal is used in a straightforward way to see
that $\chi\left(1-\frac{\rho}{\varepsilon}\right)\chi\left(\frac{\rho}{r}-1\right)q$ converges in $\mathbb{H}_{\mathrm{B}}$ as $\varepsilon$ limits to zero.)

With the preceding understood, the subsequent paragraphs derive \eqref{eq:10.42}.  To this end, note first that \eqref{eq:10.42} is obeyed with $q$ replaced by $q_\rho$ because $\chi(\frac{\rho}{r}-1)\,q_\rho$ is in $\mathbb{H}_{\mathtt{B}}$ and thus the $\psi=\chi(\frac{\rho}{r}-1)\,q_\rho$ version of \eqref{eq:4.2} holds.

The rest of $q$ is $q-q_\rho\,d\rho$; this is written in what follows as $\hat{q}$ and it is viewed as a $\rho$--dependent, $\mathcal{I}$--valued $1$--form on $S^2-\mathfrak{p}$.  (To be sure about what this means: At any given value of $\rho$, the norm of $\hat{q}$ as an $\mathcal{I}$--valued section of $T^*S^2$ is $\sin\rho$ times its norm as an $\mathcal{I}$--valued section of $T^*S^3$.)  The $d\rho$ component of the identity $*dq=\E y$ can be depicted using $\hat{q}$ in the manner of the leftmost equation in \eqref{eq:10.44}. The rightmost equation in \eqref{eq:10.44} uses $\hat{q}$ to depict the identity $d{*}q=0$.  The notation in \eqref{eq:10.44} uses $\hat{d}$ and $\hat{*}$ to denote the respective exterior derivative along $S^2$ and the Hodge star on $S^2$.  Also, \eqref{eq:10.44} uses $y_\rho$ to denote the $d\rho$ component of $y$.
\begin{equation}\label{eq:10.44}
\frac{1}{\sin^2\rho}\,\hat{*}\,\hat{d}\hat{q}=\E\,y_\rho
\qquad\text{and}\qquad
\frac{1}{\sin^2\rho}\,\hat{*}\bigl(\hat{d}\,\hat{*}\hat{q}\bigr)
=-\frac{\partial}{\partial\rho}q_\rho-\frac{2\cos\rho}{\sin\rho}\,q_\rho.
\end{equation}

Now comes the first key observation: The square of the norm of the right-hand side of both equations in \eqref{eq:10.44} is integrable on $\mathtt{B}_r$.  As a consequence, there exists $c_{\ddagger}$ such that for any given small but positive $\varepsilon$,
\begin{equation}\label{eq:10.45}
\int_{\mathtt{B}_r-\mathtt{B}_\varepsilon}
\frac{1}{\sin^4\rho}\Bigl(|\hat{d}\hat{q}|^2+|\hat{d}\,\hat{*}\hat{q}|^2\Bigr)
<c_{\ddagger}.
\end{equation}

Now comes the second key observation: Since $|\nabla q|^2$ is integrable near points in the interior of the edges of $\Gamma$, the $\mathcal{I}$--valued $1$--form $\hat{q}$ at all but a measure zero set of $\rho$ values is in the $T^*S^2\otimes\mathcal{I}$ version of the Hilbert space $\mathbb{H}_\mathfrak{p}$ (which implies that the square of the norm of its $S^2$--covariant derivative is integrable on $S^2$ at these values of $\rho$).  That implies in turn that some integration by parts can be used (with no hang-up on the edges of $\Gamma$ in $\mathtt{B}_r-\mathtt{B}_\varepsilon$) to write \eqref{eq:10.45} as done below (with $\hat{\nabla}$ denoting the round metric's covariant derivative and with $|\cdot|_{S^2}$ denoting the round metric's norm on $1$--forms and tensors).
\begin{equation}\label{eq:10.46}
\int_{\mathtt{B}_r-\mathtt{B}_\varepsilon}
\frac{1}{\sin^4\rho}\Bigl(|\hat{\nabla}\hat{q}|_{S^2}^2+|\hat{q}|_{S^2}^2\Bigr)
<c_{\ddagger}.
\end{equation}

Because the $S^3$ norm of $\hat{q}$ is $\frac{1}{\sin\rho}|\hat{q}|_{S^2}$, this last inequality leads directly to \eqref{eq:10.42} with $q$ replaced by $\hat{q}$ by taking $\varepsilon$ ever smaller with limit zero.

\part{Polytope Symmetry, Equivariance, and Topological Invariants}

\medskip
 \section{Quaternionic Representation, Regular 4-Polytopes, and Schl\"afli Symbols}\label{ch:pre} 

Part~I develops the analytic framework for $\mathbb{Z}/2$--harmonic objects on $S^3\setminus\Gamma$ with values in a real line bundle $\mathcal{I}\to S^3\setminus\Gamma$.
Part~II verifies the geometric properties needed to apply the analytic results of Part~I to the graphs $\Gamma\subset S^{3}$ obtained from regular $4$--polytopes.  Concretely, we will need (i) an explicit description of the induced action of $G=\Sym^{+}(P)$ on $S^{3}\setminus\Gamma$ and its compatibility with the $\{\pm1\}$--local system $I$, (ii) the existence of odd--order cyclic stabilizers along edges of $\Gamma$ acting by rotations on the normal $2$--planes, and (iii) a coherent lift of the $G$--action to the associated connected double
cover, yielding fiberwise linear automorphisms of $\mathcal{I}$.

This section develops the background for Part~II.
Section~11.1 records the quaternionic model for \(\mathrm{SO}(4)\), used to
write elements of \(G\) explicitly and to compute their fixed--point sets,
vertex actions, and facet actions. Sections~11.2 and~11.3 recall the
regular \(4\)-polytopes, their chambers, radial projections, Schl\"afli symbols,
and local link/figure conventions. Finally, Sections~11.4 and~11.5 extract the
structural consequences used later: even valency of the radial graph,
odd--order cyclic edge stabilizers, liftability of the \(G\)-action to \(I\),
and standard presentations of the relevant rotation subgroups.

\subsection{Quaternionic representation of $\SO(4)$ and the $G$--action}\label{subsec:quat}

Since each regular $4$--polytope $P\subset\R^{4}$ is centered at the origin, its rotational symmetry group $G=\Sym^{+}(P)\subset\SO(4)$ acts linearly on $\R^{4}$ and hence on $S^{3}$. It is convenient to work in the quaternionic model $\R^{4}\cong\mathbb H$, where elements of $\SO(4)$ are represented (up to sign) by pairs in $S^{3}\times S^{3}$ and many fixed--set computations reduce to quaternion algebra.
Let
\[
\HH=\{\,w+xi+yj+zk : w,x,y,z\in\R\,\},
\qquad
i^{2}=j^{2}=k^{2}=ijk=-1.
\]
Write $\Re(p)=w$ and $\Im(p)=xi+yj+zk$, so that $\HH=\R\oplus \Im\HH$ with $\Im\HH\cong\R^{3}$.  Quaternion conjugation is $\overline{p}=w-xi-yj-zk$ and the norm is
\[
|p|^{2}=p\overline{p}=w^{2}+x^{2}+y^{2}+z^{2}.
\]
We identify \(\mathbb R^4 \cong \mathbb H\) by
\[
(x_{1},x_{2},x_{3},x_{4}) \longleftrightarrow x_{1}+x_{2}i+x_{3}j+x_{4}k,
\]
so that
\(
\langle p,q\rangle = \Re(p\bar q),
|p|^2 = p\bar p.
\)
In particular, the quaternion norm agrees with the Euclidean norm.\\

\noindent\textbf{The Spin(4) covering of \(SO(4)\).}
Let \(S^{3}=\{q\in\HH:|q|=1\}\cong \SU(2)\).
Define
\[
\Pi:S^{3}\times S^{3}\to SO(4),\qquad
\Pi(q_{L},q_{R})(p)=q_{L}\,p\,q_{R}^{-1}.
\]
Then \(\Pi\) is a surjective homomorphism with kernel \(\{(1,1),(-1,-1)\}\cong \mathbb{Z}_{2}\); hence \(SO(4)\cong (S^{3}\times S^{3})/\{\pm(1,1)\}\). Indeed, \(\Pi\) is a group homomorphism because \(\Pi(q_{L},q_{R})=\!L_{q_{L}}\,R_{q_{R}^{-1}}\) and composition corresponds to quaternion multiplication.  If \(\Pi(q_{L},q_{R})=\mathrm{id}\), then \(q_{L}p \, q_{R}^{-1}=p\) for all \(p\), whence \(q_{L}=q_{R}\) and \(q_{L}\) commutes with \(i,j,k\), so \(q_{L}=\pm1\). Surjectivity follows from \(\Spin(4)=S^{3}\!\times\!S^{3}\) being the universal cover of \(SO(4)\). Thus,
\[
\SO(4)\ \cong\ \dfrac{\SU(2)_{L}\times\SU(2)_{R}}{\{(1,1),(-1,-1)\}}.
\]
Thus every $g\in SO(4)$ can be represented (nonuniquely) by a pair
$(q_L,q_R)\in S^{3}\times S^{3}$, unique up to the simultaneous sign change
$(q_L,q_R)\mapsto(-q_L,-q_R)$ \cite[\S 2.7, p.~42]{stillwell2008naive}.\\

\noindent\textbf{Left/right multiplication and the special case \(\Ad_q\).}
Identifying \(\HH\cong\R^{4}\) via the ordered basis \((1,i,j,k)\), left and right
multiplication are \(\R\)--linear maps
\[
L_{q}(p)=q\,p,\qquad R_{q}(p)=p\,q.
\]
If \(|q|=1\) then \(L_{q},R_{q}\in SO(4)\). An important special case is \(q_L=q_R=a\in S^3\), where \(\Ad_a(x):=\Pi(a,a)(x)=axa^{-1}\). More precisely, suppose $a\notin\{\pm1\}$, so write
$a=\cos\theta+u\sin\theta$ with $u\in\Im\mathbb H$, $|u|=1$, and
$0<\theta<\pi$. For $x\in u^\perp\subset\Im\mathbb H$, one has $ux=-xu$, and
\[
\operatorname{Ad}_a(x)=axa^{-1}
 = \cos(2\theta)x+\sin(2\theta)ux .
\]
Since $x\mapsto ux$ is the quarter-turn on $u^\perp$, this is rotation by
angle $2\theta$ about the axis $\operatorname{Span}_{\mathbb R}\{u\}$ in
$\Im\mathbb H$. Also, $\operatorname{Ad}_a(1)=1$. Consequently, for
$a\notin\{\pm1\}$, the fixed-point set of $\operatorname{Ad}_a$ in
$\mathbb R^4\simeq\mathbb H=\mathbb R\oplus\Im\mathbb H$ is the $2$-plane
$\operatorname{Span}_{\mathbb R}\{1,u\}$. In the exceptional cases
$a=\pm1$, $\operatorname{Ad}_a$ is the identity and its fixed-point set is all
of $\mathbb R^4$. 

\subsection{Regular 4-polytopes, chambers, and barycentric subdivision}\label{subsec:reg}
This section introduces definitions leading to convex regular $4$--polytopes, whose 1-skeleta supply local models for $\mathbb{Z}/2$ harmonic forms. Regular 4-polytopes are defined in terms of chambers. The following exhibits some basic notions.

\begin{definition}[Convex polytope {\cite[Def.~A.1.1]{Davis2008}}]\label{def:conv}
A \emph{convex polytope} $P$ in an affine space $\mathbb{A}$ is the convex hull of a finite subset of $\mathbb{A}$. Its dimension is the dimension of the affine subspace it spans. 
\end{definition}

\begin{definition}[Supporting hyperplane {\cite[p.~402, Appendix~A.1]{Davis2008}}]
Suppose $X$ is a closed convex subset of $\mathbb{A}$. An affine hyperplane $H$ is a supporting hyperplane of $X$ if $X\cap H \neq \emptyset$ and if $X$ is contained in one of the half-spaces bounded by $H$. (That is to say, $X$ lies on one side of $H$.) 
\end{definition}

\begin{definition}[Face of a polytope {\cite[p.~402, Appendix~A.1]{Davis2008}}]
Suppose $P$ is a convex polytope and $H$ a \emph{supporting hyperplane}. Then $P\cap H$ is also a convex polytope, called a face of $P$. A proper face is a face in the set of all nonempty faces of a convex polytope $P$
other than $P$ itself.
The collection of all proper faces $\partial P$ is the boundary complex of $P$.
\end{definition}

\begin{definition}[Chain {\cite[p.~126, Sec.~7.2]{Davis2008}}]
For any poset $\mathcal{P}$, a chain is a nonempty totally ordered subset of $\mathcal{P}$. A chain not properly contained in a larger chain is called a maximal chain.
\end{definition}

\begin{figure}
    \centering\hskip 30mm
    \includegraphics[width=0.45\linewidth]{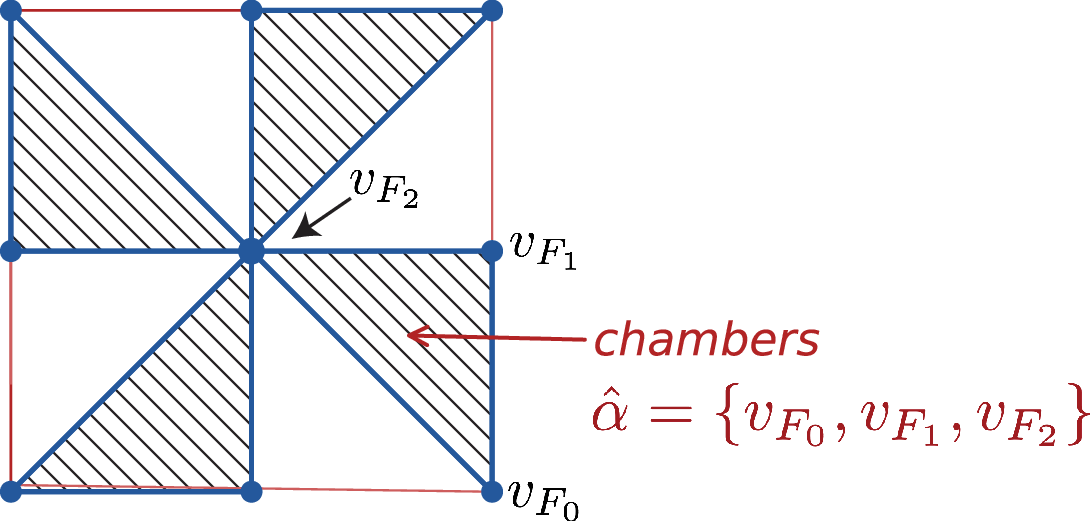}
\caption{Illustration of barycentric subdivision \(b(\partial P)\) of a two-dimensional boundary
complex of a 3-dimensional polytope $P$. For a maximal chain of proper faces
\(\alpha=(F_0<F_1<F_2)\) in $\partial P$,
the simplex
\(
\widehat{\alpha}=[v_{F_0},v_{F_1},v_{F_2}]
\)
is a top-dimensional simplex of \(b(\partial P)\), i.e. a chamber.}    \label{fig:chamber}
\end{figure}

\begin{definition}[Barycentric subdivision {\cite[Prop.~A.3.9 and following]{Davis2008}}]\label{def:bary}
Let $P$ be a convex polytope. Choose for each face $F\le P$ a point $v_F$ in the relative interior of $F$. For each chain $\alpha=(F_1<\cdots<F_k)\in\Flag(F(P))$, let $\widehat{\alpha}$ be the simplex spanned by $v_{F_1},\dots,v_{F_k}$. The simplices $\{\widehat{\alpha}\}$ form a simplicial complex $bP$, called the \emph{barycentric subdivision} of $P$.
\end{definition}

Applied to the boundary complex \(\partial P\), this construction gives the barycentric
subdivision \(b(\partial P)\); see Fig.~\ref{fig:chamber}. Its top-dimensional simplices are precisely the chambers.

\begin{definition}[Chamber {\cite[p.~94]{Davis2008}}]\label{def:chamber}
A \emph{chamber} of a convex polytope $P^n$ is a maximal chain of proper faces of $P$, $F_0 < F_1 < \cdots < F_{n-1}$, where $\dim F_i = i$. Equivalently, a chamber is a top-dimensional simplex of $b(\partial P)$, and such a simplex is identified with a maximal chain $\{F_0<F_1<\cdots<F_{n-1}\}$ of faces of $\partial P$.
\end{definition}

\begin{definition}[Regular polytope {\cite[p.~94]{Davis2008}; \cite[Def.~B.2.1, p.~424]{Davis2008}}]
\label{def:regular}A polytope in Euclidean space is \emph{regular} if its symmetry group $\Isom(P)$
acts transitively on the set of chambers, which is denoted $\Cham(b(\partial P))$.
\end{definition}

Recall the radial projection of the $1$--skeleton of a regular $4$--polytope is defined as
\[
\Gamma \;=\;\Bigl\{\frac{x}{|x|}\;:\;x\in P^{(1)}\setminus\{0\}\Bigr\}\;\subset\;S^3.
\]
For each polytope edge \([u,v]\) of a centered regular polytope, the radial
image of \([u,v]\) is contained in
\(\operatorname{span}_{\mathbb R}\{u,v\}\cap S^3\). Since \(0\notin [u,v]\),
the endpoints \(u\) and \(v\) are not antipodal, so
\(\operatorname{span}_{\mathbb R}\{u,v\}\) is a two-plane and its intersection
with \(S^3\) is a great circle. Thus the radial image of \([u,v]\) is a
geodesic arc in \(S^3\). This verifies the geodesic-edge hypothesis used in Part~I, most directly in
the local edge estimates of Lemmas~6.1, 6.2, and~6.4, and hence in the proof
of Lemma~4.2.

\subsection{Links, edge figures, Schl\"afli symbols and local models}\label{sec:schla}
Schl\"afli symbols record the local incidence data of a regular $4$--polytope in a way compatible with the chamber structure of the barycentric subdivision. Schl\"afli symbols play an essential role in determining the geometric and topological requirements on polytopes prescribed by Section 3, and also provide a description of the symmetry group of regular polytopes. We use them to encode vertex figures, edge figures, and the cyclic order of facets around an edge. 

We first recall the notion of a link.
For a simplex \(\tau\) of an abstract simplicial complex \(\Delta\), the link is
\[
\operatorname{Lk}_{\Delta}(\tau)
:=\{\sigma\in\Delta \mid \sigma\cap\tau=\emptyset,\ \sigma\cup\tau\in\Delta\}
\]
\cite[Def.~2.13, p.~11]{kozlov2008combinatorial}.
Geometrically, when the complex is embedded, the link of a vertex can be realized
by intersecting the complex with a sufficiently small sphere centered at that vertex \cite[p.~27]{kozlov2008combinatorial}. For a general face \(\sigma\), the face link is described using the local transversal model \(B^d\times Q\), where \(d=\dim\sigma\). The link of \(\sigma\) is the link of the central vertex in the transverse complex \(Q\) \cite[p.~27]{kozlov2008combinatorial}. Equivalently, for an edge \(E\subset \partial P\), the edge link is obtained by taking an interior point \(x\in E\) and considering the small spherical section \(S_\varepsilon(x)\cap\partial P\). For an edge $E$ of a $4$--polytope, the edge link is this one-dimensional cycle in the local spherical section. The edge figure is the polygonal $2$--cell associated to that cycle; equivalently, if $r$ facets meet cyclically around $E$, then the edge figure is an $r$--gon whose boundary is the edge link; see Examples~\ref{eg:edge-link}, \ref{eg:edge-link2}, and Fig.~\ref{fig:beach}, where the edge link is the equatorial polygonal cycle.

\begin{figure}
    \centering
    \includegraphics[width=0.9\linewidth]{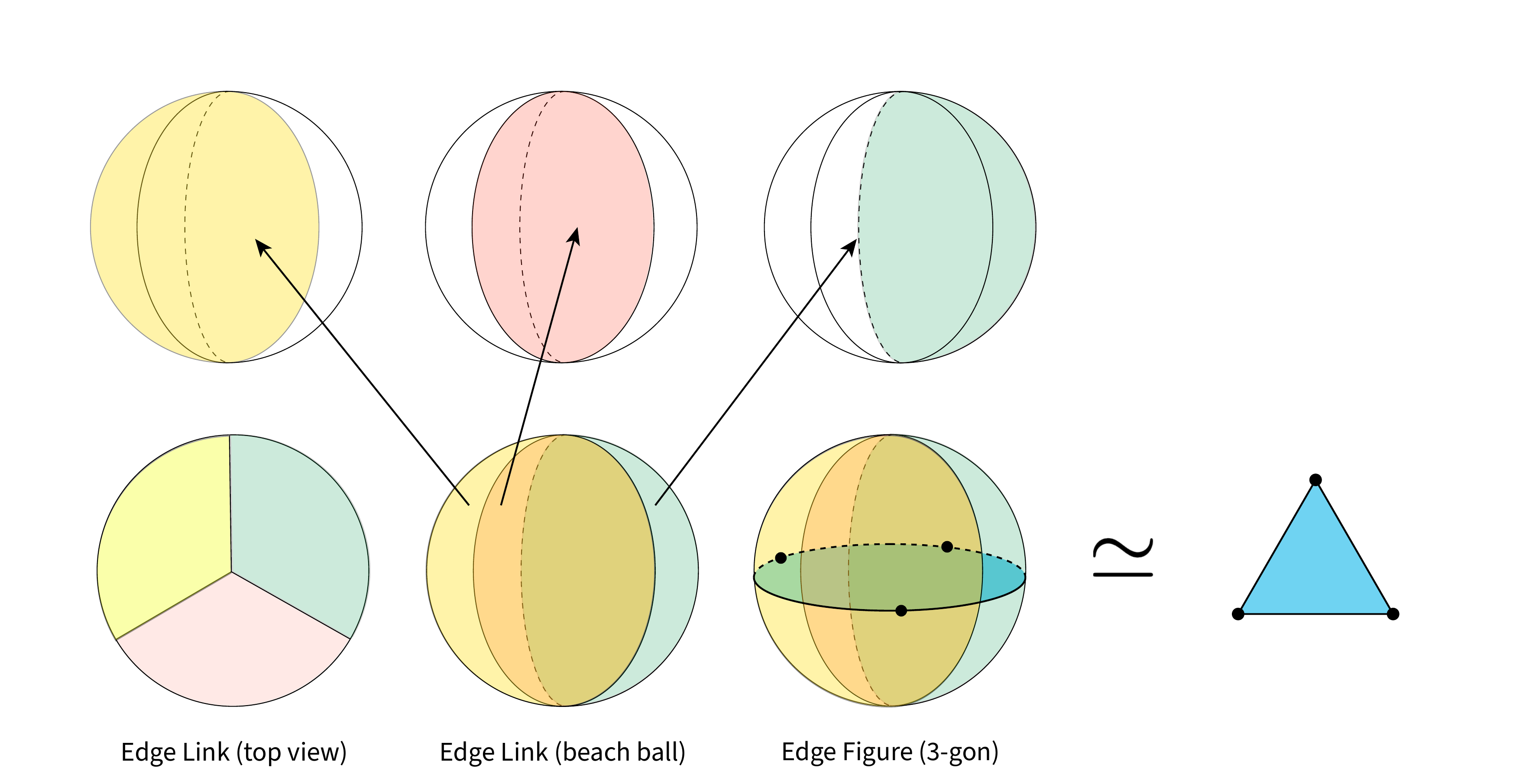}
    \caption{Local spherical section \(\partial P\cap S_\varepsilon(x)\) at a point \(x\in E^\circ\) of an edge \(E\), for \(\varepsilon>0\) sufficiently small. Its equatorial polygonal cycle records the cyclic order of the facets meeting along \(E\). The illustration shows the case \(r=3\), where the edge figure is a triangle, and colors distinguish the incident facets.}
    \label{fig:beach}
\end{figure}

\begin{example}\label{eg:edge-link}
 Let \(P=[0,1]^3\) be the ordinary cube and let
\[
E=\{(t,1,1):0\leq t\leq 1\}
\]
be an edge. Near an interior point of \(E\), the boundary complex \(\partial P\) is locally \(B^1\times Q\), where \(B^1\) is the interval direction along \(E\), and \(Q\) is a \(V\)-shaped complex coming from the two square faces incident to \(E\). The edge link can be obtained by removing the $\sigma$ direction and taking the link at $\sigma$ (now a vertex). In this example, the face direction \(\sigma\) is the \(E\)-direction; after removing it, the link of the central vertex in \(Q\) consists of two points, which is the edge figure.
\end{example}

\begin{example}\label{eg:edge-link2}
Let \(P=[0,1]^4\) be the \(4\)-cube and let
\[
E=\{(t,1,1,1):0\leq t\leq 1\}.
\]
Near an interior point of \(E\), the boundary complex \(\partial P\) is locally \(B^1\times Q\), where  \(B^1\) is the interval direction along \(E\), and \(Q\) is the transverse complex determined by the three cubic facets incident to \(E\). The edge link is a $3$--cycle recording the cyclic order of the three incident cubic facets, and the edge figure is the associated $3$--gon.
\end{example}

\subsubsection{Schl\"afli symbols and edge figures.}
For a regular \(4\)-polytope with Schl\"afli symbol \(\{p,q,r\}\), each
\(2\)--face is a \(p\)--gon; each facet is a regular polyhedron of type
\(\{p,q\}\), so within each facet \(q\) such \(p\)--gons meet at each vertex;
and exactly \(r\) facets meet cyclically around each edge, so the edge figure
is an \(r\)-gon. Fig.~\ref{fig:beach} shows the case \(r=3\), in which the edge figure is a triangle.
For an edge $E$, the full stabilizer in the symmetry group acts on the edge figure by the dihedral group $D_r$, and the orientation--preserving stabilizer contains a cyclic subgroup $C_r$ acting on the normal $2$--plane by rotation through angle $2\pi/r$.
The vertex figure is the regular polyhedron of type \(\{q,r\}\); equivalently, the local figure at each vertex is \(\{q,r\}\). In particular, the valence of each vertex is the number of vertices of the vertex figure. The Schl\"afli symbols for regular 4-polytopes \cite[p.~414]{coxeter1969introduction} are exhibited in Table~\ref{tab:regular4polytopes-figures}.

\begin{table}[htbp]
\centering
\small
\renewcommand{\arraystretch}{1.2}
\setlength{\tabcolsep}{8pt}
\begin{tabular}{@{}lclll@{}}
\toprule
\textbf{Polytope} 
& \textbf{Schl\"afli symbol} 
& \thead{Facet} 
& \thead{Vertex figure} 
& \thead{Edge figure} \\
\midrule
\(5\)-cell   
& \(\{3,3,3\}\) 
& \(\{3,3\}\) tetrahedron  
& \(\{3,3\}\) tetrahedron 
& \(\{3\}\) 3-gon \\

\(8\)-cell   
& \(\{4,3,3\}\) 
& \(\{4,3\}\) cube
& \(\{3,3\}\) tetrahedron 
& \(\{3\}\) 3-gon\\

\(16\)-cell  
& \(\{3,3,4\}\) 
& \(\{3,3\}\) tetrahedron   
& \(\{3,4\}\) octahedron 
& \(\{4\}\) 4-gon\\

\(24\)-cell  
& \(\{3,4,3\}\) 
& \(\{3,4\}\) octahedron
& \(\{4,3\}\) cube
& \(\{3\}\) 3-gon\\

\(120\)-cell 
& \(\{5,3,3\}\) 
& \(\{5,3\}\) dodecahedron
& \(\{3,3\}\) tetrahedron 
& \(\{3\}\) 3-gon\\

\(600\)-cell 
& \(\{3,3,5\}\) 
& \(\{3,3\}\)  tetrahedron   
& \(\{3,5\}\) icosahedron 
& \(\{5\}\) 5-gon\\
\bottomrule
\end{tabular}
\caption{Vertex and edge figures of the convex regular \(4\)-polytopes.}
\label{tab:regular4polytopes-figures}
\end{table}

\subsection{Even valency of the radial graph.}

This subsection extracts from Table~\ref{tab:regular4polytopes-figures} the
even-valency condition needed for the line bundle \(\mathcal I\) on
\(S^3\setminus\Gamma\).

\begin{proposition}[Even valency of the regular 4-polytopes]\label{prop:even}
Let \(P\subset\R^{4}\) be the regular $4$-polytope with vertex set \(\left\{V_{1},\dots,V_{n}\right\}\subset \mathbb{R}^{4}\), and let \(\Gamma\subset S^{3}\) be the radial projection of the \(1\)--skeleton of \(P\). Set \(G:=\Sym^{+}(P)\subset SO(4)\). Then each vertex of \(\Gamma\) has even valency.
\end{proposition}

\begin{proof}
As the valency of each vertex is the number of vertices of the vertex figure, it is enough to list the vertex figures of the regular \(4\)--polytopes.
For the \(5\)--cell and \(8\)--cell, the vertex figure is a tetrahedron;
for the \(16\)--cell, it is an octahedron;
for the \(24\)--cell, it is a cube;
for the \(120\)--cell, it is a tetrahedron;
and for the \(600\)--cell, it is an icosahedron.
Hence, the corresponding numbers of vertices are
\(
4, 4, 6, 8, 4, 12,
\)
respectively. All of these are even. Therefore, each vertex of \(\Gamma\) has even valency.
\end{proof}

\subsection{Odd-order edge stabilizers and lift of the symmetry action to \(\mathcal{I}\).}

The following proposition is verified case by case in Sections~\ref{sec:5cell}--\ref{sec:600cell}; see
Propositions~~\ref{prop:12.3(2)}, \ref{prop:13.1}, \ref{prop:14.1}, \ref{prop:15.1}, and \ref{prop:16.1} for the model-dependent proofs of the odd-order edge stabilizers.

\begin{proposition}[Odd-order edge stabilizers]\label{prop:odd-stab}
Let \(P\) be one of the regular \(4\)-polytopes appearing in Theorem~\ref{thm:main}, let \(\Gamma \subset S^3\) be the radial projection of its \(1\)-skeleton, and let \(G \le SO(4)\) be the orientation-preserving symmetry group of \(P\). For each edge \(E \subset \Gamma\), let \(W_E \subset \R^4\) denote the unique \(2\)-plane through the origin such that \(E \subset W_E \cap S^3\). Then there exists an odd integer \(m>1\) and a cyclic subgroup \(C_m \subset G\) such that \(C_m\) fixes \(E\) pointwise. Moreover, the induced action of \(C_m\) on \(W_E^\perp\) is faithful, and a generator of \(C_m\) acts on \(W_E^\perp\) by rotation through angle \(2\pi/m\).
\end{proposition}

\begin{proof}
Let \(P\) have Schl\"afli symbol \(\{p,q,r\}\). For a regular \(4\)-polytope, exactly \(r\) facets meet cyclically around each edge, so the edge figure is a regular \(r\)-gon. Hence the orientation-preserving pointwise stabilizer of an edge acts faithfully on the normal \(2\)-plane as the rotation group of this \(r\)-gon, and therefore contains a cyclic subgroup \(C_r\) whose generator acts by rotation through angle \(2\pi/r\).
For the regular \(4\)-polytopes in Theorem~\ref{thm:main}, one has \(r=3\) for the \(5\)-, \(8\)-, \(24\)-, and \(120\)-cell, and \(r=5\) for the \(600\)-cell. Thus \(r\) is odd in every case, and the proposition follows with \(m=r\).
\end{proof}

\smallskip
\noindent{\it Remark.}
Among the six convex regular $4$-polytopes, the only one not discussed here is the $16$-cell. In the local models of Section~5, the $\mathbb{Z}/m\mathbb{Z}$-action lifts to the line bundle $\mathcal{I}$ only when $m$ is odd, and in all of the cases listed here, the corresponding edge stabilizer has odd order. By contrast, the $16$-cell has Schl\"afli symbol $\{3,3,4\}$, so its edge figure is a square; equivalently, four tetrahedral facets meet around each edge, and the nontrivial rotations in the cyclic edge stabilizer have orders $2$ and $4$, not odd order.
Therefore, the odd-order lifting criterion from Section~5 fails at an interior edge of the 16-cell, so the construction used in Part~II does not apply to it.

\begin{proposition}[Lift of symmetry to \(\mathcal{I}\)]\label{prop:lift}
Let \(P\) be one of the regular \(4\)-polytopes from Theorem~\ref{thm:main}, let \(\Gamma \subset S^3\) be the radial projection of its \(1\)-skeleton, and let \(G \le SO(4)\) be the orientation-preserving symmetry group of \(P\). Then the induced action of \(G\) on \(S^3 \setminus \Gamma\) lifts to an action on the total space of the line bundle \(\mathcal{I} \to S^3 \setminus \Gamma\) defined in Section~\ref{sec:3}.
\end{proposition}

\begin{proof}
The claim is verified case-by-case in Sections~\ref{sec:5cell}--\ref{sec:600cell}. In the \(5\)-cell case,  Proposition~\ref{prop:5cell-splitting} proves that the resulting \(\mathbb{Z}/2\)-extension splits, and Lemma~\ref{lem:12.4} converts the coherent lift on the connected double cover into fiberwise linear bundle automorphisms of the associated line bundle \(\mathcal{I}_5 \to M_5\). The remaining 8-, 24-, 120-, and 600-cell cases are proved in
Propositions~\ref{prop:13.2}, \ref{prop:14.2}, \ref{prop:15.2}, and \ref{prop:16.2}, respectively. Hence the induced \(G\)-action on \(S^3 \setminus \Gamma\) lifts to the total space of \(\mathcal{I}\) in every case listed in Theorem~1.2.
\end{proof}

We record the following standard presentations for the rotation subgroups of the regular 4-polytopes~\cite[Eq.~(4), p.~1311]{johnson1999quadratic}.
For the regular $4$-simplex $\{3,3,3\}$, the group representation for the rotation subgroup is:
\begin{equation}\label{eq:5}
[3,3,3]^+
=
\left\langle
\sigma_1,\sigma_2,\sigma_3
\;\middle|\;
\sigma_1^3=\sigma_2^3=\sigma_3^3
=
(\sigma_1\sigma_2)^2
=
(\sigma_2\sigma_3)^2
=
(\sigma_1\sigma_2\sigma_3)^2
=
1
\right\rangle.
\end{equation}
This group is isomorphic to 
\[
\langle r,s\mid r^2=s^3=(sr)^5=1\rangle\cong A_5
\]
under the map
\(
s\longmapsto \sigma_1, r\longmapsto \sigma_3\sigma_2.
\)

The regular $8$-cell has rotation subgroup:
\begin{equation}\label{eq:8}
[4,3,3]^+
=
\left\langle
\sigma_1,\sigma_2,\sigma_3
\;\middle|\;
\sigma_1^4=\sigma_2^3=\sigma_3^3
=
(\sigma_1\sigma_2)^2
=
(\sigma_2\sigma_3)^2
=
(\sigma_1\sigma_2\sigma_3)^2
=
1
\right\rangle.
\end{equation}

The rotation subgroup of the symmetry group of the regular $24$-cell is:
\begin{equation}\label{eq:24}
[3,4,3]^+
=
\left\langle
\sigma_1,\sigma_2,\sigma_3
\;\middle|\;
\sigma_1^3=\sigma_2^4=\sigma_3^3
=
(\sigma_1\sigma_2)^2
=
(\sigma_2\sigma_3)^2
=
(\sigma_1\sigma_2\sigma_3)^2
=
1
\right\rangle.
\end{equation}

For the regular $120$-cell, its rotation subgroup is:
\begin{equation}\label{eq:120}
[5,3,3]^+
=
\left\langle
\sigma_1,\sigma_2,\sigma_3
\;\middle|\;
\sigma_1^5=\sigma_2^3=\sigma_3^3
=
(\sigma_1\sigma_2)^2
=
(\sigma_2\sigma_3)^2
=
(\sigma_1\sigma_2\sigma_3)^2
=
1
\right\rangle.
\end{equation}

For the regular $600$-cell, its rotation subgroup is:
\begin{equation}\label{eq:600}
[3,3,5]^+
=
\left\langle
\sigma_1,\sigma_2,\sigma_3
\;\middle|\;
\sigma_1^3=\sigma_2^3=\sigma_3^5
=
(\sigma_1\sigma_2)^2
=
(\sigma_2\sigma_3)^2
=
(\sigma_1\sigma_2\sigma_3)^2
=
1
\right\rangle.
\end{equation}

 \section{Monodromy, Connected Double Covers, and Fiberwise Lifts}
\label{sec:monodromy-double-cover-lifts}

This section discusses the covering-space and line-bundle lemmas used later. Starting from a graph complement $M=S^3\setminus \Gamma$ and a real line
bundle $\mathcal J\to M$ with monodromy character
\(
\mu_{\mathcal J}:\pi_1(M)\to \{\pm 1\},
\)
this section explains:
(i) how monodromy behaves under pullback (Lemma~\ref{lem:12.1}),
(ii) when a homeomorphism of $M$ lifts to the connected double cover associated to
$\ker(\mu_{\mathcal J})$ (Lemma~\ref{lem:12.2}),
(iii) how such lifts form a central $\mathbb Z/2$-extension of the
symmetry group (Lemma~\ref{lem:12.5}), and
(iv) how a coherent choice of lifts induces fiberwise linear automorphisms of the
associated real line bundle (Lemma~\ref{lem:12.4}).

In the polytope applications, Lemma~\ref{lem:12.3} is the concrete criterion used to verify the
hypothesis of Lemma~\ref{lem:12.2}: if a symmetry preserves $\Gamma$ and permutes its edges, then it
preserves the monodromy character. Sections~\ref{sec:5cell}-\ref{sec:600cell} verify the odd-order edge stabilizer and coherent lift for each model.

\subsection{Monodromy and the associated connected double cover}
\label{subsec:monodromy-connected-double-cover}

We begin with the relation between pullback of real line bundles and pullback of their
monodromy characters. This gives the bundle criterion underlying the lifting
statements used later.

\begin{lemma}[Pullback monodromy]\label{lem:12.1}
Let \(f:(X,x_0)\to (Y,y_0)\) be continuous between path-connected spaces, and let
$\mathcal J\to Y$ be a real line bundle. Let
\(
\mu_{\mathcal J}:\pi_1(Y,y_0)\to\{\pm1\}
\)
be its monodromy character. Then
\(
\mu_{f^*\mathcal J}=\mu_{\mathcal J}\circ f_*:\pi_1(X,x_0)\to\{\pm1\}.
\)
If in addition, \(X=Y\) and \(X\) is paracompact, then
\(f^*\mathcal J\simeq \mathcal J\) if and only if \(\mu_{\mathcal J}\circ f_*=\mu_{\mathcal J}.\)
\end{lemma}

\begin{proof}
Fix $x_0\in X$ and set $y_0:=f(x_0)$. Let $\alpha:(S^1,1)\to (X,x_0)$ be a based loop. Then there is a canonical isomorphism of real line bundles over $S^1$,
\(
\alpha^*(f^*\mathcal J)\cong (f\circ \alpha)^*\mathcal J.
\)
Hence the monodromy of $f^*\mathcal J$ along $\alpha$ equals the monodromy of $\mathcal J$ along $f\circ\alpha$, so
\(
\mu_{f^*\mathcal J}([\alpha])
=
\mu_{\mathcal J}([f\circ\alpha])
=
(\mu_{\mathcal J}\circ f_*)([\alpha]).
\)
Thus
\(
\mu_{f^*\mathcal J}=\mu_{\mathcal J}\circ f_*.
\)

Now assume $X=Y$ and $X$ is paracompact. If $f^*\mathcal J\simeq \mathcal J$, then isomorphic line bundles have the same monodromy character, so
\(
\mu_{\mathcal J}
=
\mu_{f^*\mathcal J}
=
\mu_{\mathcal J}\circ f_*.
\)

Conversely, suppose that $\mu_{\mathcal J}\circ f_*=\mu_{\mathcal J}$. By the first part,
\(
\mu_{f^*\mathcal J}=\mu_{\mathcal J}.
\)
Under the canonical identification
\(
H^1(X;\mathbb Z/2)\cong \operatorname{Hom}(\pi_1(X,x_0),\mathbb Z/2),
\)
the first Stiefel--Whitney class $w_1(\mathcal L)$ of a real line bundle $\mathcal L$ corresponds to its monodromy character $\mu_{\mathcal L}$ (after identifying $\{\pm1\}\cong \mathbb Z/2$). Therefore,
\(
w_1(f^*\mathcal J)=w_1(\mathcal J).
\)
Since $X$ is paracompact, the first Stiefel--Whitney class classifies real line bundles on $X$ 
\cite[Ch.~17, Theorem~3.4, p.~250]{husemoller1966fibre}. Therefore $w_1(f^*\mathcal J)=w_1(\mathcal J)$ implies $f^*\mathcal J\simeq \mathcal J$.
\end{proof}

\begin{lemma}[Edge--permuting symmetries preserve monodromy]\label{lem:12.3}
Let $\Gamma\subset S^{3}$ be a finite graph and set $X:=S^{3}\setminus\Gamma$. Let $\mathcal J\to X$ be a real line bundle with monodromy character $\mu_{\mathcal J}:\pi_1(X)\to\{\pm1\}$ such that \(\mu_{\mathcal J}([m_e])=-1\) for every edge meridian $m_e$ about an edge $e\subset\Gamma$. Let $f:S^{3}\to S^{3}$ be a homeomorphism preserving $\Gamma$ and permuting its edges, and denote by $f:X\to X$ the induced homeomorphism. Then
\(
\mu_{\mathcal J}\circ f_*=\mu_{\mathcal J}.
\)
\end{lemma}
\begin{proof}
Since $f$ preserves $\Gamma$ and permutes its edges, the class $f_*([m_e])$ is conjugate in $\pi_1(X)$ to $[m_{f(e)}]^{\pm1}$. Here $[m_{f(e)}]^{\pm1}$ denotes either $[m_{f(e)}]$ or $[m_{f(e)}]^{-1}$; the sign records
whether $f$ preserves the chosen orientation of the meridian. Because every element of $\{\pm1\}$ equals its own inverse, $\mu_{\mathcal J}(\gamma^{-1})=\mu_{\mathcal J}(\gamma)$ for all $\gamma\in\pi_1(X)$. Also, the target $\{\pm1\}$ is abelian; therefore, $\mu_{\mathcal J}$ is invariant under conjugacy. Thus, for each edge $e$,
\(
(\mu_{\mathcal J}\circ f_*)([m_e])
=\mu_{\mathcal J}(f_*[m_e])
=\mu_{\mathcal J}\bigl([m_{f(e)}]^{\pm1}\bigr)
=-1
=\mu_{\mathcal J}([m_e]).
\)

Thus $\mu_{\mathcal J}\circ f_*$ and $\mu_{\mathcal J}$ agree on all edge meridians. Since $\{\pm1\}$ is abelian, any homomorphism $\chi:\pi_1(X)\to\{\pm1\}$ factors through the
abelianization $\mathrm{ab}:\pi_1(X)\twoheadrightarrow H_1(X;\mathbb Z)$, i.e.\ there is a unique
$\bar\mu:H_1(X;\mathbb Z)\to\{\pm1\}$ such that
\[
\begin{array}{ccc}
\pi_1(X) & \xrightarrow{\ \chi\ } & \{\pm1\}\\
\downarrow{\mathrm{ab}} & & \|\\
H_1(X;\mathbb Z) & \xrightarrow{\ \bar\mu\ } & \{\pm1\}.
\end{array}
\]
If $\bar\mu:H_1(X;\mathbb Z)\to\{\pm1\}$ is the induced map, then for every $a\in H_1(X;\mathbb Z)$ we have
\(
\bar\mu(2a)=\bar\mu(a)^2=1.
\)
Hence $\bar\mu$ kills $2H_1(X;\mathbb Z)$, so it factors through
\(
H_1(X;\mathbb Z)/2H_1(X;\mathbb Z)\cong H_1(X;\mathbb Z/2).
\)
As in the proof of Lemma~3.1, the classes of the edge meridians generate $H_1(X;\mathbb Z/2)$, hence a $\{\pm1\}$-valued character factoring through $H_1(X;\mathbb Z/2)$ is determined by its values on the classes $\{[m_e]\}$. It follows that $\mu_{\mathcal J}\circ f_*=\mu_{\mathcal J}$.
\end{proof}

\begin{lemma}[Lifting criterion]\label{lem:12.2}
Let \(X\) be path-connected and locally path-connected, and let
\(\mu:\pi_1(X)\to\{\pm1\}\) be a nontrivial character. Suppose
\(p:\widetilde X\to X\) is the connected double cover corresponding to
\(H:=\ker(\mu)\), and let \(\tau\) denote its nontrivial deck involution.
For a homeomorphism $f:X\to X$, the following are equivalent:
\begin{enumerate}
\item $f$ admits a lift $\widetilde f:\widetilde X\to\widetilde X$ with $p\circ \widetilde f=f\circ p$;
\item $f_*(H)=H$;
\item $\mu\circ f_*=\mu$.
\end{enumerate}
When these hold, any two lifts differ by composition with $\tau$.
\end{lemma}
\begin{proof}
Fix $x_0\in X$ and $\tilde x_0\in p^{-1}(x_0)$, and set $x_1:=f(x_0)$. Choose $\tilde x_1\in p^{-1}(x_1)$ and write, with $p_*$ denoting
the homomorphism induced by $p$ on the fundamental groups,
\[
H_0:=p_*\pi_1(\widetilde X,\tilde x_0)\le \pi_1(X,x_0),
\qquad
H_1:=p_*\pi_1(\widetilde X,\tilde x_1)\le \pi_1(X,x_1).
\]
Since $p:\widetilde X\to X$ is a connected double cover, the corresponding subgroups
have index $2$ \cite[Proposition~1.32]{HatcherAT}:
\[
[\pi_1(X,x_0):H_0]=[\pi_1(X,x_1):H_1]=2.
\]
By the covering-space lifting criterion \cite[Proposition~1.33]{HatcherAT}, there exists a lift $\widetilde f$ with $\widetilde f(\tilde x_0)=\tilde x_1$ and $p\circ\widetilde f=f\circ p$ if and only if
\(
f_*(H_0)\subset H_1,
\)
where $f_*:\pi_1(X,x_0)\to\pi_1(X,x_1)$ is induced by $f$. Indeed, applying the criterion to
$f\circ p:(\widetilde X,\tilde x_0)\to (X,x_1)$ and to the covering
$p:(\widetilde X,\tilde x_1)\to (X,x_1)$ gives the condition
\[
(f\circ p)_*\pi_1(\widetilde X,\tilde x_0)
\subset
p_*\pi_1(\widetilde X,\tilde x_1),
\]
which is precisely $f_*(H_0)\subset H_1$.

If such a lift $\widetilde f$ exists, then $(p\circ\widetilde f)_*=(f\circ p)_*$ gives a commutative diagram
\[
\begin{tikzcd}
\pi_1(\widetilde X,\tilde x_0) \arrow[r,"\widetilde f_*"] \arrow[d,"p_*"'] &
\pi_1(\widetilde X,\tilde x_1) \arrow[d,"p_*"] \\
\pi_1(X,x_0) \arrow[r,"f_*"'] &
\pi_1(X,x_1)
\end{tikzcd}
\]
Then for any
$[\widetilde\gamma]\in \pi_1(\widetilde X,\tilde x_0)$ we have
\[
f_*\bigl(p_*([\widetilde\gamma])\bigr)
=
(f\circ p)_*([\widetilde\gamma])
=
(p\circ \widetilde f)_*([\widetilde\gamma])
=
p_*\bigl(\widetilde f_*([\widetilde\gamma])\bigr).
\]
Since $\widetilde\gamma$ is a loop at $\tilde x_0$ and $\widetilde f(\tilde x_0)=\tilde x_1$, the path $\widetilde f\circ \widetilde\gamma$ is a loop at $\tilde x_1$. Hence
\(
\widetilde f_*([\widetilde\gamma])\in \pi_1(\widetilde X,\tilde x_1),
\)
so
\[
f_*\bigl(p_*\pi_1(\widetilde X,\tilde x_0)\bigr)
\subset
p_*\pi_1(\widetilde X,\tilde x_1).
\]
Therefore,
\[
f_*(H_0)
=f_*\bigl(p_*\pi_1(\widetilde X,\tilde x_0)\bigr)
\subset p_*\pi_1(\widetilde X,\tilde x_1)
=H_1.
\]
Moreover,
\[
[\pi_1(X,x_1):f_*(H_0)]
=
[\pi_1(X,x_0):H_0]
=
2,
\]
because $f_*:\pi_1(X,x_0)\to\pi_1(X,x_1)$ is an isomorphism. Since $H_1$ also has index $2$ in $\pi_1(X,x_1)$, the inclusion $f_*(H_0)\subset H_1$ forces equality:
\(
f_*(H_0)=H_1.
\)
Conversely, $f_*(H_0)=H_1$ implies $f_*(H_0)\subset H_1$, hence the required lift exists. This proves \textnormal{(1)}$\Leftrightarrow$\textnormal{(2)}.

For (2)$\Leftrightarrow$(3), if $f_*(H)=H$, then for any loop $[\alpha]\in\pi_1(X,x_0)$ we have $[\alpha]\in H \iff f_*([\alpha])\in H$, hence
\[
\mu(f_*([\alpha]))=
\begin{cases}
1 & \text{if } [\alpha]\in H,\\
-1 & \text{if } [\alpha]\notin H,
\end{cases}
\qquad\text{so}\qquad
\mu\circ f_*=\mu.
\]
Conversely, if $\mu\circ f_*=\mu$ and $[\alpha]\in H$, then $\mu([\alpha])=1$ implies $\mu(f_*([\alpha]))=1$, hence $f_*([\alpha])\in H$; thus $f_*(H)\subset H$. Since $\mu\circ f_*=\mu$ implies $\mu\circ (f^{-1})_*=\mu$, the same argument applied to $f^{-1}$ gives $(f^{-1})_*(H)\subset H$.
Applying $f_*$ to both sides gives $H\subset f_*(H)$, hence $f_*(H)=H$.

Finally, if $\widetilde f_1,\widetilde f_2$ are two lifts of $f$, then $\widetilde f_1(\tilde x_0)$ and $\widetilde f_2(\tilde x_0)$ both lie in the fiber $p^{-1}(x_1)$. Since $\Deck(p)=\{\id,\tau\}$ acts transitively on this two-point fiber, there exists $\delta\in\{\id,\tau\}$ such that
\(
\delta\bigl(\widetilde f_1(\tilde x_0)\bigr)=\widetilde f_2(\tilde x_0).
\)
Then $\delta\circ \widetilde f_1$ and $\widetilde f_2$ are lifts of $f$ that agree at $\tilde x_0$, so by uniqueness of lifts they are equal. Hence
\(
\widetilde f_2=\delta\circ \widetilde f_1,
\)
i.e.\ $\widetilde f_2=\widetilde f_1$ or $\widetilde f_2=\tau\circ \widetilde f_1$.
\end{proof}

\subsection{Symmetries preserving monodromy and the lift extension}
\label{subsec:symmetries-monodromy-lift-extension}

In the polytope cases, the relevant symmetries preserve the graph $\Gamma$ and permute its edges. Since the monodromy character is prescribed on edge meridians, this implies invariance of monodromy and hence existence of lifts to the associated connected double cover.

Given a group $G$ acting on $M$ by homeomorphisms, define the group of lifts by
\[
\widetilde G
:=
\{\widetilde g\in \mathrm{Homeo}(\widetilde M)\mid
\exists\, g\in G \text{ such that } p\circ \widetilde g=g\circ p\},
\]
where $p:\widetilde M\to M$ is the connected double cover associated to
$\ker(\mu_{\mathcal J})$.
The next lemma packages the individual lifts into a short exact sequence and identifies splitting with a choice of lifting group actions to the double cover preserving the group homomorphism.

\begin{lemma}[Group of lifts and the central extension]\label{lem:12.5}
Let \(p:\widetilde M\to M\) be a connected double cover with deck involution \(\tau\), and let \(G\) be a group acting on \(M\) by homeomorphisms. Let \(\widetilde G\) denote the set of lifts of elements of \(G\) to \(\widetilde M\). Assume that every \(g\in G\) admits at least one lift to \(\widetilde M\). Then \(\widetilde G\) is a group under composition, and the projection \(\pi:\widetilde G\to G\) defined by \(\pi(\tilde g)=g\) fits into a short exact sequence
\[
1\longrightarrow \langle\tau\rangle\longrightarrow \widetilde G
\xrightarrow{\ \pi\ } G\longrightarrow 1.
\]
Moreover, this sequence splits if and only if one can choose, for each \(g\in G\), a lift
\(s(g)\in \widetilde G\) such that \(s(g_1g_2)=s(g_1)s(g_2)\) for all \(g_1,g_2\in G\).
\end{lemma}

\begin{proof}
Since \(p\circ \mathrm{id}_{\widetilde M}=\mathrm{id}_M\circ p\), one has
\(\mathrm{id}_{\widetilde M}\in \widetilde G\). If \(\tilde g_i\) is a lift of \(g_i\) for
\(i=1,2\), then
\[
p\circ(\tilde g_1\tilde g_2)=(p\circ\tilde g_1)\tilde g_2
=(g_1\circ p)\tilde g_2
=g_1\circ(p\circ\tilde g_2)
=(g_1g_2)\circ p,
\]
so \(\tilde g_1\tilde g_2\) is a lift of \(g_1g_2\).

Now let \(\tilde g\in \widetilde G\) be a lift of \(g\in G\). Since \(g^{-1}\in G\), by hypothesis there exists a lift \(\tilde h\) of \(h=g^{-1}\). 
Then $\tilde h\tilde g$ and $\mathrm{id}_{\widetilde M}$ are lifts of $\mathrm{id}_M$ under $p$, since $p\circ(\tilde h\tilde g)=(p\circ\tilde h)\circ\tilde g=(h\circ p)\circ\tilde g
=h\circ(g\circ p)=\mathrm{id}_M\circ p$
and
$p\circ\mathrm{id}_{\widetilde M}=\mathrm{id}_M\circ p$.
Choose \(\tilde x_0\in \widetilde M\), and replace \(\tilde h\) by \(\tau\circ \tilde h\) if necessary so that \(\tilde h(\tilde g(\tilde x_0))=\tilde x_0\). Then \(\tilde h\tilde g\) and \(\mathrm{id}_{\widetilde M}\) are lifts of \(\mathrm{id}_M\) which agree at \(\tilde x_0\), so by uniqueness of lifts \(\tilde h\tilde g=\mathrm{id}_{\widetilde M}\). Similarly, \(\tilde g\tilde h=\mathrm{id}_{\widetilde M}\). Hence \(\tilde h=\tilde g^{-1}\), so \(\tilde g^{-1}\in \widetilde G\). Therefore \(\widetilde G\) is a group under composition.

The projection \(\pi\) is well defined. Indeed, if
\(\widetilde g\) covers both \(g_1\) and \(g_2\), then
\(
g_1 \circ p = p \circ \widetilde g = g_2 \circ p .
\)
Since \(p\) is surjective, \(g_1=g_2\). Hence \(\pi(\widetilde g)\) is
unambiguous. The map \(\pi\) is surjective by the hypothesis that every \(g\in G\)
admits a lift. It remains to check that \(\pi\) is a homomorphism.
Suppose that \(\pi(\tilde g_i)=g_i\) for \(i=1,2\). Then
\(
p\circ(\tilde g_1\tilde g_2)
=
(p\circ \tilde g_1)\circ \tilde g_2
=
(g_1\circ p)\circ \tilde g_2
=
g_1\circ(g_2\circ p)
=
(g_1g_2)\circ p .
\)
Therefore
\(
\pi(\tilde g_1\tilde g_2)=g_1g_2
=
\pi(\tilde g_1)\pi(\tilde g_2).
\)

Its kernel consists of lifts of \(\mathrm{id}_M\), that is, the deck transformations of the connected double cover \(p\). Hence \(\ker(\pi)=\{\mathrm{id}_{\widetilde M},\tau\}=\langle\tau\rangle\), and therefore
\[
1\longrightarrow \langle\tau\rangle\longrightarrow \widetilde G
\xrightarrow{\ \pi\ } G\longrightarrow 1
\]
is exact. The subgroup \(\langle\tau\rangle\) is central in \(\widetilde G\). Indeed,
if \(\widetilde g\in\widetilde G\), then
\(\widetilde g\tau\widetilde g^{-1}\) is a deck transformation of \(p\).
It is not the identity, since otherwise \(\tau=\mathrm{id}\). Hence
\(\widetilde g\tau\widetilde g^{-1}=\tau\).

Finally, a splitting \(s:G\to \widetilde G\) with \(\pi\circ s=\mathrm{id}_G\) assigns to each \(g\in G\) a lift \(s(g)\) of \(g\), and since \(s\) is a homomorphism it satisfies \(s(g_1g_2)=s(g_1)s(g_2)\) for all \(g_1,g_2\in G\). Conversely, any choice of lifts \(s(g)\) with \(s(g_1g_2)=s(g_1)s(g_2)\) defines a homomorphism \(s:G\to \widetilde G\) such that \(\pi\circ s=\mathrm{id}_G\). This proves the final claim.
\end{proof}

\subsection{From cover lifts to fiberwise linear bundle automorphisms}
\label{subsec:cover-lifts-to-bundle-automorphisms}

A lift on the connected double cover determines a fiberwise linear automorphism of the
associated real line bundle. Replacing a chosen lift by its composition with the deck
involution changes the induced bundle map by a sign.
Lemma~\ref{lem:12.4} connects lifts of $G$ on the double cover $p:\widetilde M\to M$ to fiberwise linear lifts on the associated real line bundle $\mathcal I\to M$ and explains how the deck involution $\tau$ produces the sign ambiguity in bundle automorphisms.

\begin{lemma}[From lifts of the double cover to fiberwise lifts]\label{lem:12.4}
Let \(p:\widetilde M\to M\) be a connected double cover with deck involution \(\tau\), and let \(\mathcal J\to M\) be the associated real line bundle, realized as the quotient
\[
\mathcal J \;\cong\; (\widetilde M\times\R)\big/\big((\tau\tilde x,u)\sim(\tilde x,-u)\big).
\]
Suppose \(g:M\to M\) is a homeomorphism admitting a lift \(\tilde g:\widetilde M\to\widetilde M\), so that \(p\circ\tilde g=g\circ p\). Then:
\begin{enumerate}
\item \(\tilde g\) commutes with \(\tau\).
\item The formula \(\widehat g\big([\tilde x,u]\big):=[\tilde g(\tilde x),u]\) defines a well-defined fiberwise linear bundle automorphism \(\widehat g:\mathcal J\to\mathcal J\) covering \(g\).
\item Replacing \(\tilde g\) by the other lift \(\tau\circ\tilde g\) replaces \(\widehat g\) by \(-\widehat g\).
\end{enumerate}
\end{lemma}
\begin{proof}
Let $\mu_\mathcal{J}:\pi_1(M)\to\{\pm1\}$ denote the monodromy character of \(\mathcal{J}\). Equivalently, \(p:\widetilde M\to M\) is the connected double cover corresponding to \(\ker(\mu_\mathcal{J})\).
Since $\tilde g$ is a lift of \(g\), Lemma~\ref{lem:12.2} gives
\(
\mu_\mathcal{J}\circ g_*=\mu_\mathcal{J}.
\)
Let $h=g^{-1}$. Since $h_*=g_*^{-1}$ and
$\mu_{\mathcal J}\circ g_*=\mu_{\mathcal J}$, composing on the right by
$g_*^{-1}$ gives
\(
\mu_{\mathcal J}\circ h_*=\mu_{\mathcal J}.
\)
Hence, by Lemma~\ref{lem:12.2}, \(h\) admits a lift \(\tilde h:\widetilde M\to\widetilde M\).
Choose \(\tilde x_0\in\widetilde M\), and replace \(\tilde h\) by \(\tau\circ \tilde h\) if
necessary so that
\(
\tilde h(\tilde g(\tilde x_0))=\tilde x_0.
\)
Then \(\tilde h\circ \tilde g\) is a lift of \(\id_M\), since
\[
p\circ (\tilde h\circ \tilde g)
=(p\circ \tilde h)\circ \tilde g
=(h\circ p)\circ \tilde g
=h\circ (g\circ p)
=\id_M \circ p.
\]
It agrees with \(\id_{\widetilde M}\) at \(\tilde x_0\), so by uniqueness of lifts,
\(
\tilde h\circ \tilde g=\id_{\widetilde M}.
\)
Similarly,
\[
p\circ (\tilde g\circ \tilde h)
=(p\circ \tilde g)\circ \tilde h
=(g\circ p)\circ \tilde h
=g\circ (h\circ p)
=p,
\]
and \((\tilde g\circ \tilde h)(\tilde g(\tilde x_0))=\tilde g(\tilde x_0)\), so again by
uniqueness of lifts,
\(
\tilde g\circ \tilde h=\id_{\widetilde M}.
\)
Thus \(\tilde h=\tilde g^{-1}\), and in particular \(\tilde g\) is a homeomorphism.

\smallskip

\noindent
(1) Since \(p\) is a connected \(2\)-fold cover, \(\Deck(p)=\{\id,\tau\}\).
The conjugate \(\tilde g\tau\tilde g^{-1}\) is a deck transformation, since
\[
p\circ(\tilde g\tau\tilde g^{-1})
=(p\circ\tilde g)\circ\tau\circ\tilde g^{-1}
=(g\circ p)\circ\tau\circ\tilde g^{-1}
=g\circ(p\circ\tau)\circ\tilde g^{-1}
=g\circ p\circ\tilde g^{-1}
=p.
\]
Hence \(\tilde g\tau\tilde g^{-1}\in\Deck(p)\).
If \(\tilde g\tau\tilde g^{-1}=\id\), then \(\tau=\id\), a contradiction. Therefore
\(
\tilde g\tau\tilde g^{-1}=\tau,
\)
that is,
\(
\tilde g\tau=\tau\tilde g.
\)

\smallskip

\noindent
(2) If \((\tau\tilde x,u)\sim(\tilde x,-u)\), then using (1),
\[
[\tilde g(\tau\tilde x),u]=[\tau\tilde g(\tilde x),u]=[\tilde g(\tilde x),-u],
\]
so \(\widehat g\) is well-defined on the quotient. It is fiberwise linear because
\begin{align*}
\widehat g([\tilde x,u]+[\tilde x,v])
&=[\tilde g(\tilde x),u+v]
 =\widehat g([\tilde x,u])+\widehat g([\tilde x,v]),\\
\widehat g(\lambda[\tilde x,u])
&=[\tilde g(\tilde x),\lambda u]
 =\lambda\,\widehat g([\tilde x,u]).
\end{align*}
Moreover, if \(\pi:\mathcal J\to M\) denotes the bundle projection, then
\[
\pi(\widehat g([\tilde x,u]))
=p(\tilde g(\tilde x))
=g(p(\tilde x))
=g(\pi([\tilde x,u])),
\]
so \(\widehat g\) covers \(g\).

Since \(\tilde g^{-1}\) is a lift of \(g^{-1}\), the same construction applied to \(\tilde g^{-1}\) gives an inverse bundle map. Hence \(\widehat g\) is a bundle automorphism.

\smallskip
\noindent
(3) If \(\tilde g'=\tau\circ\tilde g\), then
\(
\widehat g'([\tilde x,u])
=[\tau\tilde g(\tilde x),u]
=[\tilde g(\tilde x),-u]
=-\,\widehat g([\tilde x,u]).
\)
\end{proof}

\begin{remark}
Part~(2) may also be viewed through the functoriality of associated bundles. Since
\(
\mathcal J \cong \widetilde M\times_{\mathbb Z/2}\mathbb R_{\mathrm{sgn}},
\)
Proposition~6.2 of Husemoller implies that this principal-bundle action induces a coherent fiberwise linear \(G\)-action on \(\mathcal I\) \cite{husemoller1966fibre}.
\end{remark}

\section{Order--$3$ Isometries on Regular $5$--cell} \label{sec:5cell}
Let $P_5\subset \mathbb R^4$ be the regular $5$-cell, let $\Gamma_5\subset S^3$ be the
radial projection of its $1$-skeleton, set $M_5:=S^3\setminus \Gamma_5$, and let
$G_5:=\mathrm{Sym}^+(P_5)\cong A_5$. By Section~\ref{sec:monodromy-double-cover-lifts}, the remaining tasks are to identify
the odd-order edge stabilizers needed for Proposition~\ref{prop:odd-stab} and to split the associated short exact sequence
\[
1\to \langle \tau_5\rangle \to \widetilde G_5 \to G_5 \to 1.
\]

\subsection{The 5--cell model.}\label{subsec:5cell}
A $5$--cell (regular $4$--simplex) is the convex hull of five points in $\mathbb R^4$ with all pairwise distances equal; see e.g.\ \cite{coxeter1973regular}. It has Schl\"afli symbol $\{3,3,3\}$, so its facets and vertex figures are regular tetrahedra. In this subsection we fix explicit vertices $V_1^{(5)},\dots,V_5^{(5)}\in S^3\subset\mathbb R^4\simeq\mathbb H$ adapted to the quaternionic order--$3$ rotation $\Ad_{q}$, and we record the induced permutation of vertices and tetrahedral facets under $\Ad_{q}$.

The symmetry group of the $5$--cell is isomorphic to \(S_{5}\), acting by permuting the vertices, and its rotational subgroup is $A_{5}$.
Let
\[
F:=\mathrm{span}_{\mathbb R}\{1,i+j+k\}\subset\mathbb H
\]
denote the fixed $2$--plane of $\Ad_{q}$ in $\mathbb R^4\simeq\mathbb H$. We choose $V_1^{(5)},V_2^{(5)}\in F\cap S^3$ and $V_3^{(5)},V_4^{(5)},V_5^{(5)}$ so that $\Ad_{q}$ cyclically permutes $\{V_3^{(5)},V_4^{(5)},V_5^{(5)}\}$.

The five vertices
\[
\begin{array}{c}
V_{1}^{(5)}=\left( 1,\; 0,\; 0,\; 0\right), \ \ \
V_{2}^{(5)}=\left( -\frac14,\; s,\; s,\; s\right), \ \ \
V_{3}^{(5)}=\left( -\frac14,\; s,\,-s,\,-s\right),\\
V_{4}^{(5)}=\left( -\frac14,\,-s,\; s,\,-s\right), \ \ \ 
V_{5}^{(5)}=\left( -\frac14,\,-s,\,-s,\; s\right)
\end{array}
\]
where $s:=\frac{\sqrt5}{4}$. One checks that $V_i^{(5)}\in S^3$ and that $\|V_i^{(5)}-V_j^{(5)}\|^2=5/2$ for $i\neq j$; hence the $V_i^{(5)}$ form a regular $4$--simplex whose centroid is the origin. Thus \(P_5\) has \(5\) vertices, \(10\) edges, and \(5\) tetrahedral \(3\)-cells. 

Using the identification $\R^{4}\cong\HH$ from Section~\ref{subsec:quat}, consider
\[
q=\frac12(1+i+j+k)\in S^{3},
\qquad
\Ad_{q}(x)=q\,x\,(q)^{-1}.
\]
On $\Im\HH\cong\R^{3}$ it is a $120^{\circ}$ rotation about the axis $\Span\{i+j+k\}$, and in $\R^{4}$ its fixed--point set is the plane $\Span\{1,\ i+j+k\}$.
This choice of $q$ will be used in Sections~\ref{sec:5cell}--\ref{sec:120cell}.

The images of vertices under \(\Ad_{q}\) are as follows:
\[
\Ad_{q} :\;
V_{3}^{(5)}\;\longrightarrow\;V_{4}^{(5)}\;\longrightarrow\;V_{5}^{(5)}\;\longrightarrow\;V_{3}^{(5)},
\qquad
V_{1}^{(5)},V_{2}^{(5)}\;\text{fixed.}
\]

For any index \(i\), let  
\[
T_{i}^{(5)}\;:=\;\Hull{\{V_{1}^{(5)},V_{2}^{(5)},V_{3}^{(5)},V_{4}^{(5)},V_{5}^{(5)}\}\setminus\{V_{i}^{(5)}\}}
\]
be the regular tetrahedral facet opposite \(V_{i}^{(5)}\).
Every facet has edge length \(\sqrt{5/2}\) and all ten triangular faces are congruent equilateral triangles of side \(\sqrt{5/2}\). Any two distinct facets share exactly one triangular face, so the adjacency graph is the complete graph \(K_{5}\).

By computing the quaternion conjugation, $\Ad_{q}$ fixes the $2$--plane $F=\operatorname{span}_{\R}\{1,i+j+k\}$ pointwise and rotates $(F)^\perp$ by angle $2\pi/3$. Hence $\Fix(\Ad_{q}|_{S^3})=F\cap S^3$, so $\Ad_{q}$ fixes $V_1^{(5)},V_2^{(5)}$ and cyclically permutes $V_3^{(5)},V_4^{(5)},V_5^{(5)}$. Consequently, the induced permutation on facets is
\(
\left(T_1^{(5)}\right)\left(T_2^{(5)}\right)\left(T_3^{(5)}\,T_4^{(5)}\,T_5^{(5)}\right),
\)
i.e.\ $\Ad_{q}:T_3^{(5)}\mapsto T_4^{(5)}\mapsto T_5^{(5)}\mapsto T_3^{(5)}$.

\subsubsection{Lifted facet decomposition and the double cover.}\label{sssec:pm-labels}
Let $\Gamma_5\subset S^3$ be the radial projection of the $1$-skeleton of $P_5$, and set
\(
M_5:=S^3\setminus \Gamma_5.
\)
Let
\(
\mu_5:\pi_1(M_5)\to\{\pm1\}
\)
denote the monodromy character determined by $\mu_5([m_e])=-1$ on edge meridians.
Let
\[
\mathcal T_i^{(5)}:=\pi\left(T_i^{(5)}\right)\subset S^3,\qquad \mathring{\mathcal T}_i^{(5)}:=\mathcal T_i^{(5)}\setminus \Gamma_5\subset M_5.
\]
Let $p_5:\widetilde M_5\to M_5$ be the associated connected double cover corresponding to the kernel of the monodromy character $\mu_5$ and let $\tau_5$ denote its nontrivial deck involution.
Therefore, the preimage $p_5^{-1}\left(\mathring{\mathcal T}_i^{(5)}\right)$ is the disjoint union of two copies of \(\mathring{\mathcal T}_i^{(5)}\); after fixing one choice of sheet we write
\[
p_5^{-1}\left(\mathring{\mathcal T}_i^{(5)}\right)=\mathcal T_i^{(5),+}\sqcup \mathcal T_i^{(5),-},
\qquad
\tau_5\left(\mathcal T_i^{(5),\pm}\right)=\mathcal T_i^{(5),\mp}.
\]
For $i\neq j$ let $F_{ij}^{(5)}:=\mathcal T_i^{(5)}\cap \mathcal T_j^{(5)}$ denote the common triangular face, and set $\mathring{F}_{ij}^{(5)}:=F_{ij}^{(5)}\setminus \Gamma_5\subset M_5$. Then 
$p_5^{-1}\left(\mathring{F}_{ij}^{(5)}\right)=F_{ij}^{(5),+}\sqcup F_{ij}^{(5),-}$ with $\tau_5\left(F_{ij}^{(5),\pm}\right)=F_{ij}^{(5),\mp}$.

Since \(\Ad_q\) preserves \(\Gamma_5\) and permutes its edges, Lemma~\ref{lem:12.3} gives \(\mu_5\circ (\Ad_q)_*=\mu_5\). Hence, by Lemma~\ref{lem:12.2}, the symmetry \(\Ad_q\in G_5\) admits a lift to \(\widetilde M_5\), and any two such lifts differ by composition with \(\tau_5\). Since $\Ad_q$ has order $3$ on $M_5$, $\widetilde{\Ad}_q^{\,3}\in\{\id,\tau_5\}$; consequently, one lift has order $3$ and the other has order $6$. We fix once and for all the order--$6$ lift $\widetilde{\Ad}_q$ characterized by
\(
\widetilde{\Ad}_q^{\,3}=\tau_5.
\)
Its action on the lifted facets is the permutation
\[
\widetilde{\Ad}_q:\quad
\left(\mathcal T_1^{(5),+}\ \mathcal T_1^{(5),-}\right)\,\left(\mathcal T_2^{(5),+}\ \mathcal T_2^{(5),-}\right)\,
\left(\mathcal T_3^{(5),+}\ \mathcal T_4^{(5),-}\ \mathcal T_5^{(5),+}\ \mathcal T_3^{(5),-}\ \mathcal T_4^{(5),+}\ \mathcal T_5^{(5),-}\right),
\]
as illustrated in Fig.~\ref{fig:tetra}.

\begin{figure}
    \centering
    \includegraphics[width=0.8\textwidth]{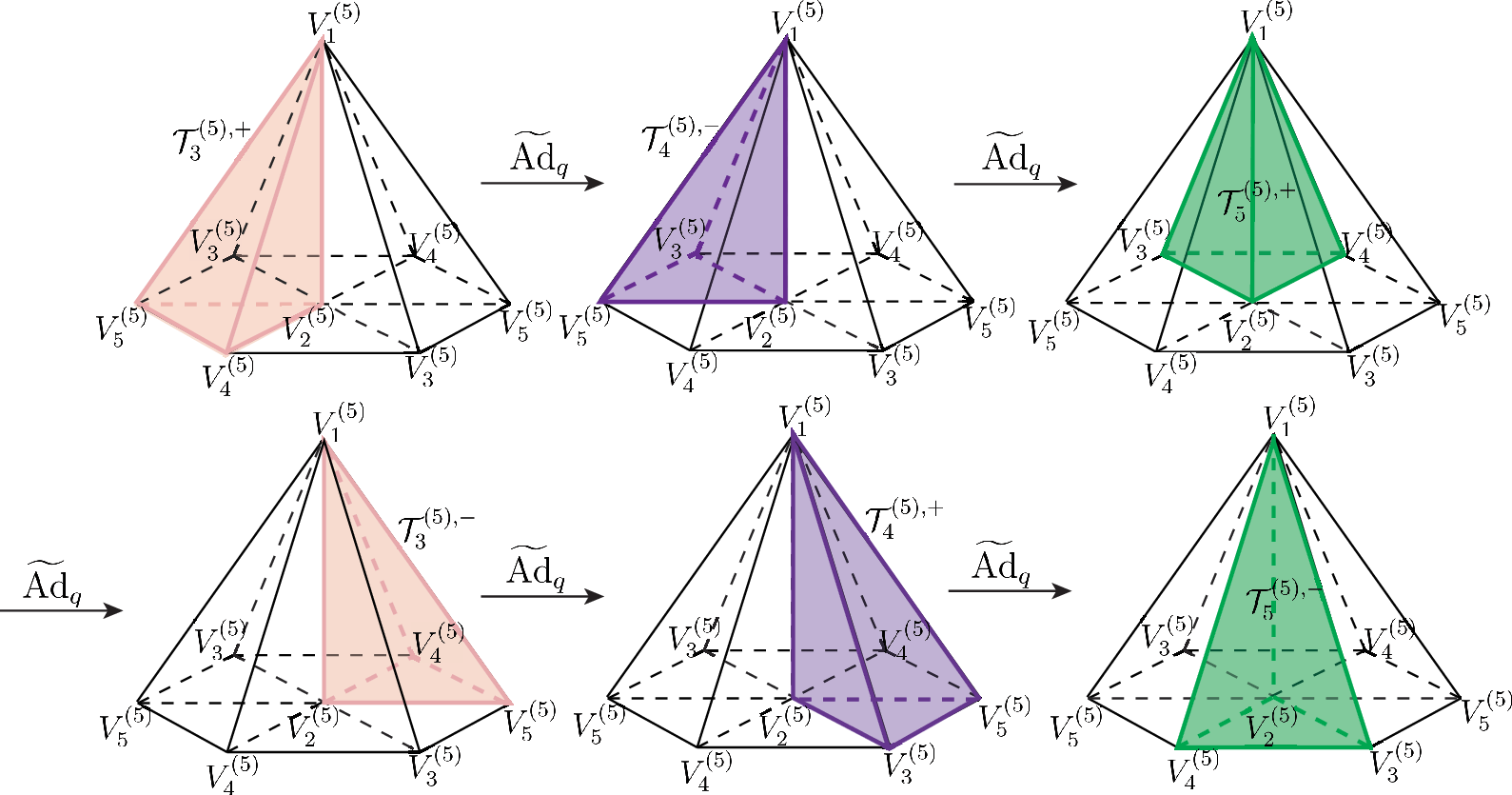}
\caption{Lifted facet dynamics for the order--6 lift
\(\widetilde{\operatorname{Ad}}_q\) over the edge
\([V_1^{(5)},V_2^{(5)}]\). Here \(T_i^{(5),\pm}\) denotes the two lifts
of the facet \(T_i^{(5)}\), and \(\tau_5\) is the deck involution.
Reading left to right across the first row and then the second row gives
\(
T_3^{(5),+}\to T_4^{(5),-}\to T_5^{(5),+}\to
T_3^{(5),-}\to T_4^{(5),+}\to T_5^{(5),-},
\)
with \(\widetilde{\operatorname{Ad}}_q^{\,3}=\tau_5\).}
\label{fig:tetra}
\end{figure}

\subsection{Invariance of $\mu_5$ and lifting to $\widetilde M_5$}
This subsection shows that for 5-cells, the natural $G_5=\Sym^+(P_5)\cong A_5$--action on $M_5$ preserves the monodromy character $\mu_5:\pi_1(M_5)\to\{\pm1\}$, and hence lifts to the connected double cover $p_5:\widetilde M_5\to M_5$. 
Let \(\mathcal{I}_5\to M_5\) denote the real line bundle associated to
\(p_5:\widetilde M_5\to M_5\). Proposition~\ref{prop:12.3(2)} supplies, for each edge of \(P_5\), an element of odd order in \(G_5\) fixing that edge pointwise, to fulfill the 5-cell case of Proposition~\ref{prop:odd-stab}. 
By Proposition~\ref{prop:even}, every vertex of \(\Gamma_5\) has valency \(4\), so the assignment \(\mu_5([m_e])=-1\) for each edge meridian \(m_e\) is well-defined. Since every \(g\in G_5\) preserves \(\Gamma_5\) and permutes its edges, Lemma~\ref{lem:12.3} implies that $\mu_5\circ g_*=\mu_5$ for every $g\in G_5$. By Lemma~\ref{lem:12.2}, this is equivalent to saying that each \(g\in G_5\) lifts to the connected double cover \(p_5:\widetilde M_5\to M_5\), uniquely up to the deck involution \(\tau_5\). Lemma~\ref{lem:12.5} packages these individual lifts into the short exact sequence
\[
1\longrightarrow \langle\tau_5\rangle \longrightarrow \widetilde G_5
\xrightarrow{\ \pi\ } G_5\longrightarrow 1.
\]
Proposition~\ref{prop:5cell-splitting} proves that this extension splits. Equivalently, the \(G_5\)-action on \(M_5\) lifts coherently to \(\widetilde M_5\). Then Lemma~\ref{lem:12.4} shows that this coherent lift induces fiberwise linear bundle automorphisms of the associated line bundle \(\mathcal{I}_5 \to M_5\).

\begin{proposition}[Order--$3$ symmetries of the $5$--cell]\label{prop:12.3(2)}
Let \(P_5\subset\R^{4}\) be the regular \(5\)--cell with vertex set \(\left\{V^{(5)}_{1},\dots,V^{(5)}_{5}\right\}\subset \mathbb{R}^{4}\), and let \(\Gamma_5\subset S^{3}\) be the radial projection of the \(1\)--skeleton of \(P_5\). Set \(G_5:=\Sym^{+}(P_5)\cong A_{5}\subset SO(4)\). Let \(E^{(5)}_0:=\left[V^{(5)}_1,V^{(5)}_2\right]\). For any edge \(E^{(5)}=\left[V^{(5)}_i,V^{(5)}_j\right]\) of \(P_5\) there exists \(h\in G_5\) with \(h\left(E^{(5)}_0\right)=E^{(5)}\), and the two nontrivial order--\(3\) elements of \(G_5\) fixing \(E^{(5)}\) pointwise are \(h\Ad_q h^{-1}\) and its inverse. Each restricts to $W_{E^{(5)}}^\perp$, where $W_{E^{(5)}}:=\Span_{\R}\left\{V^{(5)}_i,V^{(5)}_j\right\}$, to a rotation by \(\pm 2\pi/3\). In particular, \(G_5\) contains exactly \(20\) elements of order \(3\).
\end{proposition}
\begin{proof}
By \S\ref{subsec:5cell}, the map \(\Ad_q\in G_5\subset SO(4)\) has order \(3\) and fixes pointwise the \(2\)--plane \(F=\Fix(\Ad_q)\), rotating \(F^\perp\) by \(2\pi/3\). Moreover, $V_1^{(5)}$ and $V_2^{(5)}$ lie in $F$, so $\mathrm{Ad}_q$ fixes the edge $E_0^{(5)}$ pointwise. Since $V_1^{(5)}$ and $V_2^{(5)}$ are linearly independent and $F$ is a $2$--plane, we have
\[
F=\operatorname{span}_{\mathbb R}\{V_1^{(5)},V_2^{(5)}\}=W_{E_0^{(5)}}.
\]
Therefore $\mathrm{Ad}_q$ restricting to $W_{E_0^{(5)}}^\perp$ is a rotation by $2\pi/3$. Since \(G_5\cong A_5\) acts transitively on the \(2\)--element subsets of \(\{1,\dots,5\}\), for any edge \(E^{(5)}\) there exists \(h\in G_5\) with \(h\left(E^{(5)}_0\right)=E^{(5)}\). Then \(h\Ad_q h^{-1}\) fixes \(E^{(5)}\) pointwise and fixes \(W_{E^{(5)}}=h\bigl(W_{E^{(5)}_0}\bigr)\) pointwise. Hence its restriction to \(W_{E^{(5)}}^\perp\) is a rotation by \(\pm 2\pi/3\).
Any order--\(3\) symmetry fixing both endpoints of \(E^{(5)}\) must act as a \(3\)--cycle on the remaining three vertices; there are exactly two such \(3\)--cycles. As \(P_5\) has \(\binom52=10\) edges, \(G_5\) contains exactly \(20\) elements of order \(3\).
\end{proof}
 
\subsection{From lifts of the double cover to fiberwise lifts.}\label{sssec:cover-to-bundle}
Given $g^*\mathcal{I} \simeq \mathcal{I}$, Lemma~\ref{lem:12.1} gives $\mu \circ g_* = \mu$, thus $\mu$ is $G$--invariant. Therefore, from Lemma~\ref{lem:12.2}, one knows that every $g\in G_5$ admits a lift. We then define the \emph{group of lifts}
\[
\widetilde G_5:=\{\widetilde g\in \mathrm{Diff}(\widetilde M_5)\mid \exists g\in G_5\text{ such that }
p_5\circ\widetilde g=g\circ p_5\}.
\]

Lemma~\ref{lem:12.4} shows that for each \(g\in G_5\), a choice of lift \(\widetilde g:\widetilde M_5\to \widetilde M_5\) induces a fiberwise linear bundle automorphism \(\widehat g:\mathcal I_5\to\mathcal I_5\), and replacing \(\widetilde g\) by \(\tau_5\circ\widetilde g\) replaces \(\widehat g\) by \(-\widehat g\). Thus a splitting of the extension in Lemma~\ref{lem:12.5} is equivalent to a coherent choice of bundle maps \(\widehat g\) satisfying \(\widehat{gh}=\widehat g\circ \widehat h\) for all \(g,h\in G_5\). Proposition~\ref{prop:5cell-splitting} constructs such a splitting when \(G_5=\Sym^{+}(P_5)\cong A_5\).

\begin{proposition}[Splitting for the $5$--cell group]\label{prop:5cell-splitting}
The \(G_5\)-action on \(M_5\) lifts to a fiberwise linear action by bundle automorphisms
\(
\widehat g:\mathcal I_5\to \mathcal I_5
\)
covering \(g\), such that
\(
\widehat{gh}=\widehat g\circ \widehat h
\)
for all \(g,h\in G_5\).
\end{proposition}
\begin{proof}
Let \(H_5:=\ker(\mu_5)\). By Lemma~\ref{lem:12.3}, for every \(g\in G_5\) we have \(\mu_5\circ g_*=\mu_5\).
Since \(p_5:\widetilde M_5\to M_5\) is the connected double cover corresponding to \(H_5\), Lemma~\ref{lem:12.2} implies that \(g_*(H_5)=H_5\) and that every \(g\in G_5\) admits a lift to \(\widetilde M_5\), unique up to composition with the deck involution \(\tau_5\).

Let \(\pi:\widetilde G_5\to G_5\) be given by \(\pi(\widetilde g)=g\).
By Lemma~\ref{lem:12.5}, this fits into a short exact sequence
\[
1\longrightarrow \langle\tau_5\rangle\longrightarrow \widetilde G_5
\xrightarrow{\ \pi\ } G_5\longrightarrow 1,
\]
so it suffices to construct a splitting \(\sigma:G_5\to\widetilde G_5\).

According to Eq.~(\ref{eq:5}), \(G_5\cong A_5\) has the presentation
\[
A_5=\langle r,s\mid r^2=1,\ s^3=1,\ (rs)^5=1\rangle.
\]

\noindent\emph{Step 1: choose a lift of \(s\) of order \(3\).}
By Lemma~\ref{lem:12.2}, \(s\) admits a lift \(\widetilde s\) to \(\widetilde M_5\). Since \(s^3=\id\), the element \(\widetilde s^{\,3}\) is a deck transformation, so
\(
\widetilde s^{\,3}\in\{\id,\tau_5\}.
\)
If \(\widetilde s^{\,3}=\tau_5\), replace \(\widetilde s\) by \(\tau_5\widetilde s\). By Lemma~\ref{lem:12.4}(1), every lift commutes with \(\tau_5\), hence
\(
(\tau_5\widetilde s)^3
=\tau_5^{\,3}\widetilde s^{\,3}
=\tau_5\cdot\tau_5
=\id.
\)
Thus we may choose a lift \(\widetilde s\) such that
\(
\widetilde s^{\,3}=\id.
\)

\smallskip\noindent\emph{Step 2: choose a lift of \(r\) of order \(2\).}
As an involution in \(A_5\), \(r\) is a double transposition on the five vertices, so it fixes exactly one vertex of \(P_5\). 
Since \(r^2=\id\), the minimal polynomial of \(r\) divides \(x^2-1=(x-1)(x+1)\), so \(r\) is diagonalizable with eigenvalues in \(\{\pm1\}\). As \(r\in SO(4)\), this implies \(\det(r)=1\), so the multiplicity of the eigenvalue \(-1\) is even. Since \(r\neq \pm\id\),
the only possibility is that both \(+1\) and \(-1\) occur with multiplicity \(2\). Hence \(\ker(r-\mathrm{id})\) is a \(2\)-plane in \(\mathbb R^4\), and
\(\operatorname{Fix}(r|S^3)=\ker(r-\mathrm{id})\cap S^3\) is a great circle.
This circle cannot be contained in \(\Gamma_5\); otherwise \(r\) would fix the
radial image of some edge pointwise, whereas the orientation-preserving
pointwise stabilizer of an edge is cyclic of order \(3\) by the triangular
edge figure. Hence \(\operatorname{Fix}(r|S^3)\cap M_5\ne\emptyset\). Choose \(x_r\in \Fix(r|_{S^3})\setminus \Gamma_5\), a lift \(\widetilde x_r\in p_5^{-1}(x_r)\), and let \(\widetilde r\) be the unique lift of \(r\) such that \(\widetilde r(\widetilde x_r)=\widetilde x_r\). Then \(\widetilde r^{\,2}\) is a deck transformation, since it covers \(r^2=\id\). Because \(\widetilde r^{\,2}\) fixes \(\widetilde x_r\) and \(\tau_5\) acts freely, it follows that
\(
\widetilde r^{\,2}=\id.
\)
The same argument also works for $\tau_5 \widetilde r.$

\smallskip
\noindent\emph{Step 3: enforce the relation \((rs)^5=1\) upstairs.}
Since \(p_5\circ(\widetilde r\widetilde s)=(rs)\circ p_5\), iterating gives
\[
p_5\circ(\widetilde r\widetilde s)^5=(rs)^5\circ p_5=p_5.
\]
Hence \((\widetilde r\widetilde s)^5\) is a lift of \(\id_{M_5}\), i.e.\ a deck transformation. If \((\widetilde r\widetilde s)^5=\tau_5\), replace \(\widetilde r\) by \(\tau_5\widetilde r\). By Lemma~\ref{lem:12.4}(1), every lift commutes with \(\tau_5\), so \(\tau_5\) is central in \(\widetilde G_5\). Therefore
\(
(\tau_5\widetilde r)^2=\tau_5^2\widetilde r^{\,2}=\id
\)
and
\(
((\tau_5\widetilde r)\widetilde s)^5
=\tau_5^{\,5}(\widetilde r\widetilde s)^5
=\tau_5\cdot\tau_5
=\id.
\)
Hence the chosen lifts satisfy the defining relations, so
\(r\mapsto \widetilde r\) and \(s\mapsto \widetilde s\) extend to a homomorphism
\(
\sigma:G_5\to \widetilde G_5
\)
with \(\pi\circ\sigma=\id_{G_5}\).

For each \(g\in G_5\), let \(\widehat g:\mathcal I_5\to\mathcal I_5\) be the bundle map induced
by the lift \(\sigma(g)\) via Lemma~\ref{lem:12.4}.
Then \(\widehat{gh}=\widehat g\circ\widehat h\) follows from
\(\sigma(gh)=\sigma(g)\sigma(h)\).
\end{proof}

\begin{remark}[The two lifts of $Ad_q$ in the 5-cell case]
Let $\sigma:G_5\to \widetilde G_5$ denote the splitting constructed in the proof of
Proposition~\ref{prop:5cell-splitting}. Then $\sigma(Ad_q)$ is a lift of $Ad_q$ of order $3$. The lift fixed
in \S\ref{sssec:pm-labels} to describe the lifted facets is the other lift, namely
\(
\widetilde{Ad}_q=\tau_5\circ \sigma(Ad_q),
\) and hence \(\widetilde{Ad}_q^{\,3}=\tau_5.\)
Thus Proposition~\ref{prop:5cell-splitting} uses the coherent order--$3$ lift, whereas \S\ref{sssec:pm-labels} exhibits the order--$6$ lift.
\end{remark}
\section{Order--$3$ Isometries on Regular $8$--cell}
\label{sec:8cell}
Let \(P_8 := [-1,1]^4 \subset \R^4\) denote the regular \(8\)-cell, let \(\Gamma_8 \subset S^3\) be the radial projection of its \(1\)-skeleton, and set \(M_8 := S^3 \setminus \Gamma_8\). Let \(G_8 := \Sym^{+}(P_8)\) denote the orientation-preserving symmetry group of \(P_8\). The goal of this section is to prove the \(8\)-cell cases of Propositions~\ref{prop:odd-stab} and~\ref{prop:lift}. Using the explicit model of \(P_8\), Proposition~\ref{prop:13.1} shows that for each edge \(E^{(8)}\subset P_8\) there are two nontrivial order--\(3\) elements of \(G_8\) fixing \(E^{(8)}\) pointwise and acting on \(W_{E^{(8)}}^\perp\) by rotation through \(\pm 2\pi/3\); this gives the \(8\)-cell case of Proposition~\ref{prop:odd-stab}. Proposition~\ref{prop:even} shows that every vertex of \(\Gamma_8\) has valency \(4\), so the assignment \(\mu_8([m_e])=-1\) on edge meridians is well defined. Proposition~\ref{prop:13.2} proves that this extension splits and therefore $G_8$ lifts via a group homomorphism.

\subsection{The $8$--cell model.}\label{subsec:8cell}

\begin{figure}[t]
  \centering
  \includegraphics[width=0.8\textwidth]{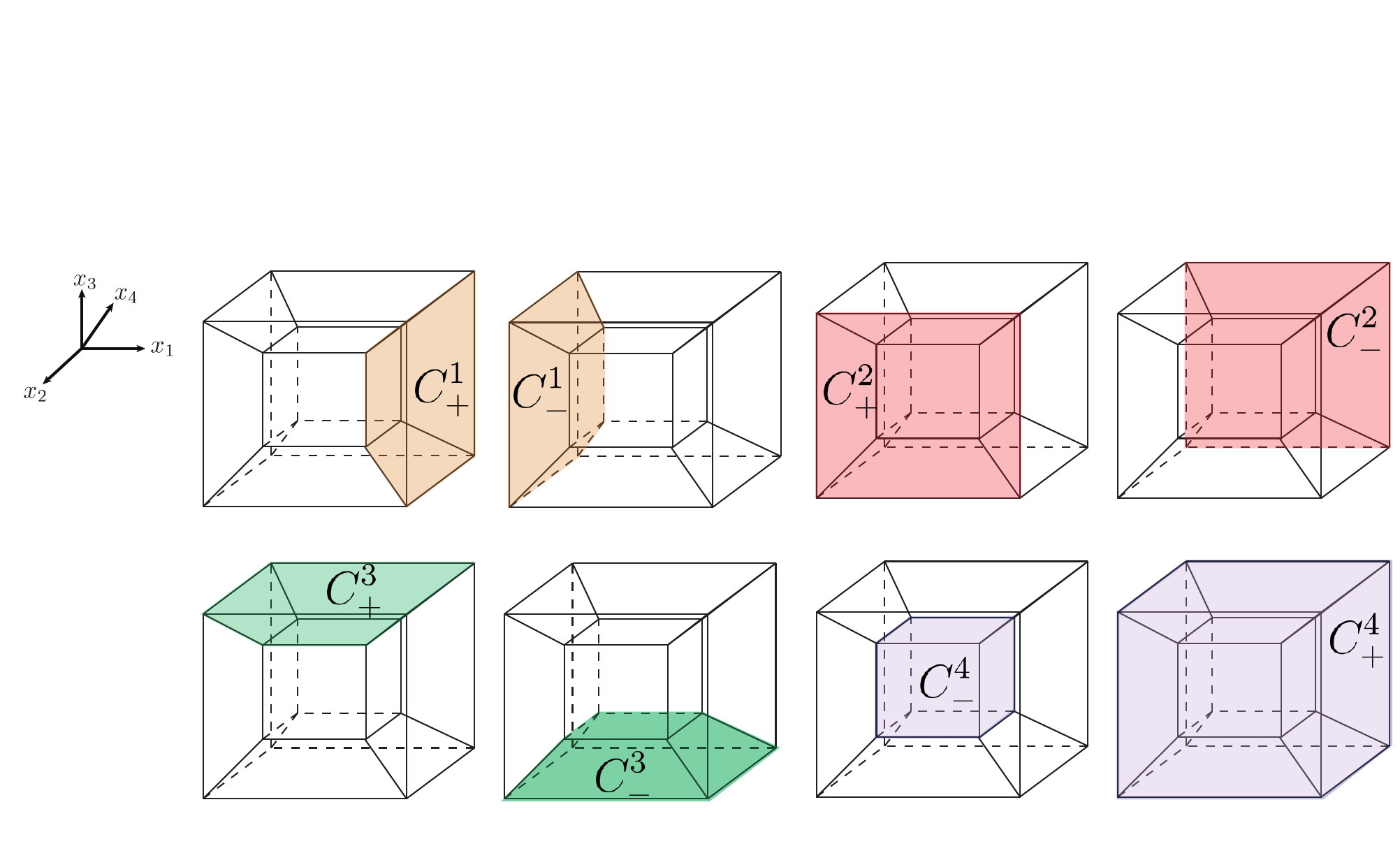}
  \caption{The eight Euclidean cubic facets of the \(8\)-cell \(P_8=[-1,1]^4\):
  \(C^1_{+}=(+1,r,s,t)\), \(C^1_{-}=(-1,r,s,t)\),
  \(C^2_{+}=(r,+1,s,t)\), \(C^2_{-}=(r,-1,s,t)\),
  \(C^3_{+}=(r,s,+1,t)\), \(C^3_{-}=(r,s,-1,t)\),
  \(C^4_{+}=(r,s,t,+1)\), \(C^4_{-}=(r,s,t,-1)\),
  with \(r,s,t\in[-1,1]\).}
  \label{fig:8cell-facets}
\end{figure}

The regular \(8\)-cell is the \(4\)-dimensional cube; see e.g.\ \cite{coxeter1973regular}. It has Schl\"afli symbol \(\{4,3,3\}\), so its facets are Euclidean cubes and its vertex figures are regular tetrahedra. 
\(P_8\) has \(16\) vertices, \(32\) edges, and \(8\) cubic \(3\)-cells. Its edge length is \(2\), and every vertex has distance \(2\) from the origin.
The vertices of \(P_{8}\) are the \(16\) sign vectors \((\pm1,\pm1,\pm1,\pm1)\); we enumerate them lexicographically as \(V^{(8)}_{1},\dots,V^{(8)}_{16}\), so that
\[
V^{(8)}_{1}=(+1,+1,+1,+1),
V^{(8)}_{8}=(+1,-1,-1,-1),
V^{(8)}_{9}=(-1,+1,+1,+1),
V^{(8)}_{16}=(-1,-1,-1,-1).
\]
For \(i\in\{1,2,3,4\}\) and \(\varepsilon\in\{\pm1\}\), let
\[
C^{i}_{\varepsilon}:=\{(x_{1},x_{2},x_{3},x_{4})\in \partial P_{8}\;:\;x_{i}=\varepsilon\}
\]
denote the cubic facet obtained by fixing the \(i\)-th coordinate. Two distinct facets \(C^{i}_{\varepsilon}\) and \(C^{j}_{\delta}\) meet along a square face if and only if \(i\neq j\); the opposite pair \(C_{+}^{i}\) and \(C_{-}^{i}\) are disjoint. Figure~\ref{fig:8cell-facets} records this labeling of the eight Euclidean cubic facets.

In coordinates, \(\Ad_{q}\) acts by
\begin{equation}\label{eq:8cell-adq}
\Ad_{q}(x_{1},x_{2},x_{3},x_{4})=(x_{1},x_{4},x_{2},x_{3}).
\end{equation}
Since \(\mathrm{Ad}_q\) fixes the first coordinate and cyclically permutes
the last three, the coordinate formula (14.1) shows that \(\mathrm{Ad}_q\)
has order \(3\) and preserves \(P_8\), hence permutes the vertices, edges,
and facets of \(P_8\). By quaternion conjugation, \(\Ad_q\) fixes the \(2\)-plane 
\(
F=\Span_{\R}\{1,\ i+j+k\}
\)
pointwise and rotates \(F^\perp\) by angle \(2\pi/3\). Hence \(\Ad_q\) fixes pointwise the two edges
$
E_{+}^{(8)}=\left[V^{(8)}_1,V^{(8)}_9\right]$ and $E_{-}^{(8)}=\left[V^{(8)}_8,V^{(8)}_{16}\right],
$
preserves the two facets \(C^1_\pm\), and cyclically permutes
\(
(C^2_+ , C^3_+ , C^4_+ )(
C^2_- , C^3_- ,C^4_- ).
\)
This facet motion is illustrated schematically in Figure~\ref{fig:8cell-adq}: the two facets \(C^1_\pm\) are setwise invariant, while \(C^2_\pm,C^3_\pm,C^4_\pm\) form two \(3\)-cycles.

\begin{figure}[t]
  \centering
  \includegraphics[width=\textwidth]{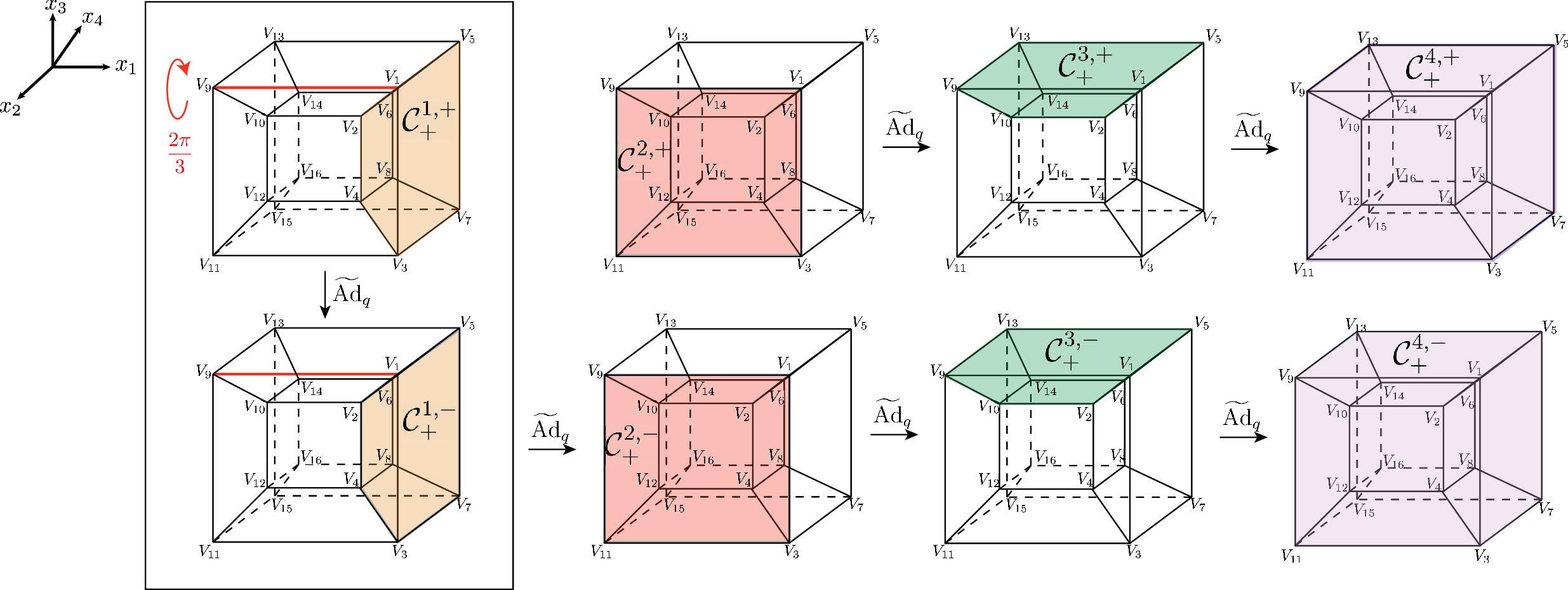}
  \caption{Lifted facet dynamics for the order--3 lift
\(\widetilde{\operatorname{Ad}}_q\). We write \(V_i:=V_i^{(8)}\).
 The figure displays the two cycles over
the base facets \(C^2_+,C^3_+,\) and \(C^4_+\). The two cycles are:
\(
C^{2,+}_{+}\to C^{3,+}_{+}\to C^{4,+}_{+}\) (first row), and \(
C^{2,-}_{+}\to C^{3,-}_{+}\to C^{4,-}_{+}\) (second row).
The analogous cycles over \(C^2_-,C^3_-,C^4_-\) are not shown. The lifted
facets over \(C^1_+\) and \(C^1_-\) are fixed by this order--3 lift.}
\label{fig:8cell-adq}
\end{figure}

\subsubsection{Lifted facet decomposition and the double cover.}\label{sssec:8cell-pm-labels}
Let
\(
\mu_8:\pi_1(M_8)\to\{\pm1\}
\)
denote the monodromy character determined by $\mu_8([m_e])=-1$ on edge meridians, let
\(
p_8:\widetilde M_8\to M_8
\)
be the associated connected double cover, and let $\tau_8$ denote its nontrivial deck involution.
For \(i\in\{1,2,3,4\}\) and \(\varepsilon\in\{+,-\}\), set
\[
\mathcal C_{\varepsilon}^i:=\pi(C^i_{\varepsilon})\subset S^3,
\qquad
\mathring{\mathcal C}^i_{\varepsilon}:=\mathcal C_{\varepsilon}^i\setminus \Gamma_8|_{\mathcal C_{\varepsilon}^i}\subset M_8.
\]
The eight sets \(\{\mathring{\mathcal C}^i_{\varepsilon}\}_{i,\varepsilon}\) are the facets in $M_8$ determined by the cubical cell decomposition of $S^3$ obtained from the boundary complex of $P_8$ by radial projection.
For \(i\neq j\) and \(\varepsilon,\delta\in\{+,-\}\), let
\[
F_{ij}^{\varepsilon,\delta}:=C^{i}_{\varepsilon}\cap C^{j}_{\delta},
\qquad
\mathcal{F}_{ij}^{\varepsilon,\delta}:=\pi(F_{ij}^{\varepsilon,\delta})\subset S^3,
\qquad
\mathring{F}_{ij}^{\varepsilon,\delta}:=\mathcal{F}_{ij}^{\varepsilon,\delta}\setminus \Gamma_8|_{\mathcal{F}_{ij}^{\varepsilon,\delta}}.
\]
Then \(\mathcal{F}_{ij}^{\varepsilon,\delta}\) is the radial projection of a square face of \(\partial P_{8}\).

The preimage \(p_8^{-1}(\mathring{\mathcal C}_{i,\varepsilon})\) is the disjoint union of two copies of
\(\mathring{\mathcal C}_{i,\varepsilon}\); after fixing one choice of sheet we write
\[
p_8^{-1}(\mathring{\mathcal C}^i_{\varepsilon})=\mathcal C_{\varepsilon}^{i,+}\sqcup \mathcal C_{\varepsilon}^{i,-},
\qquad
\tau_8(\mathcal C_{\varepsilon}^{i,\pm})=\mathcal C_{\varepsilon}^{i,\mp}.
\]
Likewise,
\[
p_8^{-1}(\mathring{\mathcal{F}}_{ij}^{\varepsilon,\delta})
=
\mathcal{F}_{ij}^{\varepsilon,\delta,+}\sqcup \mathcal{F}_{ij}^{\varepsilon,\delta,-},
\qquad
\tau_8(\mathcal{F}_{ij}^{\varepsilon,\delta,\pm})=\mathcal{F}_{ij}^{\varepsilon,\delta,\mp}.
\]

By Lemma~\ref{lem:12.3} and Lemma~\ref{lem:12.2}, the symmetry \(\Ad_q\in G_8 \subset SO(4)\) admits exactly two lifts to \(\widetilde M_8\), differing by composition with $\tau_8$. Since \(\Ad_q\) has order \(3\), one lift has order \(3\) and the other has order \(6\). 
Let \(\widetilde{\operatorname{Ad}}_q\) denote the order--3 lift. Then
\(\tau_8\circ\widetilde{\operatorname{Ad}}_q\) is the order--6 lift, characterized by
\(
(\tau_8\circ\widetilde{\operatorname{Ad}}_q)^3=\tau_8 .
\)
After choosing the sheet labels accordingly, its action on the lifted facets is
\[
\tau_8 \widetilde{\Ad}_q:\quad
(\mathcal C_{+}^{1,+}\ \mathcal C_{+}^{1,-})\, (\mathcal C_{-}^{1,+}\ \mathcal C_{-}^{1,-})\,
(\mathcal C_{+}^{2,+}\ \mathcal C_{+}^{3,-}\ \mathcal C_{+}^{4,+}\ \mathcal C_{+}^{2,-}\ \mathcal C_{+}^{3,+}\ \mathcal C_{+}^{4,-})\,
(\mathcal C_{-}^{2,+}\ \mathcal C_{-}^{3,-}\ \mathcal C_{-}^{4,+}\ \mathcal C_{-}^{2,-}\ \mathcal C_{-}^{3,+}\ \mathcal C_{-}^{4,-})\,
\]
while the action of $\widetilde{\Ad}_q$ 
\[
\widetilde{\Ad}_q:(\mathcal C_{+}^{2,+}\ \mathcal C_{+}^{3,+}\ \mathcal C_{+}^{4,+}) (\mathcal C_{+}^{2,-}\ \mathcal C_{+}^{3,-}\ \mathcal C_{+}^{4,-})(\mathcal C_{-}^{2,+}\ \mathcal C_{-}^{3,+}\ \mathcal C_{-}^{4,+}) (\mathcal C_{-}^{2,-}\ \mathcal C_{-}^{3,-}\ \mathcal C_{-}^{4,-})
\]
is illustrated in Fig.~\ref{fig:8cell-adq}.

\begin{proposition}[Order--3 symmetries of the $8$-cell]\label{prop:13.1}
Let \(P_8 \subset \mathbb{R}^4\) be the regular \(8\)-cell with vertices
\(
\left\{V^{(8)}_1,\dots,V^{(8)}_{16}\right\}\subset \mathbb{R}^4,
\)
and let \(\Gamma_8 \subset S^3\) be the radial projection of the \(1\)-skeleton of \(P_8\). Set
\(
G_8 := \operatorname{Sym}^{+}(P_8)\subset SO(4).
\)
Let
\(
E^{(8)}_0 := \left[V^{(8)}_1,V^{(8)}_9\right].
\)
For any edge \(E^{(8)}=\left[V_i^{(8)},V_j^{(8)}\right]\) of \(P_8\), there exists \(h\in G_8\) with \(h\left(E^{(8)}_0\right)=E^{(8)}\), and the two nontrivial order--\(3\) elements of \(G_8\) fixing \(E^{(8)}\) pointwise are
\(
h\operatorname{Ad}_qh^{-1}
\quad\text{and}\quad
\left(h\operatorname{Ad}_qh^{-1}\right)^{-1}.
\)
Each restricts on
\(
W_{E^{(8)}}^\perp
\), where \(W_{E^{(8)}} := \operatorname{span}_{\mathbb{R}}\left\{V_i^{(8)},V_j^{(8)}\right\},
\)
to a rotation by \(\pm 2\pi/3\). In particular, $G_8$ contains exactly 32 elements of order 3.
\end{proposition}
\begin{proof}The coordinate formula
\(
\operatorname{Ad}_q(x_1,x_2,x_3,x_4)=(x_1,x_4,x_2,x_3)
\)
shows that \(\operatorname{Ad}_q\in G_8\) has order \(3\) and preserves \(P_8=[-1,1]^4\). By quaternion conjugation, \(\operatorname{Ad}_q\) fixes the \(2\)-plane
\[
F=\operatorname{span}_{\mathbb{R}}\{1,i+j+k\}
\]
pointwise and rotates \(F^\perp\) by \(2\pi/3\). Since \(V^{(8)}_1\) and \(V^{(8)}_9\) lie in \(F\), the edge
\(
E^{(8)}_0=\bigl[V^{(8)}_1,V^{(8)}_9\bigr]
\)
is fixed pointwise by \(\operatorname{Ad}_q\). Because \(V^{(8)}_1\) and \(V^{(8)}_9\) are linearly independent and \(F\) is a \(2\)-plane, one has
\[
F=\operatorname{span}_{\mathbb{R}}\{V^{(8)}_1,V^{(8)}_9\}=W_{E^{(8)}_0}.
\]
Therefore \(\operatorname{Ad}_q\) restricting to \(W_{E^{(8)}_0}^\perp\) is a rotation by \(2\pi/3\).

Since \(P_8\) is regular, \(G_8\) acts transitively on its edges, so for any edge \(E^{(8)}\) there exists \(h\in G_8\) with \(h\left(E^{(8)}_0\right)=E^{(8)}\). Then \(h\operatorname{Ad}_q h^{-1}\) fixes \(E^{(8)}\) pointwise and restricts on \(W_{E^{(8)}}^\perp\) to a rotation by \(\pm2\pi/3\).

To see that there are no further nontrivial order--\(3\) elements fixing $E^{(8)}$ pointwise, note that the \(8\)-cell has Schl\"afli symbol \(\{4,3,3\}\), so the edge figure is a triangle. Hence, the orientation-preserving action fixing \(E\) pointwise is cyclic of order \(3\). 

Now we show the rotation group $G_8$ of the $8$-cell has exactly $32$ elements of order $3$. View $G_8$ as the orientation-preserving permutation group of $P_8$, every $g\in G_8$ acts on the standard basis $\{e_1,e_2,e_3,e_4\}$ by
\(
g(e_i)=\pm e_{\sigma(i)}
\)
for some permutation $\sigma\in S_4$.
If $g$ has order $3$, then $\sigma$ must have order dividing $3$. Since $\sigma\neq \mathrm{id}$, it follows
that $\sigma$ must be a $3$-cycle. There are exactly $8$ such $3$-cycles in $S_4$. To determine when $g$ has order $3$, compute $g^3$ on the basis vectors, i.e.,
\(
g^3(e_i)=\varepsilon^3_i e^3_{\sigma(i)}.
\)
So $\varepsilon_a=1$ for some $a\in\{1,2,3,4\}$ and $\varepsilon_b\varepsilon_c\varepsilon_d=1$ for some $\{b,c,d\}\in\{1,2,3,4\}\setminus\{a\}$. For $g^3$ to be the identity, the product must be $+1$, and there are exactly four triples $(\varepsilon_b,\varepsilon_c,\varepsilon_d)\in\{\pm1\}^3$ with product $1$:
\[
(+,+,+),\qquad (+,-,-),\qquad (-,+,-),\qquad (-,-,+).
\]
So each of the $8$ underlying $3$-cycles has exactly $4$ lifts of order $3$, for a total of 32 actions. Finally, if $g^3=\mathrm{id}$, then
\(
\det(g)^3=1.
\)
Since $\det(g)=\pm1$, this forces $\det(g)=1$. Thus every element of order $3$ is orientation-preserving, so all such elements already lie in $G_8$.
\end{proof}

\subsection{The lift group \(\widetilde G_{8}\) and the associated $\mathbb{Z}/2$--extension.}

From Lemma~\ref{lem:12.2} and Lemma~\ref{lem:12.3}, one knows that \(\mu_8\) is \(G_8\)--invariant and every \(g\in G_8\) admits a lift. Define the \emph{group of lifts}
\[
\widetilde G_8
:=
\{\widetilde g\in \operatorname{Diff}(\widetilde M_8)\mid \exists g\in G_8 \text{ such that }
p_8\circ\widetilde g=g\circ p_8\}.
\]
Then Lemma~\ref{lem:12.5} gives a short exact sequence
\[
1\longrightarrow\langle\tau_8\rangle\longrightarrow \widetilde G_8
\xrightarrow{\ \pi\ } G_8\longrightarrow 1.
\]
Lemma~\ref{lem:12.4} shows that for each \(g\in G_8\), a choice of lift \(\widetilde g:\widetilde M_8\to\widetilde M_8\) induces a fiberwise linear bundle automorphism \(\widehat g:\mathcal I_8\to \mathcal I_8\), and replacing \(\widetilde g\) by \(\tau_8\circ\widetilde g\) replaces \(\widehat g\) by \(-\widehat g\). Thus a splitting of the extension above is equivalent to a coherent choice of bundle maps \(\widehat g\) satisfying \(\widehat{gh}=\widehat g\circ\widehat h\) for all \(g,h\in G_8\).

\begin{proposition}[Splitting for the $8$--cell group]\label{prop:13.2}
The \(G_8\)--action on \(M_8\) lifts to a coherent choice of fiberwise linear bundle automorphisms
\(
\widehat g:\mathcal I_8\to \mathcal I_8
\)
covering \(g\), such that \(\widehat{gh}=\widehat g\circ\widehat h\) for all \(g,h\in G_8\).
\end{proposition}
\begin{proof}
Let $H_8:=\ker(\mu_8)$. Since $\mu_8\circ g_*=\mu_8$ for every $g\in G_8$, the connected double cover $p_8$ corresponding to $H_8$ is preserved by $G_8$. Hence every $g\in G_8$ admits a lift to $\widetilde M_8$, unique up to composition with the deck involution $\tau_8$. By Lemma~\ref{lem:12.5}, there is a short exact sequence
\[
1\longrightarrow \langle \tau_8\rangle \longrightarrow \widetilde G_8
\xrightarrow{\ \pi\ } G_8 \longrightarrow 1.
\]
It suffices to construct a splitting
\(
\sigma:G_8\to \widetilde G_8.
\)
Choose generators $\sigma_1,\sigma_2,\sigma_3\in G_8$ by
\[
\sigma_1(x_1,x_2,x_3,x_4)=(-x_2,x_1,x_3,x_4),
\]
\[
\sigma_2(x_1,x_2,x_3,x_4)=(x_3,x_1,x_2,x_4),
\]
\[
\sigma_3(x_1,x_2,x_3,x_4)=(x_1,x_4,x_2,x_3)=\operatorname{Ad}_q(x_1,x_2,x_3,x_4).
\]
These are orientation--preserving symmetries of $P_8=[-1,1]^4$. According to Eq.~(\ref{eq:8}), the group representation is:
\[
G_8 \cong [4,3,3]^+
=
\langle \sigma_1,\sigma_2,\sigma_3 \mid \sigma_1^4=\sigma_2^3=\sigma_3^3=(\sigma_1\sigma_2)^2=(\sigma_2\sigma_3)^2=(\sigma_1\sigma_2\sigma_3)^2=1\rangle.
\]
We now choose lifts of $\sigma_1,\sigma_2,\sigma_3$ with the same orders upstairs.

\smallskip
\noindent
{\it Step 1: choose a lift of $\sigma_1$ with $ \widetilde{\sigma}_1^{\,4}=1$.}
The fixed-point set of $\sigma_1$ on $S^3$ is
\[
\operatorname{Fix}(\sigma_1|_{S^3})=\{x_1=x_2=0\}\cap S^3,
\]
a great circle. This circle cannot be contained in \(\Gamma_8\); otherwise
\(\sigma_1\) would fix the radial image of some edge pointwise, whereas the
orientation-preserving pointwise stabilizer of an edge is cyclic of order
\(3\) by the triangular edge figure. Hence
\(\operatorname{Fix}(\sigma_1)\cap M_8\ne\emptyset\). Choose
\(
x_{\sigma_1}\in \operatorname{Fix}(\sigma_1)\cap M_8\), \(\widetilde x_{\sigma_1}\in p_8^{-1}(x_{\sigma_1}),
\)
and let $\widetilde{\sigma}_1$ be the unique lift of $\sigma_1$ such that
\(
\widetilde{\sigma}_1(\widetilde x_{\sigma_1})=\widetilde x_{\sigma_1}.
\)
Then $\widetilde{\sigma}_1^{\,4}$ is a deck transformation, since it covers $\sigma_1^4=\mathrm{id}$. Because $\widetilde{\sigma}_1^{\,4}$ fixes $\widetilde x_{\sigma_1}$ and the nontrivial deck involution $\tau_8$ acts freely, it follows that
\(
\widetilde{\sigma}_1^{\,4}=\mathrm{id}.
\)

\smallskip

\noindent \it Step 2: choose lifts of \(\sigma_2\) and \(\sigma_3\) with \(\widetilde{\sigma}_2^{\,3}=\widetilde{\sigma}_3^{\,3}=\id\). \normalfont
By Lemma~\ref{lem:12.2}, \(\sigma_2\) and \(\sigma_3\) admit lifts \(\widetilde{\sigma}_2,\widetilde{\sigma}_3\) to \(\widetilde M_8\). Since \(\sigma_2^3=\sigma_3^3=\id\), the elements \(\widetilde{\sigma}_2^{\,3}\) and \(\widetilde{\sigma}_3^{\,3}\) are deck transformations, so
\(
\widetilde{\sigma}_2^{\,3},\widetilde{\sigma}_3^{\,3}\in\{\id,\tau_8\}.
\)
If \(\widetilde{\sigma}_2^{\,3}=\tau_8\), replace \(\widetilde{\sigma}_2\) by \(\tau_8\widetilde{\sigma}_2\); similarly for \(\widetilde{\sigma}_3\). Since every lift commutes with \(\tau_8\),
\[
(\tau_8\widetilde{\sigma}_2)^3=\tau_8^{\,3}\widetilde{\sigma}_2^{\,3}=\tau_8\cdot\tau_8=\id,
\qquad
(\tau_8\widetilde{\sigma}_3)^3=\tau_8^{\,3}\widetilde{\sigma}_3^{\,3}=\tau_8\cdot\tau_8=\id.
\]
Thus we may choose lifts \(\widetilde{\sigma}_2,\widetilde{\sigma}_3\) such that
\(
\widetilde{\sigma}_2^{\,3}=\widetilde{\sigma}_3^{\,3}=\id.
\)

\smallskip
\noindent
{\it Step 3: verify the involution relations upstairs.}
A direct computation gives
\[
\sigma_1\sigma_2(x_1,x_2,x_3,x_4)=(-x_1,x_3,x_2,x_4),
\]
\[
\sigma_2\sigma_3(x_1,x_2,x_3,x_4)=(x_2,x_1,x_4,x_3),
\]
\[
\sigma_1\sigma_2\sigma_3(x_1,x_2,x_3,x_4)=(-x_1,x_2,x_4,x_3).
\]
Hence
\[
\operatorname{Fix}((\sigma_1\sigma_2)|_{S^3})=\{x_1=0,\ x_2=x_3\}\cap S^3,
\]
\[
\operatorname{Fix}((\sigma_2\sigma_3)|_{S^3})=\{x_1=x_2,\ x_3=x_4\}\cap S^3,
\]
\[
\operatorname{Fix}((\sigma_1\sigma_2\sigma_3)|_{S^3})=\{x_1=0,\ x_3=x_4\}\cap S^3.
\]
Each is a great circle. None is contained in \(\Gamma_8\); otherwise the
corresponding involution would fix the radial image of some edge pointwise,
impossible because the orientation-preserving pointwise stabilizer of an
edge is cyclic of order \(3\). Hence each meets \(M_8\).
Now set
\(
\widetilde g_{12}:=\widetilde{\sigma}_1\,\widetilde{\sigma}_2\),
\(\widetilde g_{23}:=\widetilde{\sigma}_2\,\widetilde{\sigma}_3\),
\(\widetilde g_{123}:=\widetilde{\sigma}_1\,\widetilde{\sigma}_2\,\widetilde{\sigma}_3.
\)
These are lifts of $\sigma_1\sigma_2$, $\sigma_2\sigma_3$, and $\sigma_1\sigma_2\sigma_3$, respectively. Since
\[
(\sigma_1\sigma_2)^2=(\sigma_2\sigma_3)^2=(\sigma_1\sigma_2\sigma_3)^2=\mathrm{id}
\]
in $G_8$, each of
$\widetilde g_{12}^{\,2}$,
$\widetilde g_{23}^{\,2}$, and
$\widetilde g_{123}^{\,2}$
is a deck transformation.

We claim that each of these deck transformations is the identity. Let $g\in\{\sigma_1\sigma_2,\sigma_2\sigma_3,\sigma_1\sigma_2\sigma_3\}$ and let $\widetilde g$ denote the corresponding lift above. Choose $x\in \operatorname{Fix}(g)\cap M_8$ and $\widetilde x\in p_8^{-1}(x)$. Since $g(x)=x$, the point $\widetilde g(\widetilde x)$ lies in the two-point fiber $p_8^{-1}(x)$, so
\(
\widetilde g(\widetilde x)\in\{\widetilde x,\tau_8(\widetilde x)\}.
\)
In either case,
\(
\widetilde g^{\,2}(\widetilde x)=\widetilde x.
\)
But $\widetilde g^{\,2}$ is a deck transformation, and the nontrivial deck involution $\tau_8$
has no fixed points. Therefore
\(
\widetilde g^{\,2}=\mathrm{id}.
\)
Applying this to $g=\sigma_1\sigma_2,\sigma_2\sigma_3,\sigma_1\sigma_2\sigma_3$ gives
\(
(\widetilde{\sigma}_1\,\widetilde{\sigma}_2)^2
=
(\widetilde{\sigma}_2\,\widetilde{\sigma}_3)^2
=
(\widetilde{\sigma}_1\,\widetilde{\sigma}_2\,\widetilde{\sigma}_3)^2
=
\mathrm{id}.
\)

Thus, the chosen lifts satisfy all defining relations of the presentation of $G_8$. Therefore
the assignment
\(
\sigma_1\mapsto \widetilde{\sigma}_1,
\sigma_2\mapsto \widetilde{\sigma}_2,
\sigma_3\mapsto \widetilde{\sigma}_3
\)
extends to a homomorphism
\(
s:G_8\to \widetilde G_8
\)
such that
\(
\pi\circ s=\mathrm{id}_{G_8}.
\)
Hence the extension splits.

Finally, for each \(g \in G_8\), let \(\widehat g : \mathcal{I}_8 \to \mathcal{I}_8\) be the fiberwise linear bundle automorphism induced by the lift \(s(g)\) via Lemma~\ref{lem:12.4}. Since \(s\) is a homomorphism, these bundle maps satisfy \(\widehat{gh} = \widehat g \circ \widehat h\) for all \(g,h \in G_8\). This proves the proposition.
\end{proof}

\begin{remark}[Order--$3$ vs.\ order--$6$ lifts for $\Ad_q$]\label{rm:13.5}
The splitting in Proposition~\ref{prop:13.2} produces a coherent lift \(\sigma(\Ad_q)\in\widetilde G_8\)
of order \(3\). The order--\(6\) lift \(\tau_8 \circ \widetilde{\Ad}_q\) fixed above for the facet dynamics is the
other lift, which satisfies
\(
(\tau_8\circ\sigma(\Ad_q))^{3}=\tau_8.
\)
\end{remark}

\section{Order--$3$ Isometries on Regular $24$--cell}\label{sec:24cell}

The goal of this section is to verify the \(24\)-cell case of Propositions~\ref{prop:odd-stab} and \ref{prop:lift}. The \(24\)-cell-specific geometric input is Proposition~\ref{prop:14.1} below, which shows that for each edge of \(P_{24}\) there is an order-\(3\) element of \(G_{24}\) fixing that edge pointwise and acting on the normal \(2\)-plane by rotation through angle \(\pm 2\pi/3\); this gives the \(24\)-cell case of Proposition~\ref{prop:odd-stab}. The even-valency statement needed to define the monodromy character \(\mu_{24}\) is already supplied by Proposition~\ref{prop:even}. Proposition~\ref{prop:14.2} then shows that the associated \(\mathbb Z/2\)-extension splits, and hence that the \(G_{24}\)-action lifts coherently to the double cover and to the associated line bundle.

\subsection{The 24-cell Model}

A $24$--cell \(P_{24}\) is the convex hull of the $24$ unit Hurwitz quaternions in
\(\mathbb H\simeq \mathbb R^4\),
which has \(24\) vertices, \(96\) edges, and \(24\) octahedral \(3\)-cells. Its edge length is \(1\), and since \(V_{24}\subset S^3\), every vertex has distance \(1\) from the origin.
Let
\[
\mathcal V_{24}
:=
\{\pm 1,\pm i,\pm j,\pm k\}
\;\cup\;
\left\{
\frac{\pm 1\pm i\pm j\pm k}{2}
\right\}
\subset S^3\subset \mathbb H,
\]
We enumerate them lexicographically as \(V^{(24)}_{1},\dots,V^{(24)}_{24}\), where
the \(16\) half-Hurwitz vertices are ordered by their sign vectors
\((\varepsilon_0,\varepsilon_1,\varepsilon_2,\varepsilon_3)\in\{\pm1\}^4\)
with \(+\) preceding \(-\). Thus
\[
V^{(24)}_{1}=1, \quad 
V^{(24)}_{8}=-k, \quad 
V^{(24)}_{9}={1\over 2}(1+i+j+k), \quad
V^{(24)}_{24}={1\over 2}(-1-i-j-k).
\]
Note that
\(
q:=\frac12(1+i+j+k)=V^{(24)}_9,\)
Recall
\[
\Ad_q(w+xi+yj+zk)=w+zi+xj+yk,
\]
Denote \(V^{(24)}_i\) by \(V_i\) for brevity,
we obtain
\[
\Ad_q
=
(V_3\,V_5\,V_7)(V_4\,V_6\,V_8)
(V_{10}\,V_{13}\,V_{11})
(V_{12}\,V_{14}\,V_{15})
(V_{18}\,V_{21}\,V_{19})
(V_{20}\,V_{22}\,V_{23}),
\]
and it fixes
\(
V_1, V_2, V_9, V_{16}, V_{17}, V_{24}.
\)
Two vertices $V^{(24)}_i$ and $V^{(24)}_j$ are joined by an edge if and only if $\left|V^{(24)}_i - V^{(24)}_j\right| = 1$. Define
\(
P_{24}:=\operatorname{Hull}(\mathcal V_{24})\subset \mathbb R^4.
\)
This is the regular $24$--cell. It has Schl\"afli symbol \(\{3,4,3\}\); in particular, its facets are regular octahedra, its $2$--faces are equilateral triangles, its vertex figures are cubes, and its edge figure is a triangle. Set \(G_{24}:=\Sym^{+}(P_{24})\subset SO(4).\)

To assign labels to Fig.~\ref{fig:24cell}, we first fix the reference edge
\(
E_0^{(24)}=[1,q]=[V_1,V_9].
\)
Let \(N(v)\) denote the set of vertices adjacent to \(v\). The distance formulas give
\begin{align*}
& N(V_1)=\{V_9,V_{10},V_{11},V_{12},V_{13},V_{14},V_{15},V_{16}\},\\
& N(V_9)=\{V_1,V_3,V_5,V_7,V_{10},V_{11},V_{13},V_{17}\},
\end{align*}
and therefore
\(
N(V_1)\cap N(V_9)=\{V_{10},V_{11},V_{13}\}.
\)

For the pair \(\{V_{11},V_{13}\}\), one finds
\[
N(V_1)\cap N(V_{11})\cap N(V_{13}) = \{V_9,V_{15}\},
\qquad
N(V_9)\cap N(V_{11})\cap N(V_{13})=\{V_1,V_7\}.
\]
Thus the first octahedral facet through \(E_0^{(24)}\) is
\(
\mathcal O_1
=
\{V_1,V_7,V_9,V_{11},V_{13},V_{15}\}.
\)

For the pair \(\{V_{10},V_{13}\}\), one finds
\[
N(V_1)\cap N(V_{10})\cap N(V_{13}) = \{V_9,V_{14}\},
\qquad
N(V_9)\cap N(V_{10})\cap N(V_{13}) = \{V_1, V_5\}.
\]
Thus the second octahedral facet through \(E_0^{(24)}\) is
\(
\mathcal O_{2}
=
\{V_1,V_5,V_9,V_{10},V_{13},V_{14}\}.
\)

For the pair \(\{V_{10},V_{11}\}\), one finds
\[
N(V_1)\cap N(V_{10})\cap N(V_{11}) = \{V_9,V_{12}\},
\qquad
N(V_9)\cap N(V_{10})\cap N(V_{11}) = \{V_1,V_3\}.
\]
Thus the third octahedral facet through \(E_0^{(24)}\) is
\(
\mathcal O_{3}
=
\{V_1,V_3,V_9,V_{10},V_{11},V_{12}\}.
\)

\begin{proposition}[Order--$3$ symmetries of the $24$--cell]
\label{prop:14.1}
Let \(P_{24}\subset \mathbb R^4\) be a regular \(24\)--cell, let \(\Gamma_{24}\subset S^3\) be the radial projection of the
\(1\)--skeleton of \(P_{24}\), and set
\(
G_{24}:=\Sym^{+}(P_{24})\subset SO(4).
\)
Let
\(
E_0^{(24)}:=[1,q].
\)
For any edge \(E^{(24)}=\left[V^{(24)}_i,V^{(24)}_j\right]\) of \(P_{24}\), there exists \(h\in G_{24}\) with
\(h\left(E_0^{(24)}\right)=E^{(24)}\), and the two nontrivial order--\(3\) elements of \(G_{24}\) fixing \(E^{(24)}\)
pointwise are \(h\Ad_q h^{-1}\) and its inverse. Each restricts on
\(
W_{E^{(24)}}^\perp,
W_{E^{(24)}}:=\operatorname{span}_{\mathbb R}\left\{V^{(24)}_i,V^{(24)}_j\right\},
\)
to a rotation by \(\pm 2\pi/3\). 
In particular, \(G_{24}\) contains exactly \(32\) elements of order \(3\) fixing an edge pointwise.
\end{proposition}
\begin{proof}
We first show that \(\Ad_q\in G_{24}\). As
\(
\Ad_q(\mathcal{V}_{24})=\mathcal{V}_{24}
\)
and \(\Ad_q \in SO(4)\),
\(
\Ad_q\in \Sym^+(P_{24})=G_{24}.
\)
Also, 
\(
\Fix(\Ad_q)=\operatorname{span}_{\mathbb R}\{1,i+j+k\}
\)
and \(\Ad_q\) rotates \(\Fix^\perp(\Ad_q)\) by angle \(2\pi/3\). So
\(\Ad_q\) fixes the reference edge
\(E_0^{(24)}\) pointwise. 
Therefore \(\Ad_q\) restricting on \({W_{E_0^{(24)}}^\perp}\) is a rotation by \(2\pi/3\).
Since \(P_{24}\) is regular, \(G_{24}\) acts transitively on its edges. Hence for any edge \(E\) there exists \(h\in G_{24}\) with \(h(E_0^{(24)})=E\). Then \(h\Ad_q h^{-1}\) fixes \(E\) pointwise and restricts on \(W_E^\perp\) to a rotation by \(2\pi/3\), while its inverse restricts to a rotation by \(-2\pi/3\), as illustrated in Fig.~\ref{fig:24cell}.

The \(24\)--cell has \(96\) edges, and
\(
\Fix(\Ad_q)\cap \mathcal V_{24}
=
\{\pm1,\ \pm q,\ \pm q^2\}.
\)
These six vertices lie on the unit circle in \(\Fix(\Ad_q)\) in the cyclic order
\(
1,\ q,\ q^2,\ -1,\ -q,\ -q^2,
\)
hence form a regular hexagon. By the adjacency criterion, two vertices $u$ and $v$ are connected by an edge if and only if \(\Re(u\bar v)=\frac12\). So the edges of \(P_{24}\) contained in \(\Fix(\Ad_q)\) are exactly the six sides of this hexagon. 
If two such hexagons share an edge
\(E=[V_i^{(24)},V_j^{(24)}]\), then their fixed planes both contain
\(W_E=\operatorname{span}_{\mathbb R}\{V_i^{(24)},V_j^{(24)}\}\). Since both
fixed planes are two-dimensional, the fixed planes coincide, and hence the
hexagons coincide.
It follows that the \(96\) edges of \(P_{24}\) are partitioned into disjoint hexagons of size \(6\), and each corresponds to exactly two nontrivial order--\(3\) elements fixing an edge pointwise, for a total of $32$.
\end{proof}

\begin{figure}[t]
  \centering
  \includegraphics[width=0.8\textwidth]{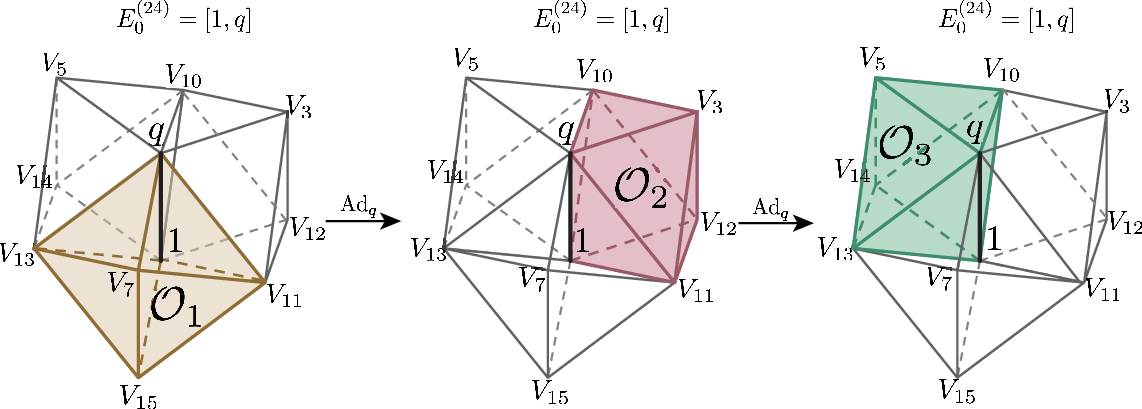}
  \caption{The three octahedral facets \(\mathcal{O}_1,\mathcal{O}_2,\mathcal{O}_3\) of the 24-cell
\(P_{24}\) meeting along the reference edge
\(E^{(24)}_0=[1,q]\), together with the induced facet dynamics under
\(\operatorname{Ad}_q\).}
  \label{fig:24cell}
\end{figure}

\subsection{The lift group \(\widetilde G_{24}\) and the associated $\mathbb{Z}/2$--extension.}
Let
\(
M_{24}:=S^3\setminus \Gamma_{24},
\)
and let
\(
\mu_{24}:\pi_1(M_{24})\to\{\pm1\}
\)
be the monodromy character determined by
\(
\mu_{24}([m_e])=-1
\)
for each edge meridian \(m_e\), which is well-defined by
Proposition~\ref{prop:even}. Let
\(
p_{24}:\widetilde M_{24}\to M_{24}
\)
be the associated connected double cover corresponding to
\(
H_{24}:=\ker(\mu_{24}),
\)
let \(\tau_{24}\) denote its nontrivial deck involution, and let
\(
\mathcal I_{24}\to M_{24}
\)
be the associated real line bundle.
Define
\[
\widetilde G_{24}
:=
\left\{
\widetilde g:\widetilde M_{24}\to \widetilde M_{24}
\;\middle|\;
\exists\, g\in G_{24}\text{ such that }p_{24}\circ \widetilde g=g\circ p_{24}
\right\}.
\]

\begin{proposition}[Splitting for the $24$--cell group]
\label{prop:14.2}
The \(G_{24}\)-action on \(M_{24}\) lifts to a coherent choice of fiberwise linear bundle
automorphisms
\(
\widehat g:\mathcal{I}_{24}\to \mathcal{I}_{24}
\)
covering \(g\in G_{24}\), such that
\[
\widehat{gh}=\widehat g\circ \widehat h
\qquad\text{for all } g,h\in G_{24}.
\]
\end{proposition}

\begin{proof}
By Proposition~\ref{prop:even}, each vertex of \(\Gamma_{24}\) has valency \(8\), so the assignment \(\mu_{24}([m_e])=-1\) on edge meridians is well-defined. Since every \(g\in G_{24}\) preserves \(\Gamma_{24}\) and permutes its
edges, Lemma~\ref{lem:12.3} gives \(\mu_{24}\circ g_*=\mu_{24}\).
Hence \(g_*(H_{24})=H_{24}\), so every \(g\in G_{24}\) admits a lift to \(\widetilde M_{24}\), unique up to composition with the deck involution \(\tau_{24}\). It follows that
\[
1\longrightarrow \langle \tau_{24}\rangle \longrightarrow \widetilde G_{24}
\xrightarrow{\ \pi\ } G_{24}\longrightarrow 1
\]
is a short exact sequence, so it suffices to construct a splitting
\(
\sigma:G_{24}\to \widetilde G_{24}.
\)
\noindent \it Step 1: choose lifts of \(\sigma_1\) and \(\sigma_3\) with \(\widetilde{\sigma}_1^{\,3}=\widetilde{\sigma}_3^{\,3}=\id\). \normalfont
By Lemma~\ref{lem:12.2}, \(\sigma_1\) and \(\sigma_3\) admit lifts \(\widetilde{\sigma}_1,\widetilde{\sigma}_3\) to \(\widetilde M_{24}\).
Since \(\sigma_1^3=\sigma_3^3=\id\), the elements \(\widetilde{\sigma}_1^{\,3}\) and \(\widetilde{\sigma}_3^{\,3}\) are deck
transformations,
\(
\widetilde{\sigma}_1^{\,3},\widetilde{\sigma}_3^{\,3}\in\{\id,\tau_{24}\}.
\)
If \(\widetilde{\sigma}_1^{\,3}=\tau_{24}\), replace \(\widetilde{\sigma}_1\) by \(\tau_{24}\widetilde{\sigma}_1\); similarly for
\(\widetilde{\sigma}_3\). Since every lift commutes with \(\tau_{24}\),
\[
(\tau_{24}\widetilde{\sigma}_1)^3=\tau_{24}^{\,3}\widetilde{\sigma}_1^{\,3}=\tau_{24}\cdot\tau_{24}=\id,
\qquad
(\tau_{24}\widetilde{\sigma}_3)^3=\tau_{24}^{\,3}\widetilde{\sigma}_3^{\,3}=\tau_{24}\cdot\tau_{24}=\id.
\]
Thus we may choose lifts \(\widetilde{\sigma}_1,\widetilde{\sigma}_3\) such that
\(
\widetilde{\sigma}_1^{\,3}=\widetilde{\sigma}_3^{\,3}=\id.
\)

\smallskip
\noindent\emph{Step 2: choose a lift of \(\sigma_2\) of order \(4\).}
Following \cite[Eq.~(4), p.~1311]{johnson1999quadratic}, take the following standard
generators of \([3,4,3]^+\): 
\[
\begin{aligned}
\sigma_1(x_1,x_2,x_3,x_4) =& (x_1,x_4,x_2,x_3),\\
\sigma_2(x_1,x_2,x_3,x_4) =& (x_1,x_2,-x_4,x_3),\\
\sigma_3(x_1,x_2,x_3,x_4) =&
\frac12\bigl(
x_1+x_2+x_3+x_4,\;
x_1+x_2-x_3-x_4,\;\\
&x_1-x_2+x_3-x_4,\;
-x_1+x_2+x_3-x_4
\bigr).
\end{aligned}
\]
Thus, \(
C:=\Fix(\sigma_2|_{S^3})=\{x_3=x_4=0\}\cap S^3
\)
is a great circle.
We claim that \(C\not\subset \Gamma_{24}\). Indeed, if \(C\subset \Gamma_{24}\), then \(C\) contains the radial projection of some edge \(E=[u,v]\) of \(P_{24}\). Hence \(\sigma_2\) fixes the edge \(E\) pointwise. 
Since the 24-cell has Schl\"afli symbol \(\{3,4,3\}\), its edge figure is a triangle. Therefore the orientation-preserving pointwise stabilizer of an edge acts faithfully on that triangle, so it is a subgroup of $C_3$ (rotation group of order 3). Hence any nontrivial element fixing an edge pointwise has order 3, whereas \(\sigma_2\) has order \(4\). Therefore, \(C\cap M_{24}\neq\varnothing\). Choose \(x_{\sigma_2}\in C\cap M_{24}\) and \(\widetilde x_{\sigma_2}\in p_{24}^{-1}(x_{\sigma_2})\), and let \(\widetilde{\sigma}_2\) be the unique lift of \(\sigma_2\) such that
\(
\widetilde{\sigma}_2(\widetilde x_{\sigma_2})=\widetilde x_{\sigma_2}.
\)
Then \(\widetilde{\sigma}_2^{\,4}\) is a deck transformation. Because \(\widetilde{\sigma}_2^{\,4}\) fixes \(\widetilde x_{\sigma_2}\), and the nontrivial deck involution acts freely, it follows that
\(
\widetilde{\sigma}_2^{\,4}=\operatorname{id}.
\)

\smallskip
\noindent\emph{Step 3: verify the involution relations upstairs.}
Set
\[
\widetilde g_{12}:=\widetilde{\sigma}_1\widetilde{\sigma}_2,\qquad
\widetilde g_{23}:=\widetilde{\sigma}_2\widetilde{\sigma}_3,\qquad
\widetilde g_{123}:=\widetilde{\sigma}_1\widetilde{\sigma}_2\widetilde{\sigma}_3.
\]
These are lifts of \(\sigma_1\sigma_2\), \(\sigma_2\sigma_3\), and \(\sigma_1\sigma_2\sigma_3\), respectively. Since
\(
(\sigma_1\sigma_2)^2=(\sigma_2\sigma_3)^2=(\sigma_1\sigma_2\sigma_3)^2=\operatorname{id}
\)
in \(G_{24}\), each of
\(
\widetilde g_{12}^{\,2},
\widetilde g_{23}^{\,2},
\widetilde g_{123}^{\,2}
\)
is a deck transformation. Let \(g\in\{\sigma_1\sigma_2,\sigma_2\sigma_3,\sigma_1\sigma_2\sigma_3\}\), and let \(\widetilde g\) denote the corresponding lift above. We first show that \(\operatorname{Fix}(g)\cap M_{24}\neq\varnothing\). Since \(g^2=\operatorname{id}\), it is enough to rule out \(g=\pm\operatorname{id}\). If \(\sigma_1\sigma_2=\pm\operatorname{id}\), then \(\sigma_1=\pm \sigma_2^{-1}\), contradicting \(\sigma_1^3=\operatorname{id}\) and \(\sigma_2^4=\operatorname{id}\); similarly, \(\sigma_2\sigma_3\neq \pm\operatorname{id}\). If \(\sigma_1\sigma_2\sigma_3=\pm\operatorname{id}\), then \(\sigma_1\sigma_2=\pm \sigma_3^{-1}\), which contradicts \((\sigma_1\sigma_2)^2=\operatorname{id}\) and \(\sigma_3^3=\operatorname{id}\). Thus \(g\neq \pm\operatorname{id}\), so by the same eigenvalue argument as in Proposition~\ref{prop:5cell-splitting},
\(
C_g:=\operatorname{Fix}(g|_{S^3})
\)
is a great circle. If \(C_g\subset \Gamma_{24}\), then \(C_g\) contains the radial projection of some edge \(E=[u,v]\) of \(P_{24}\). 
But no element of order \(2\) fixes an edge pointwise, because the
edge figure of the \(24\)-cell is a triangle, so the orientation-preserving
pointwise stabilizer of an edge is cyclic of order \(3\).
This contradiction shows that \(C_g\cap M_{24}\neq\varnothing\). Choose \(x\in C_g\cap M_{24}\) and \(\widetilde x\in p_{24}^{-1}(x)\), then \(\widetilde g(\widetilde x)\) lies in the two-point fiber
\(
p_{24}^{-1}(x)=\{\widetilde x,\tau_{24}(\widetilde x)\}.
\)
In either case,
\(
\widetilde g^{\,2}(\widetilde x)=\widetilde x.
\)
Since \(\widetilde g^{\,2}\) is a deck transformation and the nontrivial deck involution \(\tau_{24}\) acts freely, it follows that
\(
\widetilde g^{\,2}=\operatorname{id}.
\)
Thus
\[
(\widetilde{\sigma}_1\widetilde{\sigma}_2)^2
=
(\widetilde{\sigma}_2\widetilde{\sigma}_3)^2
=
(\widetilde{\sigma}_1\widetilde{\sigma}_2\widetilde{\sigma}_3)^2
=
\operatorname{id}.
\]
Hence, the chosen lifts satisfy all defining relations of the presentation; therefore, the assignments
\(
\sigma_1\mapsto \widetilde{\sigma}_1,
\sigma_2\mapsto \widetilde{\sigma}_2,
\sigma_3\mapsto \widetilde{\sigma}_3
\)
extend to a homomorphism
\(
\sigma:G_{24}\longrightarrow \widetilde G_{24}
\)
with \(\pi\circ \sigma=\mathrm{id}_{G_{24}}\). Hence the extension splits.
For each \(g\in G_{24}\), let \(\widehat g:\mathcal{I}_{24}\to \mathcal{I}_{24}\) be the fiberwise linear bundle
automorphism induced by the lift \(\sigma(g)\) via Lemma~\ref{lem:12.5}. Since \(\sigma\) is a
homomorphism, these satisfy \(\widehat{gh}=\widehat g\circ \widehat h\) for all \(g,h\in G_{24}.\) This proves the proposition.
\end{proof}

\section{Order--$3$ Isometries on Regular $120$--cell}\label{sec:120cell}

The goal of this section is to verify the \(120\)--cell case of Propositions~\ref{prop:lift} and~\ref{prop:odd-stab}. Using the model in \S\ref{subsec:120cell-model}, Proposition~\ref{prop:15.1} supplies the odd--order edge stabilizers required for Proposition~\ref{prop:odd-stab}. The even--valency statement is already provided by Proposition~\ref{prop:even}. Proposition~\ref{prop:15.2} then proves that the associated \(\mathbb Z/2\)--extension splits.

\subsection{The $120$--cell model.}\label{subsec:120cell-model}

A $120$--cell is the regular convex $4$--polytope with Schl\"afli symbol
\(
\{5,3,3\}.
\)
Its facets are regular dodecahedra, its $2$--faces are regular pentagons, its vertex figures are regular tetrahedra, and its edge figure is a triangle.
\(P_{120}\) has \(600\) vertices, \(1200\) edges, and \(120\) dodecahedral \(3\)-cells, denoted $\mathcal{D}_i$ as depicted in Fig.~\ref{fig:120cell}. Its edge length is \(3-\sqrt{5}\), and every listed vertex has norm \(2\sqrt{2}\), so every vertex has distance \(2\sqrt{2}\) from the origin.
Let
\(
\phi:=\frac{1+\sqrt5}{2}.
\)
We use the standard coordinate model of the regular $120$--cell,
whose $600$ vertices are the union of the following seven families:
\begin{align*}
\text{Family 1: } & 2(\pm1,\pm1,0,0)\ \text{and all permutations},\\
\text{Family 2: } &(\pm\sqrt5,\pm1,\pm1,\pm1)\ \text{and all permutations},\\
\text{Family 3: } &(\pm\phi^{-2},\pm\phi,\pm\phi,\pm\phi)\ \text{and all permutations},\\
\text{Family 4: } &(\pm\phi^2,\pm\phi^{-1},\pm\phi^{-1},\pm\phi^{-1})\ \text{and all permutations},\\
\text{Family 5: } &(\pm\phi^2,\pm\phi^{-2},\pm1,0)\ \text{and all even permutations},\\
\text{Family 6: } &(\pm\sqrt5,\pm\phi^{-1},\pm\phi,0)\ \text{and all even permutations},\\
\text{Family 7: } &(\pm2,\pm1,\pm\phi,\pm\phi^{-1})\ \text{and all even permutations}.
\end{align*}
Recall that \(\Ad_q(x):=q\,x\,q^{-1}\) with \(q:=\frac12(1+i+j+k)\in S^3,\) and \(\Ad_q(x_1,x_2,x_3,x_4)=(x_1,x_4,x_2,x_3).\) Thus \(\Ad_q\) fixes the first coordinate and cyclically permutes the last three. Therefore, \(\Ad_q\) preserves the seven-family vertex set above: Families \(1\)--\(4\) are closed under all permutations, and Families \(5\)--\(7\) are closed under even permutations. Thus
\(
\Ad_q\in G_{120}.
\)

Let
\(
V_{120}:=\{V^{(120)}_1,\dots,V^{(120)}_{600}\}\subset \mathbb R^4\simeq \mathbb H
\)
denote this set, and define
\(
P_{120}:=\operatorname{Hull}(V_{120})\subset \mathbb R^4.
\)
Let \(\Gamma_{120}\subset S^3\) be the radial projection of the \(1\)--skeleton of \(P_{120}\),
set
\(
M_{120}:=S^3\setminus \Gamma_{120},
\)
and write
\(
G_{120}:=\Sym^{+}(P_{120})\subset SO(4).
\)

\begin{lemma}\label{lem:120}
Let \(e=[u,v]\) be an edge of the regular \(120\)-cell. Then the plane
\(
P_e:=\operatorname{span}\{u,v\}
\)
contains exactly \(12\) vertices of the \(120\)-cell and exactly \(6\) edges of the \(120\)-cell.
\end{lemma}

\begin{proof}
Since the \(120\)-cell is regular, its symmetry group is transitive on edges. Hence it is enough to prove the claim for one specific edge.
Take
\(
u=(2,2,0,0),
v=(\sqrt5,\phi,0,\phi^{-1}).
\)
A direct computation gives
\(
\|u-v\|^2
=(3-\sqrt5)^2,
\)
so \([u,v]\) is an edge.
Thus
\[
\operatorname{span}\{u,v\}=\{(x_1,x_2,x_3,x_4)\in\mathbb R^4 : x_3=0,\ x_1=x_2+x_4\}.
\]
Family \(1\) contributes exactly the six vertices
\[
A_0=(2,2,0,0),\quad
A_1=(2,0,0,2),\quad
A_2=(0,-2,0,2),
\]
\[
A_3=(-2,-2,0,0),\quad
A_4=(-2,0,0,-2),\quad
A_5=(0,2,0,-2).
\]
Family \(6\) contributes exactly the six vertices
\[
B_0=(\sqrt5,\phi,0,\phi^{-1}),\quad
B_1=(\phi,-\phi^{-1},0,\sqrt5),\quad
B_2=(-\phi^{-1},-\sqrt5,0,\phi),
\]
\[
B_3=(-\sqrt5,-\phi,0,-\phi^{-1}),\quad
B_4=(-\phi,\phi^{-1},0,-\sqrt5),\quad
B_5=(\phi^{-1},\sqrt5,0,-\phi).
\]
Therefore,
\(P_e\) contains exactly \(12\) vertices of the \(120\)-cell.
A direct computation shows that the distance \(3-\sqrt5\) occurs exactly for the six pairs:
\(
A_0B_0,\ A_1B_1,\ A_2B_2,\ A_3B_3,\ A_4B_4,\ A_5B_5.
\)
\end{proof}

\begin{proposition}[Order--$3$ symmetries of the $120$--cell]\label{prop:15.1}
Let \(P_{120}\subset \mathbb R^{4}\) be a regular $120$--cell. Let \(\Gamma_{120}\subset S^{3}\) be the radial projection of its $1$--skeleton, and set \(G_{120}:=\Sym^{+}(P_{120})\subset SO(4)\). Let \(E^{(120)}_{0}:=\left[V^{(120)}_1,V^{(120)}_2\right]\) denote the edge incident to \(V^{(120)}_1\) that is fixed by \(\Ad_q\). For any edge \(E^{(120)}_{ij}:=\left[V^{(120)}_i,V^{(120)}_j\right]\) of \(P_{120}\), there exists \(h\in G_{120}\) with \(h(E^{(120)}_{0})=E^{(120)}_{ij}\), and the two nontrivial order--$3$ elements of \(G_{120}\) fixing \(E^{(120)}_{ij}\) pointwise are \(h\Ad_{q}h^{-1}\) and its inverse. Each restricts on
\(
W_{E^{(120)}_{ij}}^{\perp},
W_{E^{(120)}_{ij}}:=\operatorname{span}_{\mathbb R}
\left\{V^{(120)}_i,V^{(120)}_j\right\},
\)
to a rotation by \(\pm 2\pi/3\). The number of elements of order \(3\) in \(G_{120}\) that fix some edge pointwise is \(400\).
\end{proposition}

\begin{figure}[t]
  \centering
  \includegraphics[width=0.95\textwidth]{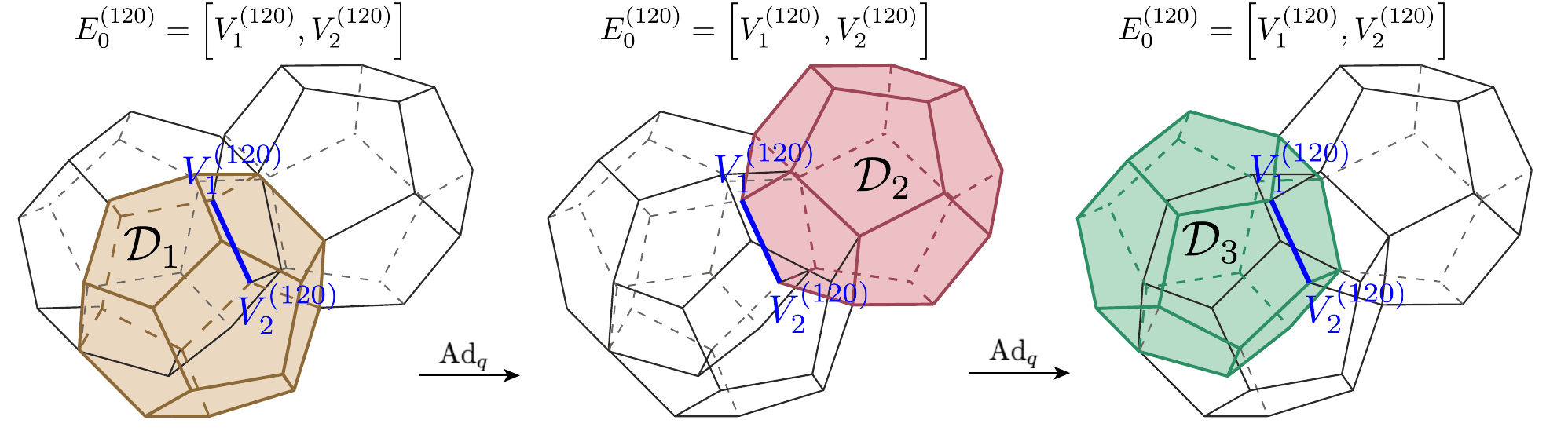}
  \caption{The three dodecahedral facets \(\mathcal{D}_1,\mathcal{D}_2,\mathcal{D}_3\) of the 120-cell
\(P_{120}\) meeting along the reference edge
\(E^{(120)}_0=[V^{(120)}_1,V^{(120)}_2]\), together with the induced facet dynamics under \(\operatorname{Ad}_q\).}
  \label{fig:120cell}
\end{figure}

\begin{proof}
The map \(\Ad_q\in G_{120}\) has order \(3\), fixes the \(2\)--plane
\(
F=\mathrm{span}_{\mathbb R}\{1,i+j+k\}
\)
pointwise, and rotates \(F^\perp\) by angle \(2\pi/3\).
Set
\[
V^{(120)}_1:=(\sqrt5,1,1,1),\qquad
V^{(120)}_2:=(\phi^2,\phi^{-1},\phi^{-1},\phi^{-1}),
\]
and let
\(
E^{(120)}_0:=[V^{(120)}_1,V^{(120)}_2].
\)
Both \(V^{(120)}_1\) and \(V^{(120)}_2\) belong to the vertex set of the standard model, and both are fixed by \(\Ad_q\), since their last three coordinates are equal. Moreover,
\(
\|V^{(120)}_1-V^{(120)}_2\|^2
=(3-\sqrt5)^2.
\)
Hence \(E^{(120)}_0\) is an edge of \(P_{120}\). Since
\(
V^{(120)}_1,V^{(120)}_2\in F
\)
are linearly independent, we have
\(
W_{E^{(120)}_0}:=\mathrm{span}_{\mathbb R}\{V^{(120)}_1,V^{(120)}_2\}=F.
\)
Therefore \(\Ad_q\) fixes \(E^{(120)}_0\) pointwise and restricts on
\(
W_{E^{(120)}_0}^{\perp}
\)
to a rotation by angle \(2\pi/3\), as illustrated in Fig.~\ref{fig:120cell}.

Since the \(120\)--cell is regular, \(G_{120}\) acts transitively on its edges. Hence for any edge
\[
E^{(120)}_{ij}:=[V^{(120)}_i,V^{(120)}_j]
\]
of \(P_{120}\), there exists \(h\in G_{120}\) with
\(
h(E^{(120)}_0)=E^{(120)}_{ij}.
\)
Then \(h\Ad_q h^{-1}\) fixes \(E^{(120)}_{ij}\) pointwise and restricts on
\(
W_{E^{(120)}_{ij}}^{\perp},
\) with
\(W_{E^{(120)}_{ij}}:=\mathrm{span}_{\mathbb R}\{V^{(120)}_i,V^{(120)}_j\},
\)
to a rotation by \(2\pi/3\), while its inverse restricts to a rotation by \(-2\pi/3\).

The \(120\)-cell has \(1200\) edges \cite{coxeter1973regular}. By Lemma~\ref{lem:120}, each edge lies in a unique plane through the origin, and each such plane contains exactly \(6\) edges. Therefore the number of such planes is \(200\). There are exactly two nontrivial rotations of order \(3\) fixing the plane \(P_e\) pointwise and rotating \(P_e^\perp\) by \(\pm 2\pi/3\).
Conversely, if an element of order \(3\) preserves an edge \(e=[u,v]\), then it fixes the plane \(\operatorname{span}\{u,v\}\) pointwise, and therefore is one of the two order-\(3\) rotations associated with that plane. Thus the number of order-\(3\) elements fixing some edge pointwise is \(400\).
\end{proof}

\subsection{The lift group \(\widetilde G_{120}\) and the associated $\mathbb{Z}/2$--extension.}\label{subsec:120cell-lifts}

By Proposition~\ref{prop:even}, each vertex of \(\Gamma_{120}\) has valency $4$, so the assignment
\(
\mu_{120}([m_{e}])=-1
\)
on edge meridians is well-defined. Let
\(
p_{120}:\widetilde M_{120}\to M_{120}
\)
be the associated connected double cover corresponding to \(H_{120}:=\ker(\mu_{120})\), let \(\tau_{120}\) denote its nontrivial deck involution, and let \(\mathcal{I}_{120}\to M_{120}\) be the associated real line bundle. Define
\[
\widetilde G_{120}:=\left\{\widetilde g:\widetilde M_{120}\to \widetilde M_{120}\ \middle|\ \exists g\in G_{120}\text{ such that }p_{120}\circ \widetilde g=g\circ p_{120}\right\}.
\]

\begin{proposition}[Splitting for the $120$--cell group]\label{prop:15.2}
The \(G_{120}\)-action on \(M_{120}\) lifts to a coherent choice of fiberwise linear bundle automorphisms
\(
\widehat g : \mathcal{I}_{120}\to \mathcal{I}_{120}
\)
covering \(g\in G_{120}\), such that
\(
\widehat{gh}=\widehat g\circ \widehat h
\text{ for all } g,h\in G_{120}.
\)
\end{proposition}

\begin{proof}
Since every \(g\in G_{120}\) preserves \(\Gamma_{120}\) and permutes its edges, Lemma~\ref{lem:12.3} gives
\(
\mu_{120}\circ g_*=\mu_{120}.
\)
Hence, by Lemma~\ref{lem:12.2}, every \(g\in G_{120}\) admits a lift to \(\widetilde M_{120}\), unique up to composition with \(\tau_{120}\). By Lemma~\ref{lem:12.5}, one therefore has a short exact sequence
\[
1 \longrightarrow \langle \tau_{120}\rangle
\longrightarrow \widetilde G_{120}
\xrightarrow{\ \pi\ }
G_{120}
\longrightarrow 1.
\]
It suffices to construct a splitting
\(
s:G_{120}\to \widetilde G_{120}.
\)

Choose generators \(\sigma_1,\sigma_2,\sigma_3\in G_{120}\) as in Equation~(\ref{eq:120}) with presentation
\[
G_{120}\cong [5,3,3]^+
=
\left\langle \sigma_1,\sigma_2,\sigma_3 \,\middle|\,
\sigma_1^{5}=\sigma_2^{3}=\sigma_3^{3}
=(\sigma_1\sigma_2)^{2}
=(\sigma_2\sigma_3)^{2}
=(\sigma_1\sigma_2\sigma_3)^{2}
=1
\right\rangle.
\]
Set
\(
g_{12}:=\sigma_1\sigma_2,
g_{23}:=\sigma_2\sigma_3,\) and \(
g_{123}:=\sigma_1\sigma_2\sigma_3.
\)
We show below that each of \(g_{12},g_{23},g_{123}\) has fixed-point set meeting
\(M_{120}\), which is enough to verify the involution relations upstairs.

\noindent \it Step 1: choose lifts of \(\sigma_1,\sigma_2,\sigma_3\) with the same odd orders upstairs. \normalfont
By Lemma~\ref{lem:12.2}, each \(\sigma_i\) admits a lift to \(\widetilde M_{120}\). Choose arbitrary lifts
\(
\widetilde\sigma_1,\widetilde\sigma_2,\widetilde\sigma_3.
\)
Since
\(
\sigma_1^{5}=\sigma_2^{3}=\sigma_3^{3}=1,
\)
the elements
\(
\widetilde\sigma_1^{5},
\widetilde\sigma_2^{3},
\widetilde\sigma_3^{3}
\in\{\id,\tau_{120}\}\). If
\(\widetilde\sigma_1^{5}=\tau_{120}\), replace \(\widetilde\sigma_1\) by
\(\tau_{120}\widetilde\sigma_1\). Since every lift commutes with \(\tau_{120}\),
\(
(\tau_{120}\widetilde\sigma_1)^5
=
\tau_{120}^{5}\widetilde\sigma_1^{5}
=
\tau_{120}\cdot \tau_{120}
=
\id.
\)
The same odd-order adjustment applies to \(\widetilde\sigma_2\) and
\(\widetilde\sigma_3\). Thus we may assume
\(
\widetilde\sigma_1^{5}
=
\widetilde\sigma_2^{3}
=
\widetilde\sigma_3^{3}
=
\id.
\)

\smallskip

\noindent\emph{Step 2: verify the involution relations upstairs.}
Set
\[
\widetilde g_{12}:=\widetilde\sigma_1\widetilde\sigma_2,\qquad
\widetilde g_{23}:=\widetilde\sigma_2\widetilde\sigma_3,\qquad
\widetilde g_{123}:=\widetilde\sigma_1\widetilde\sigma_2\widetilde\sigma_3.
\]
These are lifts of
\[
g_{12}:=\sigma_1\sigma_2,\qquad
g_{23}:=\sigma_2\sigma_3,\qquad
g_{123}:=\sigma_1\sigma_2\sigma_3.
\]
Since \(g_{12}^2=g_{23}^2=g_{123}^2=1\) in \(G_{120}\), each of \(\widetilde g_{12}^{\,2},\widetilde g_{23}^{\,2},\widetilde g_{123}^{\,2}\) is a deck transformation.

We first claim that each of \(g_{12},g_{23},g_{123}\) is a nontrivial involution in \(SO(4)\). Indeed, \(g_{12}\neq \pm \mathrm{id}\), for otherwise \(\sigma_1=\pm \sigma_2^{-1}\), impossible since \(\sigma_1\) has order \(5\) while \(\pm \sigma_2^{-1}\) has order \(3\) or \(6\). Likewise, \(g_{23}\neq \pm \mathrm{id}\), for otherwise \(g_{123}=\pm \sigma_1\), contradicting \(g_{123}^2=1\) and \(|\sigma_1|=5\). Finally, \(g_{123}\neq \pm \mathrm{id}\), for otherwise \(g_{12}=\pm \sigma_3^{-1}\), contradicting \(g_{12}^2=1\) and \(|\sigma_3|=3\).

Hence, for each
\(
g\in\{g_{12},g_{23},g_{123}\},
\)
the fixed-point set \(\Fix(g|_{S^3})\) is a great circle. We claim that
\(
\Fix(g)\cap M_{120}\neq\varnothing.
\)
If not, then \(\Fix(g)\subset \Gamma_{120}\). Since \(\Fix(g)\) is a great circle contained in the radial graph, it contains the radial image of some edge of \(P_{120}\), so \(g\) fixes that edge pointwise. But the 120-cell has triangular edge figure, so the orientation-preserving
pointwise stabilizer of an edge is cyclic of order \(3\), as in the proof of
Proposition~11.11. In particular, no involution can fix an edge pointwise.

Therefore, we may choose
\(
x_{12}\in \Fix(g_{12})\cap M_{120},
x_{23}\in \Fix(g_{23})\cap M_{120},
x_{123}\in \Fix(g_{123})\cap M_{120},
\)
and lifts
\(
\widetilde x_{12}\in p_{120}^{-1}(x_{12}),
\widetilde x_{23}\in p_{120}^{-1}(x_{23}),
\widetilde x_{123}\in p_{120}^{-1}(x_{123}).
\)
Because each \(g\) fixes the corresponding point, the lift \(\widetilde g\) sends the chosen lift into the same two-point fiber, so \(\widetilde g^{\,2}\) fixes that chosen lift. Since \(\widetilde g^{\,2}\) is a deck transformation and the nontrivial deck involution acts freely, it follows that
\(
(\widetilde\sigma_1\widetilde\sigma_2)^2=
(\widetilde\sigma_2\widetilde\sigma_3)^2=
(\widetilde\sigma_1\widetilde\sigma_2\widetilde\sigma_3)^2=\mathrm{id}.
\)

Thus, the chosen lifts satisfy all defining relations of the presentation of \(G_{120}\).
Therefore, the assignment \(
\sigma_1\mapsto \widetilde\sigma_1,
\sigma_2\mapsto \widetilde\sigma_2,
\sigma_3\mapsto \widetilde\sigma_3
\)
extends to a homomorphism
\(
s:G_{120}\longrightarrow \widetilde G_{120}
\)
with \(\pi\circ s=\mathrm{id}_{G_{120}}\). Hence the extension splits.
For each \(g\in G_{120}\), let \(\widehat g:\mathcal I_{120}\to \mathcal I_{120}\) be the fiberwise linear bundle automorphism induced by the lift \(s(g)\) via Lemma~\ref{lem:12.4}.
Since \(s\) is a homomorphism, these satisfy \(\widehat{gh}=\widehat g\circ \widehat h\) for all \(g,h\in G_{120}\). This proves the proposition.

\end{proof}

\section{Order--$5$ Isometries on Regular $600$--cell}\label{sec:600cell}

The goal of this section is to verify the \(600\)--cell case of Propositions~\ref{prop:lift} and~\ref{prop:odd-stab}. Using the model in \S\ref{subsec:600cell-model}, Proposition~\ref{prop:16.1} supplies the odd--order edge stabilizers required for Proposition~\ref{prop:odd-stab}. Proposition~\ref{prop:16.2} then proves that the associated \(\mathbb Z/2\)--extension splits.

\subsection{The $600$--cell model.}\label{subsec:600cell-model}

A \(600\)--cell is the regular convex \(4\)--polytope with Schl\"afli symbol
\(
\{3,3,5\}.
\)
Its facets are regular tetrahedra, its \(2\)--faces are equilateral triangles, its vertex figures are regular icosahedra, and its edge figure is a pentagon.
\(P_{600}\) has \(120\) vertices, \(720\) edges, and \(600\) tetrahedral \(3\)-cells, denoted $\mathcal{F}_i$ as depicted in Fig.~\ref{fig:600cell}. Its edge length is \(\phi^{-1}\), and since \(V_{600}\subset S^3\), every vertex has distance \(1\) from the origin. Explicitly, the vertices are the union of the following three families:
\begin{align*}
&A_1 = \{(\pm1,0,0,0)\ \text{and all permutations}\},\\
&A_2 = \left\{\frac12(\pm1,\pm1,\pm1,\pm1)\right\},\\
&A_3 = \left\{\frac12(0,\pm1,\pm\phi,\pm\phi^{-1})\ \text{and all even permutations}\right\},
\end{align*}
where \(\phi := {1+\sqrt5\over 2}\). These three families contribute \(8\), \(16\), and \(96\) vertices, respectively.
Let
\(
\mathcal V_{600}:=\{V_1^{(600)},\dots,V_{120}^{(600)}\}\subset S^3
\)
denote this set, and define
\(
P_{600}:=\operatorname{Hull}(\mathcal V_{600})\subset \mathbb R^4.
\)
Let \(\Gamma_{600}\subset S^3\) be the radial projection of the \(1\)--skeleton of \(P_{600}\), set
\(
M_{600}:=S^3\setminus \Gamma_{600},
\)
and write
\(
G_{600}:=\Sym^{+}(P_{600})\subset SO(4).
\)

The set \(\mathcal V_{600}\) is the binary icosahedral group \(2I\subset S^3\), so it is closed under quaternion multiplication and inversion. Set
\[
p:=\frac12(\phi+\phi^{-1}i+k).
\]
Then \(p\in A_3 \subset\mathcal{V}_{600}\), because its coordinate vector
\(
(\phi,\phi^{-1},0,1)
\)
is the even permutation of \((0,1,\phi,\phi^{-1})\). Since \(p\in S^3\), its inverse is its quaternionic conjugate:
\[
p^{-1}=\bar p=\frac12(\phi-\phi^{-1}i-k).
\]
For \(x\in\mathbb H\), define
\(
\Ad_p(x):=pxp^{-1}.
\)
Direct computation gives
\[
\Ad_p(1)=1,
\ \ 
\Ad_p(i)=\frac12\bigl(i+\phi j+\phi^{-1}k\bigr),
\ \ 
\Ad_p(j)=\frac12\bigl(-\phi i+\phi^{-1}j+k\bigr),
\ \ 
\Ad_p(k)=\frac12\bigl(\phi^{-1}i-j+\phi k\bigr).
\]
By \(\mathbb R\)-linearity,
\(\Ad_p(\pm1)=\pm1\in A_1, \Ad_p(\pm i),\Ad_p(\pm j),\Ad_p(\pm k)\in A_3\)
because the coordinate vectors
\[
(0,1,\phi,\phi^{-1}),\qquad (0,-\phi,\phi^{-1},1),\qquad (0,\phi^{-1},-1,\phi)
\]
are all even permutations of \((0,\pm1,\pm\phi,\pm\phi^{-1})\).

Let \(x\in\mathcal{V}_{600}\). Since \(\mathcal{V}_{600}=2I\) is a group, and both \(p\) and \(p^{-1}\) lie in \(\mathcal{V}_{600}\), the product
\(
pxp^{-1}
\)
also lies in \(\mathcal{V}_{600}\). To show injectivity, suppose \(x,y\in\mathcal{V}_{600}\) and
\(
\Ad_p(x)=\Ad_p(y).
\)
Then
\(
pxp^{-1}=pyp^{-1}.
\)
Multiplying on the left by \(p^{-1}\) and on the right by \(p\), we get
\(
x=y.
\)
Hence \(\Ad_p\) is injective on \(\mathcal{V}_{600}\). Since \(\Ad_p(\mathcal{V}_{600})\subseteq\mathcal{V}_{600}\) and \(\mathcal{V}_{600}\) is finite, every injective self-map of \(\mathcal{V}_{600}\) is bijective (also \(\Ad_p\) has explicit inverse \(\Ad_{p^{-1}}\)).

\begin{figure}[t]
  \centering
  \includegraphics[width=\textwidth]{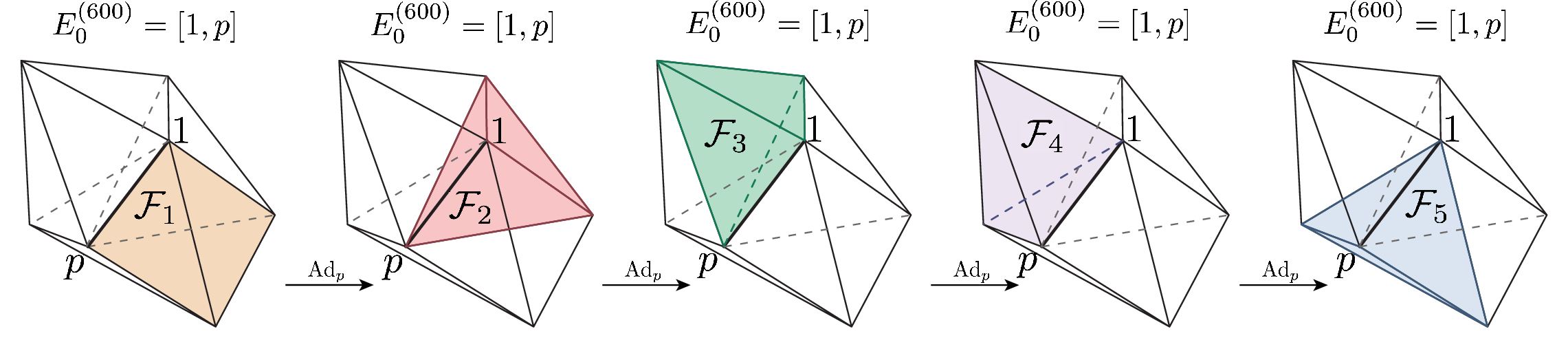}
  \caption{The five tetrahedral facets \(\mathcal{F}_1,\dots,\mathcal{F}_5\) of the 600-cell
\(P_{600}\) meeting along the reference edge \(E^{(600)}_0=[1,p]\),
together with the induced 5-cycle of facets under \(\operatorname{Ad}_p\).}
  \label{fig:600cell}
\end{figure}

\begin{lemma}\label{lem:600cell-decagon}
Let
\(
\phi:=\frac{1+\sqrt5}{2},
p:=\frac12(\phi+\phi^{-1}i+k)\in \mathcal V_{600}\subset S^3,
\)
and 
\(
F:=\operatorname{span}_{\mathbb R}\{1,\Im(p)\}\subset \mathbb H\simeq \mathbb R^4.
\)
Then
\(
F\cap \mathcal V_{600}
=
\{1,p,p^2,p^3,p^4,-1,-p,-p^2,-p^3,-p^4\}.
\)
These ten vertices lie on the unit circle \(F\cap S^3\) as a regular decagon. Moreover, the edges of \(P_{600}\) contained in \(F\) are exactly the ten sides of this decagon.
\end{lemma}

\begin{proof}
Let \(u=\operatorname{Im}(p)/|\operatorname{Im}(p)|\). Since
\(p \in \operatorname{span}_{\mathbb R}\{1,u\}\), \((a+bu)(c+du)=(ac-bd)+(ad+bc)u.\) So \(\operatorname{span}_{\mathbb R}\{1,u\}\) is closed under multiplication. Hence every power \(p^n\in \operatorname{span}_{\mathbb R}\{1,u\}.\) Because \(p^5=-1\), these are ten distinct points. We now check \(F\cap \mathcal V_{600}\) consists of exactly these ten vertices. A quaternion
\(
x=a+bi+cj+dk
\)
lies in \(F\) if and only if
\(
c=0, d=\phi b.
\)
Equivalently,
\(
F=\{\,a+bi+\phi b\,k : a,b\in\mathbb R\,\}.
\)
\begin{itemize}
    \item \(A_1:\) The condition \(c=0\) and \(d=\phi b\) restricts to
\(
A_1\cap F=\{\pm1\}.
\)
\item \(A_2\): Every such point has \(j\)-coordinate \(\pm\frac12\neq0\), so none lies in \(F\). Hence
\(
A_2\cap F=\emptyset.
\)
\item \(A_3\): 
Among even permutations of
\((0,1,\phi,\phi^{-1})\), the ones with \(0\) in the third slot are:
\[
(1,\phi,0,\phi^{-1}),\qquad
(\phi,\phi^{-1},0,1),\qquad
(\phi^{-1},1,0,\phi).
\]
The latter two satisfy \(d=\phi b\). Therefore the only \(A_3\) vertices in \(F\) are the 8 points:
\[
\frac12(\varepsilon\phi+\delta\phi^{-1}i+\delta k),
\qquad
\frac12(\varepsilon\phi^{-1}+\delta i+\delta\phi k),
\qquad
\varepsilon,\delta\in\{\pm1\}.
\]
\end{itemize}
Since \( p=\cos{\pi \over 5}+u\sin {\pi \over 5}\), the ten points
\(
\{1, p, p^2, p^3, p^4,-1,- p,- p^2,- p^3,- p^4\}
\)
are equally spaced on the circle \(F\cap S^3\), hence form a regular decagon.

Finally, the edge length of \(P_{600}\) is \(\phi^{-1}\), and
\(
\| p^{m+1}- p^m\|=\|1- p\|=\phi^{-1}.
\)
So consecutive vertices in the decagon are joined by edges of \(P_{600}\). 
For any integers \(m,r\),
\[
\|p^{m+r}-p^m\|
=\|p^m(p^r-1)\|
=\|p^m\|\,\|p^r-1\|
=\|p^r-1\|
=2\left|\sin\frac{r\pi}{10}\right|.
\]
If \(r\) is chosen modulo \(10\) with \(0\le |r|\le 5\), this simplifies to
\(
\|p^{m+r}-p^m\|=2\sin\!\left(\frac{|r|\pi}{10}\right)>\phi^{-1}
\)
if \(|r|\ge 2\).
So no non-consecutive pair is joined by an edge. Therefore the edges of \(P_{600}\) contained
in \(F\) are exactly the ten sides of the regular decagon.
\end{proof}

\begin{proposition}[Order--$5$ symmetries of the $600$--cell]
\label{prop:16.1}
Let \(P_{600}\subset \mathbb R^4\) be the regular \(600\)--cell, let
\(\Gamma_{600}\subset S^3\) be the radial projection of its \(1\)--skeleton, and set
\(
G_{600}:=\Sym^{+}(P_{600})\subset SO(4).
\)
Let
\(
E^{(600)}_0:=[1,p].
\)
For any edge
\(
E^{(600)}=\left[V_i^{(600)},V_j^{(600)}\right]
\)
of \(P_{600}\), there exists \(h\in G_{600}\) with
\(
h\!\left(E^{(600)}_0\right)=E^{(600)}.
\)
Moreover, the four nontrivial order--\(5\) elements of \(G_{600}\) fixing \(E^{(600)}\) pointwise are
\(
\left(h\,\Ad_p\,h^{-1}\right)^k, k=1,2,3,4.
\)
Each restricts to
\(
W_{E^{(600)}}^\perp,\) where \(W_{E^{(600)}}:=\operatorname{span}_{\mathbb R}\{V_i^{(600)},V_j^{(600)}\},\) to a rotation by angle \(2\pi k/5\). In particular, the orientation--preserving pointwise stabilizer of \(E^{(600)}\) is cyclic of order \(5\).
\end{proposition}
\begin{proof}
By the quaternionic description from \S\ref{subsec:quat}, \(\Ad_p\) fixes pointwise the \(2\)--plane
\[
F:=\operatorname{span}_{\mathbb R}\{1,u\}
=
\operatorname{span}_{\mathbb R}\{1,\Im(p)\}.
\]
and rotates \(F^\perp\) by angle
\(\frac{2\pi}{5}.\)
By Lemma~\ref{lem:600cell-decagon}, the segment
\(
E^{(600)}_0=[1,p]
\)
is an edge of \(P_{600}\), and it is contained in \(F\). Therefore, \(\Ad_p\) fixes \(E^{(600)}_0\) pointwise and restricts on \(W^\perp_{E^{(600)}_0}\) to a rotation by angle \(2\pi/5\), as illustrated in Fig.~\ref{fig:600cell}. Since \(P_{600}\) is regular, \(G_{600}\) acts transitively on its edges. Hence for any edge \(E^{(600)}\) there exists \(h\in G_{600}\) such that
\(
h\!\left(E^{(600)}_0\right)=E^{(600)}.
\)
Then
\(
h\Ad_p h^{-1}
\)
fixes \(E^{(600)}\) pointwise and acts on \(W_{E^{(600)}}^\perp\) by rotation through
\(2\pi/5\). Its powers
\(
(h\Ad_p h^{-1})^k, k=1,2,3,4,
\)
act by rotations through angles \(2\pi k/5\).
\end{proof}

\begin{corollary}\label{cor:600cell-order5-edgecount}
The number of order \(5\) elements in \(G_{600}\) preserving an edge of \(P_{600}\) is
\(
288.
\)
\end{corollary}
\begin{proof}
The \(600\)--cell has \(120\) vertices, and by Proposition~\ref{prop:even} each vertex has valency \(12\). Hence the number of edges is
\(
720.
\)
By Proposition~\ref{prop:16.1}, the orientation--preserving pointwise stabilizer of any edge is cyclic of order \(5\). Therefore each edge is preserved by exactly four nontrivial elements of order \(5\). By Proposition~\ref{prop:16.1}, every fixed plane of such an element is a conjugate of \(F\). By Lemma~\ref{lem:600cell-decagon}, each such plane contains exactly \(10\) edges of \(P_{600}\). Each edge spans a \(2\)--plane through the origin, so each edge lies in exactly one such fixed plane. Thus \(720\) edges determine \(72\) such planes, and each such plane determines four nontrivial order \(5\) rotations. Therefore there are \(72\cdot 4 = 288\) such elements.
\end{proof}

\subsection{The lift group \(\widetilde G_{600}\) and the associated $\mathbb{Z}/2$--extension.}
\label{subsec:600cell-lifts}

By Proposition~\ref{prop:even}, each vertex of \(\Gamma_{600}\) has valency \(12\), so the assignment
\(
\mu_{600}([m_e])=-1
\)
on edge meridians is well-defined. Let
\(
p_{600}:\widetilde M_{600}\to M_{600}
\)
be the associated connected double cover corresponding to
\(
H_{600}:=\ker(\mu_{600}),
\)
let \(\tau_{600}\) denote its nontrivial deck involution, and let
\(
\mathcal{I}_{600}\to M_{600}
\)
be the associated real line bundle. Define
\[
\widetilde G_{600}:=
\left\{
\widetilde g:\widetilde M_{600}\to \widetilde M_{600}
\ \middle|\
\exists\, g\in G_{600}\text{ such that }
p_{600}\circ \widetilde g=g\circ p_{600}
\right\}.
\]

\begin{proposition}[Splitting for the $600$--cell group]
\label{prop:16.2}
The \(G_{600}\)--action on \(M_{600}\) lifts to a coherent choice of
fiberwise linear bundle automorphisms
\(
\widehat g:\mathcal{I}_{600}\to \mathcal{I}_{600}
\)
covering \(g\in G_{600}\), such that
\(
\widehat{gh}=\widehat g\circ \widehat h
\qquad\text{for all }g,h\in G_{600}.
\)
\end{proposition}
\begin{proof}
Since every \(g\in G_{600}\) preserves \(\Gamma_{600}\) and permutes its edges, Lemma~\ref{lem:12.3} gives
\(
\mu_{600}\circ g_*=\mu_{600}.
\)
Hence, by Lemma~\ref{lem:12.2}, every \(g\in G_{600}\) admits a lift to \(\widetilde M_{600}\), unique up to composition with \(\tau_{600}\). By Lemma~\ref{lem:12.5}, one therefore has a short exact sequence
\[
1\longrightarrow \langle \tau_{600}\rangle
\longrightarrow \widetilde G_{600}
\stackrel{\pi}{\longrightarrow}
G_{600}
\longrightarrow 1.
\]
It suffices to construct a splitting \(s:G_{600}\to \widetilde G_{600}\).
Choose generators \(\sigma_1,\sigma_2,\sigma_3\in G_{600}\) as in Equation~(\ref{eq:600}), with presentation
\[
G_{600}\cong [3,3,5]^+
=
\left\langle \sigma_1,\sigma_2,\sigma_3 \,\middle|\,
\sigma_1^{3}=\sigma_2^{3}=\sigma_3^{5}
=(\sigma_1\sigma_2)^{2}
=(\sigma_2\sigma_3)^{2}
=(\sigma_1\sigma_2\sigma_3)^{2}
=1
\right\rangle.
\]

\noindent \it Step 1: choose lifts of \(\sigma_1,\sigma_2,\sigma_3\) with the same odd orders upstairs. \normalfont
By Lemma~\ref{lem:12.2}, each \(\sigma_i\) admits a lift to \(\widetilde M_{600}\). Choose arbitrary lifts
\(
\widetilde\sigma_1,\widetilde\sigma_2,\widetilde\sigma_3.
\)
Since
\(
\sigma_1^{3}=\sigma_2^{3}=\sigma_3^{5}=1,
\)
the elements
\(
\widetilde\sigma_1^{\,3},
\widetilde\sigma_2^{\,3},
\widetilde\sigma_3^{\,5}
\)
are deck transformations, hence lie in \(\{\id,\tau_{600}\}\). If
\(\widetilde\sigma_1^{\,3}=\tau_{600}\), replace \(\widetilde\sigma_1\) by
\(\tau_{600}\widetilde\sigma_1\). Since every lift commutes with \(\tau_{600}\),
\(
(\tau_{600}\widetilde\sigma_1)^3
=\id.
\)
The same odd-order adjustment applies to \(\widetilde\sigma_2\) and \(\widetilde\sigma_3\). Thus we may assume
\(
\widetilde\sigma_1^{\,3}
=
\widetilde\sigma_2^{\,3}
=
\widetilde\sigma_3^{\,5}
=
\id.
\)

\smallskip

\noindent \it Step 2: verify the involution relations upstairs. \normalfont Set
\(
\widetilde\sigma_1\widetilde\sigma_2,
\widetilde\sigma_2\widetilde\sigma_3,\) and \(
\widetilde\sigma_1\widetilde\sigma_2\widetilde\sigma_3
\) as lifts of \(\sigma_1\sigma_2\), \(\sigma_2\sigma_3\), and
\(\sigma_1\sigma_2\sigma_3\). Since
\(
(\sigma_1\sigma_2)^2=(\sigma_2\sigma_3)^2=(\sigma_1\sigma_2\sigma_3)^2=1
\)
in \(G_{600}\), each of
\(
(\widetilde\sigma_1\widetilde\sigma_2)^{\,2},
(\widetilde\sigma_2\widetilde\sigma_3)^{\,2},\) and \(
(\widetilde\sigma_1\widetilde\sigma_2\widetilde\sigma_3)^{\,2}
\)
is a deck transformation.
We claim that each of \(\sigma_1 \sigma_2,\sigma_2 \sigma_3\), and \(\sigma_1 \sigma_2 \sigma_3\) has a fixed point in \(M_{600}\).
First, none of them is equal to \(-\id\). Indeed, if \(\sigma_1 \sigma_2=-\id\), then
\(
\sigma_1=-\sigma_2^{-1},
\)
so
\(
\sigma_1^3=(-\sigma_2^{-1})^3=-\id,
\)
contrary to \(\sigma_1^3=\id\). If \(\sigma_2 \sigma_3=-\id\), then
\(
\sigma_2=-\sigma_3^{-1},
\)
so
\(
\sigma_2^{15}=(-\sigma_3^{-1})^{15}=-\id,
\)
contrary to \(\sigma_2^{15}=\id\). If \(\sigma_1 \sigma_2 \sigma_3=-\id\), then
\(
\sigma_1\sigma_2=-\sigma_3^{-1},
\)
but \(\sigma_1\sigma_2\) has order dividing \(2\), whereas \(-\sigma_3^{-1}\) has order \(10\),
a contradiction.

Now let \(g\in\{\sigma_1 \sigma_2,\sigma_2 \sigma_3,\sigma_1 \sigma_2 \sigma_3\}\). If \(g=\id\), then trivially
\(
\Fix(g)\cap M_{600}=M_{600}\neq\varnothing.
\)
Assume \(g\neq \id\). Then \(g\) is an involution in \(SO(4)\) distinct from \(\pm\id\), following from the previous paragraph.
Hence its \(+1\)- and \(-1\)-eigenspaces are both \(2\)-dimensional, and therefore
\(
\Fix(g|_{S^3})
\)
is a great circle.
We claim that \(\Fix(g|_{S^3})\not\subset \Gamma_{600}\). If \(\Fix(g|_{S^3})\subset \Gamma_{600}\), then since \(\Gamma_{600}\) is a finite union of geodesic edges, the circle \(\Fix(g|_{S^3})\) contains a non-vertex point of \(\Gamma_{600}\), hence the radial projection of some edge \(E\) of \(P_{600}\). Because every point of \(\Fix(g|_{S^3})\) is fixed by \(g\), the edge \(E\) is fixed pointwise by \(g\). But this is impossible by Proposition~\ref{prop:16.1}, since the orientation-preserving pointwise stabilizer of any edge is cyclic of order \(5\). Therefore,
\(
\Fix(g|_{S^3})\cap M_{600}\neq\emptyset.
\)

Choose
\(
x_{12}\in \Fix(\sigma_1 \sigma_2)\cap M_{600},
x_{23}\in \Fix(\sigma_2 \sigma_3)\cap M_{600},\) and \(
x_{123}\in \Fix(\sigma_1 \sigma_2 \sigma_3)\cap M_{600},
\)
and lifts
\(
\widetilde x_{12}\in p_{600}^{-1}(x_{12}),
\widetilde x_{23}\in p_{600}^{-1}(x_{23}),\) and \(
\widetilde x_{123}\in p_{600}^{-1}(x_{123}).
\)
Since \(\widetilde\sigma_1\widetilde\sigma_2\) covers \(\sigma_1 \sigma_2\) and \(\sigma_1 \sigma_2(x_{12})=x_{12}\), the point
\(\widetilde\sigma_1\widetilde\sigma_2(\widetilde x_{12})\) lies in the two-point fiber
\(
p_{600}^{-1}(x_{12})=\{\widetilde x_{12},\tau_{600}(\widetilde x_{12})\}.
\)

Hence
\(
(\widetilde\sigma_1\widetilde\sigma_2)^{\,2}(\widetilde x_{12})=\widetilde x_{12}.
\)
Since \((\widetilde\sigma_1\widetilde\sigma_2)^{\,2}\) is a deck transformation and the nontrivial deck involution \(\tau_{600}\) acts freely, it follows that
\(
(\widetilde\sigma_1\widetilde\sigma_2)^2=\id.
\)
Similarly,
\(
(\widetilde\sigma_2\widetilde\sigma_3)^{\,2}(\widetilde x_{23})=\widetilde x_{23},
(\widetilde\sigma_1\widetilde\sigma_2\widetilde\sigma_3)^{\,2}(\widetilde x_{123})=\widetilde x_{123}.
\)
Since both squares are deck transformations, the same argument gives
\(
(\widetilde\sigma_2\widetilde\sigma_3)^2=
(\widetilde\sigma_1\widetilde\sigma_2\widetilde\sigma_3)^2=
\id.
\)

Thus, the chosen lifts satisfy all defining relations of Equation~\eqref{eq:600}. Therefore the map
\[
\sigma_1\mapsto \widetilde\sigma_1,\qquad
\sigma_2\mapsto \widetilde\sigma_2,\qquad
\sigma_3\mapsto \widetilde\sigma_3
\]
extends to a homomorphism
\(
s:G_{600}\to \widetilde G_{600}
\)
with
\(
\pi\circ s=\id_{G_{600}}.
\)

Hence, by Lemma~\ref{lem:12.5}, the extension splits. For each \(g\in G_{600}\), let
\(\widehat g:\mathcal{I}_{600}\to \mathcal{I}_{600}\) be the fiberwise linear bundle automorphism
induced by the lift \(s(g)\) via Lemma~\ref{lem:12.4}. Since \(s\) is a homomorphism,
these bundle maps satisfy
\(\widehat{gh}=\widehat g\circ \widehat h\) for all \(g,h\in G_{600}\).
This proves the proposition.
\end{proof}

\bibliographystyle{plain}
\bibliography{bib}
\end{document}